\documentclass[12pt,reqno]{amsart}
\usepackage{amsthm,amsfonts,amssymb,euscript,fancyhdr}

\newcommand{\bl}{\begin{lemma}}
\newcommand{\el}{\end{lemma}}
\def\beaa{\begin{eqnarray*}}
\def\eeaa{\end{eqnarray*}}
\def\ba{\begin{array}}
\def\ea{\end{array}}
\def\be#1{\begin{equation} \label{#1}}
\def \eeq{\end{equation}}

\def\a{{\alpha}}

\def\b{{\beta}}
\def\be{{\beta}}
\def\ga{\gamma}
\def\Ga{\Gamma}
\def\de{\delta}

\def\ep{\epsilon}

\def\si{\sigma}
\def\Si{\Sigma}

\def\Om{\Omega}
\def\th{\theta}

\def\nab{\nabla}
\def\varep{\varepsilon}

\def\pr{{\partial}}

\def\rh{{\rho}}
\def\et{{\eta}}

\def\BB{{\mathcal B}}

\def\NN{{\mathcal N}}

\def\FF{{\mathcal F}}

\def\HH{{\mathcal H}}

\def\GG{{\mathcal G}}

\def\SS{{\mathcal S}}
\def\NN{{\mathcal N}}

\def\Lie{{\mathcal L}}

\def\Lie{{\mathcal L}}

\def\F{{\mathcal F}}

\def\RRR{{\mathbb R}}

\def\f12{{\frac 1 2}}

\def\dual{{\,\,^*}}
\DeclareMathOperator{\Div}{\mathrm{div}}
\DeclareMathOperator*{\Curl}{\mathrm{curl}}
\def\half{\frac{1}{2}}

\newcommand{\eh}{\hat \eta}

\newcommand{\rhon}{\left(\rho_N \right)}
\newcommand{\rhod}{\rho \mkern-7mu /\ \mkern-7mu}
\newcommand{\sigmaNd}{\mkern-6mu \dual\sigma_N \mkern-9mu /\ \mkern-7mu}
\newcommand{\sigmaNdd}{\mkern-5mu \dual\widehat{(\sigmad)} }

\newcommand{\sigmaN}{ \sigma_N\mkern-11mu/\ \mkern-7mu}
\newcommand{\sigmad}{\sigma\mkern-9mu/\ \mkern-7mu}

\newcommand{\RRRic}{\mathrm{Ric}}

\newcommand{\nr}{\nonumber}
\newcommand{\isinf}{\int\limits_1^\infty}
\newcommand{\isr}{\int\limits_1^r}
\newcommand{\isphere}{\int\limits_{S_r}}

\newcommand{\sumzero}{\sum\limits_{l\geq0} \sum\limits_{m=-l}^l }

\newcommand{\sumone}{\sum\limits_{l\geq1} \sum\limits_{m=-l}^l }

\newcommand{\sumtwo}{\sum\limits_{l\geq2} \sum\limits_{m=-l}^l }

\newcommand{\ol}{\overline} 
\newcommand{\Divd}{\Div \mkern-17mu /\ }
\newcommand{\Curld}{\Curl \mkern-17mu /\ }
\newcommand{\Nd}{\nabla \mkern-13mu /\ }
\newcommand{\Ndn}{{\nabla \mkern-13mu /\ \mkern-4mu}_N \mkern-0mu}
\newcommand{\Ld}{\triangle \mkern-12mu /\ }
\newcommand{\trd}{\mathrm{tr} \mkern-12mu /\ }

\newcommand{\iin}{\in \mkern-16mu /\ \mkern-5mu}

\newcommand{\LIE}{\mathcal{L} \mkern-11mu /\  \mkern-3mu }

\newcommand{\cDd}{\mathcal{D} \mkern-11mu /\  \mkern-4mu}

\newcommand{\HHl}{{\HH^w_{-1/2}}}

\newcommand{\HHm}{{\HH^{w-1}_{-3/2}}}

\newcommand{\HHn}{{\HH^{w-2}_{-5/2}}}

\newcommand{\HHlo}{{\ol{\HH}^w_{-1/2}}}

\newcommand{\HHmo}{{\ol{\HH}^{w-1}_{-3/2}}}

\newcommand{\HHno}{{\ol{\HH}^{w-2}_{-5/2}}}

\newcommand{\RRRwo}{ {\RRR^3 \setminus \ol{B_1}}}

\newcommand{\srp}{\int\limits_{S_{r'}}}

\def\sr{\int\limits_{S_r}}

\def\ni{\noindent}

\def\tr{\mathrm{tr}}

\def\th{\theta}

\def\f{\widetilde{f}}

\def\sr{\int\limits_{S_r}}

\newcommand{\gac}{\overset{\circ}{\ga}}
\newcommand{\Lied}{\mathcal{L} \mkern-9mu/\ \mkern-7mu}

\newtheorem{theorem}{Theorem}[section]
\newtheorem{lemma}[theorem]{Lemma}
\newtheorem{proposition}[theorem]{Proposition}
\newtheorem{corollary}[theorem]{Corollary}
\newtheorem{definition}[theorem]{Definition}
\newtheorem{remark}[theorem]{Remark}

\newtheorem{claim}[theorem]{Claim}

\numberwithin{equation}{section}

\begin{document}

\title[An extension procedure]{An extension procedure   \\ for the constraint equations}
\author{Stefan Czimek}
\address{Laboratoire Jacques-Louis Lions, Universit\'e Pierre et Marie Curie (Paris 6)}

\begin{abstract} Let $( \bar g, \bar k)$ be a solution to the maximal constraint equations of general relativity on the unit ball $B_1$ of $\RRR^3$. We prove that if $(\bar g,\bar k)$ is sufficiently close to the initial data for Minkowski space, then there exists an asymptotically flat solution $(g,k)$ on $\RRR^3$ that extends $(\bar g, \bar k)$.
Moreover, $(g,k)$ is bounded by $(\bar g, \bar k)$ and has the same regularity. \newline
Our proof uses a new method of solving the prescribed divergence equation for a tracefree symmetric $2$-tensor, and a geometric variant of the conformal method to solve the prescribed scalar curvature equation for a metric. Both methods are based on the implicit function theorem and an expansion of tensors based on spherical harmonics. They are combined to define an iterative scheme that is shown to converge to a global solution $(g,k)$ of the maximal constraint equations which extends $(\bar g,\bar k)$. 
\end{abstract}

\maketitle
\tableofcontents

\section{Introduction} \label{sec:Introduction}

\subsection{The Cauchy problem and the maximal constraint equations}

The Einstein vacuum equations on a Lorentzian $4$-manifold $(\mathcal{M},{\bf g})$ are given by
\begin{align*}
\RRRic( { \bf g} ) = 0,
\end{align*}
where $\RRRic$ denotes the Ricci tensor of ${ \bf g}$. Initial data for the Cauchy problem is given by a triple $(\Si,g,k)$, where $(\Si,g)$ is a complete Riemannian $3$-manifold and $k$ a symmetric $2$-tensor on $\Si$ satisfying the \emph{constraint equations} on $\Si$
\begin{align} \begin{aligned} 
R(g) &= \vert k \vert_g^2 - (\tr_g k )^2, \\
\Div_g k &= d(\tr_g k). \label{eq:CONST1}
\end{aligned} \end{align}
Here $R(g)$ denotes the scalar curvature of $g$, $d$ is the exterior derivative and
\begin{align*}
\vert k \vert_g^2:= g^{ij} g^{lm} k_{il} k_{jm}, \, \tr_g k: = g^{ij}k_{ij}, \, (\Div_g k )_l := g^{ij} \nab_i k_{jl},
\end{align*}
where $\nab$ denotes the covariant derivative on $(\Si,g)$ and we use, as in the rest of this paper, the Einstein summation convention.\\

Let $(\mathcal{M},{\bf g})$ be the solution of the Einstein vacuum equations corresponding to initial data $(\Si,g,k)$. Then $\Si \subset (\mathcal{M}, {\bf g})$ is a space-like Cauchy hypersurface with induced metric $g$ and second fundamental form $k$. See for example \cite{Wald} for details. \\

The trivial solution to the Einstein vacuum equations is Minkowski spacetime which follows in particular from the trivial initial data
\begin{align*}
(\Si,g,k)=(\RRR^3, e,0),
\end{align*}
where $e$ denotes the Euclidean metric.\\

In this work, we consider initial data that satisfies two further properties.
\begin{itemize}
\item The initial data is \emph{asymptotically flat}, which means
\begin{align*}
g \to e, \, k \to 0
\end{align*}
as $\vert x \vert \to \infty$ on $\Si$. For a more precise definition, see Definition \ref{def:asmypt}. Such initial data corresponds to the description of isolated gravitational systems, see for example \cite{Wald}.
\item We assume that $\Si$ is \emph{maximal}, that is,
\begin{align*}
\tr_g k = 0.
\end{align*}
This assumption is sufficiently general for our purposes, see for example \cite{BartnikExistence}.
\end{itemize}

By the second assumption, the constraint equations \eqref{eq:CONST1} reduce to the \emph{maximal constraint equations}
\begin{align} \begin{aligned}
R(g)&=\vert k \vert^2, \\
\Div_g k &=0, \\
\tr_g k &=0.
\end{aligned}
\label{eq:EinsteinMaximalConstraints}
\end{align}

\subsection{The extension problem and the main theorem}

The maximal constraint equations are an under-determined geometric non-linear elliptic system of partial differential equations. In this paper, we are interested in the following problem.\\

{\bf Extension problem.} \emph{Given initial data $(\bar g, \bar k)$ on the unit ball $B_1 \subset \RRR^3$, does there exist a regular asymptotically flat initial data set $(g,k)$ on $\RRR^3$ that isometrically contains $(\bar g, \bar k)$ and is bounded by it?} \\

This problem has received considerable attention in the literature. It appears for example
\begin{itemize}
\item when analysing the space of solutions to the maximal constraint equations, see for example \cite{Bartnik} \cite{SmithWeinstein2} \cite{ShiTam} \cite{Isenberg} \cite{SmithWeinstein1}  \cite{BartnikSharples} , 
\item when considering the rigidity of the equations, as in the celebrated gluing construction \cite{Corvino} \cite{CorvinoSchoen}, see also \cite{ChruscielDelay} \cite{Chrusciel2} \cite{Delay},
\item in the context of Bartnik's conjecture \cite{BartnikConj} \cite{Bartnik3mass}, see for example \cite{Huisken} \cite{Miao} \cite{ShiTam2} \cite{AndersonKhuri} \cite{Anderson}
\item in the proof of the bounded $L^2$ curvature theorem \cite{KRS}, where it is used to reduce the local existence for the Cauchy problem of general relativity to the small data case, see for example Section 2.3 in that paper.
\end{itemize}

Our main motivation to consider the extension problem is to prove a \emph{localised version} of the bounded $L^2$ curvature theorem of \cite{KRS}, see \cite{Cz2}.\\

In this paper, we resolve the extension problem for small data. The next theorem is a rough version of our main result, see Theorem \ref{MainTheorem12} for a precise formulation.
\begin{theorem}[Main theorem, rough version] \label{maintheoremintro0}
Let $(\bar g,\bar k)$ be a solution on the unit ball $B_1 \subset \RRR^3$ of the maximal constraint equations
\begin{align*}
R(\bar g)&=\vert \bar k \vert_{\bar g}^2, \\
\Div_{\bar g} \bar k &=0, \\
\tr_{\bar g} \bar k &=0.
\end{align*}
If $(\bar g, \bar k)$ is sufficiently close to $(e,0)$ in a suitable topology, where $e$ denotes the Euclidean metric, then there exists asymptotically flat $(g, k)$ of the same regularity as $(\bar g, \bar k)$ such that
\begin{align*}
(g, k) \vert_{B_1} = (\bar g, \bar k),
\end{align*}
and solving the maximal constraint equations on $\RRR^3$,
\begin{align*}
R(g)&=\vert k \vert_{g}^2, \\
\Div_{g} k &=0, \\
\tr_{g} k &=0.
\end{align*}
Moreover, the norm of $(g,k)-(e,0)$ is bounded by the norm of $(\bar g, \bar k)-(e,0)$. 
\end{theorem}

\textit{Comments on the result.}
\begin{enumerate}
\item The novelty of our result lies in the following facts.
\begin{itemize} \item Compared to the gluing constructions \cite{Corvino} \cite{CorvinoSchoen} and \cite{Delay}, it does not need a gluing region. Indeed, these literature results consider the problem of gluing together two given solutions within an annulus, and do not study the problem of extending a given solution $(g,k)$ as solution to the constraints onto a larger domain. This feature is crucial for localising the bounded $L^2$ curvature theorem \cite{KRS}, see the forthcoming \cite{Cz2}. 
\item The extension results in \cite{SmithWeinstein2} \cite{ShiTam} \cite{SmithWeinstein1} \cite{BartnikSharples} lose regularity across the boundary of the domain by using a parabolic equation to solve the prescribed scalar curvature equation. Our result, on the other hand, uses a geometric perturbation argument at the Euclidean metric that preserves regularity.
\end{itemize}
\item The closeness of $(\bar g, \bar k)$ to $(e,0)$ is measured in the topology corresponding to the space
$$(\bar g, \bar k) \in \HH^2(B_1) \times \HH^{1}(B_1),$$
where $\HH^2(B_1)$ denotes a Sobolev space of tensors over $B_1$ corresponding to $2$ derivatives in $L^2$. We note that this low regularity setting is in accordance with the bounded $L^2$ curvature theorem \cite{KRS}. In view of the scaling of \eqref{eq:EinsteinMaximalConstraints}, we expect Theorem \ref{maintheoremintro0} to hold also under the assumption of closeness of $(\bar g, \bar k)$ to $(e,0)$ in the topology corresponding to $\HH^w(B_1) \times \HH^{w}(B_1)$ for reals $w>3/2$. 
\item Theorem \ref{maintheoremintro0} completes the proof of the reduction step to small data in the proof of the bounded $L^2$ curvature theorem \cite{KRS}, see Section 2.3 in that paper.
\item The methods used in the proof of Theorem \ref{maintheoremintro0} could be relevant to other problems such as, for example, solving the divergence equation in context of the Maxwell-Klein-Gordon and Euler equations, see for example \cite{Oh2} \cite{Oh1}.
\end{enumerate}

\subsection{Strategy of the proof of the main theorem}

In this section we sketch the proof of Theorem \ref{maintheoremintro0}. The idea is to set up an iterative scheme consisting of pairs $\Big( (g_i,k_i)\Big)_{i\geq1}$ that extend $(\bar g,\bar k)$ from $B_1$ to $\RRR^3$. In general, the $(g_i,k_i)$ do not solve the maximal constraints \eqref{eq:EinsteinMaximalConstraints} on $\RRR^3$. However, by a fixed point argument, we show that the sequence converges to a solution $(g,k)$ as $i \to \infty$. \\

More precisely, let $(\bar g,\bar k)$ be small given initial data on $B_1$, and assume we have already obtained $(g_i,k_i)$ for some $i\geq1$. We construct the next pair $(g_{i+1},k_{i+1})$ by the following two steps.
\begin{itemize}
\item {\bf Step A.} Given $(g_i,k_i)$ on $\RRR^3$, construct $g_{i+1}$ on $\RRR^3$ such that
\begin{align*}
g_{i+1} \vert_{B_1} &= \bar g,\\
R(g_{i+1}) &= \vert k_i \vert_{g_i}^2.
\end{align*}
\item {\bf Step B.} Given $g_{i+1}$ on $\RRR^3$, construct $k_{i+1}$ on $\RRR^3$ such that
\begin{align*}
k_{i+1} \vert_{B_1} =\bar k,\\
\Div_{g_{i+1}} k_{i+1} =0, \\
\tr_{g_{i+1}} k_{i+1}=0.
\end{align*}
\end{itemize}

Step A and B rely on Theorems \ref{thm:extmetrough1} and \ref{thm:extrough1}, respectively, to be introduced now. 

\begin{theorem}[Extension of divergence-free tracefree symmetric $2$-tensors, rough version] \label{thm:extrough1}
Let $g$ be an asymptotically flat Riemannian metric on $\RRR^3$ and $\bar k$ a symmetric $2$-tensor on $B_1$ solving 
\begin{align} \begin{aligned}
\Div_{g}\bar  k = 0,\\
\tr_{g} \bar  k = 0.
\end{aligned} \label{eq:mars30}
\end{align}
If $g$ and $\bar k$ are sufficiently close to $e$ and $0$, respectively, in a suitable topology, then there exists an asymptotically flat symmetric $2$-tensor $k$ on $\RRR^3$ that extends $\bar k$, that is, 
$$k \vert_{B_1} = \bar  k$$
and solves on $\RRR^3$
\begin{align} \begin{aligned}
\Div_{g} k = 0,\\
\tr_{g}  k = 0. \end{aligned} \label{eq:mars30second}
\end{align}
Moreover, $k$ is bounded by $\bar k$.
\end{theorem}

\begin{theorem}[Metric extension theorem, rough version] \label{thm:extmetrough1}
Let $\bar g$ be a Riemannian metric on $B_1$ and $R$ a scalar function on $\RRR^3$ such that 
$$R \vert_{B_1}=R(\bar g),$$
where $R(\bar g)$ denotes the scalar curvature of $\bar g$. If $\bar g$ and $R$ are sufficiently close to $e$ and $0$, respectively, then there exists an asymptotically flat Riemannian metric $g$ on $\RRR^3$ such that
$$g \vert_{B_1} =\bar  g,$$ 
and such that its scalar curvature on $\RRR^3$ is given by
\begin{align*}
R(g) = R.
\end{align*}
Moreover, $g$ is bounded in terms of $\bar g$ and $R$.
\end{theorem}

Precise versions of the above are stated in Theorems \ref{thm:kextension1} and \ref{thm:MainExtensionScalarTheorem}, respectively. Both Theorems \ref{thm:extrough1} and \ref{thm:extmetrough1} are proved by the Implicit Function Theorem and showing the surjectivity of a linearisation of the corresponding operators at the Euclidean metric $e$. Concerning Theorem \ref{thm:extrough1}, we show in Section \ref{sec:EuclideanSurjectivity} that the operator
\begin{align*}
k \mapsto \rh := \Div_e \left( \widehat{k}^e \right)
\end{align*}
is surjective. Here, for any symmetric $2$-tensor $V$, we denote its tracefree part with respect to the Euclidean metric $e$ by
\begin{align*}
\widehat{V}^e := V- \frac{1}{3} \tr_e(V) e.
\end{align*}
Concerning Theorem \ref{thm:extmetrough1}, we show in Section \ref{sec:SurjectivityR} that the linearisation of the scalar curvature with respect to a suitable geometric variation is surjective.\\

The two proofs of surjectivity at the Euclidean metric use, among others, the following mathematical tools.
\begin{enumerate}
\item In Section \ref{ssec:tensordecomposition}, we decompose tensors with respect to the foliation of $\RRR^3$ by spheres
\begin{align*}
S_r = \{ \vert x \vert =r \}, r>0.
\end{align*}
\item In Section \ref{ssec:Hodgebasis}, relying on spherical harmonics, we define complete orthonormal bases of the spaces of $L^2(S_r)$-integrable functions, vectorfields and symmetric tracefree $2$-tensors on $S_r$. These bases have been studied before in physics literature, see for example \cite{HillSpherical} and \cite{SandbergSpherical}. We call these bases Hodge-Fourier bases. Projecting onto these bases allows us to split up the linearised operators into radial ODEs and elliptic systems on $S_r$ and $\RRR^3 \setminus \ol{B_1}$.
\item In order to force the regularity of the extension at the boundary of $B_1$, it is necessary to control the Dirichlet-to-Neumann maps associated to the elliptic systems on $\RRR^3 \setminus \ol{B_1}$. This is achieved in particular by exploiting the underdetermined character of the constraint equations.
\end{enumerate}
\vspace{3mm}
The rest of the paper is organised as follows. In Section \ref{sec:Notation}, we introduce the notation and weighted Sobolev spaces and bases of functions and tensors. In Section \ref{sec:Exactstatements} we state a precise version of Theorem \ref{maintheoremintro0}. In Section \ref{sec:DivergenceEquation}, we first reduce the proof of Theorem \ref{thm:extrough1} to the surjectivity at the Euclidean space which is then proved in Section \ref{sec:EuclideanSurjectivity}. Similarly, in Section \ref{sec:MetricPrescription}, we first reduce the proof of Theorem \ref{thm:extmetrough1} to the surjectivity at the Euclidean space which is then proved in Section \ref{sec:SurjectivityR}. In Section \ref{sec:Iteration}, we set up the iterative scheme and prove Theorem \ref{maintheoremintro0}. In Appendix \ref{sec:ProofCompleteness}, we prove the completeness of the bases of tensors defined in Section \ref{sec:Notation}. Two lemmas of Section \ref{sec:Notation} are proved in Appendix \ref{sec:appidentities}. In Appendix \ref{sec:WEIGHTEDellipticity} we derive elliptic estimates in weighted Sobolev spaces.

\subsection{Acknowledgements} This work forms part of my Ph.D. thesis and I am grateful to my Ph.D. advisor J\'er\'emie Szeftel for his kind supervision and careful guidance. Furthermore, I am grateful to the RDM-IdF for financial support.

\section{Preliminaries} \label{sec:Notation}

\subsection{Basic notation} \label{basicsection}
In this work, lowercase Latin indices range over $a,b,c,d,i,j=1,2,3$, uppercase Latin indices over $A,B,C,D=1,2$ and $n\in \mathbb{N}$. The index pairs $(lm)$ take as values integers $l\geq0, m \in \{ -l,\dots,l \}$. We apply the Einstein summation convention. The notation $A\lesssim B$ means $A\leq c B$ where $c>0$ is a numerical constant that does not depend on $A$ or $B$. In estimates, we generally use $C_w>0$ as constant which only depends on $w$. \\

An open subset $\Om \subset \RRR^3$ is called a \emph{domain} if it is connected and its boundary $\pr \Om := \ol{\Om} \setminus \overset{\circ}{\Om}$ is smooth. Let $\chi:\RRR \to \left[ 0,1 \right]$ be a fixed smooth transition function such that
\begin{align}
\chi(x) &= \begin{cases}0 \,\,\, \text{for } x \leq 1/10,\\
1 \,\,\, \text{for } x \geq1.
\end{cases} 
\label{eq:transfct}
\end{align}

We work in a fixed Cartesian coordinate system $(x^1,x^2,x^3)$ of $\RRR^3$. Consequently, given a $n$-tensor $T$, we can equivalently denote it by its coordinate components $T_{i_1 \dots i_n}$. \\

Let $e$ denote the Euclidean metric on $\RRR^3$ with components
\begin{align*}
e_{ij}= \begin{cases} 1 & \text{if } i=j, \\
0& \text{if } i\neq j.
\end{cases}
\end{align*}

Let $g$ be a Riemannian metric and $V$ a symmetric $2$-tensor. Let the divergence, the symmetrized curl, the trace and the tracefree part of $V$ with respect to $g$ be
\begin{align*}
(\Div_g V )_j &:= \nabla^i V_{ij}, \\
(\mathrm{curl}_g V)_{ij} &:= \half (\in_i^{~a b} \nabla_a V_{bj} + \in_j^{~a b} \nab_a V_{bi}), \\
\tr_g V &:= g^{ab} V_{ab}, \\
\widehat{V}^g&:= V- \frac{1}{3} \tr_g (V) g,
\end{align*}
where $\nab$ denotes the covariant derivative and $\in$ the volume form of $g$. \\

\subsection{The radial foliation of $\RRR^3$ by spheres $S_r$ and tensor decomposition} \label{subsec:DifferentialGeometry} \label{ssec:tensordecomposition}

Let $$(r,\th^1,\th^2) \in [0,\infty) \times [0, \pi) \times [0,2\pi)$$ denote the standard polar coordinates on $\RRR^3$. By definition they are related to the Cartesian coordinates $(x^1,x^2,x^3)$ by
\begin{align*} \begin{aligned}
x^1&= r \sin \th^1 \cos \th^2, \\
x^2&= r \sin \th^1 \sin \th^2, \\
x^3&= r \cos \th^1.
\end{aligned} 
\end{align*}
The coordinate spheres and balls of radius $r_0 >0$ centered at the origin are respectively defined as
\begin{align*}
S_{r_0} &:= \{ x\in \RRR^3 :  r(x) = r_0 \}, \nr\\
B_{r_0} &:= \{ x\in \RRR^3 : r(x) < r_0 \}. \nr
\end{align*}

In standard polar coordinates, the Euclidean metric is given by
\begin{align*}
e = dr^2 + r^2 ( (d\th^1)^2 + \sin^2 \th^1 (d\th^2)^2 ).
\end{align*}
Let the induced metric on $S_r \subset (\RRR^3, e)$ be denoted by
$$\gac := r^2 ((d\th^1)^2 + \sin^2 \th^1 (d\th^2)^2 ).$$
When integrating over $(S_r, \gac)$ we do not write out the standard volume element.\\

The \emph{standard polar frame} on $\RRR^3 \setminus \{ x^1=x^2=0\}$ is defined as
\begin{align}
\left\{ \pr_r, e_1:=\frac{1}{r} \pr_{\th^1}, e_2 := \frac{1}{r\sin \th^1} \pr_{\th^2} \right\}, \label{polarnormalframe}
\end{align}
where $\pr_r, \pr_{\th^1}, \pr_{\th^2}$ are the coordinate vectorfields in the coordinate system $(r,\th^1,\th^2)$, respectively. \\ 

Every Riemannian metric $g$ on $\RRR^3 \setminus \{0\}$ can be uniquely written as
\begin{align}
g = a^2 dr^2 + \ga_{AB} \left( \be^A dr + d\th^A \right) \left( \be^Bdr + d\th^B \right), \label{eq:generalmetricform}
\end{align}
where
\begin{itemize}
\item $a(x)>0$ for all $x\in \RRR^3 \setminus \{ 0\}$ is the positive lapse function,
\item $\ga$ is the Riemannian metric induced by $g$ on $S_r$, $r>0$,
\item $\be$ is the $S_r$-tangent shift vector.
\end{itemize}
The $a, \ga, \be$ are called the \emph{polar components} of $g$.\\

The following lemma is proved by direct calculation.
\begin{lemma}\label{lem:EuclideanProperties}
Let $g$ be a Riemannian metric given on $\RRR^3 \setminus \{ 0 \}$,
\begin{align*}
g = a^2 dr^2 + \ga_{AB} \left( \be^A dr + d\th^A \right) \left( \be^Bdr + d\th^B \right).
\end{align*}
Then the following holds for any $r>0$.
\begin{enumerate}
\item The outward-pointing\footnote{That is, pointing into the unbounded connected component of $\RRR^3 \setminus S_r$.} unit normal $N$ to $S_r$ with respect to $g$ is given by
\begin{align*}
N =  \frac{1}{a} \partial_r - \frac{1}{a} \be.
\end{align*}
\item The second fundamental form\footnote{Here we use the sign convention that $\Theta(X,Y):= -g(X, \nab_{Y} N)$.} $\Theta$ of $S_r$ with respect to $g$ equals in any coordinates on $S_r$, $A,B=1,2$,
\begin{align}
\Theta_{AB} = - \frac{1}{2a} \partial_r \left( \ga_{AB} \right)+ \frac{1}{2a}( \Lied_\be \ga)_{AB}, \label{eq:explicit2feb1}
\end{align}
where $\Lied$ denotes the Lie derivative on $S_r$.
\end{enumerate}
\end{lemma}

\begin{remark} \label{remark:polarcoorde} The polar components of the Euclidean metric $e$ are
\begin{align*}
a&=1, \,\,\, \be = 0, \\
\ga_{AB}&= {\gac}_{AB}=\begin{cases} r^2 & \text{if } A=B=1, \\
r^2 \sin^2 \th^1 & \text{if } A=B=2, \\
0 &\text{if } A\neq B. \end{cases}
\end{align*}
Furthermore, $N = \pr_r$ and
\begin{align*}
\trd_{\gac} \Theta := {\gac}^{AB} \Theta_{AB} = -\frac{2}{r}, \,\,\, \vert \Theta \vert^2_{\gac} := \gac^{AC} \gac^{BD} \Theta_{AB} \Theta_{CD} =   \frac{2}{r^2}.
\end{align*}
\end{remark}

More generally, we now decompose vectorfields and symmetric $2$-tensors on $\RRRwo$ with respect to the foliation of $\RRR^3$ by spheres $S_r$. Given a vectorfield $X$, decompose it into
\begin{itemize}
\item the scalar function $X_N$, 
\item the $S_r$-tangent vectorfield ${X \mkern-14mu/\ \mkern-5mu }_A = X_A$,
\end{itemize}
where $A=1,2$ denote components in any frame on $S_r$. \\

Given a symmetric $2$-tensor $V$, decompose it into
\begin{itemize}
\item the scalar function $V_{NN}$, 
\item the $S_r$-tangent vectorfield $\left( {V \mkern-14mu/\ \mkern-5mu}_{N}\right)_A := V_{NA}$, 
\item the $S_r$-tangent $2$-tensor ${V \mkern-14mu/\ \mkern-5mu}_{AB} := V_{AB}$,
\end{itemize}
where $A,B=1,2$ denote components in any frame on $S_r$. 

\begin{definition} \label{def:ndn11}
Let $X$ be a $S_r$-tangent vectorfield and $V$ a $S_r$-tangent symmetric $2$-tensor on $\RRR^3 \setminus \{ 0 \}$. Define the $S_r$-tangential vectorfield $ \Ndn X$ and symmetric $2$-tensor $ \Ndn V$, respectively, by
\begin{align*}
\left( \Ndn X \right)_a &:= (\Pi_{T_{S_r}} )_a^{\,\, c} \nab_N X_c, \\ 
\left( \Ndn V \right)_{ab} &:= (\Pi_{T_{S_r}} )_{a}^{\,\, c}  (\Pi_{T_{S_r}})_b^{\,\, d} \nab_N V_{cd},
\end{align*}
where $a,b = 1,2,3$ and $$(\Pi_{T_{S_r}})_i^{\,\,j}:= e_i^{\,\,j} - N_i N^j$$ denotes the projection onto $TS_r$.
\end{definition}


\subsection{Function spaces} \label{subsec:FunctionSpaces}

\begin{definition}[Sobolev spaces]
Let $\Om \subset \RRR^3$ be a domain and $w \geq0$ integer. Let $H^w (\Om)$ denote the standard Sobolev space
\begin{align}
H^w(\Om) := \left\{ f \in L^2(\Om) : \sum\limits_{\vert \a \vert \leq w} \Vert \partial^\a f \Vert_{L^2(\Om)} < \infty \right\}. \nr
\end{align}
Here $\a= (\a_1, \a_2, \a_3)\in \mathbb{N}^{3}$ is a multi-index and 
\begin{align*}
\vert \a \vert &:= \a_1+ \a_2  + \a_3,\,\,  \partial^\a := \partial_{x^1}^{\a_1} \partial_{x^2}^{\a_2} \partial_{x^3}^{\a_3}.
\end{align*}
\end{definition}

\begin{definition}
Let $\Om \subset \RRR^3$ be a domain and $w \geq0$ an integer. Define $H^w_{loc}(\Om)$ as 
\begin{align*}
H^w_{loc}(\Om) := \bigcap_{\Om' \subset \subset \Om} H^w(\Om'),
\end{align*}
where $\Om' \subset \subset \Om$ denotes all domains $\Om'$ such that $\overline{\Om'}$ is compact and $\overline{\Om'} \subset \Om$.
\end{definition}

See for example \cite{Adams03} for properties of the above function spaces. Our analysis of the constraint equations is set in the following weighted Sobolev spaces.
\begin{definition}[Weighted Sobolev spaces] \label{def:WeightedSobolevSpaces1} Let $\Om \subset \RRR^3$ be a domain, $w\geq0$ an integer and $\de \in \RRR$. Let
\begin{align}
H^w_\de(\Om) := \left\{ f \in H^w_{loc}(\Om) : \sum\limits_{\vert \b \vert \leq w} \Vert (1+r)^{-\de-3/2+ \vert \b \vert} \pr^\b f \Vert_{L^2(\Om)} < + \infty \right\}. \nr
\end{align}
Furthermore, define $$H^w_\de:= H^w_\de(\RRR^3).$$ 
\end{definition}
For $w\geq0$ integer and $\de\in\RRR$, $H^w_\de(\Om)$ is a Hilbert space. The next lemma follows from Lemmas 2.1, 2.4 and 2.5 and Corollary 2.6 in \cite{Maxwell}, see also \cite{MaxwellDissertation}.
\begin{lemma} \label{SobolevEmbeddingsAndNonlinear} Let $\de, \de_1,\de_2 \in \RRR$, let $w,w_1,w_2 \geq0$ be integers and let $f$ be a scalar function on $\RRR^3$. The following holds.
\begin{itemize}
\item If $w \geq1$ and $f\in H^w_\de$, then $\partial f \in H^{w-1}_{\de-1}$.
\item If $0\leq w_1\leq w_2$ and $\de_1 \leq \de_2$, then $H^{w_1}_{\de_1} \subset H^{w_2}_{\de_2}$.
\item For $w \geq2$, the space $H^w_\de$ continuously embeds into
\begin{align}
\left\{ f \in L^\infty_{loc}(\RRR^3) : \sum\limits_{\vert \beta \vert \leq w-2}  \sup\limits_{x \in \RRR^3}  (1+r)^{-\de+\vert \beta \vert} \vert \partial^\beta f \vert < \infty \right\}. \nr
\end{align}
\item Let $F: \RRR \to \RRR$ be a smooth function such that $F(0)=0$. Let $u \in H^2_{\de}$ for some $\de<0$. Then there exists a constant $C=C(\Vert u \Vert_{\HH^2_{\de}}, F)>0$ such that
\begin{align*}
\Vert F(u) \Vert_{\HH^2_{\de}} \leq C \Vert u \Vert_{\HH^2_{\de}}.
\end{align*}
\end{itemize}
\end{lemma}

In sections \ref{sec:analysisonafsets} and \ref{subsec:geometry}, we use Gagliardo-Nirenberg-Moser estimates in weighted Sobolev spaces; see the next lemma. We refer to \cite{Taylor} for standard product estimates and \cite{BruhatChr} \cite{Bartnik} \cite{Maxwell} \cite{MaxwellDissertation} for their generalisation to weighted Sobolev spaces.
\begin{lemma}[Product estimates] \label{ProductEstimates} Let $w\geq2$ be an integer and $\de_1, \de_2 <0$ be two reals. Then the following holds.
\begin{enumerate}
\item For all scalar functions $u \in H^1_{\de_1}$ and $v \in H^{1}_{\de_2}$, we have
\begin{align*}
\Vert uv \Vert_{H^{0}_{\de_1+ \de_2}} \lesssim& \Vert u \Vert_{H^1_{\de_1}} \Vert v \Vert_{H^1_{\de_2}}.
\end{align*}
\item For all scalar functions $u \in H^w_{\de_1}$ and $v \in H^{w-1}_{\de_2}$, we have
\begin{align*}
\Vert uv \Vert_{H^{w-1}_{\de_1+ \de_2}} \lesssim& \Vert u \Vert_{H^w_{\de_1}} \Vert v \Vert_{H^1_{\de_2}} +  \Vert u \Vert_{H^2_{\de_1}} \Vert v \Vert_{H^{w-1}_{\de_2}} + C_w \Vert u \Vert_{H^2_{\de_1}} \Vert v \Vert_{H^1_{\de_2}}.
\end{align*}
\item For all scalar functions $u \in H^w_{\de_1}$ and $v \in H^{w}_{\de_2}$, we have
\begin{align*}
\Vert uv \Vert_{H^{w}_{\de_1+ \de_2}} \lesssim& \Vert u \Vert_{H^w_{\de_1}} \Vert v \Vert_{H^2_{\de_2}} +  \Vert u \Vert_{H^2_{\de_1}} \Vert v \Vert_{H^{w}_{\de_2}} + C_w \Vert u \Vert_{H^2_{\de_1}} \Vert v \Vert_{H^2_{\de_2}}.
\end{align*}
\end{enumerate}
\end{lemma}
\begin{corollary} \label{cor:algebraprop}
For $w\geq2, \de<0$, the space $H^w_{\de}$ forms an algebra.
\end{corollary}
We have the following corollary from Lemmas \ref{SobolevEmbeddingsAndNonlinear} and \ref{ProductEstimates}.
\begin{corollary} \label{expHigherReg}
Let $w\geq2$ be an integer and $\de<0$ a real. There exists an $\varep>0$ such that for all scalar functions $u \in H^w_{\de}$ with
\begin{align*}
\Vert u \Vert_{\HH^2_\de} \leq \varep,
\end{align*}
it holds that
\begin{align*}
\Vert e^{2u} - 1 \Vert_{\HH^w_{\de}} \lesssim  \Vert u \Vert_{H^w_{\de}} + C_w \Vert u \Vert_{H^2_{\de}}.
\end{align*}
\end{corollary}
\begin{proof} On the one hand, by Lemma \ref{SobolevEmbeddingsAndNonlinear},
\begin{align*}
\Vert e^{2u}-1 \Vert_{H^2_{\de}} \lesssim \Vert u \Vert_{H^2_{\de}}.
\end{align*}
On the other hand, for any $w\geq0$,
\begin{align*}
\pr^w (e^{2u}-1) = e^{2u} \Big( \pr^w u + \pr^{w-1} u \pr u + \dots + \pr^2 u (\pr u)^{w-2} + (\pr u)^w \Big).
\end{align*}
Therefore, by product estimates (see \cite{Taylor}), the result follows. \end{proof}

\begin{definition} \label{definitionolspace} Let $\Om \subset \RRR^3$ be a domain, $w\geq0$ an integer and $\de \in \RRR$. Define $\overline{H}^w_\de(\Om)$ to be the closure of $C^\infty_{c}(\Om)$ with respect to the norm $\Vert \cdot \Vert_{H^w_\de(\Om)}$. Further, define $$\ol{H}^w_\de:= \ol{H}^w_\de(\RRRwo).$$
\end{definition}


The following useful characterisation of $\ol{H}^w_\de$ is left to the reader, see for example Exercise 3 of Section 4.5 in \cite{Taylor}.
\begin{proposition} \label{prop:TrivialExtensionRegularity}
Let $w\geq2$ be an integer, $\de \in \RRR$. Let $u \in H^w_{\de}(\RRR^3 \setminus \overline{B_1})$. The following are equivalent.
\begin{enumerate}
\item The trivial extension of $u$ to $B_1$ is regular, that is $\overline{u} \in H^w_\de$, where
\begin{align*}
\overline{u} = \begin{cases} u(x) & \text{if } x \in \RRR^3 \setminus \overline{B_1}, \\
0 & \text{if } x \in \overline{B_1}.
\end{cases}
\end{align*}
\item For $l=0, \dots, w-1$, it holds that in the trace sense,
\begin{align*}
 \pr_r^{\, l} u \vert_{r=1} = 0.
\end{align*}
\item It holds that $u \in \overline{H}^w_{\de}$.
\end{enumerate}
\end{proposition}

In dimension $1$, the following Sobolev embedding holds. This is similar to Lemma \ref{SobolevEmbeddingsAndNonlinear} and its proof is left to the reader.
\begin{lemma} \label{sobolev1d}
Let $\de \in \RRR$. Let $u: (1, \infty) \to \RRR$ be a scalar function. If
\begin{align*}
\int\limits_{1}^\infty (1+r)^{-2\de-1} u^2(r) dr, \, \int\limits_{1}^\infty (1+r)^{-2\de+1} \left(\pr_r u\right)^2(r)dr , \, \int\limits_{1}^\infty (1+r)^{-2\de+3} \left( \pr_r^2 u\right)^2(r) dr < +\infty,
\end{align*}
then $u, \pr_r u \in C^0((1,\infty))$ and
\begin{align*}
\sup\limits_{r \in (1,\infty)} (1+r)^{-\de} u(r), \,  \sup\limits_{r \in (1,\infty)} (1+r)^{-\de+1} \pr_r u(r) <+ \infty.
\end{align*}
\end{lemma}

For functions on $(S_r, \gac)$, we define the following norm.
\begin{definition}
Let $w\geq0$ be an integer. Let $f$ be a function on $S_r$ for some $r>0$. Then
\begin{align*}
\Vert f \Vert^2_{H^w(S_r)} := \sum\limits_{0 \leq n \leq w} \sr \vert \Nd^{n} f \vert_{\gac}^2,
\end{align*}
where $\Nd$ denotes the covariant derivative on $(S_r,\gac)$ and
\begin{align*}
\vert \Nd^{n} f \vert_{\gac}^2 = \gac^{A_1B_1} \cdots \gac^{A_n B_n} \Nd_{A_1} \dots \Nd_{A_n} f  \Nd_{B_1} \dots \Nd_{B_n} f, 
\end{align*}
see Definition \ref{def:tensornorm111}. Denote further $H^0(S_r)= L^2(S_r)$.
\end{definition}

We use the following extension result, see for example \cite{SteinExtension}, to extend functions from $B_1$  to  $\RRR^3$.
\begin{proposition}[Existence of extension operator] \label{functionextcts} There exists a linear operator $\mathcal{E}$ mapping scalar functions on $B_1$ to scalar functions on $\RRR^3$ with the properties
\begin{itemize} 
\item $\mathcal{E}(f) \vert_{B_1} = f$, that is, $\mathcal{E}$ is an extension operator.
\item $\mathcal{E}$ maps $H^{w-2}(B_1)$ continuously into $H^{w-2}(\RRR^3)$ for all integers $w\geq2$.
\end{itemize}
\end{proposition}
We extend tensors from $B_1$ to $\RRR^3$ by extending each coordinate component of the tensor, using the above proposition.


\subsection{Tensor spaces}

More generally, we now define tensor spaces on $\RRR^3$. 
\begin{definition}  \label{def:tensornorm111} Given an $n$-tensor $T$ and a Riemannian metric $g$, let
\begin{align*}
\vert T \vert_g^2 := g^{i_1j_1} \cdots g^{i_nj_n} T_{i_1 \dots i_n } T_{j_1 \dots j_n}.
\end{align*}
\end{definition}
In case of the Euclidean metric $e$, for an $n$-tensor $T$,
\begin{align*}
\vert T \vert_e^2 = \sum\limits_{i_1, \dots, i_n=1}^3 \vert T_{i_1 \dots i_n} \vert^2.
\end{align*}

The norm of a tensor is defined as follows.
\begin{definition}[Tensor norms] Let $\Om \subset \RRR^3$ be a domain. Let $n\geq1$ and $w\geq0$ be integers. For an $n$-tensor $T$ on $\Om$, define its $\HH^w(\Om)$-norm by
\begin{align*}
\Vert T \Vert^2_{\HH^w (\Om)} := \sum\limits_{\vert \alpha \vert \leq w} \int\limits_{\Om} \vert \pr^\alpha T \vert^2_e dx^1dx^2dx^3,
\end{align*}
where $(\pr^\alpha T)_{i_1 \cdots i_n} = \pr^\alpha( T_{i_1 \cdots i_n})$. We write $T \in \HH^w(\Om)$ if this norm is finite. We similarly define tensors in $\HH^w_{loc}(\Om)$, $\HH^w_\de(\Om)$, $\ol{\HH}^w_\de(\Om)$, $\HH^w_\de$ and $\ol{\HH}^w_\de$.
\end{definition}

We define tensor norms on $(S_r, \gac)$ as follows.
\begin{definition}
Let $w\geq0$ be an integer. Let $T$ be a $S_r$-tangent tensor on $(S_r, \gac)$ for some $r>0$. Then
\begin{align*}
\Vert T \Vert^2_{\HH^w(S_r)} := \sum\limits_{0 \leq n \leq w} \sr \vert \Nd^{n} T \vert_{\gac}^2,
\end{align*}
where $\Nd$ denotes the covariant derivative on $(S_r,\gac)$. We say that tensors in $\HH^0(S_r)$ are $L^2$-integrable.
\end{definition}


The next lemma is practical for calculations.
\begin{lemma} \label{lem:practical4j}
Let $w\geq0$ be an integer. Then the following holds.
\begin{itemize}
\item Let $X$ be a vectorfield on $\RRRwo$. Then,
\begin{align*}
&\Vert X \Vert_{\HH^w(\RRRwo)}^2\\
 \lesssim& \Vert X_N \Vert_{H^w(\RRRwo)}^2 + \Vert X \mkern-14mu/\ \mkern-5mu  \, \Vert_{\HH^w(\RRRwo)}^2 + C_w \Big( \Vert X_N \Vert_{H^0(\RRRwo)}^2 + \Vert X \mkern-14mu/\ \mkern-5mu  \, \Vert_{\HH^0(\RRRwo)}^2 \Big), \\
\lesssim& \Vert X_N \Vert_{H^w(\RRRwo)}^2 + \sum\limits_{n_1 + n_2 \leq w} \int\limits_1^\infty \int\limits_{S_r} \vert {\Nd}^{n_1} \Nd_N^{n_2}X \mkern-14mu/\ \mkern-3mu \vert_{\gac}^2 dr + C_w \Big( \Vert X_N \Vert_{H^0(\RRRwo)}^2 + \Vert X \mkern-14mu/\ \mkern-5mu  \, \Vert_{\HH^0(\RRRwo)}^2 \Big),
\end{align*}
and
\begin{align*}
&\Vert X_N \Vert_{H^w(\RRRwo)}^2 + \sum\limits_{n_1 + n_2 \leq w} \int\limits_1^\infty \int\limits_{S_r} \vert {\Nd}^{n_1} \Nd_N^{n_2}X \mkern-14mu/\ \mkern-3mu \vert_{\gac}^2 dr \\
\lesssim& \Vert X_N \Vert_{H^w(\RRRwo)}^2 + \Vert X \mkern-14mu/\ \mkern-5mu  \, \Vert_{\HH^w(\RRRwo)}^2 + C_w \Big( \Vert X_N \Vert_{H^0(\RRRwo)}^2 + \Vert X \mkern-14mu/\ \mkern-5mu  \, \Vert_{\HH^0(\RRRwo)}^2 \Big) \\
\lesssim& \Vert X \Vert_{\HH^w(\RRRwo)}^2 + C_w \Vert X \Vert_{\HH^0(\RRRwo)}^2.
\end{align*}
\item Let $V$ be a symmetric $2$-tensor on $\RRRwo$.
\begin{align*}
&\Vert V \Vert_{\HH^w(\RRRwo)}^2  \\
\lesssim & \Vert V_{NN} \Vert_{H^w(\RRRwo)}^2 + \Vert {V \mkern-14mu/\ \mkern-5mu}_{N} \Vert_{\HH^w(\RRRwo)}^2 + \Vert V \mkern-14mu/\ \mkern-3mu \Vert_{\HH^w(\RRRwo)}^2\\
&+ C_w\left( \Vert V_{NN} \Vert_{H^0(\RRRwo)}^2 + \Vert {V \mkern-14mu/\ \mkern-5mu}_{N} \Vert_{\HH^0(\RRRwo)}^2 + \Vert V \mkern-14mu/\ \mkern-3mu \Vert_{\HH^0(\RRRwo)}^2\right) \\
\lesssim& \Vert V_{NN} \Vert_{H^w(\RRRwo)}^2 + \sum\limits_{n_1 + n_2 \leq w} \int\limits_{1}^\infty \int\limits_{S_r} \vert \Nd^{n_1} \Nd_N^{n_2} {V \mkern-14mu/\ \mkern-5mu}_{N} \vert_{\gac}^2 dr + \sum\limits_{n_1 + n_2 \leq w} \int\limits_{1}^\infty \int\limits_{S_r} \vert \Nd^{n_1} \Nd_N^{n_2} V \mkern-14mu/\ \mkern-3mu \vert_{\gac}^2 dr \\
&+ C_w\left( \Vert V_{NN} \Vert_{H^0(\RRRwo)}^2 + \Vert {V \mkern-14mu/\ \mkern-5mu}_{N} \Vert_{\HH^0(\RRRwo)}^2 + \Vert V \mkern-14mu/\ \mkern-3mu \Vert_{\HH^0(\RRRwo)}^2\right),
\end{align*}
and
\begin{align*}
&\Vert V_{NN} \Vert_{H^w(\RRRwo)}^2 + \sum\limits_{n_1 + n_2 \leq w} \int\limits_{1}^\infty \int\limits_{S_r} \vert \Nd^{n_1} \Nd_N^{n_2} {V \mkern-14mu/\ \mkern-5mu}_{N} \vert_{\gac}^2 dr + \sum\limits_{n_1 + n_2 \leq w} \int\limits_{1}^\infty \int\limits_{S_r} \vert \Nd^{n_1} \Nd_N^{n_2} V \mkern-14mu/\ \mkern-3mu \vert_{\gac}^2 dr \\
\lesssim&  \Vert V_{NN} \Vert_{H^w(\RRRwo)}^2 + \Vert {V \mkern-14mu/\ \mkern-5mu}_{N} \Vert_{\HH^w(\RRRwo)}^2 + \Vert V \mkern-14mu/\ \mkern-3mu \Vert_{\HH^w(\RRRwo)}^2\\
&+ C_w\left( \Vert V_{NN} \Vert_{H^0(\RRRwo)}^2 + \Vert {V \mkern-14mu/\ \mkern-5mu}_{N} \Vert_{\HH^0(\RRRwo)}^2 + \Vert V \mkern-14mu/\ \mkern-3mu \Vert_{\HH^0(\RRRwo)}^2\right) \\
\lesssim& \Vert V \Vert_{\HH^w(\RRRwo)}^2 + C_w \Vert V \Vert_{\HH^0(\RRRwo)}^2,
\end{align*}
where $\Nd$ denotes the tangential gradient and $\Ndn$ was defined in Definition \ref{def:ndn11}. Analogously for $\HH^w_{loc}(\RRRwo)$, $\HH^w_\de(\RRRwo)$, $\ol{\HH}^w_\de$.
\end{itemize}
\end{lemma}
The proof of the above lemma is left to the reader. By using the radial tensor decomposition of Section \ref{ssec:tensordecomposition}, it follows that for a vectorfield $X$ and a symmetric $2$-tensor $V$,
\begin{align*}
\vert X \vert_e^2 &= X_{N}^2 + \vert X \mkern-14mu/\ \mkern-5mu \, \vert^2_{\gac}, \\
\vert V \vert_e^2 &= V_{NN}^2 + \vert {V \mkern-14mu/\ \mkern-5mu}_{N} \vert^2_{\gac} + \vert V \mkern-14mu/\ \mkern-5mu \, \vert^2_{\gac}.
\end{align*}


\subsection{Asymptotically flat initial data}

The following definition is standard, see for example \cite{BartnikMassAF} \cite{Maxwell}.
\begin{definition}[Asymptotically flat initial data] \label{def:asmypt}
Let $w \geq 2$ be an integer. Let $g \in \HH^w_{loc}(\RRR^3)$ be a Riemannian metric and $k \in \HH^{w-1}_{loc}(\RRR^3)$ a symmetric $2$-tensor on $\RRR^3$. The metric $g$ is called \emph{$\HH^w_{-1/2}$-asymptotically flat} if 
\begin{align*}
g-e \in \HH^w_{-1/2},
\end{align*}
where $e$ denotes the Euclidean metric on $\RRR^3$. The pair $(g,k)$ is called $\HH^w_{-1/2}$-asymptotically flat if 
\begin{align}
g-e \in \HH^w_{-1/2}, \,\,\, k \in \HH^{w-1}_{-3/2}, \label{eq:normofmetric}
\end{align}
Similarly, a metric $g$ on $\RRRwo$ is called $\HH^w_{-1/2}$-asymptotically flat if $$g-e \in \HH^w_{-1/2}(\RRRwo).$$
\end{definition}

\begin{remark}
The norms associated to \eqref{eq:normofmetric} are explicitly
\begin{align}
&\sum\limits_{i,j=1}^3 \sum\limits_{\vert \a \vert \leq w} \Vert (1+r)^{-1+ \vert \a \vert} \partial^\a (g_{ij}-e_{ij} ) \Vert_{L^2(\RRR^3)} < +\infty, \nr \\
&\sum\limits_{i,j=1}^3 \sum\limits_{\vert \a \vert \leq w-1} \Vert (1+r)^{\vert \a \vert} \partial^\a k_{ij} \Vert_{L^2(\RRR^3)} < +\infty. \nr
\end{align}
\end{remark}

The next lemma is used to estimate the components of the inverse metric. Its proof follows by using Kramer's rule to express the inverse metric components in terms of the metric components, and product estimates as in Lemma \ref{ProductEstimates}. Details are left to the reader.
\begin{lemma} \label{lem:nonlinearities} \label{cor:UsefulCor} \label{GInverseAnalysis} Let $w\geq2$ be an integer. There exists a universal $\varep>0$ such that the following holds.
\begin{enumerate}
\item The mapping $g\mapsto g^{-1}$ is smooth from
\begin{align*}
\{ g-e \in \HH^w_{-1/2} : \Vert g-e \Vert_{\HH^2_{-1/2}} < \varep \} \to \HH^w_{-1/2}.
\end{align*}
\item Let $g$ and $g'$ be two $\HH^2_{-1/2}$-asymptotically flat metrics such that
\begin{align*}
\Vert g-e \Vert_{\HH^2_{-1/2}} < \varep, \Vert g'-e \Vert_{\HH^2_{-1/2}}< \varep.
\end{align*}
Then it holds that
\begin{align*}
\Vert g^{-1} - g'^{-1}\Vert_{\HH^2_{-1/2}} \lesssim \Vert g-g' \Vert_{\HH^2_{-1/2}}.
\end{align*}
\item Let $g$ be an $\HH^w_{-1/2}$-asymptotically flat metric such that
\begin{align*}
\Vert g-e \Vert_{\HH^2_{-1/2}} < \varep.
\end{align*}
Then
\begin{align*}
\Vert g^{-1}-e \Vert_{\HH^w_{-1/2}} \lesssim \Vert g-e \Vert_{\HH^w_{-1/2}} + C_w \Vert g-e \Vert_{\HH^2_{-1/2}}.
\end{align*}
\end{enumerate}
\end{lemma}


The next lemma allows us to directly work with the polar components of an $\HHl$-asymptotically flat metric.
\begin{lemma} \label{lem:coordinatechangeAF}
Let $w\geq2$ be an integer. There exists a universal $\varep>0$ small such that the following holds. 
\begin{enumerate}
\item Let $g$ be an $\HH^w_{-1/2}$-asymptotically flat Riemannian metric on $\RRRwo$ such that 
$$ \Vert g-e \Vert_{\HH^2_{-1/2}(\RRRwo)} < \varep,$$
and denote its polar components on $\RRRwo$ by
\begin{align*}
g = a^2 dr^2 + \ga_{AB} \left( \be^A dr + d\th^A \right) \left( \be^Bdr + d\th^B \right).
\end{align*}
Then it holds that
\begin{align*}
a^2-1 \in H^w_{-1/2}(\RRRwo),\be, \ga-{\gac} \in \HH^w_{-1/2}(\RRRwo)
\end{align*}
with the estimate
\begin{align*}\begin{aligned}
&\Vert a^2-1 \Vert_{H^w_{-1/2}(\RRRwo)} + \Vert \be \Vert_{\HH^w_{-1/2}(\RRRwo)} + \Vert \ga- {\gac} \Vert_{\HH^w_{-1/2}(\RRRwo)} \\
&\lesssim \Vert g - e \Vert_{\HH^w_{-1/2}(\RRRwo)} + C_w \Vert g - e \Vert_{\HH^2_{-1/2}(\RRRwo)}.
\end{aligned} 
\end{align*}
\item Let $a$ be a positive scalar function, $\be$ a $S_r$-tangent vectorfield and $\ga$ a Riemannian metric on $S_r$ on $\RRRwo$ such that
\begin{align*}
\Vert a^2-1 \Vert_{H^w_{-1/2}(\RRRwo)} + \Vert \be \Vert_{\HH^w_{-1/2}(\RRRwo)} +  \Vert \ga-{\gac} \Vert_{\HH^w_{-1/2}(\RRRwo)} < \infty,
\end{align*}
and
\begin{align*}
\Vert a^2-1 \Vert_{H^2_{-1/2}(\RRRwo)} + \Vert \be \Vert_{\HH^2_{-1/2}(\RRRwo)} +  \Vert \ga-{\gac} \Vert_{\HH^2_{-1/2}(\RRRwo)} < \varep.
\end{align*}
Then the symmetric $2$-tensor $g$ defined on $\RRRwo$ by
\begin{align*}
g = a^2 dr^2 + \ga_{AB} \left( \be^A dr + d\th^A \right) \left( \be^Bdr + d\th^B \right).
\end{align*}
is an $\HH^w_{-1/2}$-asymptotically flat Riemannian metric and bounded by
\begin{align*}
&\Vert g - e \Vert_{\HH^w_{-1/2}(\RRRwo)} \\
\lesssim& \Vert a^2-1 \Vert_{H^w_{-1/2}(\RRRwo)} + \Vert \be \Vert_{\HH^w_{-1/2}(\RRRwo)} +  \Vert \ga- {\gac} \Vert_{\HH^w_{-1/2}(\RRRwo)} \\
&+C_w \Big( \Vert a^2-1 \Vert_{H^2_{-1/2}(\RRRwo)} + \Vert \be \Vert_{\HH^2_{-1/2}(\RRRwo)} +  \Vert \ga- {\gac} \Vert_{\HH^2_{-1/2}(\RRRwo)}\Big).
\end{align*}
\end{enumerate}
\end{lemma}
\begin{proof} For the metric $g$, it holds that
\begin{align*}
g_{NN} = a^2, \left( {g \mkern-10mu/\ \mkern-5mu}_{N}  \right)_{A} = \ga_{AB} \be^B, g \mkern-10mu/\ \mkern-5mu = \ga.
\end{align*}
By Lemma \ref{lem:practical4j} and Remark \ref{remark:polarcoorde}, we can bound 
\begin{align*}
&\Vert a^2 -1 \Vert_{H^w_{-1/2}(\RRRwo)} + \Vert \ga( \be, \cdot ) \Vert_{\HH^w_{-1/2}(\RRRwo)} + \Vert \ga-\gac \Vert_{\HH^w_{-1/2}(\RRRwo)} \\
\lesssim& \Vert g-e \Vert_{\HH^w_{-1/2}(\RRRwo)} + C_w \Vert g-e \Vert_{\HH^2_{-1/2}(\RRRwo)}.
\end{align*}

For $\varep>0$ sufficiently small, it follows by the above that $\ga$ is invertible. Therefore the vectorfield $\be^A = \ga^{AB} (\ga_{BC}\be^C)$ is controlled by product estimates as in Lemma \ref{ProductEstimates}, see also Lemma \ref{GInverseAnalysis}. This proves part (1) of Lemma \ref{lem:coordinatechangeAF}. Part (2) is demonstrated similarly and left to the reader. \end{proof}


\subsection{$L^2$-Hodge theory on $S_r$ } \label{sec:HodgeTheory}

In this section, we recall basic Hodge theory on Euclidean spheres $(S_r, \gac)$, $r>0$. This is a special case of the Hodge theory on Riemannian $2$-spheres in \cite{ChrKl93}. All tensors are assumed to be $S_r$-tangent. Let
\begin{itemize}
\item $\Nd$ denote the covariant derivative on $(S_r, \gac)$.
\item $\iin$ denote the volume element on $(S_r,\gac)$.
\item $\Ld := \gac^{AB} \Nd_A \Nd_B$ denote the Laplace-Beltrami operator\footnote{Here we follow the convention that Laplacians have negative eigenvalues.} on $(S_r, \gac)$.
\item the divergence and curl of a vectorfield $X$ be defined as
\begin{align}
\Divd \xi &:= \Nd_A X^A, \nr\\
\Curld \xi &:= \iin_{AB} \Nd^A X^B. \nr
\end{align}
\item the divergence and trace of a symmetric $2$-tensor $V_{AB}$ be defined as
\begin{align*}
\left( \Divd V \right)_B := \Nd^A V_{AB},\\
\tr \mkern-11mu /\  \mkern-3mu  V := \gac^{AB} V_{AB}.
\end{align*}
\item for a vectorfield $X$ the tracefree symmetric $2$-tensor $\Nd \widehat{ \otimes } X $ be defined as
\begin{align}
(\Nd \widehat{ \otimes } X )_{AB} 
:= \Nd_A X_B + \Nd_B X_A - (\Divd X) {\gac}_{AB}. \nr
\end{align}
\item for two vectorfields $X,Y$ the symmetric tracefree $2$-tensor $X \widehat{\otimes} Y$ be defined as
\begin{align}
\left( X \widehat{\otimes} Y \right)_{AB} := X_A Y_B + X_B Y_A  - \gac(X,Y) \gac_{AB}. \nr
\end{align}
\item the Hodge dual of a vectorfield $X$ be defined as
\begin{align}
\dual X_A := \iin_{AB} X^B.\nr
\end{align}
\item the left Hodge dual of a symmetric tracefree $2$-tensor $V$ be defined as the tracefree symmetric tensor
\begin{align}
\dual V_{AB} := \iin_{AC} V^{C}_{\,\, \, B}.\nr
\end{align}
\item the modulus of an $n$-tensor $V$ be defined as
\begin{align*}
\vert V \vert^2 := \gac^{A_1B_1} \cdots \gac^{A_nB_n} V_{A_1 \cdots A_n} V_{B_1 \cdots B_n}. 
\end{align*}
\end{itemize}
We note that for a vectorfield $X$ and a symmetric tracefree $2$-tensor $V$,
\begin{align} \label{eq:lefthodgedualidentity}
{}^{\ast} ({}^{\ast} X) := -X, {}^{\ast} ( {}^{\ast}V) := -V.
\end{align}

Introduce two Hodge systems on $(S_r,\gac)$ as follows. Let $X$ be a vectorfield on $S_r$ that verifies
\begin{align*}
\Divd X &= f, \\
\Curld X &= f_*,
\tag{{\bf H1}}
\end{align*}
where $f, f_*$ are scalar functions on $S_r$.\\

Let $V$ be a tracefree symmetric $2$-tensor on $S_r$ that verifies
\begin{align*}
\Divd V = F,
\tag{{\bf H2}}
\end{align*}
where $F$ is a $1$-form on $S_r$.\\

\ni The following is the Euclidean version of Proposition 2.2.1 in \cite{ChrKl93}.
\begin{proposition}[Ellipticity of Hodge systems] \label{prop:EllipticityHodgejan}
The following holds.
\begin{itemize}
\item Assume that the vectorfield $X$ is a solution of $\mathbf{H_1}$. Then
\begin{align}
\int\limits_{S_r} \left( \vert \Nd X\vert^2 + \frac{1}{r^2} \vert X \vert^2 \right) = \int\limits_{S_r} \left( \vert f \vert^2 + \vert f_* \vert^2 \right). \nr
\end{align}
\item Assume that the symmetric tracefree $2$-tensor $V$ is a solution of $\mathbf{H_2}$. Then
\begin{align}
\int\limits_{S_r} \left( \vert \Nd V \vert^2 + \frac{2}{r^2} \vert V \vert^2 \right)= 2 \int\limits_{S_r} \vert F \vert^2. \nr
\end{align}
\end{itemize}
\end{proposition}

Furthermore, the next higher regularity estimates hold.
\begin{proposition}[Higher regularity for Hodge systems on $S_r$] \label{prop:EllipticityHodgefeb23}
Let $w\geq1$ be an integer. The following holds.
\begin{itemize}
\item Assume that the vectorfield $X$ is a solution of $\mathbf{H_1}$ for $f,f_* \in H^{w-1}(S_r)$. Then
\begin{align*}
\sum\limits_{0\leq n\leq w} \int\limits_{S_r}  \vert r^{n} \Nd^{{n}} X \vert^2 \lesssim&  \sum\limits_{0\leq {n}\leq w-1} \int\limits_{S_r}  r^2 \left( \vert r^{{n}} \Nd^{{n}} f \vert^2 + \vert r^{{n}} \Nd^{{n}} f_* \vert^2 \right) +C_w \int\limits_{S_r}  r^2 \left( \vert f \vert^2 + \vert f_* \vert^2 \right).
\end{align*}
\item Assume that the symmetric tracefree $2$-tensor $V$ is a solution of $\mathbf{H_2}$ for $F \in \HH^{w-1}(S_r)$. Then
\begin{align*}
\sum\limits_{0\leq {n}\leq w}  \int\limits_{S_r} \vert r^{n} \Nd^{{n}} V \vert^2 \lesssim& \sum\limits_{0\leq {n}\leq w-1} \int\limits_{S_r} r^2 \vert r^{{n}} \Nd^{{n}} F \vert^2 + C_w \int\limits_{S_r} r^2 \vert F \vert^2.
\end{align*}
\end{itemize}
\end{proposition}

\begin{proof} We only give a sketch of the proof, because it follows from Lemmas 2.2.2 and 2.2.3 in \cite{ChrKl93} and the fact that we work on the round sphere $(S_r,\gac)$. The proof is by induction on $w$. The case $w=1$ is Proposition \ref{prop:EllipticityHodgejan}. The induction step $w \to w+1$ follows by showing that the symmetrized derivative of a totally symmetric tensor $\xi$, 
\begin{align*}
\tilde D \xi_{A_1 A_2 \dots A_{k+1}B}:= \frac{1}{k+2} \left( \Nd_B \xi_{A_1 \dots A_{k+1}} + \sum\limits_{i=1}^{k+1} \Nd_{A_i} \xi_{A_1 \dots B \dots A_{k+1}} \right)
\end{align*}
satisfies a Hodge system whose source terms can be controlled\footnote{Thereby it is used that in the Euclidean case, the Gauss curvature $K=1/r^2$ is spherically symmetric, so in particular $\Nd K =0$.} in lower order norms of $\xi$. Lemma 2.2.2 in \cite{ChrKl93} shows the ellipticity of this Hodge system. Generally on $S_r$, the symmetrized derivative and the curl of a tensor control the full covariant derivative, see Chapter 2 of \cite{ChrKl93}. The curl is estimated via the Hodge system by the induction assumption so that the full control of $\Nd \xi$ follows. This finishes the proof of Proposition \ref{prop:EllipticityHodgefeb23}. \end{proof}
The following relations are from Chapter 2 in \cite{ChrKl93}:
\begin{lemma} \label{lem:Chrlem1}
Let $\cDd_1$ be the operator that takes a vectorfield $X$ on $S_r$ into the pair of functions $( \Divd \xi, \Curld \xi)$. The $L^2$-adjoint of $\cDd_1$ is the operator $\cDd_1^*$ which takes pairs of functions $(f, f_\ast)$ into vectorfields on $S_r$ given by
\begin{align*}
\cDd_1^* (f, f_\ast) = - \Nd^A f + \in^{AB} \Nd_B f_\ast. 
\end{align*}
Let $\cDd_2$ be the operator that takes a symmetric tracefree $2$-tensor $X$ into the $1$-form $\Divd X$. The $L^2$-adjoint of $\cDd_2$ is $\cDd_2^*$ which takes $1$-forms $F$ into symmetric tracefree $2$-tensors given by
\begin{align*}
\cDd_2^* F = -\half (\Nd \widehat{ \otimes } F )_{AB},
\end{align*}
where we recall that 
\begin{align*}
(\Nd \widehat{ \otimes } F )_{AB} := \Nd_A F_B + \Nd_B F_A - (\Divd F) \gac_{AB}.
\end{align*}
The following relations hold.
\begin{align*}
\cDd_1 \cDd_1^\ast &= - \Ld,\\
\cDd_2 \cDd_2^\ast &= -\half \Ld - \half \frac{1}{r^2}, \\
\cDd_1^\ast \cDd_1 &= - \Ld + \frac{1}{r^2},\\
\cDd_2^\ast \cDd_2 &= -\half \Ld + \frac{1}{r^2}.
 \end{align*}
\end{lemma}
\begin{remark} \label{remmars30}
By the above, the kernel of $\cDd_2^*$ can be identified with the conformal Killing vectorfields on $(S_r, \gac)$. This implies that the image of $\cDd_2$ is $L^2(S_r)$-orthogonal to the conformal Killing vectorfields of $(S_r, \gac)$.
\end{remark}


\subsection{The expansion of $S_r$-tangential tensors.} \label{ssec:Hodgebasis}

In Sections \ref{sec:EuclideanSurjectivity} and \ref{sec:SurjectivityR}, we analyse Hodge systems on Euclidean spheres. The main technical tools for this analysis are the bases of tensors defined here in the following. In this section, all differential operators are on Euclidean spheres $(S_r,\gac)$ and all tensors are $S_r$-tangent.\\

\begin{itemize}
\item \underline{Real spherical harmonics:} For $r>0$, let 
\begin{align*}
\Big\{ Y^{(lm)}(r,\th, \phi) : l\geq0, m \in \{ -l,\dots,l \} \Big\}
\end{align*}
denote the set of normalised real spherical harmonics on $S_r$. In particular, for each $l\geq0, m \in \{ -l,\dots,l \}$ they solve
\begin{align}
\Ld Y^{(lm)} = -\frac{l(l+1)}{r^2} Y^{(lm)}. \label{av112}
\end{align}
The next lemma is standard, see for example \cite{CourantHilbert1}.
\begin{lemma}\label{sphericalharmonics}
For each $r>0$, the set
\begin{align}
\Big\{ Y^{(lm)}(r) : l \geq 0, m \in \{ -l,\dots,l \} \Big\} \nr
\end{align}
forms a complete orthonormal basis of $L^2(S_r)$-integrable scalar functions on $S_r$. 
\end{lemma}

\item \underline{Vector spherical harmonics:} For $r>0$, let the vectorfields $E^{(lm)}, H^{(lm)}$ on $S_r$ be defined for $l \geq1, m\in \{-l, \dots, l \}$ by
\begin{align} \begin{aligned}
{E^{(lm)}}(r)&:= \frac{r}{\sqrt{l(l+1)}} \cDd_1^*\left(Y^{(lm)},0\right), \\ 
{H^{(lm)}}(r) &:= \frac{r}{\sqrt{l(l+1)}} \cDd_1^*\left(0,Y^{(lm)}\right),
\end{aligned}  \label{def:Vectorsjan} \end{align}
where $\cDd_1^\ast$ is given in Lemma \ref{lem:Chrlem1}.


\item \underline{$2$-covariant spherical harmonics:} For $r>0$, let the tracefree symmetric $2$-tensors $\psi^{(lm)}, \phi^{(lm)}$ on $S_r$ be defined for $l \geq 2, m \in \{ -l, \dots, l \}$ by
\begin{align} \begin{aligned}
\psi^{(lm)}_{AB}(r) &:= \frac{r}{\sqrt{\half l (l+1)-1}} \cDd_2^\ast \left( E^{(lm)} \right),\\
\phi_{AB}^{(lm)}(r) &:=\frac{r}{\sqrt{\half l (l+1)-1}}\cDd_2^\ast \left(H^{(lm)}\right),
\end{aligned} \label{def:2tensorsjan} \end{align}
where $\cDd_2^\ast$ is given in Lemma \ref{lem:Chrlem1}.
\end{itemize}


\begin{remark} \label{remarkharm29} The tensors defined in \eqref{def:Vectorsjan} and \eqref{def:2tensorsjan} are spherical harmonics in the sense that by Lemma \ref{lem:Chrlem1}, 
\begin{itemize}
\item for $l\geq1,m \in \{ -l,\dots,l \}$,
\begin{align*}
\Ld E^{(lm)} &= \frac{1-l(l+1)}{r^2} E^{(lm)},\\ 
\Ld H^{(lm)} &= \frac{1-l(l+1)}{r^2} H^{(lm)}.
\end{align*}
\item for $l\geq2,m \in \{ -l,\dots,l \}$,
\begin{align*}
\Ld \psi^{(lm)} &= \frac{4-l(l+1)}{r^2} \psi^{(lm)},\\ 
\Ld \phi^{(lm)} &=\frac{4-l(l+1)}{r^2} \phi^{(lm)}.
\end{align*}
\end{itemize}
\end{remark}

The next proposition shows that these sets of tensors form complete orthonormal bases. First, we introduce some notation.
\begin{definition} \label{defnotationsep}
Let $r>0$. Let $f$ be a scalar function, $X$ a vectorfield and $V$ a symmetric tracefree $2$-tensor on $S_r$. Define then
\begin{itemize}
\item for $l\geq0:$ $f^{(lm)}(r) := \sr Y^{(lm)} f$,
\item for $l\geq1:$ $X^{(lm)}_E(r) := \sr X \cdot E^{(lm)}$, $X^{(lm)}_H(r) := \sr X \cdot H^{(lm)}$,
\item for $l\geq2:$ $V^{(lm)}_\psi(r) := \sr V \cdot \psi^{(lm)}$, $V^{(lm)}_\phi(r) := \sr V \cdot \phi^{(lm)}$,
\end{itemize}
where $\cdot$ denotes the contraction of tensors with respect to $\gac$. \end{definition}

\begin{proposition} \label{prop:Completeness} For all $r>0$, the set
\begin{align*}
 \Big\{ E^{(lm)}(r), H^{(lm)}(r) : l \geq 1, m \in \{ -l,\dots,l \} \Big\} 
\end{align*} 
forms a complete orthonormal basis of the space of $L^2$-integrable vectorfields on $S_r$. For all $r>0$, the set
\begin{align*}
\Big\{ \psi^{(lm)}(r), \phi^{(lm)}(r): l \geq 2, m \in \{ -l, \dots, l \} \Big\} 
\end{align*}
forms a complete orthonormal basis of the set of $L^2$-integrable tracefree symmetric $2$-tensors on $S_r$. Moreover,
\begin{itemize}
\item for any scalar function $f \in L^2(S_r)$,
\begin{align*}
\Vert f \Vert_{L^2(S_r)}^2 = \sum\limits_{l\geq 0} \sum\limits_{m =-l}^l  \left( f^{(lm)}\right)^2, 
\end{align*}
\item for any $S_r$-tangent vectorfield $X \in \HH^0(S_r)$,
\begin{align*}
\Vert X \Vert_{\HH^0(S_r)}^2 = \sum\limits_{l\geq 1} \sum\limits_{m =-l}^l \left( \left( X_E^{(lm)} \right)^2 + \left( X_H^{(lm)} \right)^2 \right),
\end{align*}
\item for any $S_r$-tangent symmetric tracefree $2$-tensor $V \in \HH^0(S_r)$,
\begin{align*}
\Vert V \Vert^2_{\HH^0(S_r)} = \sum\limits_{l\geq 2} \sum\limits_{m =-l}^l  \left( \left( V_\psi^{(lm)} \right)^2 + \left(V_\phi^{(lm)} \right)^2 \right).
\end{align*}
\end{itemize}
\end{proposition}
A proof is given in Appendix \ref{sec:ProofCompleteness}.


\begin{remark} \label{imageanalysis} For all $r>0$, the vectorfields with $l=1$,
\[
\Big\{ E^{(1m)}(r), H^{(1m)}(r): m \in \{-1,0,1\} \Big\},  \nr
\]
form an orthonormal basis of the six-dimensional space of conformal Killing fields on $(S_r, \gac)$.
\end{remark}

The next expansion notation is used throughout Sections \ref{sec:EuclideanSurjectivity}, \ref{sec:SurjectivityR} and Appendix \ref{sec:WEIGHTEDellipticity}.
\begin{definition}
Let $f\in L^2(S_r)$ be a scalar function, $X \in \HH^0(S_r)$ a $S_r$-tangent vectorfield and $V \in \HH^0(S_r)$ a $S_r$-tangent tracefree symmetric $2$-tensor. Denote
\begin{align*}
f &= \underbrace{f^{(00)} Y^{(00)}}_{:= f^{[0]}} + \underbrace{\sum\limits_{m=-1}^1 f^{(1m)} Y^{(1m)}}_{:=f^{[1]}} + \underbrace{ \sumtwo f^{(lm)} Y^{(lm)}}_{:=f^{[\geq2]}}, \\
X &=  \underbrace{\sum\limits_{m=-1}^1 X^{(1m)}_E E^{(1m)}  }_{:=X_E^{[1]}} + \underbrace{\sum\limits_{m=-1}^1 X^{(1m)}_H H^{(1m)} }_{:=X_H^{[1]}} + \underbrace{ \sumtwo X^{(lm)}_E E^{(lm)} }_{:=X_E^{[\geq2]}} + \underbrace{ \sumtwo X^{(lm)}_H H^{(lm)} }_{:=X_H^{[\geq2]}}, \\
V &= \underbrace{ \sumtwo V^{(lm)}_\psi \psi^{(lm)} }_{:=V_\psi} + \underbrace{ \sumtwo V^{(lm)}_\phi \phi^{(lm)} }_{:=V_\phi}, 
\end{align*} 
and let $X^{[1]} = X_E^{[1]} + X_H^{[1]}, X^{[\geq2]} = X_E^{[\geq2]} + X_H^{[\geq2]}$.
\end{definition}

We have the following identities.
\begin{lemma}[Hodge-Fourier calculus] \label{lem:RelationsSpherical} Let $f\in L^2(S_r)$ be a scalar function, $X \in \HH^0(S_r)$ a vectorfield and $V \in \HH^0(S_r)$ a symmetric tracefree $2$-tensor. It holds that
\begin{itemize}
\item for $l \geq 1, m\in \{-l, \dots,l\}$,
\begin{align}
-(\Nd f)_E^{(lm)}&= \frac{\sqrt{l(l+1)}}{r} f^{(lm)},  \qquad
-(\Nd f)^{(lm)}_H = 0, \nr \\
(\Divd X)^{(lm)} &= \frac{\sqrt{l(l+1)}}{r} X_E^{(lm)}, \qquad
(\Curld X)^{(lm)} = \frac{\sqrt{l(l+1)}}{r} X_H^{(lm)}, \nr
\end{align}
\item for $l\geq 2, m\in \{-l, \dots,l\}$,
\begin{align}
 -\half \left( \Nd \widehat{\otimes} X \right)_\psi^{(lm)} &=  \frac{\sqrt{\half l(l+1) -1}}{r} X_E^{(lm)}, 
&-\half \left( \Nd \widehat{\otimes} X\right)_\phi^{(lm)} &=  \frac{\sqrt{\half l(l+1) -1}}{r} X_H^{(lm)},\nr\\
(\Divd V)_E^{(lm)} &=\frac{\sqrt{\half l(l+1)-1}}{r} V^{(lm)}_\psi, 
&(\Divd V)_H^{(lm)}&= \frac{\sqrt{\half l(l+1)-1}}{r} V_\phi^{(lm)}.\nr
\end{align}
\end{itemize}
\end{lemma}
The proof of this lemma follows by \eqref{def:Vectorsjan}, \eqref{def:2tensorsjan}, Lemma \ref{lem:Chrlem1} and integration by parts. Details are left to the reader.\\

The next three results are handy for the estimates in Section \ref{sssec:EuclideanSurjectivity} and \ref{sec:SurjectivityR}.
\begin{proposition}  \label{prop:Howtoestimatefunctions}
Let $u$ be a scalar function and $X$ a vectorfield on $S_r$ for some $r>0$. For all integers $w\geq0$,
\begin{align*}
\Vert \Nd^w u \Vert_{\HH^0(S_r)}^2 \lesssim& \sumzero \left( \frac{l(l+1)}{r^2} \right)^w  \left( u^{(lm)}\right)^2+\frac{C_w}{r^{2w}} \sumzero \left( u^{(lm)}\right)^2, \\
\Vert \Nd^w X \Vert_{\HH^0(S_r)}^2 \lesssim& \sumone \left( \frac{l(l+1)-1}{r^2} \right)^w \left( \left( X_E^{(lm)}\right)^2+ \left( X_H^{(lm)}\right)^2 \right) \\
&+\frac{C_w}{r^{2w}}\sumone\left( \left( X_E^{(lm)}\right)^2+ \left( X_H^{(lm)}\right)^2 \right).
\end{align*}
and
\begin{align*}
\sumzero \left( \frac{l(l+1)}{r^2} \right)^w  \left( u^{(lm)}\right)^2 \lesssim& \Vert \Nd^w u \Vert_{\HH^0(S_r)}^2 +\frac{C_w}{r^{2w}}\Vert u \Vert_{\HH^0(S_r)}^2, \\
\sumone \left( \frac{l(l+1)-1}{r^2} \right)^w \left( \left( X_E^{(lm)}\right)^2+ \left( X_H^{(lm)}\right)^2 \right) \lesssim& \Vert \Nd^w X \Vert_{\HH^0(S_r)}^2 + \frac{C_w}{r^{2w}} \Vert X \Vert_{\HH^0(S_r)}^2.
\end{align*}
\end{proposition}
A proof is provided in Appendix \ref{sec:appidentities}. 

\begin{lemma} \label{lem:commutationrelation}
Let $w\geq0$ be an integer and $r>0$. The following holds.
\begin{enumerate}
\item Let $f$ be a scalar function. Then, for $l\geq0, m \in \{ -l, \dots, l\}$, 
\begin{align} \label{resscalar1}
\left\vert \left( \pr^w_r f \right)^{(lm)}(r) \right\vert &\leq \sum\limits_{n=0}^w \frac{\left\vert \pr_r^{w-n} (f^{(lm)}) \right\vert}{r^{n}}.
\end{align}
\item Let $X$ be a $S_r$-tangent vectorfield. Then, for for $l\geq1, m \in \{ -l, \dots, l\}$,
\begin{align} \label{resscalar2}
\left\vert \left( \Nd_N^w X \right)^{(lm)}_E(r) \right\vert &\leq \sum\limits_{n=0}^w \frac{\left\vert \pr_r^{w-n} \left(X_E^{(lm)}\right) \right\vert}{r^n}, \\
\left\vert \left( \Nd_N^w X \right)^{(lm)}_H(r) \right\vert &\leq \sum\limits_{n=0}^w \frac{\left\vert \pr_r^{w-n} \left(X_H^{(lm)} \right) \right\vert}{r^n}.
\end{align}
Moreover,
\begin{align*}
\pr_r \left( r \Divd X \right)  &= r \Divd \left( \Ndn X \right), \\ 
\pr_r \left( r \Curld X \right) &= r \Curld \left( \Ndn X \right).
\end{align*}
\item Let $V$ be a $S_r$-tangent tracefree symmetric $2$-tensor. Then, for $l\geq2, m \in \{ -l, \dots, l\}$,
\begin{align} \label{resscalar3}
\left\vert \left( \Nd_N^n V \right)^{(lm)}_\psi(r) \right\vert &\leq \sum\limits_{n=0}^w \frac{\left\vert \pr_r^{w-n} \left(V_\psi^{(lm)}\right) \right\vert}{r^n}, \\
 \left\vert \left( \Nd_N^w V \right)^{(lm)}_\phi(r) \right\vert &\leq \sum\limits_{n=0}^w \frac{\left\vert \pr_r^{w-n} \left(V_\phi^{(lm)} \right) \right\vert}{r^n}.
\end{align}
Moreover, 
\begin{align*}
\Ndn \left( r \Divd V \right) = r \Divd \left( \Ndn V \right).
\end{align*}
\end{enumerate}
\end{lemma}
A proof is provided in Appendix \ref{sec:appidentities}.

\begin{lemma} \label{lem:inversetoDDDD} Let $w\geq0$ be an integer. Let $f^{[\geq1]}, f^{*[\geq1]} \in H^w(S_r)$ be scalar functions and $X^{[\geq2]} \in \HH^w(S_r)$ a vectorfield. Then the inverse maps
\begin{align*}
&\cDd_1^{-1}: (f^{[\geq1]}, f^{*[\geq1]}) \mapsto \cDd_1^{-1}(f^{[\geq1]}, f^{*[\geq1]}), \\
&\cDd_2^{-1}: (X^{[\geq2]}) \mapsto \cDd_2^{-1}(X^{[\geq2]}),
\end{align*}
into vectorfields and tracefree symmetric $2$-tensors, defined such that $\cDd_1^{-1}(f^{[\geq1]}, f^{*[\geq1]})$ and $\cDd_2^{-1}(X^{[\geq2]})$ respectively solve on $S_r$
\begin{align*}
&\begin{cases}
\Divd \cDd_1^{-1}(f^{[\geq1]}, f^{*[\geq1]}) = f^{[\geq1]}, \\
\Curld \cDd_1^{-1}(f^{[\geq1]}, f^{*[\geq1]}) = f^{*[\geq1]},
\end{cases}\\
&\Divd \cDd_2^{-1}(X^{[\geq2]}) = X^{[\geq2]},
\end{align*}
are well-defined and continuous maps into $\HH^{w+1}(S_r)$, respectively. Moreover, for any scalar function $f^{[\geq1]}$ on $S_r$,
\begin{align*}
\left( \cDd_1^{-1}(f^{[\geq1]},0) \right)^{[\geq1]}_{H} &= 0, &\left( \cDd_1^{-1}(0,f^{[\geq1]}) \right)^{[\geq1]}_{E} &= 0, \\
\left( \cDd_2^{-1}(E^{(lm)}) \right)^{[\geq1]}_{\phi} &= 0, &\left( \cDd_2^{-1}(H^{(lm)}) \right)^{[\geq1]}_{\psi} &= 0.
\end{align*}
\end{lemma}
The above lemma is a consequence of Lemma \ref{lem:Chrlem1}, Remarks \ref{remmars30} and \ref{imageanalysis} and Propositions \ref{prop:EllipticityHodgejan} and \ref{prop:EllipticityHodgefeb23} and its proof is left to the reader.


\subsection{The implicit function theorem and Lipschitz estimates for operators}

For completeness we state the standard Implicit Function Theorem that is used in Sections \ref{ssec:Reduction1} and \ref{ssec:Reduction2}, see for example Theorem 2.5.7 in \cite{MarsdenImplicit} for a proof.
\begin{theorem} \label{thm:InverseMars14}
Let $X,Y,Z$ be Hilbert spaces. Let $U \subset X, V \subset Y$ be open subsets and $\FF: U \times V \to Z$ be a $C^r$-mapping, $r\geq1$. For some $x_0 \in U, y_0 \in V$ assume that the linearisation in the first argument $$D_1\FF\vert_{(x_0,y_0)}: X \to Z$$ is an isomorphism. Then there are open neighbourhoods $V_0 \subset V$ of $y_0$ and $W_0 \subset Z$ of $\FF(x_0,y_0)$ and a unique $C^r$-mapping $\GG: V_0 \times W_0 \to U$ such that for all $(y,z) \in V_0 \times W_0$,
\begin{align*}
\FF(\GG(y,z),y)=z.
\end{align*}
\end{theorem}

We also need the following lemma.
\begin{lemma} \label{lem:opth12j}
Let $X,Y,Z$ be Hilbert spaces and let $\varep>0$ be a real. Let $T: X \times Y \to Z$ be a $C^r$-mapping for $r\geq2$ in $B_\varep(0) \times B_\varep(0) \subset X \times Y$ such that for all $x \in X$,
\begin{align*}
T(x,0)=0.
\end{align*}
Then there exists a constant $C=C(X,Y,Z,T)>0$ such that the following holds.
\begin{itemize}
\item For $(x,y) \in B_\varep(0) \times B_\varep(0)\subset X \times Y$,
\begin{align*}
\Vert T(x,y) \Vert_{Z} \leq C \Vert y \Vert_{Y}.
\end{align*}
\item For $x,x' \in B_\varep(0) \subset X$ and $y \in B_\varep(0) \subset Y$,
\begin{align*}
\Vert T(x,y) -T(x',y) \Vert_Z \leq C \Vert x-x' \Vert_{X} \Vert y \Vert_Y.
\end{align*}
\end{itemize}
\end{lemma}
\begin{proof}
First,
\begin{align*}
\Vert T(x,y) \Vert_Z &= \Vert T(x,y)-T(x,0) \Vert_Z\\
&\leq \int\limits_0^1 \Vert D_2 T \vert_{(x,ty)}(y) \Vert_Z dt\\
&\leq C \Vert y \Vert_Y,
\end{align*}
where we used that the operator $T$ is $C^1$ on $B_\varep(0) \times B_\varep(0) \subset X \times Y$. Second,
\begin{align} \begin{aligned} 
\Vert T(x,y) -T(x',y) \Vert_Z &= \Vert T(x,y) - T(x,0) - T(x',y) + T(x',0) \Vert_Z \\
&\leq  \int\limits_0^1 \Vert  D_2T \vert_{(x,ty)} (y) - D_2T \vert_{(x',ty)} (y) \Vert_Z dt\\
&\leq \int\limits_0^1 \left( \int\limits_0^1  \Vert D_1D_2T \vert_{(sx+(1-s)x',ty)}(y)(x-x')  \Vert_Z ds \right)dt \\
&\leq C \Vert x-x' \Vert_X \Vert y \Vert_Y,
\end{aligned} \label{sep6} \end{align}
where we used that the operator $T$ is $C^2$ on $B_\varep(0) \times B_{\varep}(0) \subset X \times Y$. This finishes the proof of Lemma \ref{lem:opth12j}.
\end{proof}

\section{Precise statement of the main theorem} \label{sec:Exactstatements}

We are now in the position to state the precise version of our main theorem.
\begin{theorem}[Main theorem, precise version]  \label{MainTheorem12}
Let $\bar g$ be a Riemannian metric and $\bar k$ a symmetric $2$-tensor on $B_1$ that solve  
\begin{align*}
R(\bar g) &= \vert \bar k \vert_{ \bar g}^2, \\
\Div_{ \bar g}  \bar k &= 0,  \\
\tr_{\bar g}  \bar k &=0.
\end{align*}
There exists a universal constant $\varep >0$ such that if
\begin{align} \label{smallnesscondition23232}
\Vert ( \bar g- e,  \bar k) \Vert_{\HH^2(B_1) \times \HH^{1}(B_1)} < \varep,
\end{align}
where $e$ denotes the Euclidean metric, then there is an $\HH^2_{-1/2}$-asymptotically flat Riemannian metric $g$ on $\RRR^3$ and a symmetric $2$-tensor $k \in \HH^1_{-3/2}$ with $$(g, k) \vert_{B_1} = ( \bar g, \bar k)$$ such that on $\RRR^3$
\begin{align*}
R(g) &= \vert k \vert_{g}^2, \\
\Div_{ g}  k &= 0,  \\
\tr_{g}  k &=0,
\end{align*}
and bounded by
\begin{align*}
\Vert ( g- e,  k) \Vert_{\HH^2_{-1/2} \times \HH^{1}_{-3/2}} \lesssim \Vert ( \bar g- e,  \bar k) \Vert_{\HH^2(B_1) \times \HH^{1}(B_1)}.
\end{align*}
Moreover, if, in addition to \eqref{smallnesscondition23232}, $\bar g -e \in \HH^{w}(B_1)$ and $\bar k \in \HH^{w-1}(B_1)$ for an integer $w\geq3$, then
\begin{align} \label{estimateshigherreg24352}
\Vert ( g- e,  k) \Vert_{\HH^w_{-1/2} \times \HH^{w-1}_{-3/2}} \leq C_w \Vert (\bar g- e, \bar k) \Vert_{\HH^w(B_1) \times \HH^{w-1}(B_1)},
\end{align} 
where the constant $C_w>0$ depends only on $w$.
\end{theorem}


\section{The divergence equation for $k$} \label{sec:DivergenceEquation}

In this section we prove the following theorem.
\begin{theorem}[Extension of divergence-free tracefree symmetric $2$-tensors, precise version] \label{thm:kextension1} There exists a small universal constant $\varep>0$ such that the following holds.
\begin{enumerate}
\item \underline{Extension result.} Let $g$ be a given $\HH^2_{-1/2}$-asymptotically flat metric on $\RRR^3$ and $\bar k \in \HH^{1}(B_1)$ a symmetric $2$-tensor such that on $B_1$ 
\begin{align}\begin{aligned}
\Div_g \bar k &=0, \\
\tr_g \bar k &=0. \end{aligned} \label{ZeroHodge}
\end{align}
If
\begin{align}
\Vert g - e \Vert_{\HH^2_{-1/2}} + \Vert \bar k \Vert_{\HH^{1}(B_1)} < \varep, \label{eq:cond2feb}
\end{align}
then there exists a symmetric $2$-tensor $k \in \HH^{1}_{-3/2}$ with $k \vert_{B_1}= \bar k$ that solves on $\RRR^3$
\begin{align*}\begin{aligned}
\Div_g k &=0, \\
\tr_g k &=0, \end{aligned} 
\end{align*}
and is bounded by
\begin{align} \label{extensionestimation240343030300}
\Vert k \Vert_{\HH^{1}_{-3/2}} \lesssim \Vert \bar k \Vert_{\HH^{1}(B_1)}.
\end{align}
\item \underline{Iteration estimates.} Let $g$ and $g'$ be two $\HH^2_{-1/2}$-asymptotically flat metrics on $\RRR^3$ such that $$ g\vert_{B_1} = g' \vert_{B_1}$$ and $\bar k \in \HH^{1}(B_1)$ a symmetric $2$-tensor on $B_1$ that solves \eqref{ZeroHodge}  with respect to $g$ (and so for $g'$). Assume that for $(g,\bar k)$ and $(g', \bar k)$ the smallness condition \eqref{eq:cond2feb} holds and let $k, k' \in \HH^{1}_{-3/2}$ denote the two extensions of $\bar k$ constructed in part (1) of this theorem with the metrics $g$ and $g'$, respectively. Then it holds that \begin{align}
\Vert k - k' \Vert_{\HH^{1}_{-3/2}} \lesssim \Vert \bar k \Vert_{\HH^{1}(B_1)} \Vert g-g' \Vert_{\HH^2_{-1/2}}. \label{eq:kiterationestimatethm}
\end{align}
\item \underline{Higher regularity estimates.} If, in addition to \eqref{eq:cond2feb}, $g$ is an $\HH^{w}_{-1/2}$-asymptotically flat metric and $\bar k \in \HH^{w-1}(B_1)$ for an integer $w\geq3$, then the $k$ constructed in part (1) of this theorem satisfies
\begin{align*}
\Vert k \Vert_{\HH^{w-1}_{-3/2}} \lesssim \Vert \bar k \Vert_{\HH^1(B_1)} \Vert g-e \Vert_{\HH^w_{-1/2}} + C_w \Vert \bar k \Vert_{\HH^{w-1}(B_1)},
\end{align*}
where the constant $C_w>0$ depends only on $w$.
\end{enumerate}
\end{theorem}

Before proving Theorem \ref{thm:kextension1}, we analyse the divergence and trace mapping on $\HH^w_{-1/2}$-asymptotically flat metrics on $\RRR^3$.


\subsection{Analysis of operators on $\HH^w_{-1/2}$-asymptotically flat metrics} \label{sec:analysisonafsets}

Recall from Section \ref{basicsection} that for a Riemannian metric $g$ and a symmetric $2$-tensor $V$,  the divergence, trace and tracefree part of $V$ are respectively defined as
\begin{align*}
(\Div_g V )_j &:= \nabla^i V_{ij}, \\
\tr_g V &:= g^{ij} V_{ij}, \\
\widehat{V}^g &:= V- \frac{1}{3} \tr_g (V) g,
\end{align*}
where $\nab$ denotes the covariant derivative of $g$. The next lemma shows basic properties of the divergence operator.
\begin{lemma} \label{lem:DivergenceBoundedOperator}
There is a universal $\varep>0$ such that the following holds.
\begin{enumerate} 
\item Let $w\geq2$ be an integer. The mapping
\begin{align*}
\Div: (V,g) \mapsto \Div_g V
\end{align*}
is a smooth mapping from $\HHm \times B_{\varep}(e)$ to $\HHn$, where 
$$B_{\varep}(e) := \Big\{ g-e \in \HH^w_{-1/2}: \Vert g-e \Vert_{\HH^2_{-1/2}} < \varep \Big\}.$$
Furthermore, $\Div$ maps $\HHmo \times B_{\varep}(e)$ into $\HHno$.
\item For all Riemannian metrics $g$ such that
\begin{align*}
\Vert g-e \Vert_{\HH^2_{-1/2}} < \varep,
\end{align*}
it holds that for all symmetric $2$-tensors $V \in \HH^1_{-3/2}$,
\begin{align}
\Vert \Div_g (V) \Vert_{\HH^{0}_{-5/2}} \lesssim \Vert V \Vert_{\HH^{1}_{-3/2}}\label{eq:firstdivfeb3}
\end{align}
\item For all Riemannian metrics $g,g'$ with
\begin{align*}
\Vert g-e \Vert_{\HH^2_{-1/2}}< \varep, \Vert g'-e \Vert_{\HH^2_{-1/2}} < \varep
\end{align*}
it holds that for all symmetric $2$-tensors $V \in \HH^1_{-3/2}$
\begin{align}
\Vert \Div_g V - \Div_{g'} V \Vert_{\HH^{0}_{-5/2}} \lesssim \Vert g-g' \Vert_{\HH^2_{-1/2}} \Vert V \Vert_{\HH^{1}_{-3/2}}.\label{eq:divLipschitzestimate3}
\end{align}
\item Let $w\geq3$ be an integer. For all Riemannian metrics $g$ such that
\begin{align*}
\Vert g-e \Vert_{\HH^2_{-1/2}} < \varep, \, \Vert g-e \Vert_{\HH^w_{-1/2}} < \infty,
\end{align*}
it holds that for all symmetric $2$-tensors $V \in \HH^{w-1}_{-3/2}$
\begin{align}
\begin{aligned}
\Vert \Div_e V - \Div_g V \Vert_{\HH^{w-2}_{-5/2}} \lesssim& \Vert g-e\Vert_{\HH^2_{-1/2}} \Vert V \Vert_{\HH^{w-1}_{-3/2}} + \Vert V \Vert_{\HH^1_{-3/2}} \Vert g-e \Vert_{\HH^{w}_{-1/2}} \\
&+ C_w \Vert V \Vert_{\HH^1_{-3/2}} \Vert g-e \Vert_{\HH^2_{-1/2}}.
\end{aligned} \label{eqestk493424}
\end{align}
\end{enumerate}
\end{lemma}

\begin{proof}[Proof of Lemma \ref{lem:DivergenceBoundedOperator}]
{\bf Proof of part (1).} By definition, 
\begin{align*} 
\left( \Div_g V \right)_i = g^{ab} \left( \pr_a V_{bi} - \Ga^{j}_{ab} V_{ji} - \Ga^j_{ai} V_{jb} \right), 
\end{align*}
where $\Ga^j_{ia} = \half g^{jb} \left( \pr_i g_{ba} + \pr_a g_{bi} - \pr_b g_{ia}\right)$ denote the Christoffel symbols. Therefore, we schematically have
\begin{align}
\Div_g V = g^{-1} \pr g \, V + g^{-1} \pr V. \label{av113}
\end{align}
By \eqref{av113} and Lemma \ref{GInverseAnalysis}, it follows that for $\varep>0$ sufficiently small, $\Div$ is a smooth mapping from $ \HHm \times B_{\varep}(e)$ to $\HHn$. The restriction of $\Div$ to $V \in \HHmo$ clearly maps into $\HHno$.\\

{\bf Proof of parts (2) and (3).} By \eqref{av113},
\begin{align} \begin{aligned}
\Vert \Div_g V \Vert_{\HH^{0}_{-5/2}} &\lesssim \Vert g^{-1} \pr g V \Vert_{\HH^{0}_{-5/2}} + \Vert g^{-1}\pr V \Vert_{\HH^{0}_{-5/2}}. \label{eq:jan295}
\end{aligned} \end{align}
By Lemma \ref{SobolevEmbeddingsAndNonlinear}, product estimates as in Lemma \ref{ProductEstimates}, and Lemma \ref{GInverseAnalysis}, there exists an $\varep>0$ such that if $g \in B_{\varep}(e) \subset \HH^{2}_{-1/2}$, then
\begin{align} \begin{aligned}
\Vert g^{-1} \pr V \Vert_{\HH^{0}_{-5/2}} &\lesssim \Vert V \Vert_{\HH^{1}_{-3/2}},\\
\Vert g^{-1} \pr g \, V \Vert_{\HH^{0}_{-5/2}} &\lesssim \Vert \pr g \, V \Vert_{\HH^{0}_{-5/2}} \\
&\lesssim \Vert \pr g \Vert_{\HH^{1}_{-3/2}} \Vert V \Vert_{\HH^{1}_{-3/2}}\\
&\lesssim \Vert g -e\Vert_{\HH^{2}_{-1/2}} \Vert V \Vert_{\HH^{1}_{-3/2}}. 
\end{aligned} \label{eq:jan262}
\end{align}
Plugging \eqref{eq:jan262} into \eqref{eq:jan295} proves \eqref{eq:firstdivfeb3}. The estimate \eqref{eq:divLipschitzestimate3} is proved similarly and left to the reader. \\

{\bf Proof of part (4).} By using \eqref{av113}, we have schematically
\begin{align} \label{diff343524}
\Vert \Div_e V - \Div_g V \Vert_{\HH^{w-2}_{-5/2}} \lesssim \Vert g^{-1} \pr g V \Vert_{\HH^{w-2}_{-5/2}} + \Vert (g^{-1}-e) \pr V \Vert_{\HH^{w-2}_{-5/2}}.
\end{align}
The first term on the right-hand side of \eqref{diff343524} is bounded by using product estimates as in Lemma \ref{ProductEstimates}, see also Lemma \ref{GInverseAnalysis},
\begin{align*}
\Vert g^{-1} \pr g V \Vert_{\HH^{w-2}_{-5/2}}  \lesssim& \Vert g-e \Vert_{\HH^2_{-1/2}} \Vert V \Vert_{\HH^{w-1}_{-3/2}} + \Vert g-e \Vert_{\HH^w_{-1/2}} \Vert V \Vert_{\HH^1_{-3/2}} \\
&+ C_w \Vert g-e \Vert_{\HH^2_{-1/2}} \Vert V \Vert_{\HH^1_{-3/2}}.
\end{align*}
The second term on the right-hand side of \eqref{diff343524} is bounded similarly by
\begin{align*}
\Vert (g^{-1}-e) \pr V \Vert_{\HH^{w-2}_{-5/2}} \lesssim& \Vert g-e \Vert_{\HH^2_{-1/2}} \Vert V \Vert_{\HH^{w-1}_{-3/2}} + \Vert g-e \Vert_{\HH^w_{-1/2}} \Vert V \Vert_{\HH^1_{-3/2}} \\
&+ C_w \Vert g-e \Vert_{\HH^2_{-1/2}} \Vert V \Vert_{\HH^1_{-3/2}}.
\end{align*}
Plugging the above two into \eqref{diff343524} proves \eqref{eqestk493424} and hence finishes the proof of Lemma \ref{lem:DivergenceBoundedOperator}. \end{proof}


\begin{lemma} \label{lem:conttrace}
There exists a universal constant $\varep>0$ such that the following holds.
\begin{itemize}
\item Let $w\geq2$ be an integer. The mapping
\begin{align*}
\tr : (V,g) \mapsto \tr_g V
\end{align*}
is a smooth mapping from $\HHm \times B_{\varep}(e)$ to $\HHm$, where 
$$B_{\varep}(e) := \Big\{ g-e \in \HH^w_{-1/2}: \Vert g-e \Vert_{\HH^2_{-1/2}} < \varep \Big\}.$$
Furthermore, $\tr$ maps $\HHmo \times B_{\varep}(e)$ into $\HHno$.
\item For all Riemannian metrics $g$ with
$$\Vert g-e \Vert_{\HH^2_{-1/2}} < \varep,$$
it holds that for all symmetric $2$-tensors $V \in \HH^{1}_{-3/2}$,
\begin{align*}
\Vert tr_g V \Vert_{H^{1}_{-3/2}} \lesssim \Vert V \Vert_{\HH^{1}_{-3/2}}.
\end{align*}
\item For two Riemannian metrics $g$ and $g'$ such that 
$$\Vert g-e \Vert_{\HH^2_{-1/2}}, \, \Vert g'-e \Vert_{\HH^2_{-1/2} }< \varep,$$ 
it holds that for all symmetric $2$-tensors $V\in \HH^1_{-3/2}$,
\begin{align*}
\Vert \tr_g V -\tr_{g'} V \Vert_{\HH^1_{-3/2}} \lesssim \Vert g-g' \Vert_{\HH^2_{-1/2}} \Vert V \Vert_{\HH^1_{-3/2}}.
\end{align*}
\item Let $w\geq3$ be an integer. For all Riemannian metrics $g$ such that
\begin{align*}
\Vert g-e \Vert_{\HH^2_{-1/2}} < \varep, \, \Vert g-e \Vert_{\HH^w_{-1/2}} < \infty,
\end{align*}
it holds that for all symmetric $2$-tensors $V \in \HH^{w-1}_{-3/2}$
\begin{align}
\begin{aligned}
\Vert \tr_e V - \tr_g V \Vert_{\HH^{w-1}_{-5/2}} \lesssim&  \Vert g-e\Vert_{\HH^2_{-1/2}} \Vert V \Vert_{\HH^{w-1}_{-3/2}} + \Vert V \Vert_{\HH^1_{-3/2}} \Vert g-e \Vert_{\HH^{w}_{-1/2}} \\
&+ C_w \Vert V \Vert_{\HH^1_{-3/2}} \Vert g-e \Vert_{\HH^2_{-1/2}}.
\end{aligned} \label{eqestk49342422222}
\end{align}
\end{itemize}
\end{lemma}
The proof of Lemma \ref{lem:conttrace} is similar to the proof of Lemma \ref{lem:DivergenceBoundedOperator} and left to the reader. \\


Lemmas \ref{lem:DivergenceBoundedOperator} and \ref{lem:conttrace} imply the following corollary. The proof is left to the reader.
\begin{corollary} \label{remark:Bounded}
There exists a universal constant $\varep>0$ such that the following holds. 
\begin{itemize}
\item Let $w\geq2$ be an integer. The mapping
\begin{align*}
(V,g) \mapsto \Div_g \left( \widehat{V}^g \right)
\end{align*}
is smooth from $\HHm \times B_{\varep}(e)$ to $\HHn$, where 
$$B_{\varep}(e) := \Big\{ g-e \in \HH^w_{-1/2}: \Vert g-e \Vert_{\HH^2_{-1/2}} < \varep \Big\}.$$ 
Furthermore, the restriction of this mapping to $V \in \HHmo$ maps into $\HHno$. 
\item For a Riemannian metric $g$ on $\RRR^3$ such that
\begin{align}
\Vert g - e \Vert_{\HH^2_{-1/2}} < \varep, \label{rem:Bounded}
\end{align}
it holds that for all symmetric $2$-tensors $V\in \HH^1_{-3/2}$,
\begin{align*}
\Vert  \widehat{V}^g \Vert_{\HH^{1}_{-3/2}} &\lesssim \Vert V \Vert_{\HH^1_{-3/2}},\\
\left\Vert \Div_g \left( \widehat{V}^g \right) \right\Vert_{\HH^0_{-5/2}} &\lesssim \Vert V \Vert_{\HH^1_{-3/2}}.
\end{align*}
\item For two Riemannian metrics $g$ and $g'$ on $\RRR^3$ such that $$\Vert g-e \Vert_{\HH^2_{-1/2}}<\varep, \Vert g'-e \Vert_{\HH^2_{-1/2}}< \varep,$$
it holds that for all symmetric $2$-tensors $V \in \HH^1_{-3/2}$,
\begin{align*}
\left\Vert \widehat{V}^g - \widehat{V}^{g'}  \right\Vert_{\HH^{1}_{-3/2}} &\lesssim \Vert g-g' \Vert_{\HH^2_{-1/2}} \Vert V \Vert_{\HH^1_{-3/2}}, \\
\left\Vert  \Div_g \left( \widehat{V}^g \right)-  \Div_{g'} \left( \widehat{V}^{g'} \right) \right\Vert_{\HH^0_{-5/2}} &\lesssim \Vert g-g' \Vert_{\HH^2_{-1/2}} \Vert V \Vert_{\HH^1_{-3/2}}.
\end{align*}
\end{itemize}
\end{corollary}


\subsection{Reduction to the Euclidean case} \label{ssec:Reduction1}

In this section, we prove Theorem \ref{thm:kextension1} under the assumption of Lemma \ref{prop:EuclideanSurjectivity} below which is proved in Section \ref{sec:EuclideanSurjectivity}. First, as an intermediate step, we prove the next proposition.
\begin{proposition} \label{prop:Dez191}
There is a universal constant $\varep>0$ such that the following holds.
\begin{enumerate}
\item \underline{Existence.} Let $g$ be an $\HH^2_{-1/2}$-asymptotically flat metric and $\rho \in \ol{\HH}^{0}_{-5/2}$ a $1$-form on $\RRR^3$ such that
\begin{align}
\Vert g -e \Vert_{\HH^2_{-1/2}} + \Vert \rh \Vert_{\ol{\HH}^{0}_{-5/2}} <\varep.\label{eq:condsmall2feb}
\end{align}
Then there exists $k \in \ol{\HH}^{1}_{-3/2}$ solving on $\RRRwo$
\begin{align*}
\begin{aligned}
\Div_g k &= \rh,\\
\tr_g k &= 0
 \end{aligned}
\end{align*}
and bounded by
\begin{align}
\Vert k \Vert_{\ol{\HH}^{1}_{-3/2}} \lesssim \Vert \rho \Vert_{\ol{\HH}^{0}_{-5/2}}. \label{12j1p}
\end{align}
\item \underline{Iteration estimates.} Moreover, for two pairs $(g,\rh)$ and $(g',\rh')$ satisfying \eqref{eq:condsmall2feb}, the respectively constructed $k, k'$ satisfy
\begin{align} \label{eq:mars15secondest1}
\Vert k-k' \Vert_{\ol{\HH}^{1}_{-3/2}} \lesssim \Vert \rh-\rh' \Vert_{\ol{\HH}^{0}_{-5/2}} + \Vert g-g' \Vert_{\HH^2_{-1/2}} \Vert \rh \Vert_{\ol{\HH}^{0}_{-5/2}}.
\end{align}
\item \underline{Higher regularity estimates.} If, in addition to \eqref{eq:condsmall2feb}, $g$ is an $\HH^{w}_{-1/2}$-asymptotically flat metric and $\rh \in \ol{\HH}^{w-2}_{-5/2}$ for an integer $w\geq3$, then
\begin{align} \begin{aligned}
\Vert k \Vert_{\ol{\HH}^{w-1}_{-3/2}} \lesssim \Vert \rh \Vert_{\ol{\HH}^0_{-5/2}} \Vert g-e \Vert_{\HH^{w}_{-1/2}} +  \Vert \rh \Vert_{\ol{\HH}^{w-2}_{-5/2}} + C_w \Vert \rh \Vert_{\ol{\HH}^0_{-5/2}},
\end{aligned} \label{hre3435389929} \end{align}
where the constant $C_w>0$ depends on $w$.
\end{enumerate}
\end{proposition}
To prove Proposition \ref{prop:Dez191}, we assume the following essential lemma proved in Section \ref{sssec:EuclideanSurjectivity}. 
\begin{lemma}[Surjectivity at the Euclidean metric] \label{prop:EuclideanSurjectivity} The following holds.
\begin{enumerate}
\item \underline{Surjectivity.} For any $\rh \in \ol{\HH}^{0}_{-5/2}$, there exists a symmetric $2$-tensor $k\in \ol{\HH}^1_{-3/2}$ solving on $\RRRwo$
\begin{align*}
\begin{aligned}
\Div_e k &= \rh, \\ 
 \tr_e k &= 0
 \end{aligned}
\end{align*}
and bounded by
\begin{align} \label{firstkestsepsep}
\Vert k \Vert_{\ol{\HH}^{1}_{-3/2}} \lesssim \Vert \rh \Vert_{\ol{\HH}^{0}_{-5/2}}.
\end{align}
In other words, the mapping $k \mapsto \Div_e ( \hat{k}^e )$ from $\overline{\HH}^{1}_{-3/2}$ to $\ol{\HH}^{0}_{-5/2}$ is surjective and has a bounded right-inverse. 
\item \underline{Higher regularity.} If in addition it holds that $\rh \in \ol{\HH}^{w-2}_{-5/2}$ for an integer $w\geq3$, then
\begin{align} \label{firstkestsepsep2}
\Vert k \Vert_{\ol{\HH}^{w-1}_{-3/2}} \lesssim \Vert \rh \Vert_{\ol{\HH}^{w-2}_{-5/2}} + C_w \Vert \rh \Vert_{\ol{\HH}^0_{-5/2}}.
\end{align}
\end{enumerate}
\end{lemma}

For the rest of this section denote
\begin{align}
\ol{\NN_e} := \mathrm{ker} \left( \Div_e \circ \, \left( \hat{\,\,\,}^e \right) \right)^\perp \subset \overline{\HH}^{1}_{-3/2},
 \label{defNNmars15}
\end{align}
where ${}^\perp$ denotes the orthogonal complement with respect to the scalar product on $\overline{\HH}^{1}_{-3/2}$.

\begin{remark} \label{remark459434} The mapping $k \mapsto \Div_e ( \hat{k}^e )$ is clearly a bounded linear mapping from $\overline{\HH}^{1}_{-3/2}$ into $\ol{\HH}^{0}_{-5/2}$. Therefore its kernel is closed, and we have the splitting
\begin{align*}
\overline{\HH}^{1}_{-3/2} = \ol{\NN_e} \oplus \mathrm{ker} \left( \Div_e \circ \, \left( \hat{\,\,\,}^e \right) \right).
\end{align*}
\end{remark}


We are now ready to prove Proposition \ref{prop:Dez191} by Lemma \ref{prop:EuclideanSurjectivity} and the Implicit Function Theorem \ref{thm:InverseMars14}.
\begin{proof}[Proof of Proposition \ref{prop:Dez191}] We prove each part separately. \\

{\bf Proof of part (1).} We apply the Implicit Function Theorem \ref{thm:InverseMars14} to the mapping
\begin{align*}
\F: \ol{\NN_e} \times \HH^2_{-1/2} &\to \ol{\HH}^{0}_{-5/2} \\ 
(k,h) &\mapsto \rh := \Div_{e+h} \Big( \hat{k}^{e+h} \Big),
\end{align*}
where $h$ is a symmetric $2$-tensor.
We verify that $\FF$ satisfies the assumptions of Theorem \ref{thm:InverseMars14} at $(k,h)=0$. On the one hand, by Corollary \ref{remark:Bounded}, there exists an $\varep'>0$ such that $\F$ is a smooth mapping from $\ol{\NN_e} \times B_{\varep'}(0)$ to $\ol{\HH}^{0}_{-5/2}$, where $B_{\varep'}(0) \subset \HH^2_{-1/2}$, and $\F(0,0)=0$. On the other hand, by Lemma \ref{prop:EuclideanSurjectivity}, the definition of $\ol{\NN_e}$ in \eqref{defNNmars15} and Remark \ref{remark459434}, the linearisation in the first argument at $h=0$,
$$D_1 \F \vert_{h=0}: \ol{\NN_e} \to \ol{\HH}^{0}_{-5/2},$$ 
is an isomorphism.\\

Consequently, by Theorem \ref{thm:InverseMars14}, there exists an open neighbourhood $V_0 \subset B_{\varep'}(0) \times \ol{\HH}^{0}_{-5/2}$ of $(h, \rh)=(0,0)$ and a unique mapping $\GG: V_0 \to \ol{\HH}^{1}_{-3/2}$ into symmetric $2$-tensors such that on $\RRRwo$ 
\begin{align*}
\Div_{e+h}  \left( \widehat{\GG}^{e+h}(h,\rh) \right)=\rh
\end{align*}
for all $(h, \rh) \in V_0$. We note that this mapping $\GG$ is smooth. By the uniqueness of $\GG$ and because $\F(0,h)=0$ for all $h \in B_{\varep'}(0)$, it holds that for all $(h,0) \in V_0$, 
\begin{align}
\GG(h,0)=0. \label{eq:uniquehel}
\end{align}

For $(h, \rh) \in V_0$, let $k := \widehat{\GG}^{e+h}(h, \rh)$. Then, on $\RRRwo$,
\begin{align*} \begin{aligned}
\Div_{e+h} k = \rh, \\
\tr_{e+h} k =0.\end{aligned}
\end{align*}

Let $0< \varep< \varep'$ be a sufficiently small real such that 
$$(h,\rh) \in B_\varep(0) \times B_{\varep}(0) \subset V_0$$
and such that we can apply Corollary \ref{remark:Bounded}. Then we have, by using also Lemma \ref{lem:opth12j},
\begin{align*} \begin{aligned}
\Vert k \Vert_{\ol{\HH}^{1}_{-3/2}} &= \Vert \widehat{\GG}^{e+h}(h, \rh) \Vert_{\ol{\HH}^{1}_{-3/2}} \\
&\lesssim \Vert \GG(h,\rh) \Vert_{\ol{\HH}^{1}_{-3/2}}\\
&\lesssim \Vert \mathcal{G}(h,\rh) - \underbrace{\mathcal{G}(h,0)}_{=0} \Vert_{\ol{\HH}^{1}_{-3/2}} \\
&\lesssim \Vert \rh \Vert_{\ol{\HH}^{0}_{-5/2}}.
\end{aligned} 
\end{align*}
This proves \eqref{12j1p}. \\


{\bf Proof of part (2).} Let two pairs $(h, \rh), (h',\rh') \in B_\varep(0) \times B_{\varep}(0) \subset V_0$. By Lemma \ref{lem:opth12j} and \eqref{eq:uniquehel}, it follows that for $\varep>0$ sufficiently small,
\begin{align} \begin{aligned}
\Vert \GG(h, \rh)-\GG(h',\rh) \Vert_{\ol{\HH}^{1}_{-3/2}} &\lesssim \Vert h-h' \Vert_{\HH^{2}_{-1/2}(\RRR^3)} \Vert \rh \Vert_{\ol{\HH}^{0}_{-3/2}}, \\
\Vert \GG(h, \rh) \Vert_{\ol{\HH}^{1}_{-3/2}} &\lesssim \Vert \rh \Vert_{\ol{\HH}^{0}_{-3/2}}.
\end{aligned}\label{eq:estmars15} \end{align}
Moreover, by the smoothness of $\GG$, for $\varep>0$ sufficiently small,
\begin{align}
\Vert \GG(h', \rh)-\GG(h',\rh') \Vert_{\ol{\HH}^{1}_{-3/2}} \lesssim \Vert \rh-\rh' \Vert_{\ol{\HH}^{0}_{-3/2}}. \label{eq:estmars15212j}
\end{align}

Let $k:= \widehat{\GG}^{e+h}(h,\rh), k':= \widehat{\GG}^{e+h'}(h',\rh')$. By \eqref{eq:estmars15} and \eqref{eq:estmars15212j},  Lemma \ref{lem:conttrace} and Corollary  \ref{remark:Bounded}, for $\varep>0$ sufficiently small,
\begin{align*} \begin{aligned}
&\Vert k-k' \Vert_{\ol{\HH}^{1}_{-3/2}} \\
\lesssim& \left\Vert \left[ \GG(h,\rh) -\GG(h',\rh') \right]^{\wedge_{e+h'}}  \right\Vert_{\ol{\HH}^{1}_{-3/2}} + \Vert (e+h') \tr_{e+h'} \GG(h,\rh) - (e+h) \tr_{e+h} \GG(h,\rh) \Vert_{\ol{\HH}^{1}_{-3/2}}\\
\lesssim& \Vert \GG(h,\rh) -\GG(h',\rh') \Vert_{\ol{\HH}^{1}_{-3/2}} + \Vert h' -h \Vert_{\HH^2_{-1/2}} \Vert \GG(h,\rh) \Vert_{\ol{\HH}^{1}_{-3/2}} \\
\lesssim& \Vert \GG(h,\rh) -\GG(h',\rh) \Vert_{\ol{\HH}^{1}_{-3/2}} + \Vert \GG(h',\rh) -\GG(h',\rh') \Vert_{\ol{\HH}^{1}_{-3/2}} + \Vert h'-h\Vert_{\HH^2_{-1/2}} \Vert \GG(h,\rh) \Vert_{\ol{\HH}^{1}_{-3/2}} \\
\lesssim& \Vert h'-h\Vert_{\HH^2_{-1/2}} \Vert \rh \Vert_{\ol{\HH}^{0}_{-5/2}}  + \Vert \rh -\rh' \Vert_{\ol{\HH}^{0}_{-5/2}}. 
\end{aligned} 
\end{align*}
This proves \eqref{eq:mars15secondest1}. \\

{\bf Proof of part (3).} By part (1) of this proposition, for given $\rh \in \ol{\HH}^{w-2}_{-5/2}$, $w\geq3$, let 
$$k \in \ol{\NN_e} := \mathrm{ker} \left( \Div_e \circ \, \left( \hat{\,\,\,}^e \right) \right)^\perp \subset \overline{\HH}^{1}_{-3/2}$$ 
on $\RRRwo$ be solution to 
\begin{align} \begin{aligned}
\Div_g k =& \rh, \\
\tr_g k =&0.
\end{aligned} \label{aligne343412535} \end{align}
By the definition of $\ol{\NN_e}$, it holds further that
\begin{align*}
\hat{k}^e := k - \frac{1}{3} \tr_e k \, e \in \ol{\NN_e}.
\end{align*}

On the other hand, by Lemma \ref{prop:EuclideanSurjectivity}, let $k' \in \ol{\HH}^1_{-3/2}$ be the constructed solution to
\begin{align} \begin{aligned}
\Div_e k' =& \rh', \\
\tr_e k' =&0
\end{aligned} \label{syshih2243} \end{align}
for $$\rh' := \Div_e k - \frac{1}{3} d(\tr_e k)= \Div_e(\hat{k}^e).$$ 

By \eqref{aligne343412535} and \eqref{syshih2243}, we have 
$\hat{k}^e - k' \in \ol{\NN_e},$
that is, $\hat{k}^e$ equals $k'$ up to a part of $k'$ in $\mathrm{ker} \left( \Div_e \circ \, \left( \hat{\,\,\,}^e \right) \right)$. \\

Therefore, by estimate \eqref{firstkestsepsep2} for $k'$ in Lemma \ref{prop:EuclideanSurjectivity}, it follows that for integers $w\geq3$,
\begin{align} \begin{aligned}
&\Vert k \Vert_{\ol{\HH}^{w-1}_{-3/2}} \\
\leq & \Vert \hat{k}^e \Vert_{\ol{\HH}^{w-1}_{-3/2}} + \frac{1}{3}\Vert \tr_e k \Vert_{\ol{\HH}^{w-1}_{-3/2}} \\
\leq& \Vert k' \Vert_{\ol{\HH}^{w-1}_{-3/2}} + \frac{1}{3} \Vert \tr_e k - \tr_g k \Vert_{\ol{\HH}^{w-1}_{-3/2}} \\
\lesssim& \Vert \Div_e k' \Vert_{\ol{\HH}^{w-2}_{-5/2}} +C\Vert \Div_e k' \Vert_{\ol{\HH}^{0}_{-5/2}} +  \Vert \tr_e k - \tr_g k \Vert_{\ol{\HH}^{w-1}_{-3/2}} \\
\lesssim& \left\Vert \Div_e k - \frac{1}{3} d(\tr_e k) \right\Vert_{\ol{\HH}^{w-2}_{-5/2}} + C\left\Vert \Div_e k - \frac{1}{3} d(\tr_e k) \right\Vert_{\ol{\HH}^{0}_{-5/2}}+  \Vert \tr_e k - \tr_g k \Vert_{\ol{\HH}^{w-1}_{-3/2}} \\
\lesssim& \Vert \Div_e k - \Div_g k \Vert_{\ol{\HH}^{w-2}_{-5/2}} + \Vert \Div_g k \Vert_{\ol{\HH}^{w-2}_{-5/2}} + \Vert \tr_e k - \tr_g k \Vert_{\ol{\HH}^{w-1}_{-3/2}} +C \Vert k \Vert_{\ol{\HH}^1_{-3/2}} \\
\lesssim& \Vert \Div_e k - \Div_g k \Vert_{\ol{\HH}^{w-2}_{-5/2}} + \Vert \rh \Vert_{\ol{\HH}^{w-2}_{-5/2}} + \Vert \tr_e k - \tr_g k \Vert_{\ol{\HH}^{w-1}_{-3/2}}+C_w \Vert \rh \Vert_{\ol{\HH}^0_{-5/2}},
\end{aligned} \label{khigher343524} \end{align}
where we used part (1) of this proposition and \eqref{aligne343412535}. \\

By Lemmas \ref{lem:DivergenceBoundedOperator} and \ref{lem:conttrace} applied to the right-hand side of \eqref{khigher343524}, we have
\begin{align*}
\Vert k \Vert_{\ol{\HH}^{w-1}_{-3/2}} \lesssim&  \Vert g-e\Vert_{\HH^2_{-1/2}} \Vert k \Vert_{\ol{\HH}^{w-1}_{-3/2}} + \Vert k \Vert_{\HH^1_{-3/2}} \Vert g-e \Vert_{\HH^{w}_{-1/2}} + \Vert \rh \Vert_{\ol{\HH}^{w-2}_{-5/2}} \\
&+ C_w \Big( \Vert k \Vert_{\ol{\HH}^1_{-3/2}} \Vert g-e \Vert_{\HH^2_{-1/2}} + \Vert \rh \Vert_{\ol{\HH}^0_{-5/2}}\Big).
\end{align*}
Therefore, for $\varep>0$ sufficiently small, we can absorb the first term on the right-hand side and get
\begin{align*}
\Vert k \Vert_{\ol{\HH}^{w-1}_{-3/2}} \lesssim& \Vert k \Vert_{\ol{\HH}^1_{-3/2} }\Vert g-e \Vert_{\HH^{w}_{-1/2}} + \Vert \rh \Vert_{\ol{\HH}^{w-2}_{-5/2}} \\
&+ C_w \Big( \Vert k \Vert_{\ol{\HH}^1_{-3/2}} \Vert g-e \Vert_{\HH^2_{-1/2}} + \Vert \rh \Vert_{\ol{\HH}^0_{-5/2}}\Big) \\
\lesssim& \Vert \rh \Vert_{\ol{\HH}^0_{-5/2}} \Vert g-e \Vert_{\HH^{w}_{-1/2}} + \Vert \rh \Vert_{\ol{\HH}^{w-2}_{-5/2}} +C_w \Vert \rh \Vert_{\ol{\HH}^0_{-5/2}}.
\end{align*}
This finishes the proof of Proposition \ref{prop:Dez191}. \end{proof}


We now turn to the proof of Theorem \ref{thm:kextension1}. 
\begin{proof}[Proof of Theorem \ref{thm:kextension1}] We prove the three parts of Theorem \ref{thm:kextension1} separately.\\

{\bf Proof of Part 1:} Let the symmetric $2$-tensor $\bar k \in \HH^{1}(B_1)$ solve on $B_1$
\begin{align*}
\Div_g \bar k =0,\\
\tr_g \bar k= 0.
\end{align*}
Using Proposition \ref{functionextcts}, extend $ \bar k$ to a symmetric $2$-tensor $\check k \in \HH^{1}_{loc}(\RRR^3)$. We can assume without loss of generality that $\check k$ is $g$-tracefree and
\begin{align}
\Vert \check k \Vert_{\HH^{1}_{-3/2}} \lesssim \Vert \bar k \Vert_{\HH^{1}(B_1)}. \label{eq:bound116}
\end{align}
Indeed, for $\Vert g-e \Vert_{\HH^2_{-1/2}}$ small enough, multiplying by a cut-off function and taking the $g$-tracefree part are both continuous endomorphisms of $\HH^{1}_{loc}(\RRR^3)$, see Corollary \ref{remark:Bounded}.\\

Let $ \rh:= \Div_g \check k$. For $\varep>0$ small enough, by Lemma \ref{lem:DivergenceBoundedOperator} and \eqref{eq:bound116},
\begin{align}\begin{aligned}
\Vert  \rh \Vert_{\HH^{0}_{-5/2}} &\lesssim \Vert \check k \Vert_{\HH^{1}_{-3/2}}\\
&\lesssim \Vert \bar k \Vert_{\HH^{1}(B_1)}.
\end{aligned} \label{mars15eq5}
\end{align}
Further, it holds that on $B_1$
\begin{align*}
 \rh = \Div_g \check k = \Div_g \bar k =0,
\end{align*}
so by Proposition \ref{prop:TrivialExtensionRegularity}, $ \rh \in \ol{\HH}^{0}_{-5/2}$. This $ \rh$ is in general non-trivial (otherwise we would be done), and the Sobolev extension $\check k$ is therefore in general not a solution to \eqref{ZeroHodge}.\\

We have by \eqref{mars15eq5}
\begin{align*}
\Vert g-e \Vert_{\HH^2_{-1/2}} + \Vert  \rh \Vert_{\ol{\HH}^0_{-5/2}} &\lesssim \Vert g-e \Vert_{\HH^2_{-1/2}} + \Vert \bar k \Vert_{\ol{H}^{1}(B_1)}.
\end{align*}
Therefore, for $\varep>0$ small enough, Proposition \ref{prop:Dez191} yields a symmetric $2$-tensor $\tilde k \in \ol{\HH}^{1}_{-3/2}$ that solves on $\RRRwo$
\begin{align*} 
\Div_g \tilde k &= - \rh, \\
\tr_g \tilde k &=0 
\end{align*}
and is bounded by
\begin{align}
\Vert \tilde k \Vert_{\ol{\HH}^{1}_{-3/2}} &\lesssim \Vert  \rh \Vert_{\ol{\HH}^{0}_{-5/2}}. \label{eq:estimate444feb2}
\end{align}
Extend $\tilde k$ trivially to $B_1$. By Proposition \ref{prop:TrivialExtensionRegularity}, $\tilde k \in \HH^{1}_{-3/2}$. Consequently, the symmetric $2$-tensor
\begin{align} \label{summationklabel}
k := \check k + \tilde k \in \HH^{1}_{-3/2}
\end{align}
is such that $k\vert_{B_1}  = \bar k$ and solves on $\RRR^3$
\begin{align*} 
\Div_g k&=0,\\
\tr_g k &=0.
\end{align*}
Finally, for $\varep>0$ sufficiently small, by the estimates \eqref{mars15eq5} and \eqref{eq:estimate444feb2},
\begin{align}
\Vert k \Vert_{\HH^{1}_{-3/2}} \lesssim \Vert \bar k \Vert_{\HH^{1}(B_1)}. \nr
\end{align}
This proves the first part of Theorem \ref{thm:kextension1}.\\

{\bf Proof of Part 2.} Extend by Proposition \ref{functionextcts} the tensor $\bar k \in \HH^{1}(B_1)$ to a symmetric $2$-tensor $\check{k} \in \HH^{1}_{-3/2}$ on $\RRR^3$ such that
\begin{align}
\Vert \check k \Vert_{\HH^{1}_{-3/2}} \lesssim \Vert \bar k \Vert_{\HH^{1}(B_1)}. \label{12j1}
\end{align} 
Taking the $g$-tracefree and $g'$-tracefree parts of $\check k$ yields two symmetric $2$-tensors $\widehat{\check{k}}^g\in \HH^{1}_{-3/2}$ and $\widehat{\check{k}}^{g'} \in \HH^{1}_{-3/2}$, respectively, that both extend $\bar k$ and satisfy for $\varep>0$ sufficiently small,
\begin{align*}
\Big\Vert \widehat{\check{k}}^g \Big\Vert_{\HH^{1}_{-3/2}} \lesssim \Vert \bar k \Vert_{\HH^{1}(B_1)}, \,
\Big\Vert \widehat{\check{k}}^{g'} \Big\Vert_{\HH^{1}_{-3/2}} \lesssim \Vert \bar k \Vert_{\HH^{1}(B_1)}.
\end{align*}
By Proposition \ref{prop:TrivialExtensionRegularity},
\begin{align*}
 \rh := \Div_g  \widehat{\check{k}}^{g} \in \ol{\HH}^0_{-5/2},  \rh' := \Div_{g'} \widehat{\check{k}}^{g'}  \in \ol{\HH}^0_{-5/2}.
\end{align*}
For $\varep>0$ sufficiently small, by Lemma \ref{lem:DivergenceBoundedOperator} and \eqref{12j1},
\begin{align}\label{12j2eq}
\Vert \rh \Vert_{\ol{\HH}^0_{-5/2}}\lesssim \Vert \bar k \Vert_{\HH^{w-1}(B_1)}, \, \Vert \rh' \Vert_{\ol{\HH}^0_{-5/2}} \lesssim \Vert \bar k \Vert_{\HH^{w-1}(B_1)}.
\end{align}
For $\varep>0$ small enough, applying Proposition \ref{prop:Dez191} to $\rh, \rh'$ with metrics $g, g'$ yields two tensors $\tilde k, \tilde k' \in \ol{\HH}^{1}_{-3/2}$, respectively, that satisfy
\begin{align*}
\Div_g \tilde k &= - \rh, \\
\tr_g \tilde k &= 0, \\
\Div_{g'} \tilde k' &= -\rh', \\
\tr_{g' } \tilde k' &= 0.
\end{align*}
By \eqref{eq:mars15secondest1} in Proposition \ref{prop:Dez191}, for $\varep>0$ sufficiently small,
\begin{align} \begin{aligned}
\Vert \tilde k - \tilde k' \Vert_{\ol{\HH}^{1}_{-3/2}} &\lesssim \Vert g-g' \Vert_{\HH^2_{-1/2}} \Vert \rh \Vert_{\ol{\HH}^{0}_{-5/2}} + \Vert  \rh - \rh' \Vert_{\ol{\HH}^{0}_{-5/2}} \\
&\lesssim \Vert g-g' \Vert_{\HH^2_{-1/2}} \Vert \bar k \Vert_{\HH^{1}(B_1)} + \left\Vert  \Div_g  \widehat{\check{k}}^{g}  - \Div_{g'} \widehat{\check{k}}^{g'} \right\Vert_{\ol{\HH}^{0}_{-5/2}} \\
&\lesssim  \Vert g-g' \Vert_{\HH^2_{-1/2}} \Vert \bar k \Vert_{\HH^{1}(B_1)} +  \Vert g-g' \Vert_{\HH^2_{-1/2}} \Vert \check k \Vert_{\HH^1_{-3/2}} \\
&\lesssim \Vert g-g' \Vert_{\HH^2_{-1/2}}  \Vert \bar k \Vert_{\HH^{1}(B_1)},
\end{aligned} \label{eq:2feb555} \end{align}
where we used \eqref{12j2eq} and Corollary \ref{remark:Bounded}. Extend $\tilde k, \tilde k'$ trivially to $B_1$. By Proposition \ref{prop:TrivialExtensionRegularity}, $\tilde k, \tilde k' \in \HH^1_{-3/2}$. \\

The tensors 
\begin{align*}
k :=  \widehat{\check{k}}^g + \tilde k \in \HH^1_{-3/2}, 
k' := \widehat{\check{k}}^{g'} + \tilde k' \in \HH^1_{-3/2}
\end{align*}
both extend $\bar k$ and satisfy on $\RRR^3$
\begin{align*}
\Div_g k = 0, \\
\tr_g k = 0, \\
\Div_{g'} k' = 0,\\
\tr_{g'} k' = 0.
\end{align*}
Moreover, their difference is bounded by
\begin{align*}
\Vert k -k' \Vert_{\HH^1_{-3/2}} &\leq \left\Vert  \widehat{\check{k}}^g -\widehat{\check{k}}^{g'} \right\Vert_{\HH^1_{-3/2}} + \Vert \tilde k - \tilde k' \Vert_{\ol{\HH}^1_{-3/2}}\\
&\lesssim \Vert g-g' \Vert_{\HH^2_{-1/2}} \Vert \check k \Vert_{\HH^1_{-3/2}} + \Vert g-g' \Vert_{\HH^2_{-1/2}}  \Vert \bar k \Vert_{\HH^{1}(B_1)}\\
&\lesssim \Vert g-g' \Vert_{\HH^2_{-1/2}}  \Vert \bar k \Vert_{\HH^{1}(B_1)},
\end{align*}
where we used Corollary \ref{remark:Bounded}, \eqref{12j1} and \eqref{eq:2feb555}. \\


{\bf Proof of Part (3).} Let for an integer $w\geq3$ the symmetric $2$-tensor $\bar k \in \HH^{w-1}(B_1)$. By Proposition \ref{functionextcts}, extend $k$ from $B_1$ to a tensor $\check k \in \HH^{w-1}_{-3/2}$ with
\begin{align*}
\Vert \check k \Vert_{\HH^{w-1}_{-3/2}} \leq C_w \Vert \bar k \Vert_{\HH^{w-1}(B_1)}.
\end{align*}
It holds that
\begin{align*}
\widehat{\check k}^g =& \check k - \frac{1}{3} (\tr_g \check k ) g \\
=& \check k - \frac{1}{3} \left( \tr_g \check k - \tr_e \check k \right) (g-e) -\frac{1}{3} \tr_e \check k (g-e)\\
& - \frac{1}{3} \tr_e \check k e - \frac{1}{3} (\tr_g \check k -\tr_e \check k )e,
\end{align*}
and therefore by Lemma \ref{lem:conttrace},
\begin{align} \begin{aligned}
\left\Vert \widehat{\check k}^g \right\Vert_{\HH^{w-1}_{-3/2}} \lesssim& \Vert \check k \Vert_{\HH^{w-1}_{-3/2}} + \Vert \check k \Vert_{\HH^1_{-3/2}} \Vert g-e \Vert_{\HH^{w}_{-1/2}} + C_w \Vert \check k \Vert_{\HH^1_{-3/2}} \Vert g-e \Vert_{\HH^2_{-1/2}} \\
\lesssim&  \Vert \bar{k} \Vert_{\HH^1(B_1)} \Vert g-e \Vert_{\HH^{w}_{-1/2}} + C_w \Big(\Vert \bar k \Vert_{\HH^{w-1}(B_1)} + \Vert \bar{k} \Vert_{\HH^1(B_1)} \Vert g-e \Vert_{\HH^{2}_{-1/2}} \Big).
\end{aligned} \label{sefsege} \end{align}

Defining $\rh := \Div_g \widehat{\check{k}}^g$, we have
\begin{align} \begin{aligned}
\Vert \rh \Vert_{\HH^{w-2}_{-5/2}} \lesssim&  \left\Vert \widehat{\check{k}}^g \right\Vert_{\HH^{w-1}_{-5/2}}+ \left\Vert \Div_g \widehat{\check{k}}^g - \Div_e \widehat{\check{k}}^g \right\Vert_{\HH^{w-2}_{-5/2}} \\
\lesssim& \left\Vert \widehat{\check{k}}^g \right\Vert_{\HH^{w-1}_{-3/2}} + \Big( \Vert g-e\Vert_{\HH^2_{-1/2}} \left\Vert \widehat{\check{k}}^g \right\Vert_{\HH^{w-1}_{-3/2}}  + \left\Vert \widehat{\check{k}}^g \right\Vert_{\HH^1_{-3/2}} \Vert g-e \Vert_{\HH^{w}_{-1/2}} \Big)\\
& + C \left\Vert \widehat{\check{k}}^g \right\Vert_{\HH^1_{-3/2}} \Vert g-e \Vert_{\HH^{2}_{-1/2}} \\
\lesssim& \left\Vert \widehat{\check{k}}^g \right\Vert_{\HH^{w-1}_{-3/2}} + \left\Vert \widehat{\check{k}}^g \right\Vert_{\HH^1_{-3/2}} \Vert g-e \Vert_{\HH^{w}_{-1/2}}+ C_w \left\Vert \widehat{\check{k}}^g \right\Vert_{\HH^1_{-3/2}} \Vert g-e \Vert_{\HH^{2}_{-1/2}} \\
\lesssim& \Vert \bar k \Vert_{\HH^1(B_1)} \Vert g-e \Vert_{\HH^{w}_{-1/2}}+ C_w \Big( \Vert \bar k \Vert_{\HH^{w-1}(B_1)} +\left\Vert \bar k \right\Vert_{\HH^1(B_1)} \Vert g-e \Vert_{\HH^{2}_{-1/2}} \Big),
\end{aligned} \label{eqextension2434} \end{align}
where we used Lemma \ref{lem:DivergenceBoundedOperator} and \eqref{sefsege}.\\

By Proposition \ref{prop:Dez191}, for $\varep>0$ sufficiently small, let $\tilde{k}$ be the constructed solution to
\begin{align*}
\Div_g \tilde{k} =&-\rh, \\
\tr_g \tilde{k} =&0.
\end{align*}
By the estimates in Proposition \ref{prop:Dez191}, we have
\begin{align*}
\Vert \tilde k \Vert_{\ol{\HH}^{w-1}_{-3/2}} \lesssim& \Vert \rh \Vert_{\HH^0_{-5/2}} \Vert g-e \Vert_{\HH^w_{-1/2}} + \Vert \rh \Vert_{\ol{\HH}^{w-2}_{-5/2}}+ C_w \Vert \rh \Vert_{\HH^0_{-5/2}} \\
\lesssim& \Vert \bar k \Vert_{\HH^1(B_1)} \Vert g-e \Vert_{\HH^w_{-1/2}} + C_w \Big( \Vert \bar{k} \Vert_{\HH^{w-1}(B_1)} +\Vert \bar{k} \Vert_{\HH^1(B_1)} \Vert g-e \Vert_{\HH^2_{-1/2}} \Big),\end{align*}
where we used \eqref{sefsege} and \eqref{eqextension2434}. \\

Therefore the tensor $k= \widehat{\check k}^g + \tilde k$ (see \eqref{summationklabel}) satisfies $k\vert_{B_1}=\bar k$ as well as on $\RRR^3$
\begin{align*}
\Div_g \, k =&0, \\
\tr_g \, k =&0,
\end{align*}
and is bounded by
\begin{align*}
\Vert k \Vert_{\HH^{w-1}_{-3/2}} \lesssim& \Vert \bar k \Vert_{\HH^1(B_1)} \Vert g-e \Vert_{\HH^w_{-1/2}} + C_w \Big(\Vert \bar k \Vert_{\HH^{w-1}(B_1)} +\Vert \bar{k} \Vert_{\HH^1(B_1)} \Vert g-e \Vert_{\HH^2_{-1/2}} \Big)\\
\lesssim& \Vert \bar k \Vert_{\HH^1(B_1)} \Vert g-e \Vert_{\HH^w_{-1/2}} + C_w \Vert \bar k \Vert_{\HH^{w-1}(B_1)}.
\end{align*}
This finishes the proof of Theorem \ref{thm:kextension1}. \end{proof}


\subsection{Surjectivity at the Euclidean metric} \label{sssec:EuclideanSurjectivity} \label{sec:EuclideanSurjectivity}

Let $w\geq2$ be an integer. In this section we prove Lemma \ref{prop:EuclideanSurjectivity}, that is, we show that for any $\rho \in \ol{\HH}^{w-2}_{-5/2}$, there exists a symmetric $2$-tensor $k \in \HHmo$ that solves on $\RRRwo$
\begin{align*}
\Div_e k &= \rho, \\
\tr_e k &= 0
\end{align*}
and is bounded by
\begin{align*}
\Vert k \Vert_{\HHmo} \lesssim \Vert \rho \Vert_{\HHno} + C_w \Vert \rho \Vert_{\ol{\HH}^0_{-5/2}}.
\end{align*}

In this section, all differential operators are with respect to the Euclidean metric $e$. The operators $\Divd, \Curld, \Nd$ are the induced operators on the spheres $(S_r,\gac) \subset (\RRR^3,e)$ for $r>0$. 

\begin{remark} Let us note the following.
\begin{itemize}
\item In general, the system on $\RRRwo$
\begin{align}
\begin{aligned}
\Div k &= \rho,\\
\tr k &= 0
\end{aligned}
\label{Hodgeunder}
\end{align}
is underdetermined and does not admit an a priori estimate for solutions $k$. We work with the following Hodge system on $\RRRwo$
\begin{align}\begin{aligned}
\Div k &= \rho,\\
\Curl k &= \si,\\
\tr k &=0, \end{aligned}
\label{Hodge}
\end{align}
where $\si$ is a tracefree symmetric $2$-tensor that we carefully choose by hand. This system admits in general a priori estimates for $k$ in terms of $\rho$ and $\si$, see for example Proposition 4.4.1 in \cite{ChrKl93}. Clearly, a solution $k$ to \eqref{Hodge} is in particular a solution to \eqref{Hodgeunder}.
\item In the following, for given $\rh \in \ol{\HH}^{w-2}_{-5/2}$, we construct tracefree symmetric $2$-tensors $\si \in \ol{\HH}^{w-2}_{-5/2}$ and $k \in \ol{\HH}^{w-1}_{-3/2}$ solving on $\RRRwo$
\begin{align*}
\begin{cases}
\Div k = \rho, \\
\Curl k = \si, \\
\tr k =0, \\
k \vert_{r=1} = 0.
\end{cases}
\end{align*}
We note that generally, for given $\rh \in \ol{\HH}^{w-2}_{-5/2}$ and $\si \in \ol{\HH}^{w-2}_{-5/2}$, this is an overdetermined boundary value problem for $k$. Solutions $k$ automatically satisfy
\begin{align*}
\nab_N k \vert_{r=1} = 0,
\end{align*}
which follows by expressing the system as in \eqref{EinsteinSphere} below.
\item First, we decompose $k$ with respect to the foliation of $\RRR^3 \setminus \{ 0 \}$ by spheres $S_r$ into scalar functions and $S_r$-tangent tensors. Second, the Hodge-Fourier expansion of $S_r$-tangent tensors introduced in Section \ref{sec:Notation} allows to decompose \eqref{Hodge} into three independent sub-systems \textbf{S0},  \textbf{S1} and \textbf{S2}, see later in this section. These sub-systems are then solved individually. 
\end{itemize}
\end{remark}

\subsubsection{Derivation of the equations} \label{sss:mars301}

In this section, we derive the new form of \eqref{Hodge} with respect to the radial foliation of $\RRR^3 \setminus \{0\}$, see Section  \ref{ssec:tensordecomposition} for notations. Decompose the tensor $k$ into 
\begin{itemize}
\item the scalar $\de := k_{NN}$,
\item the $S_r$-tangent vectorfield $\ep_A:= \left( {k \mkern-10mu/\ \mkern-5mu}_{N}\right)_A$,
\item the $S_r$-tangent symmetric $2$-tensor $\et_{AB} := \left( {k \mkern-10mu/\ \mkern-5mu}\right)_{AB}$.
\end{itemize}
\ni Furthermore, let $\eh$ be the tracefree part of $\et$, that is\footnote{Here we use that $\gac^{AB} k_{AB} = -\de$ by the third equation of \eqref{Hodge}.}
\begin{align}
\eh_{AB} := \et_{AB} + \half \de \ga_{AB}. \nr
\end{align}
Decompose $\rh$ into 
\begin{itemize}
\item the scalar $\rh_N,$
\item the $S_r$-tangent vectorfield $\rhod_A := \rh_A,$
\end{itemize}

and $\si$ into
\begin{itemize}
\item the scalar $\si_{NN},$
\item the $S_r$-tangent vectorfield $\sigmaN_A := \si_{AN},$
\item the $S_r$-tangent symmetric $2$-tensor $\sigmad_{AB}:= \si_{AB}.$
\end{itemize}

The system \eqref{Hodge} is equivalent to (this the Euclidean version of Proposition 4.4.3 in \cite{ChrKl93})
\begin{align} \begin{aligned}
\Divd \ep &= \rh_N - \nab_N \de - \frac{3}{r} \de,  \\
\Curld \ep &= \si_{NN},  \\
\Ndn \ep + \frac{2}{r} \ep &=  \half \rhod + \sigmaNd  + \Nd \de,  \\
\Divd \eh &= \half \rhod -\sigmaNd - \half \Nd \de - \frac{1}{r} \ep, \\
\Ndn \eh + \frac{1}{r} \eh &= \dual\widehat{(\sigmad)} + \half \Nd \widehat{\otimes} \ep, \end{aligned} \label{EinsteinSphere}
\end{align}
where $\sigmaNd$ denotes the Hodge dual of $\sigmaN$ and $\mkern-7mu \dual\widehat{(\sigmad)}$ the Hodge dual of $\widehat{\sigmad}$, the tracefree part of $\sigmad$. See Section \ref{sec:HodgeTheory} for details. \\

The Hodge system \eqref{EinsteinSphere} is linear and its coefficients depend only on $r$. Therefore, we may project the equations of \eqref{EinsteinSphere} onto the Hodge-Fourier basis elements. This uses Remark \ref{remmars30} and Proposition \ref{prop:Completeness}. We split \eqref{EinsteinSphere} into the modes $l=0,1$ and $l \geq2$, which yields the following three subsystems \textbf{S0}, \textbf{S1} and \textbf{S2}.
\begin{align}
0&= \rhon^{[0]} - \frac{1}{r^3} \pr_r \left( r^3 \de^{[0]} \right), \label{EH1a0} \tag{\textbf{S0.1}}\\
0&= \si_{NN}^{[0]}, \label{EH2a0} \tag{\textbf{S0.2}}
\end{align}

\begin{align}
\Divd \ep^{[1]} &= \rhon^{[1]} - \frac{1}{r^3} \pr_r \left( r^3 \de^{[1]} \right), \label{EH1a1} \tag{\textbf{S1.1}}\\
\Curld \ep^{[1]} &= \si_{NN}^{[1]},  \label{EH2a1} \tag{\textbf{S1.2}} \\
\frac{1}{r^3} \Ndn \left( r^3 \ep^{[1]} \right) &= \rhod^{[1]} + \half \Nd \de^{[1]}, \label{EH3a1} \tag{\textbf{S1.3}}\\
\sigmaNd^{[1]} &= \half \rhod^{[1]} - \half \Nd \de^{[1]} - \frac{1}{r} \ep^{[1]},  \label{EH4a1} \tag{\textbf{S1.4}}
\end{align}
and, using that $\eh^{[\geq2]} = \eh, \sigmaNdd^{[\geq2]} = \sigmaNdd$,
\begin{align} 
\Divd \ep^{[\geq2]} &= \rh_N^{[\geq2]} - \frac{1}{r^3} \pr_r \left( r^3 \de^{[\geq2]} \right), \label{EH1a2} \tag{\textbf{S2.1}}\\
\Curld \ep^{[\geq2]} &= \si_{NN}^{[\geq2]},  \label{EH2a2} \tag{\textbf{S2.2}} \\
\frac{1}{r^2} \Ndn \left( r^2 \ep^{[\geq2]} \right) &=  \half \rhod^{[\geq2]} + \sigmaNd^{[\geq2]}  + \Nd \de^{[\geq2]},  \label{EH3a2} \tag{\textbf{S2.3}} \\
\Divd \eh &=   \half \rhod^{[\geq2]} - \sigmaNd^{[\geq2]} - \half \Nd \de^{[\geq2]} -\frac{1}{r } \ep^{[\geq2]},  \label{EH4a2} \tag{\textbf{S2.4}} \\
\Ndn \eh + \frac{1}{r} \eh &= \sigmaNdd + \half \Nd \widehat{\otimes} \ep^{[\geq2]}. \label{EH5a2} \tag{\textbf{S2.5}} 
\end{align}


\subsubsection{Definition of the $2$-tensors $k$ and $\si$} \label{constrmars30}

In this section, we explicitly exhibit the two symmetric $2$-tensors $k$ and $\si$. We show in Section \ref{proofmars30} that they form a regular solution to \eqref{EinsteinSphere}. \\

Let $\rho = (\rh_N, \rhod) \in \HHno$. Let the Hodge-Fourier decomposition of $\rh_N, \rhod$ be
\begin{align}
\rho_N &= \rhon^{[0]} + \rhon^{[1]} + \rhon^{[\geq2]}, \nr \\
\rhod &= \rhod_{E}^{[1]} + \rhod_{H}^{[1]}  +  \rhod_{E}^{[\geq2]} + \rhod_{H}^{[\geq2]}. \nr
\end{align}

Define symmetric tracefree $2$-tensors $k$ and $\si$ on $\RRR^3 \setminus \ol{B_1}$ as follows.
\begin{itemize}
\item \textbf{Definition of $\de$.} Let the scalar function
\begin{align}
\de &= \de^{[0]} + \de^{[1]} + \de^{[\geq2]}, \label{eq:delta1}
\end{align}
where $\de^{[0]}$ is defined as
\begin{align}
\de^{[0]} := \frac{1}{r^3} \int\limits_1^r (r')^3 \rhon^{[0]} dr'  \label{eq:delta2}
\end{align}
and $\de^{[1]}$ is defined as the solution to the second-order ODE on $r>1$
\begin{align}
\begin{cases}
\pr_r^2 \de^{[1]} + \frac{7}{r} \pr_r \de^{[1]} + \frac{8}{r^2} \de^{[1]} = \frac{1}{r^4} \pr_r ( r^4 \rhon^{[1]}) -\Divd \rhod^{[1]}, \\ 
\de^{[1]}\vert_{r=1} = \pr_r \de^{[1]}\vert_{r=1} = 0.
\end{cases} \label{eq:delta3}
\end{align}
The function $\de^{[\geq2]}$ is defined as the solution to the following elliptic boundary value problem on $\RRR^3 \setminus \ol{B_1}$ (see Appendix \ref{sec:WEIGHTEDellipticity}),
\begin{align}
\begin{cases}
\triangle \de^{[\geq2]} + \frac{4}{r} \pr_r \de^{[\geq2]} + \frac{6}{r^2} \de^{[\geq2]} = \frac{1}{r^3} \pr_r \left(r^3 \rhon^{[\geq2]} \right) - \Divd \left( \rhod_E^{[\geq2]} + \zeta_E \right), \\
\de^{[\geq2]} \vert_{r=1} =0.
\end{cases} \label{eq:delta4}
\end{align}
Here, the $S_r$-tangent vectorfield $\zeta_E$ is defined on $\RRR^3$ by
\begin{align} \begin{aligned}
\zeta_E&:= \sumtwo \zeta_E^{(lm)} E^{(lm)}, \\
\zeta_E^{(lm)}(r) &:= c_E^{(lm)} r^{l-1} \partial_r (\chi( l(r-1))), \end{aligned} \label{eq:defofZetaE}
\end{align}
where $\chi$ is the standard transition function defined in \eqref{eq:transfct} and for $l\geq2$,
\begin{align}
c_E^{(lm)} :=  \isinf r^{-l+1} \left( \frac{l}{\sqrt{l(l+1)}}  \left( \rh_N \right)^{(lm)} - \rhod_E^{(lm)} \right) dr. \label{eq:defofZetaE2}
\end{align}

\item \textbf{Definition of $\si_{NN}$.} Let the scalar function
\begin{align}
\si_{NN} =  \si_{NN} ^{[1]}+ \si_{NN}^{[\geq2]}, \label{eq:sigma0}
\end{align}
where $ \si_{NN}^{[1]}$ is defined as
\begin{align}
\si_{NN}^{[1]} := \frac{1}{r^4} \isr (r')^4 \Curld \rhod^{[1]} dr', \label{eq:sigma1}
\end{align}
and $\si_{NN}^{[\geq2]}$ is defined as solution to the following elliptic boundary value problem on $\RRR^3 \setminus \ol{B_1}$ (see Appendix \ref{sec:WEIGHTEDellipticity}),
\begin{align}\begin{cases}
\triangle \si_{NN}^{[\geq2]} + \frac{1}{r} \pr_r \si_{NN}^{[\geq2]} - \frac{3}{r^2}\si_{NN}^{[\geq2]} = \pr_r \left(  \Curld \left( \rhod^{[\geq2]}_H +  \zeta_H^{[\geq2]} \right) \right), \\
\si_{NN}^{[\geq2]} \vert_{r=1} = 0.
\end{cases} \label{eq:sigma2}
\end{align}
Here, the $S_r$-tangent vectorfield $\zeta_H$ is defined by
\begin{align}
\begin{aligned}
\zeta_H &:= \sumtwo \zeta_H^{(lm)}  H^{(lm)},\\
\zeta_H^{(lm)}(r) &:= c_H^{(lm)} r^{1+\sqrt{l(l+1)+4}} \partial_r (\chi( l(r-1))),  \end{aligned} \label{eq:defofZetaH}
\end{align}
and for $l\geq2$,
\begin{align}
c_H^{(lm)} := - \isinf r^{-1-\sqrt{l(l+1)+4}} \rhod^{(lm)}_H dr. \label{constant8feb}
\end{align}

\item \textbf{Definition of $\ep$.} Let the $S_r$-tangent vectorfield $\ep$ be on each $S_r$, $r\geq1$, the solution to 
\begin{align} 
\cDd_1 \ep &= \Big( \rh_N- \frac{1}{r^3} \pr_r\left( r^3 \de \right),  \si_{NN}  \Big). \label{eq:ep2}
\end{align}

\item \textbf{Definition of $\sigmaNd$.} Let the $S_r$-tangent vectorfield 
\begin{align}
\sigmaNd = \sigmaNd^{[1]} +\sigmaNd^{[\geq2]}_E +\sigmaNd^{[\geq2]}_H, \label{eq:sigmaNd1}
\end{align}
where $\sigmaNd^{[1]}, \sigmaNd^{[\geq2]}_E$ are defined as
\begin{align} 
\sigmaNd^{[1]} &:= \half \rhod^{[1]} - \half \Nd \de^{[1]} - \frac{1}{r} \ep^{[1]}, \label{eq:sigmaNd2} \\
\sigmaNd^{[\geq2]}_E &:= \half \rhod^{[\geq2]}_E + \zeta_E, \label{eq:sigmaNd2feb9}
\end{align}
and $\sigmaNd^{[\geq2]}_H$ is defined to be on each $S_r$, $r\geq1$, the solution of
\begin{align} 
\cDd_1 \left(  \sigmaNd_{H}^{[\geq2]}  \right) = -\cDd_1 \left(  \half \rhod_H^{[\geq2]}  \right) +  \left( 0, \frac{1}{r^3} \pr_r \left( r^3 \si_{NN}^{[\geq2]} \right) \right).  \label{eq:sigmaNd3}
\end{align}

\item \textbf{Construction of $\sigmaNdd$.} Let the symmetric $\gac$-tracefree $2$-tensor $\sigmaNdd$ be on each $S_r$, $r\geq1$, the solution to
\begin{align} \begin{aligned}
\cDd_2 \left( \sigmaNdd \right) =& - \cDd_2 \left( \half \Nd \widehat{\otimes} \ep^{[\geq2]} \right)\\
&+ \frac{1}{r^2} \Ndn \left( r^2 \left( \half \rhod^{[\geq2]} - \sigmaNd^{[\geq2]} - \half \Nd \de^{[\geq2]} - \frac{1}{r} \ep^{[\geq2]} \right) \right).
\end{aligned} \label{eq:sigmaNdd2}
\end{align}

\item \textbf{Construction of $\eh$.} Let the symmetric $\gac$-tracefree $2$-tensor $\eh$ be on each $S_r$, $r\geq1$, the solution to 
\begin{align} \begin{aligned}
\cDd_2 \eh =  \half \rhod^{[\geq2]} - \sigmaNd^{[\geq2]} - \half \Nd \de^{[\geq2]} - \frac{1}{r}\ep^{[\geq2]}.
\end{aligned}  \label{eq:eh1} \end{align}

\end{itemize}
\begin{remark} For ease of presentation, we defined $k$ and $\si$ via the quantities that appear in \eqref{EinsteinSphere}. Indeed, by the Hodge duality relation \eqref{eq:lefthodgedualidentity} and the third equation of \eqref{Hodge}, all components of $k$ and $\si$ are uniquely specified this way. \end{remark}
\begin{remark} \label{Remark2DtN}
The auxiliary $\zeta_E$ and $\zeta_H$ in \eqref{eq:defofZetaH} and \eqref{eq:defofZetaH} are introduced to control the Dirchlet-to-Neumann map of the elliptic boundary value problems \eqref{eq:delta4} and \eqref{eq:sigma2} for $\de^{[\geq2]}$ and $\si_{NN}^{[\geq2]}$, respectively, to achieve the additional boundary condition
\begin{align*}
\pr_r \de^{[\geq2]} \vert_{r=1} = \pr_r \si_{NN}^{[\geq2]} \vert_{r=1} = 0.
\end{align*}
This is necessary such that $\varep$ and $\sigmaNd_{H}^{[\geq2]} $ vanish on $S_1$. Indeed, see their definitions in \eqref{eq:ep2} and \eqref{eq:sigmaNd3}.
\end{remark}

\subsubsection{Proof of surjectivity} \label{proofmars30}

In this section, we prove Lemma \ref{lem:formalsolfeb5}, Proposition \ref{prop:boundaryvalues} and Lemma \ref{lem:bdrycntrlfeb5} below, that together imply surjectivity. Especially Proposition \ref{prop:boundaryvalues} is essential and only holds due to our delicate choice of $\zeta_E, \zeta_H$ in \eqref{eq:defofZetaE} and \eqref{eq:defofZetaH}, as well as our particular choice of $\si_{NN}^{[\geq2]}$ to be a solution to \eqref{eq:sigma2}. See also Remark \ref{Remark2DtN}.
\begin{lemma} \label{lem:formalsolfeb5}
For given $\rho$, the symmetric $2$-tensors $k$ and $\si$ defined by \eqref{eq:delta1}-\eqref{eq:eh1} are a formal solution to \eqref{EinsteinSphere}, that is, on $\RRRwo$,
\begin{align*}
\Div k &= \rho,\\
\Curl k &= \si,\\
\tr k &=0. 
\end{align*}
\end{lemma}

\begin{proof}  For each of the three subsystems \textbf{S0}, \textbf{S1} and \textbf{S2}, we show that the corresponding parts of $k, \si$ are solutions.\\

{\bf Analysis of \textbf{S0}.} The two functions $\de^{[0]}, \si_{NN}^{[0]}$ are radial. Integration of \eqref{EH1a0} 
along $r$ with the trivial boundary condition $\de^{[0]} \vert_{r=1}
=0$ directly leads to \eqref{eq:delta2}. \eqref{EH2a0} is satisfied by \eqref{eq:sigma0}. Therefore, $\de^{[0]}$ and $\si_{NN}^{[0]}$ solve \textbf{S0}. \\


{\bf Analysis of \textbf{S1}.} The equations \eqref{EH1a1}, \eqref{EH2a1} are automatically satisfied by the definition of $\ep^{[1]}$ in \eqref{eq:ep2}. The same holds for \eqref{EH4a1} by the definition of $\sigmaNd^{[1]}$ in \eqref{eq:sigmaNd2}. It remains to verify that \eqref{EH3a1} is satisfied.\\

Applying $\Divd$ and $\Curld$ to \eqref{EH3a1}, plugging in the definition of $\ep^{[1]}$ in \eqref{eq:ep2} and using Lemma \ref{lem:commutationrelation}, we get that $\de^{[1]}$ and $\si_{NN}^{[1]}$ must satisfy the compatibility conditions
\begin{align}
\half \Ld \de^{[1]} + \pr_r^2 \de^{[1]} + \frac{7}{r} \pr_r \de^{[1]} + \frac{9}{r^2} \de^{[1]} &= \frac{1}{r^4} \pr_r \left( r^4 \rhon^{[1]}\right) - \Divd \rhod^{[1]}, \label{eq:condDez1} \\
\frac{1}{r^4} \pr_r \left( r^4  \si_{NN}^{[1]} \right) &= \Curld \rhod^{[1]}. \label{eq:condDez2}
\end{align}

With regard to the Hodge-Fourier decomposition, it holds that $\Ld \de^{[1]} = -\frac{2}{r^2} \de^{[1]}$. So, \eqref{eq:condDez1} is satisfied by the definition of $\de^{[1]}$ in \eqref{eq:delta3}. Moreover, $\si_{NN}^{[1]}$ satisfies \eqref{eq:condDez2} by \eqref{eq:sigma1}. Note that at the level of $l\geq1$, $\cDd_1$ is a bijection, see Lemma \ref{lem:inversetoDDDD}, 
so the above shows that \eqref{EH3a1} is satisfied. To summarize, we showed that $\ep^{[1]},\de^{[1]}, \si_{NN}^{[1]}$ and $\sigmaNd^{[1]}$ solve {\bf S1}.\\


{\bf Analysis of \textbf{S2}.} In the following, we use that $\triangle = \pr_r^2 + \frac{2}{r}\pr_r + \Ld$. The equations \eqref{EH1a2}, \eqref{EH2a2} and \eqref{EH4a2} are satisfied in view of \eqref{eq:ep2} and \eqref{eq:eh1}. It remains to prove \eqref{EH3a2} and \eqref{EH5a2}. We start with \eqref{EH3a2}.\\

Applying $\Divd$ and $\Curld$ to \eqref{EH3a2} and using \eqref{eq:ep2} leads to the compatibility conditions
\begin{align}
\triangle \de^{[\geq2]} + \frac{4}{r} \pr_r \de^{[\geq2]} + \frac{6}{r^2} \de^{[\geq2]} &= \frac{1}{r^3} \pr_r \left( r^3 \rhon^{[\geq2]} \right) - \Divd \left( \half \rhod^{[\geq2]} + \sigmaNd^{[\geq2]} \right), \label{eq:condDez17}\\
\frac{1}{r^3} \pr_r \left( r^3 \si_{NN}^{[\geq2]} \right) &= \Curld \left( \half \rhod^{[\geq2]} + \sigmaNd^{[\geq2]} \right). \label{eq:condDez172}
\end{align}

The function $\de^{[\geq2]}$ defined in \eqref{eq:delta4} satisfies \eqref{eq:condDez17} by the definition of $\sigmaNd_E^{[\geq2]}$ in \eqref{eq:sigmaNd2feb9} and the fact that 
\begin{align*}
\Divd \left(  \half \rhod_H^{[\geq2]} + \sigmaNd_H^{[\geq2]} \right) = 0,
\end{align*}
see the construction of $H^{(lm)}$ in \eqref{def:Vectorsjan}. Furthermore, the $\sigmaNd_H^{[\geq2]}$ defined in \eqref{eq:sigmaNd3} satisfies \eqref{eq:condDez172} because 
\begin{align*}
\Curld \left(  \half \rhod_E^{[\geq2]} + \sigmaNd_E^{[\geq2]} \right) = 0
\end{align*}
by the construction of $E^{(lm)}$ in \eqref{def:Vectorsjan}. This shows that \eqref{EH3a2} is satisfied.\\

We turn now to \eqref{EH5a2}. Applying the divergence operator $\cDd_{2}$ to \eqref{EH5a2} and using \eqref{eq:eh1} leads to 
\begin{align} \begin{aligned}
\cDd_2 \left( \sigmaNdd + \half \Nd \widehat{\otimes} \ep^{[\geq2]} \right) &= \frac{1}{r^2} \Ndn \left( r^2 \left( \half \rhod^{[\geq2]} - \sigmaNd^{[\geq2]} - \half \Nd \de^{[\geq2]} - \frac{1}{r} \ep^{[\geq2]} \right) \right).
\label{eq:condDez173}
\end{aligned}\end{align}
This coincides with the definition of $\sigmaNdd$ in \eqref{eq:sigmaNdd2} and thus shows that \eqref{EH5a2} is satisfied.\\

To summarise, we showed that $\de^{[\geq2]}, \ep^{[\geq2]}, \si_{NN}^{[\geq2]}, \sigmaNd^{[\geq2]}, \sigmaNdd^{[\geq2]}, \eh$ solve \textbf{S2}. This finishes the proof of Lemma \ref{lem:formalsolfeb5}. \end{proof}


We continue by controlling the regularity and boundary behaviour at $S_1$ of $\zeta_E, \zeta_H$ and $\de^{[\geq2]}, \si_{NN}^{[\geq2]}$.
\begin{proposition} \label{prop:boundaryvalues} Let $w\geq2$ be an integer. Let $\rh =(\rho_N,\rhod) \in \HHno$ be given. Let $\zeta_E, \zeta_H$ be the vectorfields defined in \eqref{eq:defofZetaE}-\eqref{eq:defofZetaE2}, \eqref{eq:defofZetaH}-\eqref{constant8feb}, and $\de^{[\geq2]}, \si_{NN}^{[\geq2]}$ be the solutions to the elliptic PDEs \eqref{eq:delta4}, \eqref{eq:sigma2}, respectively. Then, the following holds.
\begin{enumerate}
\item {\bf Regularity and boundary behaviour of $\zeta_E, \zeta_H$.} The vectorfields $\zeta_E$ and $\zeta_H$ satisfy
\begin{align} \begin{aligned}
\Vert \zeta_E \Vert_\HHno \lesssim& \Vert \rh \Vert_\HHno + C_w\Vert \rh \Vert_{\ol{\HH}^0_{-5/2}}, \\
\Vert \zeta_H \Vert_\HHno \lesssim& \Vert \rh \Vert_\HHno+ C_w\Vert \rh \Vert_{\ol{\HH}^0_{-5/2}}. \end{aligned} \label{eq:zetabound8feb}
\end{align}
Moreover, for each $l\geq2$, $m \in \{-l, \dots, l\}$, $\zeta_E^{(lm)}$ and $\zeta_H^{(lm)}$ satisfy 
\begin{align} \begin{aligned}
\isinf r^{1-l} \left( \zeta_E^{(lm)} - \left( \frac{l}{\sqrt{l(l+1)}} \rhon^{(lm)} - \rhod_E^{(lm)}  \right) \right) dr&=0, \\
 \isinf r^{-1-\sqrt{l(l+1)+4}}  \left( \zeta_H^{(lm)}  + \rhod_H^{(lm)} \right)dr &=0.
\end{aligned} \label{eq:propo8feb} \end{align}
\item {\bf Precise estimate for $\zeta_E$ and $\zeta_H$. } It holds that
\begin{align}
\begin{aligned}
\Vert \cDd^{-1}_2 \left( \Ndn \zeta_E \right) \Vert_{\HH^{w-2}_{-5/2}(\RRRwo)}& \lesssim \Vert \rh \Vert_{\HHno}+ C_w\Vert \rh \Vert_{\ol{\HH}^0_{-5/2}},\\
\Vert \cDd^{-1}_2 \left( \Ndn \zeta_H \right) \Vert_{\HH^{w-2}_{-5/2}(\RRRwo)}& \lesssim \Vert \rh \Vert_{\HHno}+ C_w\Vert \rh \Vert_{\ol{\HH}^0_{-5/2}},
\end{aligned}
\label{eq:zetaEestimatefeb15}
\end{align}
with $\cDd^{-1}_2 \left( \Ndn \zeta_E \right), \cDd^{-1}_2 \left( \Ndn \zeta_H \right) \in \ol{\HH}^{w-2}_{-5/2}$.
\item {\bf Elliptic regularity and boundary behaviour of $\de^{[\geq2]}, \si_{NN}^{[\geq2]}$.} It holds that 
\begin{align} \begin{aligned}
\left\Vert \de^{[\geq2]} \right\Vert_{\HHm(\RRRwo)} &\lesssim \Vert \rh \Vert_{\HHno}+ C_w \Vert \rh \Vert_{\ol{\HH}^0_{-5/2}}, \\ 
\left\Vert \si_{NN}^{[\geq2]} \right\Vert_{\HHn(\RRRwo)} &\lesssim \Vert \rh \Vert_{\HHno}+ C_w\Vert \rh \Vert_{\ol{\HH}^0_{-5/2}}. 
\end{aligned}\label{mars2222}
\end{align}
Furthermore,
\begin{align*}
\de^{[\geq2]} \in \HHmo, \si_{NN}^{[\geq2]} \in \HHno.
\end{align*}
In particular, for $w>2$,
\begin{align}
\pr_r \de^{[\geq2]} \Big\vert_{r=1}=0,  \label{eq:boundarycond5feb1}
\end{align}
and for $w>3$,
\begin{align}
\pr_r \si_{NN}^{[\geq2]} \Big\vert_{r=1} =0. \label{eq:boundarycond5feb1111}
\end{align}
\item {\bf Precise estimate for  $\pr_r \si_{NN}^{[\geq2]}$.} It holds that 
\begin{align}
\left\Vert \cDd_1^{-1}(0, \pr_r \si_{NN}^{[\geq2]}) \right\Vert_{\HHn(\RRRwo)} \lesssim \Vert \rho \Vert_{\HHno}+ C_w\Vert \rh \Vert_{\ol{\HH}^0_{-5/2}}. \label{eq:mars22estim1}
\end{align}
and moreover, $\cDd_1^{-1}(0, \pr_r \si_{NN}^{[\geq2]}) \in \HHno$. \end{enumerate}
\end{proposition}
Here $\cDd_1^{-1}, \cDd_2^{-1}$ denote the inverse operators to the elliptic $\cDd_1, \cDd_2$ on $(S_r,\gac)$, respectively, see Lemma \ref{lem:inversetoDDDD}.

\begin{remark}
The quantities $\de^{[\geq2]}, \si_{NN}^{[\geq2]}$ are solutions to the elliptic equations \eqref{eq:delta4}, \eqref{eq:sigma2} on $\RRRwo$, respectively. Therefore their boundary regularity at $S_1$ is harder to estimate than for the other components of $k$ and $\si$ which all satisfy first order transport equations in $r$ or Hodge systems on $S_r$.
\end{remark}


\begin{proof} We first analyse $\zeta_E$ and $\zeta_H$. \\

{ \bf (1) Regularity and boundary behaviour of $\zeta_E, \zeta_H$.} We begin by showing that the constants $c^{(lm)}_E, c^{(lm)}_H$ in \eqref{eq:defofZetaE2}, \eqref{constant8feb} are well-defined. By Cauchy-Schwarz, for all $l\geq2$, $m \in \{-l, \dots, l \}$,
\begin{align} \begin{aligned}
\left\vert c^{(lm)}_E \right\vert &= \left\vert \int\limits_1^\infty r^{-l+1} \left( \frac{l}{\sqrt{l(l+1)}} {\rhon}^{(lm)} - \rhod_E^{(lm)} \right) dr \right\vert\\
&\leq \left( \isinf r^{-2l} dr \right)^{1/2} \left( \isinf  r^2\left( \frac{l}{\sqrt{l(l+1)}} {\rhon}^{(lm)} - \rhod_E^{(lm)} \right)^2 dr \right)^{1/2} \\
&\lesssim \frac{1}{\sqrt{2l-1}} \left( \isinf \left( r {\rhon}^{(lm)} \right)^2 dr  + \isinf \left( r \rhod_E^{(lm)} \right)^2 dr \right)^{1/2}, 
\end{aligned}  \label{eq:constantestimate1112} 
\end{align}
and also for all $l\geq2$,
\begin{align} \begin{aligned}
\left\vert c^{(lm)}_H \right\vert  &= \left\vert \isinf r^{-1-\sqrt{l(l+1)+4}} \rhod_H^{(lm)}dr \right\vert \\
&\lesssim \left( \frac{1}{2\sqrt{l(l+1)+4}+3}\right)^{1/2} \left( \isinf \left( r \rhod_H^{(lm)} \right)^2 dr \right)^{1/2}.
\end{aligned} \label{eq:constantestimate8feb1} \end{align}

We show next that for all integers $w\geq2$,
\begin{align} \begin{aligned}
\Vert \zeta_E \Vert_{\HHn(\RRRwo)} &\lesssim \Vert \rh \Vert_{\HHno}+ C_w\Vert \rh \Vert_{\ol{\HH}^0_{-5/2}},\\
 \Vert \zeta_H \Vert_{\HHn(\RRRwo)} &\lesssim \Vert \rh \Vert_{\HHno}+ C_w\Vert \rh \Vert_{\ol{\HH}^0_{-5/2}}. \end{aligned} \label{eq:zetabound8feb2}
\end{align}

Consider first the case $w=2$ of \eqref{eq:zetabound8feb2} for $\zeta_E$, that is,
\begin{align} 
\Vert \zeta_E \Vert_{\HH^0_{-5/2}(\RRR^3 \setminus \ol{B_1})} &\lesssim \Vert \rh \Vert_{\HH^0_{-5/2}(\RRR^3 \setminus \ol{B_1})}.
\label{eq:feb1011} \end{align}
By \eqref{eq:defofZetaE} and \eqref{eq:defofZetaH}, for all $l\geq2$, $\zeta_E^{(lm)}$ and $\zeta_H^{(lm)}$ vanish outside the interval $(1, 1+1/l )$. Therefore, for $l\geq2$,
\begin{align} \begin{aligned}
\isinf \left(r \zeta_E^{(lm)}\right)^2 dr &= \left(c_E^{(lm)} \right)^2 l^2 \int\limits_1^{1+1/l} r^{2l} \Big(  (\pr_r \chi)(l(r-1)) \Big)^2 dr \\
&\lesssim \left(c_E^{(lm)} \right)^2 \frac{l^2}{2l+1}  \bigg[ \left(1 + \frac{1}{l} \right)^{2l+1} -1 \bigg] \\
&\lesssim 
\left( \isinf \left( r \rhon^{(lm)} \right)^2 dr + \isinf \left( r \rhod_E^{(lm)} \right)^2 dr \right),
\end{aligned} \label{eq:feb9estimateE111} \end{align}
where we uniformly estimated $\pr_r \chi$ and used \eqref{eq:constantestimate1112} in the last step. Similarly, for $l\geq2$, using \eqref{eq:constantestimate8feb1},
\begin{align*}
\isinf (r \zeta_H^{(lm)})^2 dr &= \left( c_H^{(lm)} \right)^2 \isinf r^{2+2 \sqrt{l(l+1)+4}} \left( \pr_r \left( \chi(l(r-1)) \right) \right)^2 dr \\
&= \left( c_H^{(lm)} \right)^2 l^2 \int\limits_1^{1+1/l} r^{2+2 \sqrt{l(l+1)+4}} \left( \pr_r \chi \right)^2 (l(r-1))dr  \\
&\lesssim  \left( c_H^{(lm)} \right)^2 l^2 \left( \frac{1}{3+2 \sqrt{l(l+1)+4}} \right) \left( \left( 1+ \frac{1}{l}\right)^{3+2 \sqrt{l(l+1)+4}} -1 \right)\\
&\lesssim \isinf \left( r \rhod_H^{(lm)} \right)^2 dr.
\end{align*}
Summing over $l$ and $m$ proves \eqref{eq:feb1011}.\\

The case $w >2$ of \eqref{eq:zetabound8feb2} for $\zeta_E$ is derived as follows. By Proposition \ref{prop:Howtoestimatefunctions}, we improve \eqref{eq:constantestimate1112} as follows, 
\begin{align} \begin{aligned}
\left\vert c_E^{(lm)} \right\vert \lesssim& \frac{1}{\sqrt{2l-1}} \left( \isinf \left( r {\rhon}^{(lm)} \right)^2 dr  +\isinf \left( r \rhod_E^{(lm)} \right)^2 dr  \right)^{1/2} \\
\lesssim& \frac{1}{\sqrt{2l-1}\sqrt{l(l+1)}^{w-2}} \left((l(l+1))^{w-2} \isinf \left(  r {\rhon}^{(lm)} \right)^2 dr \right)^{1/2}\\
 &+\frac{1}{\sqrt{2l-1}\sqrt{l(l+1)}^{w-2}} \left( (l(l+1))^{w-2} \isinf \left(r \rhod_E^{(lm)} \right)^2 dr  \right)^{1/2} \\
\lesssim& \frac{1}{\sqrt{2l-1}\sqrt{l(l+1)}^{w-2}} \left( \isinf r^{2({w-2})} \left( \frac{l(l+1)}{r^2} \right)^{w-2} \left(  r {\rhon}^{(lm)} \right)^2 dr \right)^{1/2} \\
& + \frac{1}{\sqrt{2l-1}\sqrt{l(l+1)}^{w-2}} \left( \isinf r^{2({w-2})} \left( \frac{l(l+1)}{r^2} \right)^{w-2} \left(r \rhod_E^{(lm)} \right)^2 dr  \right)^{1/2}.
\end{aligned}\label{eq:improvedfeb9} \end{align}
The terms in brackets on the right-hand side, in view of Proposition \ref{prop:Howtoestimatefunctions}, correspond after summing over $l,m$ to the $\HHno$-norm of $\rh$ which is bounded. These terms are therefore in particular summable.\\

On the other hand, we can explicitly calculate
\begin{align} \begin{aligned}
\pr_r \left( \zeta_E^{(lm)}\right) &= c_E^{(lm)} \frac{l-1}{r} r^{l-1} \pr_r \left( \chi(l(r-1)) \right) + c_E^{(lm)} r^{l-1} l \pr_r \Big( \left( \pr_r \chi \right)(l(r-1)) \Big) \\
& \approx \frac{l}{r} \zeta_E^{(lm)},\\
\left( \Divd \zeta_E \right)^{(lm)} &=- \frac{\sqrt{l(l+1)}}{r} \zeta_E^{(lm)}, \,\,\, \left( \Curld \zeta_E \right)^{(lm)} =0. \\
\end{aligned} \label{feb25estref} \end{align}
Combining \eqref{eq:improvedfeb9}, \eqref{feb25estref} and using Propositions \ref{prop:Completeness} and \ref{prop:Howtoestimatefunctions} and Lemma \ref{lem:commutationrelation}, we can estimate the derivatives of $\zeta_E$ similarly as in \eqref{eq:feb1011}. This proves \eqref{eq:zetabound8feb2} for $\zeta_E$ for all $w\geq2$. The estimates \eqref{eq:zetabound8feb2} for $\zeta_H$ are derived analogously and left to the reader. This proves \eqref{eq:zetabound8feb2} for all $w\geq2$.\\

We next show that $\zeta_E, \zeta_H \in \HHno$ by proving that there exist sequences $(\zeta_E)_n, (\zeta_H)_n$ of smooth vectorfields with 
\begin{align} \label{conddensity}
\mathrm{supp} (\zeta_E)_n, \mathrm{supp} (\zeta_H)_n \subset \subset \RRRwo
\end{align}
that converge as $n\to \infty$ in $\HHn$ to
\begin{align} \label{conddensity2}
(\zeta_E)_n \to \zeta_E, \, (\zeta_H)_n \to \zeta_H.
\end{align}

Indeed, let 
\begin{align*}
(\zeta_E)_n &:= \sum\limits_{l=2}^n \sum\limits_{m=-l}^l \zeta_E^{(lm)} E^{(lm)}, \\
(\zeta_H)_n &:= \sum\limits_{l=2}^n \sum\limits_{m=-l}^l \zeta_H^{(lm)} E^{(lm)}.
\end{align*}
By \eqref{eq:transfct}, \eqref{eq:defofZetaE} and \eqref{eq:defofZetaH}, it follows that for each $n$, these are smooth vectorfields satisfying \eqref{conddensity}. By \eqref{eq:zetabound8feb2}, the convergence \eqref{conddensity2} follows. \\
 
Next, we prove the integral identities \eqref{eq:propo8feb}. The first one follows by
\begin{align} \begin{aligned}
\isinf r^{1-l} \zeta_E^{(lm)} dr &= \isinf c_E^{(lm)} \pr_r \left( \chi((r-1)l) \right)dr \\
&= c_E^{(lm)} \left[ \chi((r-1)l)  \right]_1^\infty \\
&= c_E^{(lm)} ( 1-0)\\
&= \isinf r^{1-l} \left( \frac{l}{\sqrt{l(l+1)}} \rhon^{(lm)} -\rhod_E^{(lm)} \right) dr.
\end{aligned} \label{eq:feb25est45} \end{align}
The second identity is proven similarly and left to the reader. This proves part (1) of Proposition \ref{prop:boundaryvalues}.\\

{\bf (2) Precise estimate for $\zeta_E$ and $\zeta_H$.} We first prove \eqref{eq:zetaEestimatefeb15}. Consider the case $w=2$ of the estimate for $\zeta_E$ in \eqref{eq:zetaEestimatefeb15}. Using the Hodge-Fourier formalism, see Proposition \ref{prop:Howtoestimatefunctions} and Lemmas \ref{lem:commutationrelation} and \ref{lem:inversetoDDDD}, it suffices to prove
\begin{align*}
\int\limits_{1}^\infty r^2 \left( \frac{r}{\sqrt{\half l(l+1)-1}} \pr_r \zeta_E^{(lm)} \right)^2 dr \lesssim \int\limits_{1}^\infty \left( r \rhod_E^{(lm)} \right)^2 +\left( r \rhod_H^{(lm)} \right)^2 +\left( r \rhon^{(lm)} \right)^2 dr.
\end{align*}

By \eqref{feb25estref}, we can estimate
\begin{align} \begin{aligned}
\isinf r^2& \left( \frac{r}{\sqrt{\half l(l+1)-1}} \pr_r \zeta_E^{(lm)} \right)^2dr \\
& \lesssim \isinf r^2 \left( \zeta_E^{(lm)} \right)^2dr + \left(c_E^{(lm)}\right)^2 \isinf r^{2l+2} l^2 \left( (\pr^2 \chi)(l(r-1)) \right)^2 dr \\
&\lesssim \isinf r^2 \left( \zeta_E^{(lm)} \right)^2dr + l^2 \left(c_E^{(lm)}\right)^2 \int\limits_1^{1+\frac{1}{l}} r^{2l+2} dr \\
&\lesssim \isinf r^2 \left( \zeta_E^{(lm)} \right)^2dr + l^2 \left(c_E^{(lm)}\right)^2 \frac{1}{2l+3} \left( \left(1+\frac{1}{l} \right)^{3+2l}- 1\right) \\
&\lesssim \isinf r^2 \left( \zeta_E^{(lm)} \right)^2dr + \frac{l^2}{2l+3}\left(c_E^{(lm)}\right)^2 \\
&\lesssim \isinf r^2 \left( \zeta_E^{(lm)} \right)^2dr +  \isinf \left( r \rhod_E^{(lm)} \right)^2 +  \left( r \rhon^{(lm)} \right)^2 dr,
\end{aligned} \label{eq:8j66} \end{align}
where we uniformly estimated $\pr^2_r \chi$, used the fact that $\mathrm{supp} \, \pr^2_r \chi \subset [1,1+1/l]$ and \eqref{eq:constantestimate1112}. Together with Proposition \ref{prop:boundaryvalues}, this proves \eqref{eq:zetaEestimatefeb15} for $w=2$. \\ 

Consider now the case $w>2$ of \eqref{eq:zetaEestimatefeb15}. On the one hand, by the higher regularity of $\rhod_E^{(lm)},\rhon^{(lm)}$, the estimate of $c_E^{(lm)}$ improves, see \eqref{eq:improvedfeb9}. On the other hand, we can differentiate the explicit formula \eqref{eq:defofZetaE} by $\pr_r$, while taking angular derivatives correspond to multiplications by $\frac{\sqrt{l(l+1)}}{r}$. All terms appearing can be bounded analogously as in \eqref{eq:8j66} by using the improved bounds for $c^{(lm)}$ and the fact that for all $n\geq1$, $$\mathrm{supp} \, \pr_r^n \chi \subset [1, 1+1/l].$$
This proves \eqref{eq:zetaEestimatefeb15} for $\zeta_E$ for $w\geq2$. The proof for $\zeta_H$ is analogous and left to the reader. \\

It remains to show that $\cDd^{-1}_2 \left( \Ndn \zeta_E \right), \cDd^{-1}_2 \left( \Ndn \zeta_E \right) \in \ol{\HH}^{w-2}_{-5/2}$. Consider the statement for $\cDd^{-1}_2 \left( \Ndn \zeta_E \right)$. For each $n\geq2$, the smooth $S_r$-tangent tracefree symmetric $2$-tensor
\begin{align*}
V_n := \sum\limits_{l=2}^n \sum\limits_{m=-l}^l \left( \cDd_2^{-1} \left( \Ndn \zeta_E \right) \right)^{(lm)}_\psi \psi^{(lm)}
\end{align*}
has compact support in $\RRRwo$ by the definition of $\zeta_E$ in \eqref{eq:defofZetaE}, see Lemma \ref{lem:RelationsSpherical}. Further, by \eqref{eq:zetaEestimatefeb15}, $V_n \to  \cDd_2^{-1} \left( \Ndn \zeta_E \right)$ as $n\to \infty$ in $\HH^{w-2}_{-5/2}$. By definition of $\ol{\HH}^{w-2}_{-5/2}$, see Definition \ref{definitionolspace}, the statement for $\cDd^{-1}_2 \left( \Ndn \zeta_E \right)$ follows. The statement for $\cDd^{-1}_2 \left( \Ndn \zeta_H \right)$ follows analogously. This finishes the proof of part (2) of Proposition \ref{prop:boundaryvalues}.\\

{\bf (3) Elliptic regularity and boundary behaviour of $\de^{[\geq2]}, \si_{NN}^{[\geq2]}$.} First, by the elliptic theory of Appendix \ref{sec:WEIGHTEDellipticity}, it follows that for integers $w\geq2$,
\begin{align*}
\de^{[\geq2]} \in \ol{H}^1 \cap H^{w-1}_{-3/2}(\RRRwo), \,\,\, \si_{NN}^{[\geq2]} \in H^{w-2}_{-5/2}(\RRRwo)
\end{align*}
with estimates
\begin{align} \begin{aligned}
\Vert \de^{[\geq2]} \Vert_{H^{w-1}_{-3/2}(\RRR^3 \setminus \ol{B_1})} &\lesssim \Vert \rh \Vert_{\HHno}+ C_w\Vert \rh \Vert_{\ol{\HH}^0_{-5/2}}, \\ 
\Vert \si_{NN}^{[\geq2]} \Vert_{H^{w-2}_{-5/2}(\RRR^3 \setminus \ol{B_1})} &\lesssim \Vert \rh \Vert_{\HHno}+ C_w\Vert \rh \Vert_{\ol{\HH}^0_{-5/2}}.
\end{aligned} \label{regularityestimates8feb} \end{align}
Indeed, $\de^{[\geq2]}$ is estimated in $\ol{H}^1_{-3/2}$ by Proposition \ref{prop:ellmars10} and $\si_{NN}^{[\geq2]}$ in $H^0_{-5/2}(\RRRwo)$ by Lemma \ref{lowreg16mars}. Higher order regularity follows from Proposition \ref{higherregmars14}. The corresponding estimates obtained for $\de^{[\geq2]}$ and $\si_{NN}^{[\geq2]}$ are in terms of norms of the right-hand sides of \eqref{eq:delta4} and \eqref{eq:sigma2}. In turn, these right-hand sides are estimated thanks to Corollary \ref{corprac8j} and the estimates of the part (1) of the proof for $\zeta_E$ and $\zeta_H$.\\

We demonstrate now the improved boundary behaviour 
\begin{align*}
\de^{[\geq2]} \in \ol{H}^{w-1}_{-3/2}, \,\,\, \si_{NN}^{[\geq2]} \in \ol{H}^{w-2}_{-5/2}.
\end{align*}
We only need to consider the cases $w>2$ for $\de^{[\geq2]}$ and $w>3$ for $\si_{NN}^{[\geq2]}$. Indeed, else the trivial extension to $B_1$ is regular and in view of the boundary conditions 
\begin{align*}
\de^{[\geq 2]} \vert_{r=1} = \si_{NN}^{[\geq2]} = 0,
\end{align*}
the statement follows by Proposition \ref{prop:TrivialExtensionRegularity}.\\

By Proposition \ref{higherregmars14444}, it suffices to prove the following claim.
\begin{claim} \label{summaryclaimj8}
If $w>2$, then it holds that
\begin{align}
\pr_r \de^{[\geq2]} \vert_{r=1} = 0, \label{8j1}
\end{align}
and if $w>3$, then 
\begin{align}
\pr_r \si_{NN}^{[\geq2]} \vert_{r=1} = 0. \label{8j2}
\end{align}
\end{claim}
First, by \eqref{regularityestimates8feb} it holds that for all $l\geq2$, $m \in \{ -l, \dots, l \}$, if $w>2, w>3$, respectively,
\begin{align*}
&\isinf \left(\de^{(lm)} \right)^2 dr, \,\, \isinf (1+r)^2 \left( \pr_r \de^{(lm)} \right)^2 dr,\,\, \isinf (1+r)^4 \left( \pr_r^2 \de^{(lm)} \right)^2 dr < \infty, \\
&\isinf (1+r)^2 \left(\si_{NN}^{(lm)} \right)^2 dr, \,\, \isinf (1+r)^4 \left( \pr_r \si_{NN}^{(lm)} \right)^2 dr,\,\, \isinf (1+r)^6 \left( \pr_r^2 \si_{NN}^{(lm)} \right)^2 dr < \infty.
\end{align*}
By Lemma \ref{sobolev1d}, it follows that 
\begin{align} \begin{aligned}
\sup\limits_{r \in (1,\infty)} (1+r)^{1/2} \de^{(lm)}, \, \, \sup\limits_{r \in (1,\infty)} (1+r)^{3/2} \pr_r \de^{(lm)} < \infty, \\
\sup\limits_{r \in (1,\infty)} (1+r)^{3/2}\si_{NN}^{(lm)}, \, \, \sup\limits_{r \in (1,\infty)} (1+r)^{5/2} \pr_r \si_{NN}^{(lm)} < \infty.
\end{aligned} \label{eq:8j3}
\end{align}

We show now that if $w>2$, then for all $l\geq2$, $m \in \{ -l, \dots, l \}$, 
$$\pr_r \de^{(lm)} \vert_{r=1} = 0.$$ 
Definition \eqref{eq:delta4} is in the Hodge-Fourier formalism equivalent to the following ODEs on $r \in (1,\infty)$ for $\de^{(lm)}$ with $l\geq2, m \in \{ -l, \dots, l\}$, see Lemma \ref{lem:RelationsSpherical},
\begin{align}
r^{l-2} \pr_r \left( r^{-2l} \pr_r \left( r^{l+2} \de^{(lm)} \right) \right) = \frac{1}{r^2} \pr_r \left( r^2 \rhon^{(lm)} \right) - \frac{\sqrt{l(l+1)}}{r} \left( \rhod_E^{(lm)} + \zeta_E^{(lm)} \right). \label{eq:5febrad1}
\end{align}

On the one hand, using that $\de^{(lm)} \vert_{r=1}=0$, $l\geq2$, and \eqref{eq:8j3}, we get
\begin{align*}
\isinf  \pr_r \left( r^{-2l} \pr_r \left( r^{l+2} \de^{(lm)} \right) \right) dr & = \left[ r^{2-l} \pr_r \de^{(lm)} + (l+2) r^{1-l} \de^{(lm)} \right]_1^\infty \\
& = - \pr_r \de^{(lm)} \vert_{r=1}. 
\end{align*}
On the other hand, by \eqref{eq:5febrad1}, 
\begin{align} \begin{aligned}
\isinf  \pr_r \left( r^{-2l} \pr_r \left( r^{l+2} \de^{(lm)} \right) \right) =& \isinf r^{-l} \pr_r \left( r^2 \rhon^{(lm)} \right) - \sqrt{l(l+1)} r^{1-l} \left( \rhod_E^{(lm)} + \zeta_E^{(lm)} \right) \\
=& \underbrace{\left[ r^{2-l} \rhon^{(lm)} \right]_1^\infty}_{=0}  \\
&+ l \isinf r^{1-l} \left( \rhon^{(lm)} - \frac{\sqrt{l(l+1)}}{l} \left( \rhod_E^{(lm)} + \zeta_E^{(lm)} \right) \right) \\
=&0,
\end{aligned} \label{eq:Dez183}  \end{align}
where the boundary term vanished because $\rh_N \in \ol{H}^{w-2}_{-5/2}$, and where we also used the integral identity \eqref{eq:propo8feb}. This shows that $$\pr_r \de^{(lm)} \vert_{r=1} = 0$$ for all $l\geq2$, $m \in \{-l, \dots, l \}$ and proves \eqref{8j1}.\\


We show now that if $w>3$, then for all $l\geq2$, $m\in \{-l, \dots, l \}$, 
$$\pr_r \si_{NN}^{(lm)} \vert_{r=1} = 0.$$
Definition \eqref{eq:sigma2} is in the Hodge-Fourier formalism equivalent to the following ODEs on $r \in (1,\infty)$ for $\si_{NN}^{(lm)}$ with $l\geq2, m \in \{ -l, \dots, l\}$, see Lemma \ref{lem:RelationsSpherical},
\begin{align}
r^{\sqrt{l(l+1)+4}-1} \pr_r \left( r^{1-2 \sqrt{l(l+1)+4}} \pr_r \left( r^{\sqrt{l(l+1)+4}} \si_{NN}^{(lm)} \right) \right) = r \pr_r \left( \frac{\sqrt{l(l+1)}}{r^2} \left( \rhod_H^{(lm)} + \zeta_H^{(lm)} \right) \right). \label{eq:sideflm8feb}
\end{align}
On the one hand, using that $\si_{NN}^{(lm)}\vert_{r=1}=0$, $l\geq2$ and \eqref{eq:8j3},
\begin{align*}
\isinf \pr_r \left( r^{1-2 \sqrt{l(l+1)+4}} \pr_r \left( r^{\sqrt{l(l+1)+4}} \si_{NN}^{(lm)} \right) \right) =& \left[ \sqrt{l(l+1)+4} r^{-\sqrt{l(l+1)+4}} \si_{NN}^{(lm)} \right]_1^\infty \\
& + \left[ r^{1-\sqrt{l(l+1)+4}} \pr_r \si_{NN}^{(lm)} \right]_1^\infty \\
=& -\pr_r \si_{NN}^{(lm)} \vert_{r=1}.
\end{align*}
On the other hand, using \eqref{eq:sideflm8feb},
\begin{align*}
&\isinf \pr_r \left( r^{1-2 \sqrt{l(l+1)+4}} \pr_r \left( r^{\sqrt{l(l+1)+4}} \si_{NN}^{(lm)} \right) \right) \\
=& \isinf r^{2-\sqrt{l(l+1)+4}} \pr_r \left( \frac{\sqrt{l(l+1)}}{r^2} \left( \rhod_H^{(lm)} + \zeta_H^{(lm)} \right) \right)\\
=& \left[ r^{-\sqrt{l(l+1)+4}} \sqrt{l(l+1)} \left( \rhod_H^{(lm)} + \zeta_H^{(lm)} \right) \right]_1^\infty\\
&+ \sqrt{l(l+1)} \left(\sqrt{l(l+1)+4}-2\right) \isinf r^{-\sqrt{l(l+1)+4}-1}  \left( \rhod_H^{(lm)} + \zeta_H^{(lm)} \right) \\
=&0,
\end{align*}
where we used the integral identity \eqref{eq:propo8feb} and the boundary term vanished because $\rh, \zeta_H \in \HHno$. This shows that for all $l\geq2$, $m\in \{-l, \dots, l \}$, 
$$\pr_r \si_{NN}^{(lm)} \vert_{r=1} = 0$$ and proves \eqref{8j2}. This finishes the proof of Claim \ref{summaryclaimj8}. Hence, we have obtained the control of $\de^{[\geq2]} \in \HHmo, \si_{NN}^{[\geq2]} \in \HHno$. This finishes the proof of part (3) of Proposition \ref{prop:boundaryvalues}.

\begin{remark} \label{rem:supportstatement}
For $l\geq2, m \in \{-l, \dots, l \}$, if $\rhod_H^{(lm)}$ is compactly supported in $\RRRwo$, then $\si_{NN}^{(lm)}$ is compactly supported in $\RRRwo$. Indeed, this follows by integrating the radial ODE \eqref{eq:sideflm8feb} and using that by the construction of $\zeta_H^{(lm)}$, $\mathrm{supp} \zeta_H^{(lm)} \subset \subset \RRRwo$ and $\pr_r \si_{NN}^{(lm)} \vert_{r=1} = 0$. \end{remark}

{\bf (4) Precise estimate for $ \pr_r \si_{NN}^{[\geq2]}$.} First consider the case $w=2$ of \eqref{eq:mars22estim1}, that is,
\begin{align}
\left\Vert \cDd_1^{-1}(0, \pr_r \si_{NN}^{[\geq2]}) \right\Vert_{\HH^0_{-5/2}(\RRRwo)} \lesssim \Vert \rho \Vert_{\HH^0_{-5/2}(\RRRwo)}. \label{eq:mars22estim2}
\end{align}
Using the Fourier-Hodge formalism and  the previous estimates for $\si_{NN}^{[\geq2]}, \zeta_H$, it suffices to prove
\begin{align}
\isinf r^2 \left( \frac{r}{\sqrt{l(l+1)}} \pr_r \si_{NN}^{(lm)} \right)^2 dr \lesssim \isinf \left(r \rhod_H^{(lm)} \right)^2 +\left(r \si_{NN}^{(lm)} \right)^2 + \left(r\zeta_H^{(lm)} \right)^2 dr.
\label{8js1} \end{align}

First, we rewrite the integrand by using \eqref{eq:sideflm8feb}. Indeed, multiplying \eqref{eq:sideflm8feb} by $r^{-\sqrt{l(l+1)+4}+1}$ and integrating from $1$ to $r\geq1$ leads, after integration by parts, to the expression
\begin{align} \begin{aligned}
&\frac{r}{\sqrt{l(l+1)}} \pr_r \si_{NN}^{(lm)} \\
=& - \frac{\sqrt{l(l+1)+4}}{\sqrt{l(l+1)}} \si_{NN}^{(lm)} + \left( \rhod_H^{(lm)} + \zeta_H^{(lm)} \right)\\
&-\left(2-\sqrt{l(l+1)+4}\right) r^{\sqrt{l(l+1)+4}}\left( \isr (r')^{-\sqrt{l(l+1)+4}-1} \left( \rhod_H^{(lm)} + \zeta_H^{(lm)} \right) dr' \right),
\end{aligned} \label{8js2} \end{align}
where the boundary terms at $r=1$ vanished because $\rh, \zeta_H \in \HHno$ and $$\si_{NN}^{[\geq2]} \vert_{r=1} =\pr_r \si_{NN}^{[\geq2]} \vert_{r=1}=0.$$ 

We are now in the position to prove \eqref{8js1}. We have
\begin{align} \begin{aligned}
&\isinf r^2 \left( \frac{r}{\sqrt{l(l+1)}} \pr_r \si_{NN}^{(lm)} \right)^2 dr \\
\lesssim& \isinf r^2 \Big[ \left( \si_{NN}^{(lm)} \right)^2 + \left(\rhod_H^{(lm)} \right)^2+ \left(\zeta_H^{(lm)} \right)^2 \Big] dr\\
&+ (\sqrt{l(l+1)+4}-2)^2 \underbrace{\isinf r^{2\sqrt{l(l+1)+4}+2}\left( \isr (r')^{-\sqrt{l(l+1)+4}-1} \left( \rhod_H^{(lm)} + \zeta_H^{(lm)} \right) dr' \right)^2dr}_{:= I_2}.
\end{aligned} \label{mars224} \end{align}
By the integral identity \eqref{eq:propo8feb} we can rewrite $I_2$ and use integration by parts to get
\begin{align}\begin{aligned}
I_2 
=&\isinf r^{2\sqrt{l(l+1)+4}+2}\left( \int\limits_r^\infty (r')^{-\sqrt{l(l+1)+4}-1} \left( \rhod_H^{(lm)} + \zeta_H^{(lm)} \right) dr' \right)^2dr \\
=&  \frac{1}{2\sqrt{l(l+1)+4}+3} \left[  r^{2\sqrt{l(l+1)+4}+3} \left( \int\limits_r^\infty (r')^{-\sqrt{l(l+1)+4}-1} \left( \rhod_H^{(lm)} + \zeta_H^{(lm)} \right) dr' \right)^2 \right]_1^\infty \\
&+2 \isinf \frac{r^{\sqrt{l(l+1)+4}+2} \left( \rhod_H^{(lm)} + \zeta_H^{(lm)} \right)}{2\sqrt{l(l+1)+4}+3}\left( \int\limits_r^\infty (r')^{-\sqrt{l(l+1)+4}-1} \left( \rhod_H^{(lm)} + \zeta_H^{(lm)} \right) dr' \right) dr.
\end{aligned} \label{eqmars223} \end{align}
The boundary term on the right-hand side can be estimated as follows
\begin{align*}
&\left\vert  
r^{2\sqrt{l(l+1)+4}+3} \left( \int\limits_r^\infty (r')^{-\sqrt{l(l+1)+4}-1} \left( \rhod_H^{(lm)} + \zeta_H^{(lm)} \right) dr' \right)^2 \right\vert \\
&\leq 
r^{2\sqrt{l(l+1)+4}+3} \left( \int\limits_r^\infty (r')^{-2\sqrt{l(l+1)+4}-4} dr' \right)\left( \isinf \left(r \rhod_H^{(lm)} \right)^2 +  \left(r \zeta_H^{(lm)} \right)^2 dr\right) \\
&\leq  \frac{1}{2\sqrt{l(l+1)+4}+3} 
\left( \isinf \left(r \rhod_H^{(lm)} \right)^2 +  \left( r \zeta_H^{(lm)} \right)^2  dr \right).
\end{align*}
The integral term on the right-hand side of \eqref{eqmars223} is estimated by Cauchy-Schwarz as 
\begin{align*}
&\isinf r^{2\sqrt{l(l+1)+4}+2} \left( \rhod_H^{(lm)} + \zeta_H^{(lm)} \right) \left( \isr (r')^{-\sqrt{l(l+1)+4}-1} \left( \rhod_H^{(lm)} + \zeta_H^{(lm)} \right) dr' \right) dr \\
 \leq& \left( I_2\right)^{1/2} \left( \isinf  \left(r \rhod_H^{(lm)} \right)^2 +  \left(r \zeta_H^{(lm)} \right)^2 dr \right)^{1/2}.
\end{align*}
Putting everything together, we arrive at 
\begin{align*}
I_2 \lesssim  \frac{1}{(2\sqrt{l(l+1)+4}+3)^2}  \left( \isinf  \left(r\rhod_H^{(lm)} \right)^2 +  \left(r\zeta_H^{(lm)} \right)^2  dr \right).
\end{align*}
Plugging this into \eqref{mars224} yields 
\begin{align*}
\isinf r^2 \left( \frac{r}{\sqrt{l(l+1)}} \pr_r \si_{NN}^{(lm)} \right)^2 dr \lesssim \isinf \left(r \rhod_H^{(lm)} \right)^2 +\left(r \si_{NN}^{(lm)} \right)^2 + \left(r\zeta_H^{(lm)} \right)^2 dr.
\end{align*} 
This proves \eqref{eq:mars22estim2}, that is, the case $w=2$ of \eqref{eq:mars22estim1}.\\

We turn now to the case $w>2$ of \eqref{eq:mars22estim1}. By differentiating \eqref{8js2} in $r$ or taking the tangential derivative $\Nd$ which on the Fourier side amounts to multiplication by $\frac{\sqrt{l(l+1)}}{r}$, and using Proposition \ref{prop:Howtoestimatefunctions} and Lemma \ref{lem:commutationrelation}, we get that for all $w\geq2$, 
\begin{align}\label{estimate3916}
\left\Vert \cDd_1^{-1}\left(0, \pr_r \si_{NN}^{[\geq2]}\right) \right\Vert_{\HHn(\RRRwo)} \lesssim \Vert \rho \Vert_{\HHno}+ C_w\Vert \rh \Vert_{\ol{\HH}^0_{-5/2}},\end{align}
that is, we proved \eqref{eq:mars22estim1}.\\

It remains to show that $\cDd_1^{-1}(0, \pr_r \si_{NN}^{[\geq2]}) \in \ol{\HH}^{w-2}_{-5/2}$. By definition, it suffices to prove that there is a sequence $X_n$ of smooth vectorfields on $\RRRwo$ with
\begin{align*}
\mathrm{supp} \, X_n \subset \subset \RRRwo
\end{align*}
that converges as $n\to \infty$ in $\HH^{w-2}_{-5/2}$,
\begin{align*}
X_{n} &\to \cDd_1^{-1}\left(0, \pr_r \si_{NN}^{[\geq2]} \right).
\end{align*}

Let $\rh_n$ be a sequence of smooth vectorfields on $\RRRwo$ such that for all $n$,
\begin{align*}
\mathrm{supp} \, \rh_n &\subset\subset \RRRwo
\end{align*}
and as $n \to \infty$, in $\HH^{w-2}_{-5/2}(\RRRwo)$,
\begin{align*}
\rh_n &\to \rh.
\end{align*}

Consider the sequence
\begin{align*}
\rh_n^{[\leq n]} := \sum\limits_{l=1}^n \sum\limits_{m=-l}^l \left( (\rh_n)_E^{(lm)} E^{(lm)} + (\rh_n)_H^{(lm)} H^{(lm)}  \right)
\end{align*}
which satisfies for all $n$,
\begin{align*}
\mathrm{supp} \, \rh_n^{[\leq n]} &\subset\subset \RRRwo
\end{align*}
and as $n \to \infty$, in $\HH^{w-2}_{-5/2}(\RRRwo)$,
\begin{align*}
\rh_n^{[\leq n]} &\to \rh.
\end{align*}

By Remark \ref{rem:supportstatement} and the higher regularity estimates \eqref{eq:zetabound8feb} and \eqref{regularityestimates8feb}, it follows that solutions $(\si_{NN}^{[\geq2]} )_n$ to \eqref{eq:sigma2} with $\rh_n^{[\leq n]}$ and corresponding $(\zeta_H)_n$ defined in \eqref{eq:defofZetaH} on the right-hand side are smooth and satisfy
\begin{align*}
\mathrm{supp} \, \left( \si_{NN}^{[\geq2]}\right)_n \subset \subset \RRRwo.
\end{align*}

This shows that
\begin{align*}
X_n := \cDd_1^{-1}\left(0, \pr_r \left( \si_{NN}^{[\geq2]} \right)_n \right),
\end{align*}
is a sequence of smooth vectorfields with 
\begin{align*}
\mathrm{supp} \, X_n \subset \subset \RRRwo.
\end{align*}
Furthermore, by linearity and \eqref{eq:mars22estim1}, as $n\to \infty$,
\begin{align*}
&\left\Vert X_n - \cDd_1^{-1}\left(0, \pr_r \si_{NN}^{[\geq2]} \right) \right\Vert_{\HHn(\RRRwo)} \\
\lesssim& \Vert (\rhod_n^{[\leq n]})_H^{[\geq 2]} - \rhod_H^{[\geq2]} \Vert_{\HHno}+C_w \Vert (\rhod_n^{[\leq n]})_H^{[\geq 2]} - \rhod_H^{[\geq2]} \Vert_{\ol{\HH}^0_{-5/2}}\\
& \to 0.
\end{align*}
The above implies that $\cDd_1^{-1}(0, \pr_r \si_{NN}^{[\geq2]}) \in \ol{\HH}^{w-2}_{-5/2}$. This finishes the control of $\cDd_1^{-1}\left(0, \pr_r \si_{NN}^{[\geq2]} \right)$ and hence concludes the proof of Proposition \ref{prop:boundaryvalues}. \end{proof}


In the following lemma, we estimate all quantities that were not yet bounded in Proposition \ref{prop:boundaryvalues} and obtain the full regularity and boundary control of $k$ and $\si$.
\begin{lemma}[Full boundary control and regularity] \label{lem:bdrycntrlfeb5}
For $\rho = (\rho_N, \rhod) \in \HHno$, the symmetric $2$-tensors $k$ and $\si$ defined in \eqref{eq:delta1}-\eqref{eq:eh1} satisfy $k \in \HHmo, \si \in \HHno$, with
\begin{align} \begin{aligned}
\Vert k \Vert_{\HHmo} &\lesssim \Vert \rho \Vert_{\HHno}+ C_w\Vert \rh \Vert_{\ol{\HH}^0_{-5/2}}, \\
\Vert \si \Vert_{\HHno} &\lesssim \Vert \rho \Vert_{\HHno}+ C_w\Vert \rh \Vert_{\ol{\HH}^0_{-5/2}}.
\end{aligned} \label{eq:regularityestimatefeb833} \end{align}
\end{lemma}

\begin{proof} In view of Lemma \ref{lem:practical4j} and the decomposition of $k$ and $\si$ introduced at the beginning of Section \ref{sss:mars301}, we prove that $\de, \ep, \eh \in \HHmo$ and $\si_{NN},\sigmaNd, \sigmaNdd \in \HHno$ together with quantitative estimates. We estimate the terms in the order they were introduced in \eqref{eq:delta1}-\eqref{eq:eh1}.\\

{\bf Control of $\de$.} \\

\textit{Control of $\de^{[0]}$.} First we show that for all $w\geq2$,
\begin{align}
\Vert \de^{[0]} \Vert_{H^{w-1}_{-3/2}} \lesssim \Vert \rh \Vert_{\HH^{w-2}_{-5/2}}+ C_w\Vert \rh \Vert_{\ol{\HH}^0_{-5/2}}.  \label{eq:j8f1}
\end{align}
First consider the case $w=2$, that is,
\begin{align*}
\Vert \de^{[0]} \Vert_{H^1_{-3/2}} \lesssim \Vert \rh \Vert_{\HH^0_{-5/2}}. 
\end{align*}
By \eqref{eq:delta2} we can rewrite
\begin{align*}
\Vert \de^{[0]} \Vert_{H^0_{-3/2}}^2 = \isinf \sr \left( \de^{[0]} \right)^2 dr &= \isinf \frac{4 \pi r^2}{r^6} \left( \isr (r')^3 \rhon^{[0]} dr' \right)^2 dr\\
&=\frac{1}{4\pi} \isinf \frac{1}{r^4} \left( \isr r' \srp \rhon^{[0]} dr' \right)^2 dr,
\end{align*}
where we used that $\de^{[0]}$ and $\rhon^{[0]}$ are radial. This expression allows us to estimate by partial integration
\begin{align} \begin{aligned}
\Vert \de^{[0]} \Vert_{H^0_{-3/2}}^2 =& \left[ \left( -\frac{1}{3r^3} \right) \left( \isr r' \srp \rhon^{[0]} dr' \right)^2 \right]_1^\infty \\
&+ \frac{2}{3} \isinf \frac{1}{r^3} \left( r \sr \rhon^{[0]} \right) \left( \isr r' \srp \rhon^{[0]} dr' \right) dr\\
\leq& \frac{2}{3} \left( \isinf \left( \sr \rhon^{[0]} \right)^2 dr \right)^{1/2} \left( \isinf \frac{1}{r^4} \left( \isr r' \srp \rhon^{[0]} dr' \right)^2 dr \right)^{1/2}\\
=& \frac{2}{3} \left( \isinf \left( \sr \rhon^{[0]} \right)^2 dr \right)^{1/2}  \Vert \de^{[0]} \Vert_{H^0_{-3/2}} \end{aligned} \label{avril111} 
\end{align}
where the boundary term was discarded because it has non-positive sign. This implies that
\begin{align}\begin{aligned}
\Vert \de^{[0]} \Vert_{H^0_{-3/2}(\RRRwo)}^2 &\lesssim \isinf \left( \sr \rhon^{[0]} \right)^2 dr \\
&\lesssim \isinf \sr \left( r \rhon^{[0]} \right)^2 dr \\
&\lesssim \Vert \rhon^{[0]} \Vert^2_{\ol{H}^0_{-5/2}(\RRRwo)}.
\end{aligned} \label{eq:de1feb9} \end{align}
The radial derivative $\pr_r \de^{[0]}$ equals by \eqref{eq:delta2}
\begin{align}
\pr_r \de^{[0]} = -\frac{3}{r} \de^{[0]} + \rhon^{[0]}. \label{eq:Dez174}
\end{align}

This yields with \eqref{eq:de1feb9} the estimate
\begin{align}
\Vert \pr_r \de^{[0]} \Vert_{H^0_{-5/2}(\RRRwo)} \lesssim \Vert \rhon^{[0]} \Vert_{\ol{H}^0_{-5/2}(\RRRwo)}.
\end{align}
The tangential derivative vanishes because $\de^{[0]}$ is radial. This proves the case $w=2$ of \eqref{eq:j8f1}.\\

Consider now the case $w>2$ of \eqref{eq:j8f1}. Higher radial regularity follows by differentiating \eqref{eq:Dez174}, and higher tangential regularity is trivial since $\de^{[0]}$ is radial. This proves \eqref{eq:j8f1} for $w\geq2$.\\

For the control of $\de^{[0]}$ it remains to show that 
$\de^{[0]} \in \ol{H}^{w-1}_{-3/2}$. Indeed, this follows by \eqref{eq:Dez174}, the fact that $\rh \in \HHno$ and Proposition \ref{prop:TrivialExtensionRegularity}. This finishes the control of $\de^{[0]}$.\\


\textit{Control of $\de^{[1]}$.} First we show that for all $w\geq2$,
\begin{align}
\Vert \de^{[1]} \Vert_{H^{w-1}_{-3/2}(\RRRwo)} \lesssim \Vert \rh \Vert_{\ol{\HH}^{w-2}_{-5/2}(\RRRwo)}+ C_w\Vert \rh \Vert_{\ol{\HH}^0_{-5/2}}. \label{eq:controlde1feb9}
\end{align}
Consider first the case $w=2$, that is,
\begin{align*}
\Vert \de^{[1]} \Vert_{H^{1}_{-3/2}(\RRRwo)} \lesssim \Vert \rh \Vert_{\ol{\HH}^{0}_{-5/2}(\RRRwo)}.
\end{align*}
Integrating \eqref{eq:delta3} yields the explicit form
\begin{align} 
\de^{[1]} = \frac{1}{r^4} \isr r' \underbrace{\left( \int\limits_1^{r'} \left( \frac{1}{r''} \pr_r((r'')^4 \rhon^{[1]}) - (r'')^3 \Divd \rhod^{[1]} \right) dr''\right)}_{:=I_1(r')}  dr'. \label{eq:Dez181} 
\end{align}
Using \eqref{eq:Dez181} and integration by parts in $r$, we estimate
\begin{align*}
\Vert \de^{[1]} \Vert_{H^0_{-3/2}(\RRRwo)}^2 =& \isinf \sr \left( \de^{[1]} \right)^2 dr\\
=& \isinf \sr \frac{1}{r^8} \left( \isr r' I_1(r') dr' \right)^2 dr \\
=& -\frac{1}{5} \left[ \sr \frac{1}{r^7} \left( \isr r' I_1(r') dr' \right)^2 \right]_1^\infty \\
& + \frac{2}{5} \isinf \sr \frac{1}{r^7} (r I_1(r)) \left( \isr r' I_1(r') dr' \right) dr\\
\leq& \frac{2}{5} \left( \isinf \sr \frac{1}{r^8} \left( \isr r' I_1(r') dr' \right)^2 dr \right)^{1/2} \left( \isinf \sr \frac{1}{r^4} (I_1)^2(r) dr \right)^{1/2}\\
=& \frac{2}{5} \Vert \de^{[1]} \Vert_{H^0_{-3/2}(\RRRwo)} \left( \isinf \sr \frac{1}{r^4} (I_1)^2(r) dr \right)^{1/2},
\end{align*}
where the boundary term was discarded because of its non-positive sign. This shows that
\begin{align}
\Vert \de^{[1]} \Vert_{H^0_{-3/2}(\RRRwo)}^2 \lesssim \isinf \sr \frac{1}{r^4} (I_1)^2(r) dr. \label{av2}
\end{align}
By a similar integration by parts, we further have
\begin{align}
\isinf \sr \frac{1}{r^4} (I_1)^2(r) dr \lesssim \Vert \rh \Vert_{\ol{\HH}^0_{-5/2}}^2, \label{av21}
\end{align}
where we used that at $l=1$,
$$\Vert \Divd \rhod^{[1]} \Vert_{\ol{H}^0_{-7/2}} \lesssim \Vert \rhod^{[1]} \Vert_{\ol{H}^0_{-5/2}}.$$  
Together, \eqref{av2} and \eqref{av21} imply 
 $$\Vert \de^{[1]} \Vert_{H^0_{-3/2}(\RRRwo)} \lesssim \Vert \rh \Vert_{\ol{\HH}^0_{-5/2}}.$$

Moreover, by \eqref{eq:Dez181} the radial derivative $\pr_r \de^{[1]}$ is
\begin{align*}
\pr_r \de^{[1]} = -\frac{4}{r} \de^{[1]} + \frac{1}{r^3} I_1(r).
\end{align*}
Therefore \eqref{av2} and \eqref{av21} imply that 
$$\Vert \pr_r \de^{[1]} \Vert_{H^0_{-5/2}(\RRRwo)} \lesssim \Vert \rh \Vert_{\ol{\HH}^0_{-5/2}}^2.$$

The tangential regularity of $\de^{[1]}$ follows immediately from the fact that $l=1$, $$\Vert \Nd \de^{[1]} \Vert_{H^0_{-5/2}} \lesssim \Vert \de^{[1]} \Vert_{H^0_{-3/2}}.$$ This proves the case $w=2$ of \eqref{eq:controlde1feb9}.\\

We turn now to the case $w>2$ of \eqref{eq:controlde1feb9}. For higher radial regularity, differentiate the defining ODE \eqref{eq:delta3},
\begin{align}
\begin{cases}
\pr_r^2 \de^{[1]} + \frac{7}{r} \pr_r \de^{[1]} + \frac{8}{r^2} \de^{[1]} = \frac{1}{r^4} \pr_r \left( r^4 \rhon^{[1]} \right) -\Divd \rhod^{[1]},  \,\,\, \text{on } \RRR^3 \setminus \ol{B_1} \\
\de^{[1]}\vert_{r=1} = \pr_r \de^{[1]}\vert_{r=1} = 0.
\end{cases} \label{eq:delta3av1}
\end{align}
Higher tangential regularity follows at the level of $l=1$ in the Hodge-Fourier decomposition by the observation that for $w\geq0$
\begin{align}\begin{aligned}
\Vert  \Nd^{w} \de^{[1]} \Vert_{H^{0}_{-3/2-w}(\RRRwo)} &\lesssim C_w \Vert \de^{[1]} \Vert_{H^{0}_{-3/2}(\RRRwo)}\\
&\lesssim C_w \Vert \rh \Vert_{\ol{\HH}^{0}_{-5/2}}.
\end{aligned} \label{eq:tanglone} \end{align}
This proves \eqref{eq:controlde1feb9} for $w\geq2$. \\

It remains to show that $\de^{[1]} \in \ol{H}^{w-1}_{-3/2}$. Indeed, this follows by \eqref{eq:delta3av1} and $\rh \in \HHno$ with Proposition \ref{prop:TrivialExtensionRegularity}. This finishes the control of $\de^{[1]}$.\\

\textit{The full control of $\de$.} Recall that 
\begin{align*}
\de = \de^{[0]} +\de^{[1]} + \de^{[\geq2]}.
\end{align*}
Above we proved that for $w\geq2$, $\de^{[0]}, \de^{[1]} \in \ol{H}^{w-1}_{-3/2}$ with the estimate
\begin{align*}
\Vert \de^{[0]} + \de^{[1]} \Vert_{H^{w-1}_{-3/2}(\RRRwo)} \lesssim \Vert \rh \Vert_{\ol{\HH}^{w-2}_{-3/2}}+ C_w\Vert \rh \Vert_{\ol{\HH}^0_{-5/2}}.
\end{align*}
In Proposition \ref{prop:boundaryvalues}, we proved that for $w\geq2$, $\de^{[\geq2]} \in \ol{H}^{w-1}_{-3/2}$ with the estimate
\begin{align*}
\Vert \de^{[\geq2]} \Vert_{H^{w-1}_{-3/2}(\RRRwo)} \lesssim \Vert \rh \Vert_{\ol{\HH}^{w-2}_{-3/2}}+ C_w\Vert \rh \Vert_{\ol{\HH}^0_{-5/2}}.
\end{align*}
Together this proves that for $w\geq2$, $\de \in \ol{H}^{w-1}_{-3/2}$ with the estimate
\begin{align*}
\Vert \de \Vert_{H^{w-1}_{-3/2}(\RRRwo)} \lesssim \Vert \rh \Vert_{\ol{\HH}^{w-2}_{-3/2}}+ C_w\Vert \rh \Vert_{\ol{\HH}^0_{-5/2}}
\end{align*}
and hence finishes the control of $\de$. \\

{ \bf Control of $\si_{NN}$.} \\

\textit{Control of $\si_{NN}^{[1]}$.} First we show that for all $w\geq2$,
\begin{align}
\Vert \si_{NN}^{[1]} \Vert_{H^{w-2}_{-5/2}(\RRRwo)} \lesssim \Vert \rhod^{[1]} \Vert_{\ol{\HH}^{w-2}_{-5/2}}+ C_w \Vert \rhod^{[1]} \Vert_{\ol{\HH}^{0}_{-5/2}}. \label{eq:sigmaNNfirstestimate}
\end{align}

Consider first the case $w=2$, that is,
\begin{align*}
\Vert \si_{NN}^{[1]} \Vert_{H^0_{-5/2}(\RRRwo)} \lesssim \Vert \rhod^{[1]} \Vert_{\HH^0_{-5/2}(\RRRwo)}. \end{align*}
Recall \eqref{eq:sigma1},
\begin{align*}
\si_{NN}^{[1]} = \frac{1}{r^4} \isr (r')^4 \Curld \rhod^{[1]} dr'.
\end{align*}
Using this expression, the case $w=2$ of \eqref{eq:sigmaNNfirstestimate} can be derived like for $\de^{[0]}$ before, see \eqref{avril111} and \eqref{eq:de1feb9}. In particular, use that at $l=1$, $$\Vert \Curld \rhod^{[1]} \Vert_{H^0_{-7/2}(\RRRwo)} \lesssim \Vert \rhod^{[1]} \Vert_{\HH^0_{-5/2}(\RRRwo)}.$$ \\

We turn now to the case $w>2$ of \eqref{eq:sigmaNNfirstestimate}. Higher radial regularity is proved by using and differentiating the defining ODE,
\begin{align}\begin{cases}
\pr_r \si_{NN}^{[1]} + \frac{4}{r} \si_{NN}^{[1]}= \Curld \rhod^{[1]} \\
\si_{NN}^{[1]} \vert_{r=1} =0.
\end{cases} \label{definingODEfeb112}
\end{align}
Higher tangential regularity is automatic at $l=1$, as in \eqref{eq:tanglone}. This proves \eqref{eq:sigmaNNfirstestimate} for all $w\geq2$.\\

It remains to show that $ \si_{NN}^{[1]} \in \ol{H}^{w-2}_{-5/2}$. This follows by \eqref{definingODEfeb112} and $\rh \in \HHno$ with Proposition \ref{prop:TrivialExtensionRegularity}. This finishes the control of $\si_{NN}^{[1]}$.\\

\textit{The full control of $\si_{NN}$. } Recall that
\begin{align*}
\si_{NN} = \si_{NN}^{[1]} + \si_{NN}^{[\geq2]}.
\end{align*}
Above, we proved that for $w\geq2$, $\si_{NN}^{[1]} \in \ol{H}^{w-2}_{-5/2}$ with the estimate
\begin{align*}
\Vert \si_{NN}^{[1]} \Vert_{\ol{H}^{w-2}_{-5/2}} \lesssim \Vert \rh \Vert_{\ol{\HH}^{w-2}_{-5/2}}+C_w\Vert \rh \Vert_{\ol{\HH}^{0}_{-5/2}}.
\end{align*}
In Proposition \ref{prop:boundaryvalues}, we proved that for $w\geq2$, $\si_{NN}^{[\geq2]} \in \ol{H}^{w-2}_{-5/2}$ with the estimate
\begin{align*}
\Vert \si_{NN}^{[\geq2]} \Vert_{\ol{H}^{w-2}_{-5/2}} \lesssim \Vert \rh \Vert_{\ol{\HH}^{w-2}_{-5/2}} +C_w\Vert \rh \Vert_{\ol{\HH}^{0}_{-5/2}}.
\end{align*}
Together this proves that for $w\geq2$, $\si_{NN} \in \ol{H}^{w-2}_{-5/2}$ with the estimate
\begin{align*}
\Vert \si_{NN} \Vert_{\ol{H}^{w-2}_{-5/2}} \lesssim \Vert \rh \Vert_{\ol{\HH}^{w-2}_{-5/2}}+C_w\Vert \rh \Vert_{\ol{\HH}^{0}_{-5/2}}
\end{align*}
and hence finishes the control of $\si_{NN}$. \\


{\bf Control of $\ep$.} First we show that for $w\geq2$
\begin{align}
\Vert \ep \Vert_{\HHm(\RRRwo)} \lesssim \Vert \rh \Vert_{\HHno}+C_w\Vert \rh \Vert_{\ol{\HH}^{0}_{-5/2}}. \label{epest23m}
\end{align}
Consider first the case $w=2$ of \eqref{epest23m},
\begin{align*}
\Vert \ep \Vert_{\HH^1_{-3/2}(\RRRwo)} \lesssim \Vert \rh \Vert_{\ol{\HH}^0_{-5/2}(\RRRwo)}.
\end{align*}
By \eqref{eq:ep2} we have on $S_r$, for $r\geq1$,
\begin{align}
\cDd_1 \ep &= \Big( \rh_N - \frac{1}{r^3} \pr_r\left( r^3 \de \right),  \si_{NN}  \Big). \label{eq:feb15}
\end{align}
Proposition \ref{prop:EllipticityHodgejan} and the estimates above for $\de$ and $\si_{NN}$ yield
\begin{align*}
&\Vert \ep \Vert_{\HH^0_{-3/2}(\RRRwo)}^2 + \Vert \Nd \ep \Vert_{\HH^0_{-5/2}(\RRRwo)}^2 \\
\lesssim& \Vert \rh \Vert^2_{\HH^0_{-5/2}(\RRRwo)} + \Vert \de \Vert^2_{\HH^1_{-3/2}(\RRRwo)} + \Vert \si_{NN} \Vert^2_{\HH^0_{-5/2}(\RRRwo)} \\
\lesssim& \Vert \rh \Vert^2_{\HH^0_{-5/2}(\RRRwo)}.
\end{align*}

For radial regularity, we use that by Lemma \ref{lem:formalsolfeb5}, $\ep$ also solves \eqref{EH3a1} and \eqref{EH3a2}, together with \eqref{eq:sigmaNd2feb9} and \eqref{eq:sigmaNd3} to obtain 
\begin{align} \begin{aligned}
\frac{1}{r^3} \Ndn \left( r^3 \ep^{[1]} \right) &= \rhod^{[1]} + \half \Nd \de^{[1]},\\
\frac{1}{r^2} \Ndn \left( r^2 \ep_E^{[\geq2]} \right) &=  \rhod_E^{[\geq2]} + \zeta_E  + \left( \Nd \de \right)_E^{[\geq2]},\\
\frac{1}{r^2} \Ndn \left( r^2 \ep_H^{[\geq2]} \right) &=  \half \rhod_H^{[\geq2]} + \sigmaNd_H^{[\geq2]} \\
&= \cDd_1^{-1} \left( 0, \frac{1}{r^3} \pr_r\left( r^3 \si_{NN}^{[\geq2]} \right) \right).
\end{aligned} \label{eq:epdrfeb11} \end{align}
By Proposition \ref{prop:boundaryvalues} and the above estimates for $\de$, this yields the bounds
\begin{align*}
\Vert \Ndn \ep \Vert_{\HH^0_{-5/2}(\RRRwo)} \lesssim \Vert \rh \Vert_{\HH^0_{-5/2}(\RRRwo)}.
\end{align*}
This proves the case $w=2$ of \eqref{epest23m}.\\

Consider now the case $w>2$ of \eqref{epest23m}. Higher tangential regularity is derived by tangentially differentiating \eqref{eq:feb15} and using Propositions \ref{prop:EllipticityHodgefeb23} and \ref{prop:boundaryvalues}. Higher radial regularity follows by applying $\Ndn$ to \eqref{eq:epdrfeb11} and using Proposition \ref{prop:boundaryvalues}. This proves \eqref{epest23m} for all $w\geq2$.\\

It remains to show that $\ep \in \HHmo$. Indeed, this follows by \eqref{eq:epdrfeb11} and the fact that $\rh, \cDd_1^{-1} \left( 0, \frac{1}{r^3} \pr_r\left( r^3 \si_{NN}^{[\geq2]} \right) \right), \zeta_E, \Nd \de  \in \HHno(\RRRwo)$ together with Proposition \ref{prop:TrivialExtensionRegularity}. This finishes the control of $\ep$.\\


{\bf Control of $\sigmaNd$.} First we show that for all $w\geq2$
\begin{align}
\Vert \sigmaNd \Vert_{\HHn(\RRRwo)} \lesssim \Vert \rh \Vert_{\HHn(\RRRwo)}+C_w\Vert \rh \Vert_{\ol{\HH}^{0}_{-5/2}}. \label{eq:control15febs}
\end{align}
This control of $\sigmaNd$ follows by the control of the previous quantities. Indeed, recall from \eqref{eq:sigmaNd1}-\eqref{eq:sigmaNd3},
\begin{align}  \begin{aligned}
\sigmaNd^{[1]} &= \half \rhod^{[1]} - \half \Nd \de^{[1]} - \frac{1}{r} \ep^{[1]}, \\ 
\sigmaNd^{[\geq2]}_E &= \half \rhod^{[\geq2]}_E + \zeta_E,\\ 
 \sigmaNd_{H}^{[\geq2]} &= - \half \rhod_H^{[\geq2]}  +  \cDd_1^{-1} \left( 0, \frac{1}{r^3} \pr_r \left( r^3 \si_{NN}^{[\geq2]} \right) \right). 
\end{aligned} \label{summary12se} \end{align}
This implies by Proposition \ref{prop:boundaryvalues} and the above control of $\de$ and $\ep$ \begin{align*}
\Vert \sigmaNd^{[1]} \Vert_{\HHn(\RRRwo)} &\lesssim \Vert \rh \Vert_{\HHn(\RRRwo)} + \Vert \de \Vert_{H^{w-1}_{-3/2}} + \Vert \ep \Vert_{\HHm(\RRRwo)}\\
&\lesssim \Vert \rh \Vert_{\HHn(\RRRwo)}+C_w\Vert \rh \Vert_{\ol{\HH}^{0}_{-5/2}}, \\
\Vert\sigmaNd^{[\geq2]}_E \Vert_{\HHn(\RRRwo)} &\lesssim \Vert \rh \Vert_{\HHn(\RRRwo)} + \Vert \zeta_E \Vert_{\HHn(\RRRwo)}\\
&\lesssim \Vert \rh \Vert_{\HHn(\RRRwo)}+C_w\Vert \rh \Vert_{\ol{\HH}^{0}_{-5/2}}, \\
\Vert \sigmaNd_{H}^{[\geq2]} \Vert_{\HHn(\RRRwo)} &\lesssim \Vert \rh \Vert_{\HHn(\RRRwo)} + \left\Vert \cDd_1^{-1} \left( 0, \frac{1}{r^3} \pr_r \left( r^3 \si_{NN}^{[\geq2]} \right) \right) \right\Vert_{\HHn(\RRRwo)} \\
&\lesssim \Vert \rh \Vert_{\HHn(\RRRwo)}+C_w\Vert \rh \Vert_{\ol{\HH}^{0}_{-5/2}}.
\end{align*}
This proves \eqref{eq:control15febs} for all $w\geq2$. \\

It remains to show that $\sigmaNd \in \HHno$. This follows by \eqref{summary12se} and $\de,\ep \in \HHmo$, $\rh, \zeta_E, \cDd_1^{-1} \left( 0, \frac{1}{r^3} \pr_r \left( r^3 \si_{NN}^{[\geq2]} \right) \right) \in \HHno$ together with Proposition \ref{prop:TrivialExtensionRegularity}. This finishes the control of $\sigmaNd$. \\


{ \bf Control of $\sigmaNdd$.} First we show that for $w\geq2$,
\begin{align} \begin{aligned}
\Vert \sigmaNdd_\psi \Vert_{\HH^{w-2}_{-5/2}(\RRRwo)} & \lesssim \Vert \rh \Vert_{\HHno}+C_w\Vert \rh \Vert_{\ol{\HH}^{0}_{-5/2}}, \\
\Vert \sigmaNdd_\phi \Vert_{\HH^{w-2}_{-5/2}(\RRRwo)} & \lesssim \Vert \rh \Vert_{\HHno}+C_w\Vert \rh \Vert_{\ol{\HH}^{0}_{-5/2}}.\end{aligned} 
\label{finalest2avr}
\end{align}

In \eqref{eq:sigmaNdd2}, $\sigmaNdd$ was defined on each $S_r$, $r\geq1$, as solution to
\begin{align*}
\cDd_2 \left( \sigmaNdd + \half \Nd \widehat{\otimes} \ep^{[\geq2]} \right) &= \frac{1}{r^2} \Ndn \left( r^2 \left( \half \rhod^{[\geq2]} - \sigmaNd^{[\geq2]} - \half \Nd \de^{[\geq2]} - \frac{1}{r} \ep^{[\geq2]} \right) \right).
\end{align*}
Using definitions \eqref{eq:sigmaNd2feb9} and \eqref{eq:sigmaNd3}, this can be decomposed into
\begin{align} \begin{aligned}
\cDd_2 \left( \sigmaNdd_{\psi} + \half \Nd \widehat{\otimes} \ep^{[\geq2]}_E \right) &= \frac{1}{r^2} \Ndn \left( r^2 \left( - \zeta^{[\geq2]}_E -\half \left( \Nd \de\right)^{[\geq2]}_E - \frac{1}{r} \ep^{[\geq2]}_E \right) \right),\\
\cDd_2 \left( \sigmaNdd_\phi + \half \Nd \widehat{\otimes} \ep^{[\geq2]}_H \right) &= \frac{1}{r^2} \Ndn \left( r^2 \left(  \rhod^{[\geq2]}_H - \cDd_1^{-1} \left( 0, \frac{1}{r^3} \pr_r \left( r^3 \si_{NN}^{[\geq2]} \right) \right)  - \frac{1}{r} \ep^{[\geq2]}_H\right) \right). \end{aligned} \label{eq:mars191}
\end{align}
To analyse these equations, we first rewrite the second equation.

\begin{claim} \label{claim23may}
The second equation of \eqref{eq:mars191} is equivalent to
\begin{align} \begin{aligned}
\cDd_2 \left( \sigmaNdd_\phi + \half \Nd \widehat{\otimes} \ep^{[\geq2]}_H \right) =&
 \frac{3}{2r} \rhod_H^{[\geq2]} + \frac{1}{r} \zeta_H^{[\geq2]}  - \frac{3}{r} \sigmaNd_H^{[\geq2]} - \frac{1}{r^2} \Ndn \left( r \ep_H^{[\geq2]} \right) \\
& + \cDd_1^{-1} \left( 0,\Ld \si_{NN}^{[\geq2]} \right) - \Ndn \zeta_H^{[\geq2]}.  
\end{aligned}\label{may23} \end{align}
\end{claim}
\begin{proof} In the following, we use that for a scalar function $f^{[\geq1]}$,
\begin{align*}
\Ndn \left( \frac{1}{r} \cDd_1^{-1}(0,f) \right) = \frac{1}{r} \cDd_1^{-1}(0, \pr_r f^{[\geq1]}),
\end{align*}
this follows by Lemma \ref{lem:commutationrelation}. \\

By the definition of $\si_{NN}^{[\geq2]}$ in \eqref{eq:sigma2} and of $\sigmaNd_H^{[\geq2]}$ in \eqref{eq:sigmaNd3}, and using Lemma \ref{lem:commutationrelation}, 
\begin{align*}
&- \frac{1}{r^2} \Ndn \left( r^2 \cDd_1^{-1} \left( 0, \frac{1}{r^3} \pr_r \left( r^3 \si_{NN}^{[\geq2]} \right) \right) \right) \\
=& - \frac{3}{r}  \cDd_1^{-1} \left( 0, \frac{1}{r^3} \pr_r \left( r^3 \si_{NN}^{[\geq2]} \right) \right) -  \cDd_1^{-1} \left( 0, \pr_r \left( \frac{1}{r^3} \pr_r \left( r^3 \si_{NN}^{[\geq2]} \right) \right) \right)\\
=&- \frac{3}{r} \left( \half \rhod_H^{[\geq2]}+ \sigmaNd_H^{[\geq2]}\right) -  \cDd_1^{-1} \left( 0, - \Ld \si_{NN}^{[\geq2]} + \pr_r \Curld \left(\rhod^{[\geq2]}_H +\zeta_H^{[\geq2]}  \right) \right)\\
=&- \frac{3}{r} \left( \half \rhod_H^{[\geq2]}+ \sigmaNd_H^{[\geq2]}\right) + \cDd_1^{-1} \left( 0,\Ld \si_{NN}^{[\geq2]} \right) \\
&-  \cDd_1^{-1} \left( 0, \pr_r \Curld \left(\rhod^{[\geq2]}_H +\zeta_H^{[\geq2]}  \right) \right) \\
=& - \frac{3}{r} \left( \half \rhod_H^{[\geq2]}+ \sigmaNd_H^{[\geq2]}\right) + \cDd_1^{-1} \left( 0,\Ld \si_{NN}^{[\geq2]} \right) \\
&- \Ndn \left( \rhod^{[\geq2]}_H +\zeta_H^{[\geq2]} \right) + \frac{1}{r} \left( \rhod^{[\geq2]}_H +\zeta_H^{[\geq2]} \right)\\
=& - \frac{3}{r} \sigmaNd_H^{[\geq2]} + \cDd_1^{-1} \left( 0,\Ld \si_{NN}^{[\geq2]} \right)  - \Ndn \left( \rhod^{[\geq2]}_H +\zeta_H^{[\geq2]} \right) + \frac{1}{r} \left( -\half \rhod^{[\geq2]}_H +\zeta_H^{[\geq2]} \right).
\end{align*}
Plugging this into the second equation of \eqref{eq:mars191} finishes the proof of Claim \ref{claim23may}. \end{proof}
 
By differentiating the first of \eqref{eq:mars191} and \eqref{may23}, and using the commutation relations of Lemma \ref{lem:commutationrelation}, we can apply Propositions \ref{prop:EllipticityHodgejan} and \ref{prop:EllipticityHodgefeb23} to get for $w\geq2$
\begin{align} \begin{aligned}
&\Vert \sigmaNdd_\psi \Vert_{\HH^{w-2}_{-5/2}(\RRRwo)} \\
\lesssim& \Vert \ep \Vert_{\HHm(\RRRwo)} + \Vert \zeta_E \Vert_{\HHn(\RRR^3)} + \Vert \cDd^{-1}_2 \left( \Ndn \zeta_E \right) \Vert_{\HH^{w-2}_{-5/2}(\RRRwo)}+ \Vert \de \Vert_{\HHm(\RRRwo)}\\
&+ C_w\Big( \Vert \ep \Vert_{\HH^1_{-3/2}} + \Vert \zeta_E \Vert_{\HH^0_{-5/2}(\RRRwo)} + \Vert \cDd^{-1}_2 \left( \Ndn \zeta_E \right) \Vert_{\HH^{0}_{-5/2}(\RRRwo)}+ \Vert \de \Vert_{\HH^1_{-3/2}(\RRRwo)}\Big) \\
\lesssim& \Vert \rh \Vert_{\HHn(\RRRwo)}+C_w\Vert \rh \Vert_{\ol{\HH}^{0}_{-5/2}}, 
\end{aligned}  \label{feb15eqestnew1}
\end{align}

\begin{align} \begin{aligned} 
&\Vert  \sigmaNdd_\phi \Vert_{\HH^{w-2}_{-5/2}(\RRRwo)}\\
 \lesssim& \Vert \ep \Vert_{\HHm(\RRRwo)} + \Vert \cDd^{-1}_2 \left( \Ndn \zeta_H \right) \Vert_{\HHn(\RRRwo)} + \Vert  \zeta_H \Vert_{\HH^{w-2}_{-5/2}(\RRRwo)} + \Vert \rhod_H^{[\geq2]}  \Vert_{\HHn(\RRRwo)} \\
&+ \Vert  \si_{NN}^{[\geq2]} \Vert_{\HHn(\RRRwo)} +  \Vert  \sigmaNd^{[\geq2]} \Vert_{\HHn(\RRRwo)}+ C_w \Vert \rh \Vert_{\ol{\HH}^{0}_{-5/2}} \\
 \lesssim& \Vert \rh \Vert_{\HHn(\RRRwo)}+C_w\Vert \rh \Vert_{\ol{\HH}^{0}_{-5/2}}, 
  \end{aligned} 
\label{feb15eqestnew2}
\end{align} 
where we used Proposition \ref{prop:boundaryvalues}. This proves \eqref{finalest2avr} for all $w\geq2$. \\

It remains to show that $\sigmaNdd \in \HHno$. This follows by \eqref{eq:mars191}, \eqref{may23}, by the fact that $\nab_N \de, \rh, \zeta_E, \zeta_H \in \HHno, \, \de, \ep \in \HHmo$ and $\cDd_{2}^{-1} (\Ndn \zeta_E), \cDd_{2}^{-1} (\Ndn \zeta_H )\in \ol{\HH}^{w-2}_{-5/2}$ together
with Proposition \ref{prop:TrivialExtensionRegularity}. Indeed, to show for a $S_r$-tangent symmetric $2$-tensor $V_\psi^{[\geq2]}$ that $V \vert_{r=1} = 0$, it suffices to prove that $\Divd \left( V \right) \vert_{r=1} =0$. Together with the commutation relations of Lemma \ref{lem:commutationrelation}, this concludes the control of $\sigmaNdd$.\\ 

{\bf Control of $\eh$.} First we prove for $w\geq2$,
\begin{align}
\Vert \eh \Vert_{\HHm(\RRRwo)} \lesssim \Vert \rh \Vert_{\HHno}+C_w\Vert \rh \Vert_{\ol{\HH}^{0}_{-5/2}}. \label{eq:8jj3}
\end{align}
The control of $\eh$ follows by the control of the above quantities. Indeed, recall \eqref{eq:eh1},
\begin{align}
\cDd_2 \eh =  \half \rhod^{[\geq2]} - \sigmaNd^{[\geq2]} - \half \Nd \de^{[\geq2]} - \frac{1}{r}\ep^{[\geq2]}. \label{eq:defofEHfeb15}
\end{align}
Higher tangential regularity follows by Propositions \ref{prop:EllipticityHodgejan} and \ref{prop:EllipticityHodgefeb23}. Indeed, tangentially deriving \eqref{eq:defofEHfeb15} yields for all $w\geq2$, by the above control of $\de, \ep,\sigmaNd$, 
\begin{align*}
\Vert \eh \Vert_{\HH^0_{-3/2}(\RRRwo)}+ \sum\limits_{n=1}^{w-1}\Vert \Nd^{n} \eh \Vert_{\HH^0_{-3/2-n}(\RRRwo)} 
\lesssim \Vert \rh \Vert_{\ol{\HH}^{w-2}_{-5/2}} + C_w\Vert \rh \Vert_{\ol{\HH}^{0}_{-5/2}}.
\end{align*}
For radial regularity, use that by Lemma \ref{lem:formalsolfeb5}, $\eh$ also solves \eqref{EH5a2}, that is,
 \begin{align*}
\Ndn \eh + \frac{1}{r} \eh &= \sigmaNdd + \half \Nd \widehat{\otimes} \ep^{[\geq2]}. 
\end{align*}
Differentiating in $r$ yields with the above control of $\de, \ep, \sigmaNd, \sigmaNdd$ for all $w\geq2$ the bound 
\begin{align*}
\Vert \Nd_N^{w-1} \eh \Vert_{\HH^0_{-1/2-w}(\RRRwo)} \lesssim \Vert \rh \Vert_{\ol{\HH}^{w-2}_{-5/2}}+C_w\Vert \rh \Vert_{\ol{\HH}^{0}_{-5/2}}.
\end{align*} 
This proves \eqref{eq:8jj3} for $w\geq2$.\\

It remains to show that $\eh \in \HHmo$. This follows by \eqref{EH5a2} and $\ep \in \HHmo, \sigmaNdd \in \HHno$ together with Proposition \ref{prop:TrivialExtensionRegularity}. This finishes the control of $\eh$ as well as the proof of Lemma \ref{lem:bdrycntrlfeb5}. \end{proof}


\section{The prescribed scalar curvature equation for $g$} \label{sec:MetricPrescription}

In this section we prove the following theorem.
\begin{theorem}[Metric extension theorem, precise version] \label{thm:MainExtensionScalarTheorem} There exists a universal constant $\varep>0$ such that the following holds.
\begin{enumerate}
\item {\bf Extension result.} Let $\bar g \in \HH^2(B_1)$ be a Riemannian metric on the unit ball $B_1 \subset \RRR^3$ with scalar curvature $R(\bar g)$, and let $R \in H^{0}_{-5/2}$ be such that $R \vert_{B_1}= R(\bar g)$. If
\begin{align}\begin{aligned}
\Vert \bar g - e \Vert_{\HH^2(B_1)} +\Vert R \Vert_{H^{0}_{-5/2}} &< \varep,
\end{aligned} \label{eq:Rextensioncond} \end{align}
then there exists an $\HH^2_{-1/2}$-asymptotically flat metric $g$ on $\RRR^3$ such that $g\vert_{B_1} = \bar g$ and its scalar curvature satisfies
\begin{align*}
R(g)&=R \,\,\, \text{on } \RRR^3,
\end{align*}
Moreover, it is bounded by
\begin{align}
\Vert g - e \Vert_{\HH^2_{-1/2}} &\lesssim \Vert \bar g - e \Vert_{\HH^2(B_1)} + \Vert R \Vert_{H^{0}_{-5/2}}, \label{eq:Dez2166}
\end{align}
\item {\bf Iteration estimates.} Let $\bar g\in \HH^2(B_1)$ be a Riemannian metric on $B_1$ and let $R, \tilde R \in H^{0}_{-5/2}$ such that $R\vert_{B_1} = \tilde R \vert_{B_1} = R(\bar g)$ and \eqref{eq:Rextensioncond} holds for $(\bar g, R)$ and $(\bar g, \tilde R)$. Let $g$ and $\tilde{g}$ denote the metrics constructed in part (1) of this theorem with respect to $R$ and $\tilde R$. Then
\begin{align}
\Vert g - \tilde{g} \Vert_{\HH^2_{-1/2}} \lesssim \Vert R - \tilde R\Vert_{H^{0}_{-5/2}}. \label{eq:Riterationestimate222}
\end{align}
\item {\bf Higher regularity.} If, in addition to \eqref{eq:Rextensioncond}, $R \in \HH^{w-2}_{-5/2}$ and $\bar g -e \in \HH^w(B_1)$ for an integer $w\geq3$, then the $g$ constructed in part (1) of this theorem satisfies
\begin{align*}
\Vert g-e \Vert_{\HH^w_{-1/2}} \lesssim \Vert R \Vert_{H^{w-2}_{-5/2}} + C_w \Big( \Vert \bar g -e \Vert_{\HH^w(B_1)} + \Vert R \Vert_{H^0_{-5/2}} \Big),
\end{align*}
where the constant $C_w>0$ depends only on $w$.
\end{enumerate}
\end{theorem}

Before turning to the proof of Theorem \ref{thm:MainExtensionScalarTheorem}, we first analyse in more detail the scalar curvature functional in the next section.

\subsection{Scalar curvature and geometry of foliations}  \label{subsec:geometry}

In this section, we analyse the scalar curvature functional with respect to the foliation of $\RRR^3$ by spheres $S_r$.
\begin{lemma} \label{scalarcurvaturefirstlem}
Let $g$ be a smooth Riemannian metric on $\RRRwo$,
\begin{align*}
g = a^2 dr + \ga_{AB} (\be^A dr +d\th^A) (\be^B dr + d\th^B).
\end{align*}
Then the scalar curvature $R(g)$ of $g$ on $\RRRwo$ is given by
\begin{align*}
R(g) = 2 N (\tr_\ga \Theta) -\frac{2}{a} \Ld_\ga a +2 K(\ga) - (\tr_\ga \Theta)^2 - \vert \Theta \vert_\ga^2,
\end{align*}
where $N$ and $\Theta$ denote the unit normal and the second fundamental form of $S_r \subset \RRR^3$ with respect to $g$, respectively, and $K(\ga)$ is the Gauss curvature of $(S_r,\ga)$.
\end{lemma}

\begin{proof} The lemma follows by the traced second variation equation\footnote{Recall our sign convention $\Theta(X,Y)=- g(X, \nab_Y N)$.}
\begin{align*}
N (\tr_\ga \Theta) = \frac{1}{a} \Ld_\ga a + \mathrm{Ric}(N,N) + \vert \Theta \vert^2_\ga
\end{align*}
and the twice traced Gauss equation
\begin{align*}
R(g) = 2 \mathrm{Ric}(N,N) + 2 K(\ga) - (\tr_\ga \Theta)^2 + \vert \Theta \vert^2_\ga,
\end{align*}
where $\mathrm{Ric}$ denotes the Ricci tensor of $g$. See Section 1 of \cite{SmithWeinstein2} for a detailed derivation. This finishes the proof of Lemma \ref{scalarcurvaturefirstlem}.
\end{proof}

We introduce the following variations of Riemannian metrics.
\begin{definition}\label{definitionmathcalSSSS2222} \label{definitionmathcalSSSS} Let $g$ be a Riemannian metric on $\RRRwo$,
\begin{align*}
g = a^2 dr + \ga_{AB} (\be^A dr +d\th^A) (\be^B dr + d\th^B),
\end{align*}
and let further $\varphi$ be a scalar function and $\be'$ a $S_r$-tangent vectorfield on $\RRRwo$. We define the \emph{variation of $g$ by $(\varphi,\be')$} as
\begin{align*}
\check{ g}_{\varphi,\be'} := a^2 dr^2 + e^{2\varphi} \ga_{AB} \left( (\be + \be')^A dr + d\th^A \right) \left( (\be + \be')^B dr +d\th^B \right),
\end{align*}
and set
\begin{align*}
\SS(\varphi, \be', g)  := R(\check{g}_{\varphi, \be'}) - R(g).
\end{align*}
\end{definition}

\begin{lemma} \label{varphibetaboundg}
Let $g$ be an $\HH^2_{-1/2}$-asymptotically flat metric on $\RRRwo$, $\varphi \in H^2_{-1/2}(\RRRwo)$ a scalar function and $\be' \in \HH^{2}_{-1/2}(\RRRwo)$ an $S_r$-tangent vectorfield. Then there exists a universal constant $\varep>0$ such that the following holds.
\begin{enumerate} 
\item If
\begin{align} \label{smallnesscondition2433534}
\Vert g-e \Vert_{\HH^2_{-1/2}(\RRRwo)} < \varep, \, \Vert \varphi \Vert_{H^2_{-1/2}(\RRRwo)} < \varep, \, \Vert \be' \Vert_{\HH^2_{-1/2}(\RRRwo)}<\varep,
\end{align}
then
\begin{align*}
\Vert \check{g}_{\varphi,\be'} -e \Vert_{\HH^2_{-1/2}(\RRRwo)} \lesssim \Vert g-e \Vert_{\HH^2_{-1/2}} + \Vert \varphi \Vert_{H^2_{-1/2}(\RRRwo)} + \Vert \be' \Vert_{\HH^2_{-1/2}(\RRRwo)}.
\end{align*}
\item If, in addition to   \eqref{smallnesscondition2433534}, the metric $g$ is $\HH^w_{-1/2}$-asymptotically flat and $\varphi \in H^w_{-1/2}(\RRRwo)$ and $\be' \in \HH^w_{-1/2}(\RRRwo)$ for an integer $w\geq3$, then 
\begin{align*}
&\Vert \check{g}_{\varphi,\be'} -e \Vert_{\HH^w_{-1/2}(\RRRwo)}\\
 \lesssim& \Vert g-e \Vert_{\HH^w_{-1/2}(\RRRwo)} + \Vert \varphi \Vert_{H^w_{-1/2}(\RRRwo)}+ \Vert \be' \Vert_{\HH^w_{-1/2}(\RRRwo)} \\
&+ C_w \Big( \Vert g-e \Vert_{\HH^2_{-1/2}(\RRRwo)} + \Vert \varphi \Vert_{H^2_{-1/2}(\RRRwo)} + \Vert \be' \Vert_{\HH^2_{-1/2}(\RRRwo)}\Big).
\end{align*}
\item Let $\tilde{\varphi} \in H^2_{-1/2}(\RRRwo)$ be a second scalar function and $\tilde{\be'} \in \HH^{2}_{-1/2}(\RRRwo)$ a second $S_r$-tangent vectorfield satisfying \eqref{smallnesscondition2433534}. Then it holds that
\begin{align*}
\Vert \check{g}_{\varphi, \be'} - \check{ g}_{\tilde{\varphi},\tilde{\be'}}\Vert_{\HH^2_{-1/2}(\RRRwo)} \lesssim \Vert \varphi - \tilde{\varphi} \Vert_{H^2_{-1/2}(\RRRwo)} + \Vert \be' - \tilde{\be}' \Vert_{\HH^2_{-1/2}(\RRRwo)}.
\end{align*}
\end{enumerate}
\end{lemma}
\begin{proof} {\bf Proof of parts (1) and (2).} By Lemma \ref{lem:coordinatechangeAF}, it suffices to show that for $\varep>0$ sufficiently small, for integers $w\geq2$,
\begin{align}\begin{aligned}
&\Vert a^2 -1 \Vert_{H^w_{-1/2}(\RRRwo)} + \Vert \be^A + \be'^A \Vert_{\HH^w_{-1/2}(\RRRwo)} + \Vert e^{2\varphi} \ga - \gac \Vert_{\HH^w_{-1/2}(\RRRwo)} \\
\lesssim& \Vert g - e \Vert_{\HH^w_{-1/2}(\RRRwo)} + \Vert \varphi \Vert_{H^{w}_{-1/2}(\RRRwo)} +\Vert \be' \Vert_{\HH^{w}_{-1/2}(\RRRwo)}  \\
& + C_w \Big(\Vert g - e \Vert_{\HH^2_{-1/2}(\RRRwo)} + \Vert \varphi \Vert_{H^{2}_{-1/2}(\RRRwo)} +\Vert \be' \Vert_{\HH^{2}_{-1/2}(\RRRwo)} \Big).
\end{aligned} \label{compoest2324HIGHER} \end{align}
The first term on the left-hand side of \eqref{compoest2324HIGHER} has not been changed in the variation and thus for $\varep>0$ sufficiently small
\begin{align*}
\Vert a^2 -1 \Vert_{H^w_{-1/2}(\RRRwo)} \lesssim \Vert g - e \Vert_{\HH^w_{-1/2}(\RRRwo)} + C_w \Vert g - e \Vert_{\HH^2_{-1/2}(\RRRwo)}.
\end{align*}
The second term on the left-hand side of \eqref{compoest2324HIGHER} is bounded by
\begin{align*}
&\Vert \be^A + \be'^A \Vert_{\HH^w_{-1/2}(\RRRwo)} \\
\lesssim & \Vert \be^A \Vert_{\HH^w_{-1/2}(\RRRwo)} + \Vert \be'^A \Vert_{\HH^w_{-1/2}(\RRRwo)} \\
\lesssim& \Vert g - e \Vert_{\HH^w_{-1/2}(\RRRwo)} + C_w \Vert g - e \Vert_{\HH^2_{-1/2}(\RRRwo)} + \Vert \be'^A \Vert_{\HH^w_{-1/2}(\RRRwo)}. 
\end{align*}
The third term on the left-hand side of \eqref{compoest2324HIGHER} is bounded by
\begin{align} \begin{aligned}
\Vert e^{2\varphi} \ga - \gac \Vert_{\HH^w_{-1/2}(\RRRwo)} \leq \Vert e^{2\varphi} (\ga-\gac) \Vert_{\HH^w_{-1/2}(\RRRwo)} + \Vert (e^{2\varphi}-1) \gac  \Vert_{H^w_{-1/2}(\RRRwo)}.
\end{aligned} \label{expest2343} \end{align}
For $\varep>0$ sufficiently small, the first term on the right-hand side of \eqref{expest2343} is bounded by using Lemmas \ref{SobolevEmbeddingsAndNonlinear}, Corollary \ref{expHigherReg} and product estimates as in Lemma \ref{ProductEstimates} by
\begin{align*}
\Vert e^{2\varphi} (\ga-\gac) \Vert_{\HH^w_{-1/2}(\RRRwo)} \lesssim& \Vert g-e \Vert_{\HH^w_{-1/2}(\RRRwo)} +  \Vert \varphi \Vert_{H^w_{-1/2}(\RRRwo)} \\
&+ C_w \Big( \Vert g-e \Vert_{\HH^2_{-1/2}} +  \Vert \varphi \Vert_{H^2_{-1/2}(\RRRwo)} \Big).
\end{align*}
The second term on the right-hand side of  \eqref{expest2343} is bounded similarly by Corollary \ref{expHigherReg},
\begin{align*}
\Vert (e^{2\varphi}-1) \gac  \Vert_{H^w_{-1/2}(\RRRwo)} \lesssim& \Vert \varphi \Vert_{H^w_{-1/2}(\RRRwo)} + C_w \Vert \varphi \Vert_{H^2_{-1/2}(\RRRwo)}.\end{align*}
This finishes the proof of \eqref{compoest2324HIGHER} and therefore finishes the proof of parts (1) and (2) of Lemma \ref{varphibetaboundg}. \\

{\bf Proof of part (3).} Indeed, by the construction of $\check{g}_{\varphi, \be'},\check{g}_{\tilde \varphi, \tilde{\be'}}$ as variations of $g$ with $(\varphi, \be'), (\tilde \varphi, \tilde{\be'})$ we can write, see Section \ref{ssec:tensordecomposition}, with $B=1,2$,
\begin{align} \begin{aligned}
(\check{g}_{\varphi, \be'} -\check{g}_{\tilde \varphi, \tilde{\be'}})_{NN} &= a^2-a^2 =0, \\
\check{g}_{\varphi, \be'} \mkern-32mu/\ \mkern+20mu -\check{g}_{\tilde \varphi, \tilde{\be'}} \mkern-32mu/\ \mkern+20mu &= (e^{2\varphi}-e^{2\tilde{\varphi}}) \ga,\\
\left({\check{g}_{\varphi, \be'} \mkern-32mu/\ \mkern+14mu}_{N} - {\check{g}_{\tilde \varphi, \tilde{\be'}} \mkern-32mu/\ \mkern+14mu}_{N}\right)_B &= e^{2\varphi} \ga_{B A} (\be^A + \be'^A)-e^{2\tilde \varphi} \ga_{B A} \left(\be^A + \tilde{\be'}^A \right) \\
&= (e^{2\varphi}-e^{2\tilde \varphi}) \ga_{B A} \left( \be^A + \be'^A \right) +e^{2\tilde{\varphi}}  \ga_{B A} \left(\be'^A-\tilde{\be'}^A\right).
\end{aligned} \label{eq:differenceform5j} \end{align}
By Lemma \ref{SobolevEmbeddingsAndNonlinear} and \eqref{eq:iterationestimateuxifeb18} it holds for $\varep>0$ sufficiently small that 
\begin{align*}
\left\Vert e^{2\varphi} -e^{2\tilde \varphi}  \right\Vert_{H^2_{-1/2}(\RRRwo)} =& \left\Vert e^{2\tilde{\varphi}} (e^{2(\varphi-\tilde{\varphi})}-1) \right\Vert_{H^2_{-1/2}(\RRRwo)}\\
\lesssim& \Vert \varphi- \tilde{\varphi} \Vert_{H^2_{-1/2}}.
\end{align*}
For $\varep>0$ small enough, we can use this estimate and the expression \eqref{eq:differenceform5j} to apply Lemma \ref{lem:nonlinearities}, product estimates as in Lemma \ref{ProductEstimates}, and Lemma \ref{lem:coordinatechangeAF} to get
\begin{align*}
\Vert \check{g}_{\varphi, \be'}-\check{g}_{\tilde \varphi, \tilde{\be'}} \Vert_{\HH^2_{-1/2}(\RRRwo)} \lesssim \Vert \varphi- \tilde{\varphi} \Vert_{H^2_{-1/2}(\RRRwo)} + \Vert \be'- \tilde{\be'} \Vert_{\HH^2_{-1/2}(\RRRwo)}.
\end{align*}
This finishes the proof of Lemma  \ref{varphibetaboundg}. \end{proof}


The next lemma shows how the scalar curvature changes under variation $g \mapsto \check{g}_{\varphi, \be'}$ defined above. 
\begin{lemma} \label{variationscalarlem2} Let $g$ be a smooth Riemannian metric, 
\begin{align*}
g = a^2 dr + \ga_{AB} (\be^A dr +d\th^A) (\be^B dr + d\th^B),
\end{align*}
and $\varphi$ a smooth scalar function, and $\be'$ a smooth $S_r$-tangent vectorfield on $\RRRwo$. Then it holds that
\begin{align} \begin{aligned}
\SS(\varphi, \be', g) =& -4 N'(N'\varphi) - 2 e^{-2\varphi}\Ld_\ga \varphi + 6 (N'\varphi) (\tr_\ga \Theta + \frac{1}{a} \Divd \be') - 6 (N'\varphi)^2 \\
&+ 2N'\left(\frac{1}{a} \Divd_\ga \be' \right) -\frac{2}{a} \tr_\ga \Theta \Divd_\ga \be' - \frac{1}{a} \Theta^{AB} (\Lied_{\be'} \ga)_{AB} \\
&-\frac{2}{a} \be'(\tr_\ga \Theta) - \frac{2}{a} (e^{-2\varphi}-1) \Ld_\ga a + 2(e^{-2\varphi}-1) K(\ga) \\
&-\frac{1}{a^2} (\Divd_\ga \be')^2 - \vert \Lied_{\be'} \ga \vert^2_\ga,
\end{aligned} \label{finalscalarcurv33333} \end{align}
where $N'= N- \frac{1}{a} \be'= \frac{1}{a} \pr_r - \frac{1}{a} \be - \frac{1}{a} \be'$.
\end{lemma}

\begin{proof}
By Lemma \ref{scalarcurvaturefirstlem}, it holds that
\begin{align} \begin{aligned} 
R(\check{g}_{\varphi,\be'}) =& 2 \check{N}_{\varphi,\be'} (\tr_{e^{2\varphi}\ga} \check{\Theta}_{\varphi,\be'} ) - \frac{2}{a} \Ld_{e^{2\varphi}\ga} a + 2 K(e^{2\varphi} \ga) \\
&- (\tr_{e^{2\varphi} \ga} \check{\Theta}_{\varphi,\be'} )^2 - \vert \check{\Theta}_{\varphi,\be'} \vert_{e^{2\varphi}\ga}^2,
\end{aligned}\label{scalarcurvforvar} \end{align}
where $\check{N}_{\varphi,\be'}$ and $\left(\check{\Theta}_{\varphi,\be'}\right)_{AB}$ are given in any coordinates on $S_r$, for $A,B=1,2$, by
\begin{align*}
\check{N}_{\varphi,\be'} =& N' = \frac{1}{a} \pr_r - \frac{1}{a} (\be + \be'), \\
\left( \check{\Theta}_{\varphi,\be'}\right)_{AB} =& -\frac{1}{2a} \pr_r (e^{2\varphi} \ga_{AB}) + \frac{1}{2a} \left(\LIE_{\be'+ \be} (e^{2\varphi} \ga) \right)_{AB}\\
=& - (N'\varphi) e^{2\varphi} \ga_{AB} + e^{2\varphi} \left( \Theta_{AB} + \frac{1}{2a} (\Lie_{\be'} \ga)_{AB}\right).
\end{align*}
We note that
\begin{align} \begin{aligned}
\tr_{e^{2\varphi}\ga} \check{\Theta}_{\varphi,\be'} =& -2 (N'\varphi) + \tr_\ga \Theta + \frac{1}{a} \Divd_\ga \be',\\
\vert \check{\Theta}_{\varphi,\be'} \vert_{e^{2\varphi}\ga}^2 
=& 2 (N'\varphi)^2 -2 (N'\varphi) \tr_\ga \Theta - \frac{2}{a} N'(\varphi) \Divd_\ga \be' + \vert \Theta \vert_\ga^2  \\
&+ \frac{1}{4a^2} \vert \Lied_{\be'} \ga \vert^2_\ga + \frac{1}{a} \Theta^{AB} (\Lie_{\be'} \ga)_{AB},
\end{aligned} \label{expreforvar} \end{align}
and also 
\begin{align} \begin{aligned}
K(e^{2\varphi} \ga) &= e^{-2\varphi} (K(\ga) - \Ld_{\ga} \varphi), \\
\Ld_{e^{2\varphi}\ga} a &= e^{-2\varphi} \Ld_\ga a.
\end{aligned} \label{dgexprvar} \end{align}
Plugging \eqref{expreforvar} and \eqref{dgexprvar} into \eqref{scalarcurvforvar} yields
\begin{align} \begin{aligned}
R(\check{g}_{\varphi,\be'}) =& 2 N'(\tr_\ga \Theta + \frac{1}{a} \Divd_\ga \be' - 2 N'(\varphi)) - \frac{2}{a} e^{-2\varphi} \Ld_\ga a + 2e^{-2\varphi}(K(\ga)-\Ld_\ga \varphi) \\
&- (\tr_\ga \Theta)^2 - 6 (N'(\varphi))^2 + 6 N'(\varphi)( \tr_\ga \Theta + \frac{1}{a} \Divd_\ga \be') \\
&- \vert \Theta \vert^2_\ga - \frac{1}{4a^2} \vert \Lied_{\be'} \ga \vert^2_{\ga} - \frac{1}{a} \Theta^{AB} (\Lied_{\be'} \ga)_{AB}.
\end{aligned} \label{finalscalarcurv22222} \end{align}
By Lemma \ref{scalarcurvaturefirstlem}, it holds that
\begin{align} \begin{aligned}
& 2 N' (\tr_\ga \Theta )-\frac{2}{a} e^{-2\varphi} \Ld_\ga a + 2 e^{-2\varphi} K(\ga) - (\tr_\ga\Theta)^2 - \vert \Theta \vert_\ga^2 \\
=& R(g) -2 \be'(\tr_\ga \Theta) - \frac{2}{a} (e^{-2\varphi}-1) \Ld_\ga a + 2(e^{-2\varphi}-1) K(\ga).
\end{aligned} \label{eqreformingbck3434} \end{align}
Plugging \eqref{eqreformingbck3434} into \eqref{finalscalarcurv22222} yields \eqref{finalscalarcurv33333}. This finishes the proof of Lemma \ref{variationscalarlem2}. \end{proof}


The scalar curvature functional is a smooth mapping.
\begin{lemma} \label{prop:Fanalysis} Let $w\geq2$ be an integer. There exists a universal small constant $\varep>0$ such that the following holds.
\begin{itemize}
\item The scalar curvature functional $g \mapsto R(g)$ is a smooth mapping  
$$\{ g-e \in \HH^w_{-1/2} : \Vert g-e \Vert_{\HH^2_{-1/2}}< \varep\} \to H^{w-2}_{-5/2}.$$
\item The mapping $(\varphi,\be',g) \mapsto R(\check{g}_{\varphi, \be'})$ is a smooth mapping 
$$H^w_{-1/2}(\RRRwo) \times \HH^w_{-1/2}(\RRRwo) \times \HH^w_{-1/2}(\RRRwo) \to H^{w-2}_{-5/2}(\RRRwo)$$ 
in an $\varep$-neighbourhood of $(0,0,e)$ in $H^2_{-1/2} \times \HH^2_{-1/2} \times \HH^2_{-1/2}$.
\item The mapping $(\varphi,\be',g) \mapsto \SS(\varphi, \be',g)$ is a smooth mapping 
$$\ol{H}^w_{-1/2} \times \ol{\HH}^w_{-1/2} \times \HH^w_{-1/2}(\RRRwo) \to \ol{H}^{w-2}_{-5/2}(\RRRwo)$$ 
in an $\varep$-neighbourhood of $(0,0,e)$ in $\ol{H}^2_{-1/2} \times \ol{\HH}^2_{-1/2} \times \HH^2_{-1/2}(\RRRwo)$.
\end{itemize}
\end{lemma}

\begin{proof} The proof follows by the explicit expression in Lemma \ref{variationscalarlem2} and product estimates as in Lemma \ref{ProductEstimates}, and is left to the reader. See for example \cite{FischerMarsden}. \end{proof}

We calculate now the linearisation of $R(\check{g}_{\varphi,\be'})$ in $\varphi$ and $\be'$ at the Euclidean metric.
\begin{lemma} \label{variationscalarlin}
The linearisation of $R(\check{g}_{\varphi,\be'})$ in $(\varphi,\be')$ at $(\varphi,\be',g)=(0,0,e)$ is given by
\begin{align*}
L := D_{\varphi,\be'}R(\check{g}_{\varphi,\be'}) \vert_{(0,0,e)}(u, \xi) = -4 \pr^2_r u -2 \Ld_{\gac} u - \frac{12}{r} \pr_r u - \frac{4}{r^2} u +\frac{2}{r^3} \pr_r \left( r^3 \Divd \xi \right).
\end{align*}
The operator $L$ is a bounded linear operator from $(u, \xi) \in \ol{H}^w_{-1/2} \times\ol{\HH}^w_{-1/2}$ to $\ol{H}^{w-2}_{-5/2}$.
\end{lemma}

\begin{proof} We only calculate the linearisation $L$; the second statement follows by Lemma \ref{SobolevEmbeddingsAndNonlinear} and is left to the reader. It suffices to calculate the variation $\de$ of each term of \eqref{finalscalarcurv33333} in Lemma \ref{variationscalarlem2}. Denoting $\de \varphi = u$ and $\de \be = \xi$, the non-vanishing variations are
\begin{align*}
\de \left( -4 N'(N'\varphi) \right) =& -4 \pr_r^2 u, \\
\de \left( -2 e^{-2\varphi} \Ld_{\ga} \varphi \right) =& -2 \Ld_{\gac} u, \\
\de \left( 6 (N'\varphi) (\tr_{\ga} \Theta + \frac{1}{a} \Divd_\ga \be')  \right) =&6 \pr_r u \left( - \frac{2}{r} \right) = -\frac{12}{r} \pr_r u,\\
\de \left( 2 N' \left( \frac{2}{a} \Divd_\ga \be' \right) \right) =& 2 \pr_r (\Divd \xi), \\
\de\left( -\frac{2}{a} \tr_\ga \Theta \Divd_\ga \be' \right) =& \frac{4}{r} \Divd \xi, \\
\de \left( -\frac{1}{a} \Theta^{AB} (\Lied_{\be'}\ga)_{AB} \right)=& \frac{2}{r} \Divd \xi, \\
\de \left( 2(e^{-2\varphi}-1)K(\ga)\right)=&-\frac{4}{r^2} u.
\end{align*}

To summarise, we have 
\begin{align*}
L (u,\xi) = -4 \pr^2_r u -2 \Ld u - \frac{12}{r} \pr_r  u - \frac{4}{r^2}  u + 2\left( \pr_r \Divd \xi + \frac{3}{r} \Divd \xi \right). 
\end{align*}
This finishes the proof of Lemma \ref{variationscalarlin}. \end{proof}

We remark that by Definition \ref{definitionmathcalSSSS2222}, the above implies also that
\begin{align*}
D_{\varphi,\be'} \SS \vert_{(0, 0, e)}(u,\xi) = L(u,\xi).
\end{align*} 
In the next proposition, we estimate the difference $\SS(\varphi,\be',g)- D_{\varphi,\be'} \SS \vert_{(0, 0, e)}(\varphi,\be')$; this is used to derive higher regularity estimates in the next section. The idea of the proof is to rewrite this difference in terms of products of $g-e$, $\be$ and $\varphi$, and then to apply to each term product estimates as in Lemma \ref{ProductEstimates}. See also the corresponding Lemma \ref{lem:DivergenceBoundedOperator} for $k$ in Section \ref{sec:analysisonafsets}.
\begin{proposition} \label{Prophigherregularityestimates23232323} Let $w\geq3$ be an integer. There is a universal $\varep>0$ small such that if $g$ is an $\HH^w_{-1/2}$-asymptotically flat Riemannian metric on $\RRRwo$ with
\begin{align*}
\Vert g-e \Vert_{\HH^2_{-1/2}(\RRRwo)} < \varep,
\end{align*}
and $\varphi \in \ol{H}^w_{-1/2}$ a scalar function and $\be' \in \ol{\HH}^w_{-1/2}$ an $S_r$-tangent vector, then it holds that
\begin{align*}
&\Vert \SS(\varphi,\be',g) - L(\varphi,\be') \Vert_{\ol{H}^{w-2}_{-5/2}(\RRRwo)} \\
\lesssim& \Big( \Vert g-e \Vert_{\HH^2_{-1/2}(\RRRwo)} + \Vert (\varphi, \be') \Vert_{\ol{H}^2_{-1/2} \times \ol{\HH}^2_{-1/2}} \Big) \Vert (\varphi, \be') \Vert_{\ol{H}^w_{-1/2} \times\ol{\HH}^w_{-1/2}} \\
&+ \Vert (\varphi,\be') \Vert_{\ol{H}^2_{-1/2} \times \ol{\HH}^2_{-1/2}} \Vert g-e \Vert_{\HH^{w}_{-1/2}(\RRRwo)} \\
&+ C_w \Big( \Vert g-e \Vert_{\HH^2_{-1/2}(\RRRwo)} \Vert (\varphi,\be') \Vert_{\ol{H}^2_{-1/2} \times \ol{\HH}^2_{-1/2}} \Big).
\end{align*}
\end{proposition}

\begin{proof} By Lemma \ref{variationscalarlem2}, we have
\begin{align} \begin{aligned}
&\SS(\varphi,\be',g) - L(\varphi,\be') \\
=& -\Big( 4N'(N' \varphi) - 4\pr^2_r \varphi \Big) - 2 \Big( e^{-2\varphi} \Ld_\ga \varphi - \Ld_{\gac} \varphi \Big) \\
&-6 (N'\varphi)^2 + \Big( 6N'(\varphi)(\tr_\ga \Theta + \frac{1}{a} \Divd_\ga \be) + \frac{12}{r} \pr_r \varphi \Big) \\
&+\Big( 2N'\left(\frac{1}{a} \Divd_\ga \be' \right) -2 \pr_r \Divd \be' \Big) - \Big( \frac{2}{a} \tr_\ga \Theta \Divd_\ga \be' + \frac{4}{r} \Divd \be' \Big) \\
&- \Big( \frac{1}{a} \Theta^{AB} (\Lied_{\be'} \ga)_{AB} + \frac{2}{r} \Divd \be' \Big) - 2 \be' (\tr_\ga \Theta) - \frac{2}{a} (e^{-2\varphi}-1) \Ld_\ga a \\
&+ 2(e^{-2\varphi}-1) K(\ga) -\frac{1}{a^2} (\Divd_\ga \be')^2 - \vert \Lied_{\be'} \ga \vert^2_{\ga}.
\end{aligned} \label{estrhsHRS} \end{align}
We rewrite now the right-hand side of \eqref{estrhsHRS} into products of $g-e$, $\be'$, $\varphi$ and their derivatives. The first term on the right-hand side of \eqref{estrhsHRS} equals
\begin{align*}
&4N'(N'\varphi) - 4\pr^2_r \varphi \\
=& \frac{4}{a} (-\be-\be')\left( \frac{1}{a}(\pr_r-\be-\be') (\varphi) \right) + \frac{4}{a^2} \pr_r((-\be-\be')(\varphi)) \\
&- \frac{4}{a^3} \pr_r a (\pr_r -\be-\be')(\varphi).
\end{align*}
The second term on the right-hand side of \eqref{estrhsHRS} equals
\begin{align*}
-2 \Big( e^{-2\varphi} \Ld_\ga \varphi +\Ld_{\gac} \varphi  \Big)= -2(e^{-2\varphi} - 1) \Ld_\ga \varphi - 2(\Ld_\ga \varphi - \Ld_{\gac} \varphi). 
\end{align*}
The second term on the right-hand side of \eqref{estrhsHRS} is already in product form,
\begin{align*}
-6 (N'\varphi)^2.
\end{align*}
The fourth term on the right-hand side of \eqref{estrhsHRS} equals
\begin{align*}
&6N'(\varphi)(\tr_\ga \Theta) + \frac{12}{r} \pr_r \varphi + \frac{6}{a} N'(\varphi) \Divd_\ga \be' \\
=& \frac{6}{a} (-\be-\be')(\varphi) \tr_\ga \Theta + \frac{6}{a} \pr_r \varphi (\tr_\ga \Theta + \frac{2}{r}) - \frac{12}{r}\frac{1-a}{a} \pr_r \varphi + \frac{6}{a}(\pr_r - \be- \be')(\varphi) \Divd_\ga \be'.
\end{align*}
The fifth term on the right-hand side of \eqref{estrhsHRS} equals
\begin{align*}
&2N'\left( \frac{1}{a} \Divd_\ga \be' \right) - 2 \pr_r \Divd \be' \\
=& 2\left(\frac{1-a}{a} \right) \pr_r \Divd_\ga \be' + 2 \pr_r \left( \Divd_\ga \be' - \Divd \be' \right) - \frac{\pr_r a}{a^2} \Divd_\ga \be' +2(-\be -\be')\left( \frac{1}{a} \Divd_\ga \be' \right).
\end{align*}
The sixth term on the right-hand side of \eqref{estrhsHRS} equals
\begin{align*}
&\frac{2}{a}\tr_\ga \Theta \Divd_\ga \be' + \frac{4}{r} \Divd \be' \\
=& \frac{2}{a} \left(\tr_\ga \Theta + \frac{2}{r} \right) \Divd_\ga \be' - 4\frac{1-a}{ar} \Divd_\ga \be' + \frac{4}{r} (\Divd \be' - \Divd_\ga \be').
\end{align*}
The seventh term on the right-hand side of \eqref{estrhsHRS} equals
\begin{align*}
&\frac{1}{a} \Theta_{AB} (\Lied_{\be'} \ga)^{AB} + \frac{2}{r} \Divd \be' \\
=& \frac{1}{a} \widehat{\Theta}_{AB}(\Lied_{\be'} \ga)^{AB} + \frac{1}{a} \left(\tr_\ga \Theta +\frac{2}{r} \right) \Divd_\ga \be' - \frac{2(1-a)}{ar} \Divd_\ga \be' + \frac{2}{r} (\Divd \be' - \Divd_\ga \be').
\end{align*}
The eight term on the right-hand side of \eqref{estrhsHRS} can be rewritten as
\begin{align*}
2 \be' (\tr_\ga \Theta) = 2 \be'\left(\tr_\ga \Theta + \frac{2}{r}\right),
\end{align*}
where we used that $\be'$ is a $S_r$-tangent vectorfield. Here we recall that $\tr_{\ga} \Theta \vert_{g=e} = -\frac{2}{r}$.\\

The ninth to twelfth terms on the right-hand side of \eqref{estrhsHRS} are already in the right product form. \\

By rewriting the terms on the right-hand side of \eqref{estrhsHRS} as above, it follows that, in view of Lemmas \ref{SobolevEmbeddingsAndNonlinear}, \ref{ProductEstimates} and \ref{GInverseAnalysis}, the difference $\SS(\varphi,\be',g) - L(\varphi,\be')$ is schematically given by
\begin{align*}
&\Big(\be' +(g-e)+ \varphi \Big) \pr^2 \varphi + \Big(\pr(g-e) + \pr \be' + \pr \varphi + (g-e) + \be' \Big) \pr \varphi \\
+& \Big( \pr^2 \be'\Big) (g-e) + \Big(\pr \be' +\pr(g-e) \be' \Big)\pr (g-e) + \be' \pr^2(g-e),
\end{align*}
where we left away lower order terms.

Applying to each of the above terms the product estimates as in Lemma \ref{ProductEstimates} yields the estimate
\begin{align*}
&\Vert \SS(\varphi,\be',g) - L(\varphi,\be') \Vert_{\ol{H}^{w-2}_{-5/2}} \\
\lesssim& \Big( \Vert g-e \Vert_{\HH^2_{-1/2}} + \Vert (\varphi, \be') \Vert_{\ol{H}^2_{-1/2} \times \ol{\HH}^2_{-1/2}} \Big) \Vert (\varphi, \be') \Vert_{\ol{H}^w_{-1/2} \times\ol{\HH}^w_{-1/2}} \\
&+ \Vert (\varphi,\be') \Vert_{\ol{H}^2_{-1/2} \times \ol{\HH}^2_{-1/2}} \Vert g-e \Vert_{\HH^{w}_{-1/2}} + C_w \Big( \Vert g-e \Vert_{\HH^2_{-1/2}} \Vert (\varphi,\be') \Vert_{\ol{H}^2_{-1/2} \times \ol{\HH}^2_{-1/2}} \Big),
\end{align*}
We leave the details to the interested reader. This finishes the proof of Proposition \ref{Prophigherregularityestimates23232323}. \end{proof}


\subsection{Reduction to the Euclidean case} \label{ssec:Reduction2}

In this section, we first prove the next pertubation result, Proposition \ref{prop:ExistenceOfPerturbation}, under the assumption of Lemma \ref{lem:scalarsurj} which is proved in Section \ref{sec:SurjectivityR}. Then we prove Theorem \ref{thm:MainExtensionScalarTheorem}.

\begin{proposition}[Perturbation of scalar curvature] \label{prop:ExistenceOfPerturbation} There is a small universal constant $\varep >0$ such that the following holds.
\begin{enumerate}
\item {\bf Extension result.} Let $g$ be an $\HH^2_{-1/2}$-asymptotically flat metric on $\RRRwo$ written in standard polar coordinates as
\begin{align*}
g = a^2 dr^2 + \ga_{AB} \left( \be^A dr + d\th^A \right) \left( \be^B dr + d\th^B \right),
\end{align*}
and $\mathfrak{s} \in \ol{H}^{0}_{-5/2}$ a scalar function. If
\begin{align}
\Vert g -e \Vert_{\HH^2_{-1/2}(\RRRwo)} + \Vert \mathfrak{s} \Vert_{\ol{H}^{0}_{-5/2}} < \varep, \label{mars7eq1}
\end{align}
then there exist a scalar function $\varphi \in \ol{H}^2_{-1/2}$ and an $S_r$-tangent vector $\be' \in \ol{\HH}^2_{-1/2}$ bounded by
\begin{align}
\Vert (\varphi,\be') \Vert_{\ol{H}^2_{-1/2} \times \ol{\HH}^2_{-1/2}} \lesssim \Vert \mathfrak{s} \Vert_{\ol{H}^{0}_{-5/2}}, \label{eq:uxiestimatefeb181}
\end{align}
and such that the metric
\begin{align}
\check{ g}_{\varphi,\be'} := a^2 dr^2 + e^{2\varphi} \ga_{AB} \left( (\be + \be')^A dr + d\th^A \right) \left( (\be + \be')^B dr +d\th^B \right) \label{eq:Dez21defvar}
\end{align}
is $\HH^2_{-1/2}$-asymptotically flat and its scalar curvature on $\RRRwo$ is given by
$$R(\check{g}_{\varphi, \be'}) = R(g) +\mathfrak{s}.$$
Furthermore, the following bound holds,
\begin{align}
\Vert \check{g}_{\varphi, \be'} - e \Vert_{\HH^2_{-1/2}(\RRRwo)} \lesssim \Vert g - e \Vert_{\HH^2_{-1/2}(\RRRwo)} + \Vert \mathfrak{s} \Vert_{\ol{H}^{0}_{-5/2}}.
\label{eq:estimatefeb181}
\end{align} 
\item {\bf Iteration estimates.} Let $g$ be an $\HH^2_{-1/2}$-asymptotically flat metric on $\RRRwo$ written in standard polar coordinates as
\begin{align*}
 g = a^2 dr^2 + \ga_{AB} \left( \be^A dr + d\th^A \right) \left( \be^B dr + d\th^B \right),
\end{align*}
and $\mathfrak{s}, \tilde{\mathfrak{s}} \in \ol{H}^{0}_{-5/2}$ two scalar functions such that \eqref{mars7eq1} holds for $(g,\mathfrak{s})$ and $(g,\tilde{\mathfrak{s}})$. Applying part (1) to $ g$ with $\mathfrak{s}$ and $\tilde{\mathfrak{s}}$ yields two pairs $(\varphi,\beta')$ and $(\tilde{\varphi},\tilde{\be'})$, respectively. Let $\check{g}_{\varphi, \be'}, \check{ g}_{\tilde{\varphi},\tilde{\be'}}$ denote the asymptotically flat metrics defined by \eqref{eq:Dez21defvar}. Then it holds that
\begin{align}
\Vert (\varphi,\be') - (\tilde{\varphi},\tilde{\be'}) \Vert_{\ol{H}^2_{-1/2} \times \ol{\HH}^2_{-1/2}} \lesssim \Vert \mathfrak{s}-\tilde{\mathfrak{s}} \Vert_{\ol{H}^{0}_{-5/2}}. \label{eq:iterationestimateuxifeb18}
\end{align}
and
\begin{align}
\Vert \check{g}_{\varphi, \be'} - \check{ g}_{\tilde{\varphi},\tilde{\be'}}\Vert_{\HH^2_{-1/2}(\RRRwo)} \lesssim \Vert \mathfrak{s}- \tilde{\mathfrak{s}} \Vert_{\ol{H}^{0}_{-5/2}}. \label{eq:iterationestimategmetricfeb18}
\end{align}
\item {\bf Higher regularity.} If, in addition to \eqref{mars7eq1}, $g$ is an $\HH^w_{-1/2}$-asymptotically flat metric and $\mathfrak{s} \in \ol{H}^{w-2}_{-5/2}$ for an integer $w\geq3$, then
\begin{align} \begin{aligned}
\Vert (\varphi,\be') \Vert_{\ol{H}^w_{-1/2} \times \HHlo} \lesssim& \Vert \mathfrak{s} \Vert_{\ol{H}^{w-2}_{-5/2}} +\Vert \mathfrak{s} \Vert_{\ol{H}^{0}_{-5/2}} \Vert g-e \Vert_{\HH^w_{-1/2}} + C_w \Vert \mathfrak{s} \Vert_{\ol{H}^0_{-5/2}},
\end{aligned}\label{eq:uxiestimatefeb181HRE} \end{align}
and
\begin{align} \begin{aligned}
\Vert \check{g}_{\varphi, \be'} - e \Vert_{\HH^w_{-1/2}(\RRRwo)} \lesssim& \Vert \mathfrak{s} \Vert_{\ol{H}^{w-2}_{-5/2}} + \Vert g-e \Vert_{\HH^w_{-1/2}} \\
&+ C_w \Big( \Vert \mathfrak{s} \Vert_{\ol{H}^0_{-5/2}} + \Vert g-e \Vert_{\HH^2_{-1/2}} \Big),
\end{aligned} \label{eq:estimatefeb181HRE}
\end{align}
where the constant $C_w>0$ depends on $w$.
\end{enumerate}
\end{proposition}

The proof of Proposition \ref{prop:ExistenceOfPerturbation} is based on the Implicit Function Theorem and the essential Lemma \ref{lem:scalarsurj} below which is proved in Section \ref{sec:SurjectivityR}.  \\


The next lemma is essential for the proof of Proposition \ref{prop:ExistenceOfPerturbation}.
\begin{lemma}[Surjectivity at the Euclidean metric] \label{lem:scalarsurj} The following holds.
\begin{enumerate}
\item \underline{Surjectivity.} For every scalar function $h \in \ol{H}^{w-2}_{-5/2}$, there exists a scalar function $u \in \ol{\HH}^2_{-1/2}$ and a $S_r$-tangent vectorfield $\xi \in \ol{\HH}^w_{-1/2}$ that solve
\begin{align*}
\begin{cases}
4 \pr_r^2 u + \frac{12}{r} \pr_r u + \frac{4}{r^2} u + 2 \Ld_{\gac} u - \frac{2}{r^3} \pr_r \left( r^3 \Divd_{\gac} \xi \right) = h\,\,\, \text{on } \RRRwo, \\
(u,\xi) \vert_{r=1} =0,
\end{cases} 
\end{align*}
and are bounded by
\begin{align*}
\Vert (u,\xi) \Vert_{\ol{H}^2_{-1/2} \times\ol{\HH}^2_{-1/2}} \lesssim \Vert h \Vert_{\ol{H}^{0}_{-5/2}}.
\end{align*}
That is, the linearisation $L: \ol{H}^2_{-1/2} \times\ol{\HH}^2_{-1/2} \to \ol{H}^{0}_{-5/2}$ is surjective and has a bounded right-inverse.
\item \underline{Higher regularity.} If in addition $h \in \ol{H}^{w-2}_{-5/2}$ for an integer $w\geq3$, then
\begin{align*}
\Vert (u,\xi) \Vert_{\ol{H}^w_{-1/2} \times\ol{\HH}^w_{-1/2}} \lesssim \Vert h \Vert_{\ol{H}^{w-2}_{-5/2}} + C_w \Vert h \Vert_{\ol{H}^0_{-5/2}},
\end{align*}
where $C_w>0$ depends on $w$.
\end{enumerate}
\end{lemma}

Let 
\begin{align*}
\ol{\NN_e} :=  \mathrm{ker}\Big( D_{\varphi,\be'} \SS \vert_{(0,0,e)}  \Big)^\perp  \subset \ol{H}^2_{-1/2} \times \ol{\HH}^2_{-1/2},
\end{align*}
where $\perp$ denotes the orthogonal complement with respect to the scalar product on $ \ol{H}^2_{-1/2} \times \ol{\HH}^2_{-1/2}$. 

\begin{remark} \label{rem52343953}
Because $D_{\varphi,\be'} \SS \vert_{(0,0,e)}$ is a bounded operator (see Lemma \ref{lem:scalarsurj} above), its kernel is closed. Therefore we have the splitting
\begin{align*}
\ol{H}^2_{-1/2} \times \ol{\HH}^2_{-1/2} = \ol{\NN_e} \oplus \mathrm{ker}\Big( D_{\varphi,\be'} \SS \vert_{(0,0,e)}  \Big).
\end{align*} 
\end{remark}
From now on, let $\SS$ be restricted to $(\varphi,\be') \in \ol{\NN_e}$.\\

We are now ready to prove Proposition \ref{prop:ExistenceOfPerturbation}.

\begin{proof}[Proof of Proposition \ref{prop:ExistenceOfPerturbation}] {\bf Proof of part (1).} On the one hand, by Remark \ref{rem52343953} and the definition of $\ol{\NN_e}$, the operator $D_{\varphi,\be'}\SS \vert_{(0,0,e)}$ is injective and surjective, that is, an isomorphism. On the other hand, clearly $\SS(0,0,e)=0$. Therefore, by the Inverse Function Theorem \ref{thm:InverseMars14}, there exist open neighourhoods $V_0$ around the Euclidean metric $e$ and $W_0 \subset \ol{H}^{0}_{-5/2}$ around $0$, together with a unique mapping 
\begin{align*}
\GG: V_0 \times W_0 &\to \ol{H}^{2}_{-1/2} \times \ol{\HH}^2_{-1/2}, \\
(g ,\mathfrak{s}) &\mapsto \GG(g ,\mathfrak{s}):=(\varphi,\be')
\end{align*}
such that for $g \in V_0, \mathfrak{s}\in W_0$, on $\RRRwo$, 
$$\SS(\GG(g ,\mathfrak{s}),g)=\mathfrak{s}.$$
We note that the mapping $\GG$ is smooth. Moreover, it holds by the uniqueness of $\GG$ that for every $g \in V_0$, 
\begin{align}
\GG(g,0)= 0, \label{eq:13j11}
\end{align}
because $\SS(0,0,g) = R(g)-R(g)=0$.\\

For $\varep>0$ sufficiently small it holds that for $(g,\mathfrak{s})$ with
\begin{align*}
\Vert g-e \Vert_{\HH^2_{-1/2}} < \varep, \Vert \mathfrak{s}\Vert_{\ol{H}^{0}_{-5/2}} < \varep
\end{align*}
we have $g \in V_0, \mathfrak{s} \in W_0$ and furthermore, for 
$$(\varphi,\be') := \GG(\tilde g,\mathfrak{s})$$
it holds that
\begin{align} \begin{aligned}
\Vert (\varphi,\be') \Vert_{\ol{H}^2_{-1/2} \times \HHlo} &= \Vert \GG(\tilde g, \mathfrak{s}) \Vert_{\ol{H}^2_{-1/2} \times \HHlo} \\
&=\Vert \GG(\tilde g, \mathfrak{s}) - \underbrace{\GG(\tilde g, 0)}_{=0} \Vert_{\ol{H}^2_{-1/2} \times \HHlo} \\
&\lesssim \Vert \mathfrak{s} \Vert_{\ol{H}^{0}_{5/2}},
\end{aligned} \label{eq:firstcalc2j}
\end{align}
see Lemma \ref{lem:opth12j}. This proves \eqref{eq:uxiestimatefeb181}. The estimate \eqref{eq:estimatefeb181} for $\check{g}_{\varphi,\be'}$ follows for $\varep>0$ sufficiently small from \eqref{eq:uxiestimatefeb181} by Lemma \ref{varphibetaboundg}. \\


{\bf Proof of part (2).} By Lemma \ref{lem:opth12j}, for $\varep>0$ sufficiently small, it holds that for $g, \mathfrak{s},\tilde{\mathfrak{s}}$ with
$$ \Vert g-e \Vert_{\HH^2_{-1/2}} < \varep, \Vert \mathfrak{s} \Vert_{\ol{H}^{0}_{-5/2}} <\varep,\Vert \tilde{\mathfrak{s}} \Vert_{\ol{H}^{0}_{-5/2}} < \varep$$
we have
\begin{align} \begin{aligned}
\Vert (\varphi,\be') - (\tilde{\varphi},\tilde{\be'}) \Vert_{\ol{H}^2_{-1/2} \times \HHlo} &= \Vert \GG(g,\mathfrak{s}) -  \GG( g, \tilde{\mathfrak{s}}) \Vert_{\ol{H}^2_{-1/2} \times \HHlo}\\
&\lesssim \Vert \mathfrak{s}-\tilde{\mathfrak{s}} \Vert_{\ol{H}^{0}_{-5/2}}.
\end{aligned} \end{align}
This proves \eqref{eq:iterationestimateuxifeb18}. The estimate \eqref{eq:iterationestimategmetricfeb18} follows for $\varep>0$ sufficiently small from \eqref{eq:iterationestimateuxifeb18} by Lemma \ref{varphibetaboundg}. This finishes the proof of part (2) of Proposition \ref{prop:ExistenceOfPerturbation}.\\


{\bf Proof of part (3).} On the one hand, by part (1) of this proposition, the metric $\check{g}_{\varphi,\be'}$ satisfies
\begin{align*}
\SS(\varphi,\be',g) = R(\check{g}_{\varphi,\be'}) - R(g) = \mathfrak{s},
\end{align*}
and we have
\begin{align*}
(\varphi,\be') \subset \ol{\NN_e} :=  \mathrm{ker}\Big( L \Big)^\perp  \subset \ol{H}^2_{-1/2} \times \ol{\HH}^2_{-1/2}.
\end{align*}

On the other hand, by Lemma \ref{lem:scalarsurj}, there exists a solution $(u,\xi) \subset \ol{H}^2_{-1/2} \times \ol{\HH}^2_{-1/2}$ to
\begin{align} \label{eq3493222018}
L(u,\xi) = h
\end{align}
for
\begin{align*}
h:= L(\varphi,\be') \in \ol{H}^{0}_{-5/2},
\end{align*}
which satisfies for integers $w\geq2$ the bound 
\begin{align} \begin{aligned}
\Vert (u,\xi) \Vert_{\ol{H}^w_{-1/2} \times\ol{\HH}^w_{-1/2}} \lesssim& \Vert h  \Vert_{\ol{H}^{w-2}_{-5/2}} + C_w \Vert h  \Vert_{\ol{H}^{0}_{-5/2}} \\
\lesssim&\Vert L(\varphi,\be')  \Vert_{\ol{H}^{w-2}_{-5/2}} + C_w \Vert L(\varphi,\be')  \Vert_{\ol{H}^{0}_{-5/2}}.
\end{aligned} \label{bound3492484353} \end{align}
where the constant $C_w>0$ depends on $w$.\\

By \eqref{eq3493222018}, it follows that
\begin{align*}
(u,\xi)-(\varphi,\be') \in \ol{\NN_e},
\end{align*}
that is, $(\varphi,\be')$ equals $(u,\xi)$ up to a part of $(u,\xi)$ in the kernel of $L$. \\

Therefore we can estimate by \eqref{bound3492484353},
\begin{align} \begin{aligned}
&\Vert (\varphi,\be') \Vert_{\ol{H}^w_{-1/2} \times\ol{\HH}^w_{-1/2}} \\
\leq& \Vert (u,\xi) \Vert_{\ol{H}^w_{-1/2} \times\ol{\HH}^w_{-1/2}} \\
\lesssim& \Vert L(\varphi,\be')  \Vert_{\ol{H}^{w-2}_{-5/2}}+ C_w\Vert L(\varphi,\be')  \Vert_{\ol{H}^{0}_{-5/2}} \\
\lesssim& \Vert L(\varphi,\be') - \SS(\varphi,\be',g) \Vert_{\ol{H}^{w-2}_{-5/2}} + \Vert \SS(\varphi,\be',g) \Vert_{\ol{H}^{w-2}_{-5/2}}+ C_w\Vert L(\varphi,\be')  \Vert_{\ol{H}^{0}_{-5/2}}\\
\lesssim& \Vert L(\varphi,\be') - \SS(\varphi,\be',g) \Vert_{\ol{H}^{w-2}_{-5/2}} + \Vert \mathfrak{s} \Vert_{\ol{H}^{w-2}_{-5/2}}+ C_w\Vert L(\varphi,\be')  \Vert_{\ol{H}^{0}_{-5/2}} \\
\lesssim&  \Big( \Vert g-e \Vert_{\HH^2_{-1/2}}+ \Vert (\varphi,\be') \Vert_{\ol{\HH}^2_{-1/2} \times \ol{\HH}^2_{-1/2}} \Big) \Vert (\varphi, \be') \Vert_{\ol{H}^w_{-1/2} \times\ol{\HH}^w_{-1/2}} \\
&+ \Vert (\varphi,\be') \Vert_{\ol{H}^2_{-1/2} \times \ol{\HH}^2_{-1/2}} \Vert g-e \Vert_{\HH^w_{-1/2}}+ \Vert \mathfrak{s} \Vert_{\ol{H}^{w-2}_{-5/2}} + C_w \Vert (\varphi,\be') \Vert_{\ol{H}^2_{-1/2} \times \ol{\HH}^2_{-1/2}},
\end{aligned} \label{rhshreest39434} \end{align}
were we used Proposition \ref{Prophigherregularityestimates23232323} for the last estimate. \\

For $\varep>0$ sufficiently small, the first term on the right-hand side of \eqref{rhshreest39434} can be absorbed and we get
\begin{align*}
\Vert (\varphi,\be') \Vert_{\ol{H}^w_{-1/2} \times\ol{\HH}^w_{-1/2}} \lesssim&  \Vert \mathfrak{s} \Vert_{\ol{H}^{0}_{-5/2}}  \Vert g-e \Vert_{\HH^w_{-1/2}} + \Vert \mathfrak{s} \Vert_{\ol{H}^{w-2}_{-5/2}} + C_w \Vert \mathfrak{s} \Vert_{\ol{\HH}^0_{-5/2}}.
\end{align*}
This proves \eqref{eq:uxiestimatefeb181HRE}. The estimate \eqref{eq:estimatefeb181HRE} follows for $\varep>0$ sufficiently small from \eqref{eq:uxiestimatefeb181HRE} by Lemma \ref{varphibetaboundg}. This finishes the proof of Proposition \ref{prop:ExistenceOfPerturbation}. \end{proof}

We are now in position to prove Theorem \ref{thm:MainExtensionScalarTheorem}.
\begin{proof}[Proof of Theorem \ref{thm:MainExtensionScalarTheorem}] {\bf Proof of Part (1).} Using Proposition \ref{functionextcts}, extend $\bar g$ from $B_1$ to an asymptotically flat metric $ \check{g}$ on $\RRR^3$ such that
\begin{align}
\Vert  \check{g} - e \Vert_{\HH^2_{-1/2}(\RRR^3)} \lesssim \Vert \bar g-e \Vert_{\HH^2(B_1)}, \label{feb23222}
\end{align}
Denote its standard polar components on $\RRRwo$ by
\begin{align*}
 \check{g} = a^2 dr^2 + \ga_{AB} \left( \be^A dr + e^A \right) \left( \be^B dr + e^B \right).
\end{align*}
Given a $R \in H^{0}_{-5/2}$ such that $R\vert_{B_1} = R(\bar g)$, let
\begin{align*}
\mathfrak{s} := R - R( \check{g}) \in \ol{H}^{0}_{-5/2},
\end{align*}
where we used Proposition \ref{prop:TrivialExtensionRegularity}. It holds for $\varep>0$ small that
\begin{align}\begin{aligned}
\Vert \mathfrak{s} \Vert_{\ol{H}^{0}_{-5/2}} &\leq \Vert R \Vert_{H^{0}_{-5/2}(\RRRwo)} + \Vert R( \check{g}) \Vert_{H^{0}_{-5/2}(\RRRwo)},\\
&\lesssim \Vert R \Vert_{H^{0}_{-5/2}(\RRRwo)} + \Vert \check{g} - e \Vert_{\HH^2_{-1/2}(\RRRwo)}.
\end{aligned}\label{eq:Dez2177} 
\end{align}
Therefore, for $\varep>0$ small enough, Proposition \ref{prop:ExistenceOfPerturbation} yields a pair $(\varphi,\be')$ such that 
\begin{align*}
\check{g}_{\varphi,\be'} = a^2 dr^2 + e^{2\varphi} \ga_{AB} \left( (\be + \be')^A dr + e^A \right) \left( (\be + \be')^B dr +e^B \right)
\end{align*}
 is a $\HH^2_{-1/2}$-asymptotically flat metric with $\check{g}_{\varphi,\be'} \vert_{B_1}= \bar{g}$ and scalar curvature $$R(\check{g}_{\varphi,\be'})=R( g)+ S = R.$$ By \eqref{eq:estimatefeb181}, \eqref{feb23222} and \eqref{eq:Dez2177},
\begin{align*}
\Vert \check{g}_{\varphi,\be'}- e \Vert_{\HH^2_{-1/2}} &\lesssim \Vert \check{g} - e \Vert_{\HH^2_{-1/2}} + \Vert \mathfrak{s} \Vert_{\ol{H}^{0}_{-5/2}},\\
&\lesssim \Vert \bar g -e \Vert_{\HH^2(B_1)} + \Vert R \Vert_{H^{0}_{-5/2}}.
\end{align*}
Letting $g :=\check{g}_{\varphi,\be'}$ then proves \eqref{eq:Dez2166}. This finishes the proof of part (1) of Theorem \ref{thm:MainExtensionScalarTheorem}.\\

{\bf Proof of Part (2).} Using Proposition \ref{functionextcts}, extend $\bar g$ from $B_1$ to an asymptotically flat metric $\check{g}$ on $\RRR^3$ such that
\begin{align*}
\Vert  \check{g} - e \Vert_{\HH^2_{-1/2}(\RRR^3)} \lesssim \Vert \bar g-e \Vert_{\HH^2(B_1)}. 
\end{align*}
Let $\mathfrak{s}:=R-R(\check{g}), \tilde{\mathfrak{s}}:=\tilde{R}-R(\check{g})$, so that
\begin{align*}
\mathfrak{s}-\tilde{\mathfrak{s}} = \left( R( \check{g}) + \mathfrak{s} \right) -\left(R( \check{g}) +\tilde{\mathfrak{s}} \right) = R- \tilde{R}.
\end{align*}
Hence, for $\varep>0$ sufficiently small, \eqref{eq:Riterationestimate222} follows from \eqref{eq:iterationestimategmetricfeb18} in Proposition \ref{prop:ExistenceOfPerturbation}. \\

{\bf Proof of Part (3).} Using Proposition \ref{functionextcts}, extend $\bar g$ from $B_1$ to an asymptotically flat metric $ \check{g}$ on $\RRR^3$ such that
\begin{align*}
\Vert  \check{g} - e \Vert_{\HH^w_{-1/2}(\RRR^3)} \leq C_w \Vert \bar g-e \Vert_{\HH^w(B_1)},
\end{align*}
where the constant $C_w>0$ depends on $w$. \\

By Proposition \ref{prop:ExistenceOfPerturbation}, there exist $(\varphi,\be') \in \ol{\HH}^w_{-1/2} \times \ol{\HH}^w_{-1/2}$ such that the metric $\check{g}_{\varphi,\be'}$ has scalar curvature
\begin{align*}
R(\check{g}_{\varphi,\be'}) = R,
\end{align*}
and the following estimate holds with $\mathfrak{s} := R - R(\check{g})$,
\begin{align*}
&\Vert \check{g}_{\varphi,\be'} -e \Vert_{\HH^w_{-1/2}}\\
 \lesssim& \Vert \mathfrak{s} \Vert_{\ol{H}^{w-2}_{-5/2}} + \Vert \check g-e \Vert_{\HH^w_{-1/2}} + C_w \Big( \Vert \mathfrak{s} \Vert_{\ol{H}^0_{-5/2}} + \Vert \check g-e \Vert_{\HH^2_{-1/2}} \Big) \\
\lesssim& \Vert R-R(\check g) \Vert_{H^{w-2}_{-5/2}} +  \Vert \check g-e \Vert_{\HH^w_{-1/2}} + C_w \Big( \Vert R-R(\check g) \Vert_{H^0_{-5/2}} + \Vert \check{g}-e \Vert_{\HH^2_{-1/2}} \Big) \\
\lesssim& \Vert R \Vert_{H^{w-2}_{-5/2}} + \Vert \check g -e\Vert_{\HH^w_{-1/2}} + C_w \Big( \Vert R \Vert_{H^0_{-5/2}} + \Vert \check g-e \Vert_{\HH^2_{-1/2}} \Big), \\
\lesssim& \Vert R \Vert_{H^{w-2}_{-5/2}} + C_w\Big( \Vert \bar g-e \Vert_{\HH^w(B_1)} + \Vert R \Vert_{H^0_{-5/2}}\Big),
\end{align*}
Letting $g:= \check{g}_{\varphi,\be'}$ on $\RRR^3$ finishes the proof of Theorem \ref{thm:MainExtensionScalarTheorem}. 
\end{proof}

\subsection{Surjectivity at the Euclidean metric} \label{sec:SurjectivityR}

In this section, we prove Lemma \ref{lem:scalarsurj}. In this section all operators are Euclidean. First, by the explicit expression of Lemma \ref{variationscalarlin}, it directly follows that 
$$D_{\varphi,\be'} \SS \vert_{(0,0,e)}: \ol{H}^w_{-1/2} \times\ol{\HH}^w_{-1/2} \to \ol{H}^{w-2}_{-5/2}$$
(see Definition \ref{definitionmathcalSSSS}) is a bounded linear operator between Hilbert spaces. \\

It remains to show that for every scalar function $h \in \ol{H}^{w-2}_{-5/2}$ there exist $(u, \xi) \in  \ol{H}^{w}_{-1/2}\times  \ol{\HH}^{w}_{-1/2}$ that solves on $\RRRwo$
\begin{align}
\pr_r^2 u + \frac{3}{r} \pr_r u + \frac{1}{r^2} u + \half \Ld u  - \frac{1}{2 r^3} \pr_r \left( r^3 \Divd \xi \right) = \frac{1}{4} h \label{eq:DifferentialF}
\end{align}
and is bounded by
\begin{align*}
\Vert (u,\xi) \Vert_{\ol{H}^{w}_{-1/2}\times  \ol{\HH}^{w}_{-1/2}} \lesssim \Vert h \Vert_{\ol{H}^{w-2}_{-5/2}} + C_w \Vert h \Vert_{\ol{H}^{0}_{-5/2}}.
\end{align*}

We consider the following more general system on $\RRRwo$
\begin{align} \begin{aligned}
\triangle u + \frac{1}{r} \pr_r u + \frac{1}{r^2} u - \half \Ld u &= \half \left(\half h + \zeta^{[\geq1]} \right), \\
\frac{1}{r^3} \pr_r \left( r^3 \Divd \xi \right) &=  \zeta^{[\geq1]}, \end{aligned} \label{eq:systemUbe}
\end{align}
where $\zeta^{[\geq1]} \in \ol{H}^{w-2}_{-5/2}$ is a scalar function. By the relation $\triangle = \pr_r^2 + \frac{2}{r} \pr_r + \frac{1}{r^2} \Ld$, it follows that for any $\zeta^{[\geq1]}$, a solution $(u,\xi)$ to \eqref{eq:systemUbe} also solves \eqref{eq:DifferentialF}. From now on, we will thus focus on \eqref{eq:systemUbe}.

\subsubsection{Construction of the solution $(u, \xi)$ and $\zeta^{[\geq1]}$}

In this section we construct two scalar functions $\zeta^{[\geq1]}, u$ and a $S_r$-tangent vectorfield $\xi$. It is shown in Section \ref{newsecav4} that they are a solution to \eqref{eq:systemUbe}.\\

Let the given $h \in \ol{H}^{w-2}_{-5/2}$ be decomposed as
\begin{align*}
h = h^{[0]} + h^{[\geq1]}.
\end{align*}
\begin{itemize}
\item \textbf{Definition of $u$ and $\zeta^{[\geq1]}$.} Let the scalar function
\begin{align*}
u = u^{[0]} + u^{[\geq1]},
\end{align*}
where the radial function $u^{[0]}(r)$ is defined as solution to the following ODE on $r>1$,
\begin{align}
\begin{cases}
\pr_r^2 u^{[0]} + \frac{3}{r} \pr_r u^{[0]} + \frac{1}{r^2} u^{[0]} = \frac{1}{4} h^{[0]},\\
u^{[0]}\vert_{r=1} = \pr_r u^{[0]} \vert_{r=1} =0,
\end{cases} \label{def:Dez20u0}
\end{align}
and $u^{[\geq1]} $ is defined as solution to the elliptic PDE on $\RRRwo$ (see Appendix \ref{sec:ELLIPTICITYweighted})
\begin{align}
\begin{cases}
\triangle u^{[\geq1]} - \half \Ld u^{[\geq1]} + \frac{1}{r} \pr_r u^{[\geq1]} + \frac{1}{r^2}u^{[\geq1]} = \half \left( \half h^{[\geq1]} + \zeta^{[\geq1]} \right),\\
u^{[\geq1]} \vert_{r=1} = 0,
\end{cases}  \label{def:Dez20u1}
\end{align}
where on $\RRR^3$,
\begin{align} \begin{aligned}
\zeta^{[\geq1]} :=& \sumone \zeta^{(lm)} Y^{(lm)}, \\
\zeta^{(lm)} :=& c^{(lm)} r^{\sqrt{l(l+1)/2}-1} \pr_r \left( \chi (l(r-1) ) \right) \\
& - \tilde{c}^{(lm)} r^{\sqrt{l(l+1)/2} -1 } \pr_r^2 \left( \chi (l(r-1) ) \right),
\end{aligned}\label{def:zet1feb24}
\end{align}
where $\chi$ denotes the smooth transition function defined in \eqref{eq:transfct} and for $l\geq1$,
\begin{align} \begin{aligned}
c^{(lm)} &:= - \half \isinf r^{1-\sqrt{l(l+1)/2}} h^{(lm)} dr, \\
\tilde{c}^{(lm)} &:= \frac{ \isinf c^{(lm)} r^{\sqrt{l(l+1)/2}+1} \pr_r \left( \chi (l(r-1) ) \right)dr }{ \isinf  r^{\sqrt{l(l+1)/2} +1 } \pr_r^2 \left( \chi (l(r-1)  ) \right)dr }.
\end{aligned} \label{def:zet2feb24} \end{align}
\item \textbf{Definition of $\xi$.} Let $\xi$ be the $S_r$-tangent vectorfield solving on $\RRRwo$
\begin{align}
\begin{cases}
\frac{1}{r^3} \pr_r \left( r^3 \Divd \xi \right) = \zeta^{[\geq1]}, \\
 \Curld \xi =0, \\
\xi \vert_{r=1}=0.
\end{cases} \label{eq:defnBelmE}
\end{align}
\end{itemize}

\begin{remark} \label{remarkexistencexi} The existence of a solution $\xi$ to \eqref{eq:defnBelmE} for smooth compactly supported $\zeta^{[\geq1]}$ is established as follows. First, integrate the first equation of \eqref{eq:defnBelmE} to get on each $S_r$, $r>0$, 
\begin{align*}
\Divd \xi =& \frac{1}{r^3} \int\limits_1^r (r')^3 \zeta^{[\geq1]} dr', \\
\Curld \xi =& 0.
\end{align*}
This Hodge system admits a unique solution $\xi$ on each $S_r$ (see Proposition \ref{prop:EllipticityHodgejan}) and hence on $\RRR^3 \setminus B_1$. The above argument can be used together with the a priori estimates of the next sections and a limit argument to deduce existence of solutions in the low regularity case.
\end{remark}

\subsubsection{Proof of surjectivity at the Euclidean metric} \label{newsecav4}

In this section, we prove first in Lemma \ref{lem:feb24formalsol} that $\zeta^{[\geq1]}, u$ and $\be$ solve \eqref{eq:systemUbe}. Then, in Propositions \ref{prop:feb24est1} and \ref{prop:feb242p}, we show that
$$\zeta^{[\geq1]} \in  \ol{H}^{w-2}_{-5/2}, \,\,\, (u, \xi) \in  \ol{H}^{w}_{-1/2} \times  \ol{\HH}^{w}_{-1/2}$$
with quantitative bounds. These results prove Lemma \ref{lem:scalarsurj}.

\begin{lemma}\label{lem:feb24formalsol} \label{claim:Formal111}
The $u, \xi, \zeta^{[\geq1]}$ defined in \eqref{def:Dez20u0}-\eqref{eq:defnBelmE} solve \eqref{eq:systemUbe}, that is, on $\RRRwo$,
\begin{align*}
\triangle u + \frac{1}{r} \pr_r u + \frac{1}{r^2} u - \half \Ld u &= \half \left( \half h + \zeta^{[\geq1]} \right), \\
\frac{1}{r^3} \pr_r \left( r^3 \Divd \xi \right) &=  \zeta^{[\geq1]}.
\end{align*}
\end{lemma}

\begin{proof}[Proof of Lemma \ref{lem:feb24formalsol}]
The coefficients of the system \eqref{eq:systemUbe} depend only on $r$. Therefore we may project the equations of \eqref{eq:systemUbe} onto the Hodge-Fourier basis elements. This uses Proposition \ref{prop:Completeness}. We split \eqref{eq:systemUbe} into the modes $l=0$ and $l\geq1$ and get the following two subsystems {\bf S0}, {\bf S1}.
\begin{align*}
\pr_r^2 u^{[0]} + \frac{3}{r} \pr_r u^{[0]} + \frac{1}{r^2} u^{[0]} = \frac{1}{4} h^{[0]},  \tag{{\bf S0}} \label{mars4S0}
\end{align*}
\begin{align*}
\triangle u^{[\geq1]} + \frac{1}{r} \pr_r  u^{[\geq1]} + \frac{1}{r^2}  u^{[\geq1]} - \half \Ld  u^{[\geq1]} &= \half \left(  \half h^{[\geq1]} +  \zeta^{[\geq1]} \right), \tag{{\bf S1.1}} \\
\frac{1}{r^3} \pr_r \left( r^3 \Divd \xi \right) &= \zeta^{[\geq1]}. \tag{{\bf S1.2}} 
\end{align*}
The $(u,\xi)$ and $\zeta^{[\geq1]}$ defined in \eqref{def:Dez20u0}-\eqref{eq:defnBelmE} directly solve these equations on $\RRRwo$. This proves Lemma \ref{lem:feb24formalsol}. \end{proof}

The next proposition is essential and only possible due to our careful choice of $\zeta^{[\geq1]}$ which ensures that for all $w\geq2$, we have the boundary behaviour $u^{[\geq1]} \in \ol{H}^w_{-1/2}$.
\begin{proposition}\label{prop:feb24est1} Let $w\geq2$ be an integer. The following holds.
\begin{itemize}
\item {\bf Regularity and boundary control of $\zeta^{[\geq1]}$. }The scalar function $\zeta^{[\geq1]}$ is well-defined by \eqref{def:zet1feb24}-\eqref{def:zet2feb24} and bounded by
\begin{align}
\Vert \zeta^{[\geq1]} \Vert_{ H^{w-2}_{-5/2}(\RRRwo)} \lesssim \Vert h^{[\geq1]} \Vert_{ \ol{H}^{w-2}_{-5/2}}+C_w \Vert h^{[\geq1]} \Vert_{ \ol{H}^{0}_{-5/2}}. \label{zetaestfeb25}
\end{align}
Moreover, $\zeta^{[\geq1]} \in \ol{H}^{w-2}_{-5/2}$ and the following integral identities hold.
\begin{align} 
\isinf r^{1-\sqrt{l(l+1)/2}} \left( \half h^{(lm)} + \zeta^{(lm)} \right) dr&=0, \label{feb25id1}\\
\isinf r^2 \zeta^{(lm)} dr &=0. \label{feb25id2}
\end{align}
\item {\bf Precise estimate for $\zeta^{[\geq1]}$.} It holds that
\begin{align} \label{inverseestimatelast}
\Vert \cDd_1^{-1}(\pr_r \zeta^{[\geq1]},0) \Vert_{\HH^{w-2}_{-5/2}(\RRRwo)} \lesssim \Vert h^{[\geq1]}\Vert_{\ol{H}^{w-2}_{-5/2}}+C_w \Vert h^{[\geq1]} \Vert_{ \ol{H}^{0}_{-5/2}}.
\end{align}
Moreover, $\cDd_1^{-1}(\pr_r \zeta^{[\geq1]},0) \in \ol{\HH}^{w-2}_{-5/2}$.
\item {\bf Regularity and boundary control of $u^{[\geq1]}$. }The scalar function $u^{[\geq1]}$ defined in \eqref{def:Dez20u1} is bounded by
\begin{align}
\Vert u^{[\geq1]} \Vert_{H^w_{-1/2}(\RRRwo)} \lesssim \Vert h^{[\geq1]} \Vert_{\ol{H}^{w-2}_{-5/2}}+C_w \Vert h^{[\geq1]} \Vert_{ \ol{H}^{0}_{-5/2}}. \label{ugeq1estimatemars2}
\end{align}
Moreover, $u^{[\geq1]} \in \ol{H}^w_{-1/2}$ and in particular, $$\pr_r u^{[\geq1]} \vert_{r=1}=0.$$
\end{itemize}
\end{proposition}

\begin{remark}
The function $u^{[\geq1]}$ satisfies the elliptic equation \eqref{def:Dez20u1}. Therefore its boundary behaviour is harder to estimate than for $u^{[0]}$ and $\xi$ which essentially satisfy transport equations in $r$, see also Remark \ref{remarkexistencexi}.
\end{remark}

\begin{proof} We prove each point separately. \\ 

{\bf Regularity and boundary control of $\zeta^{[\geq1]}$.} We show at first that the constants $c^{(lm)}, \tilde{c}^{(lm)}$ are well-defined in \eqref{def:zet2feb24}. Concerning $c^{(lm)}$, for $l\geq1$, $m \in \{-l, \dots, l\}$,
\begin{align} \begin{aligned}
\left\vert c^{(lm)} \right\vert &=\half \left\vert\isinf r^{1-\sqrt{l(l+1)/2}} h^{(lm)} dr\right\vert \\
&\leq \half \left( \isinf r^{-2 \sqrt{l(l+1)/2}} dr \right)^{1/2} \left( \isinf \left( r h^{(lm)} \right)^2 dr \right)^{1/2}\\
&= \half \frac{1}{(2\sqrt{l(l+1)/2} -1)^{1/2} }  \left( \isinf \left( r h^{(lm)} \right)^2 dr \right)^{1/2}.
\end{aligned} \label{eq:estc1feb25} \end{align}
To estimate $\tilde{c}^{(lm)}$, estimate first its denominator. Integrating by parts twice and using that $\mathrm{supp } (\pr_r\chi)(l(r-1)) \subset [1,1+ \frac{1}{l}]$ yields
\begin{align*}
&\isinf r^{\sqrt{l(l+1)/2} +1 } \pr_r^2 \left( \chi (l(r-1)  ) \right) dr \\
=& - \left( \sqrt{l(l+1)/2} +1 \right) \int\limits_1^{1+\frac{1}{l}} r^{\sqrt{l(l+1)/2}} \pr_r \left( \chi(l(r-1)) \right) dr \\
=& - \left( \sqrt{l(l+1)/2} +1 \right)  \left( 1+ \frac{1}{l} \right)^{\sqrt{l(l+1)/2}} \\
& + \left( \sqrt{l(l+1)/2} +1 \right) \left( \sqrt{l(l+1)/2} \right) \int\limits_1^{1+\frac{1}{l}} r^{\sqrt{l(l+1)/2}-1}\chi(l(r-1)) dr\\
 \leq& -  \left( \sqrt{l(l+1)/2} +1 \right),
\end{align*}
where we uniformly bounded $\vert \chi \vert \leq 1$. Consequently, for all $l\geq1$,
\begin{align} \begin{aligned}
\vert \tilde{c}^{(lm)} \vert &\leq \frac{1}{\sqrt{l(l+1)/2} +1} \left\vert \isinf c^{(lm)} r^{\sqrt{l(l+1)/2}+1} \pr_r \left( \chi (l(r-1) ) \right) dr \right\vert\\
&\leq \frac{\vert c^{(lm)} \vert l}{\sqrt{l(l+1)/2} +1} \int\limits_1^{1+\frac{1}{l}} \left(1+\frac{1}{l}\right)^{\sqrt{l(l+1)/2}+1} \vert \pr_r \chi \vert(l(r-1)) dr \\
&\lesssim \frac{\vert c^{(lm)} \vert}{l}\\
&\lesssim   \frac{1}{l(2\sqrt{l(l+1)/2} +1 )^{1/2}} \left( \isinf \left( r h^{(lm)} \right)^2 dr \right)^{1/2}
\end{aligned} \label{eq:estc2feb25} \end{align}
where we used \eqref{eq:estc1feb25} and the fact that $\vert \pr_r \chi \vert$ is universally bounded. This shows that $c^{(lm)}, \tilde{c}^{(lm)}$ are well-defined.\\

We now prove the case $w=2$ of \eqref{zetaestfeb25}, that is, 
\begin{align*}
\Vert \zeta^{[\geq1]} \Vert_{ H^{0}_{-5/2}(\RRRwo)} \lesssim \Vert h^{[\geq1]} \Vert_{ \ol{H}^{0}_{-5/2}}.
\end{align*} 
Indeed, by plugging in \eqref{def:zet1feb24}, for $l\geq1$,
\begin{align} \begin{aligned}
& \isinf r^2 \left( \zeta^{(lm)} \right)^2 dr \\
 \lesssim&  \isinf r^{2\sqrt{l(l+1)/2} } \left[  \left(c^{(lm)}\right)^2  \left( \pr_r \left( \chi(l(r-1) ) \right)\right)^2 + \left( \tilde{c}^{(lm)} \right)^2 \left( \pr_r^2 \left( \chi( l (r-1) ) \right) \right)^2    \right] dr \\
 \lesssim&  \isinf r^{2\sqrt{l(l+1)/2} } \left[  \left(c^{(lm)}\right)^2  l^2 ( \pr_r  \chi )^2(l(r-1)) + \left( \tilde{c}^{(lm)} \right)^2 l^4 ( \pr_r^2 \chi)^2( l (r-1) )    \right] dr \\
\lesssim& \left( \isinf \left( r h^{(lm)} \right)^2 dr \right) l \int\limits_1^{1+ \frac{1}{l}} r^{2 \sqrt{l(l+1)/2}} dr \\
 \lesssim&  \left(1 +\frac{1}{l} \right)^{2 \sqrt{l(l+1)/2}} \isinf \left( r h^{(lm)} \right)^2 dr \\
\lesssim&   \isinf \left( r h^{(lm)} \right)^2 dr,
\end{aligned} \label{feb25est23} \end{align}
where we used that $\mathrm{supp } \, \pr_r \chi(l(r-1)) \subset [1,1+\frac{1}{l}]$ and \eqref{eq:estc1feb25},\eqref{eq:estc2feb25}. Summing over $l\geq1, m \in \{ -l, \dots, l\}$ proves the case $w=2$ of \eqref{zetaestfeb25}.\\

We turn now to the case $w>2$ of \eqref{zetaestfeb25}. On the one hand, the estimates \eqref{eq:estc1feb25}, \eqref{eq:estc2feb25} improve,
\begin{align} \begin{aligned}
\vert c^{(lm)} \vert &\lesssim \frac{1}{(2\sqrt{l(l+1)/2} -1)^{1/2} } \frac{1}{\sqrt{l(l+1)}^{w-2}}  \left( \isinf  \left( \frac{l(l+1)}{r^2} \right)^{w-2} \left( r^{w-1} h^{(lm)} \right)^2 dr \right)^{1/2}, \\
\vert \tilde{c}^{(lm)} \vert &\lesssim  \frac{1}{l(2\sqrt{l(l+1)/2} +1 )^{1/2}}\frac{1}{\sqrt{l(l+1)}^{w-2}} \left( \isinf  \left( \frac{l(l+1)}{r^2} \right)^{w-2} \left( r^{w-1} h^{(lm)} \right)^2 dr \right)^{1/2},
\end{aligned}\label{improvedmars41} \end{align}
where the integrals on the right-hand side correspond to the norm $\Vert h \Vert_{\ol{H}^{w-2}_{-5/2}}$ and are therefore summable, see Proposition \ref{prop:Howtoestimatefunctions}. \\

On the other hand, by differentiating \eqref{def:zet1feb24} and using Lemma \ref{lem:RelationsSpherical}, derivatives of $\zeta^{(lm)}$ generally are of the form
\begin{align} \begin{aligned}
\pr_r \zeta^{(lm)} \approx \frac{l}{r} \zeta^{(lm)}, \,\,\, \left(\Nd \zeta\right)_E^{(lm)} =- \frac{\sqrt{l(l+1)}}{r} \zeta^{(lm)}. \end{aligned} \label{eq:av412}
\end{align}
Combining \eqref{improvedmars41} with \eqref{eq:av412}, yields estimates for higher derivatives of $\zeta^{[\geq1]}$ analogously to \eqref{feb25est23}. This proves \eqref{zetaestfeb25} for all $w\geq2$, see Lemma \ref{lem:commutationrelation}.\\

Next, we claim that $\zeta^{[\geq1]} \in \ol{H}^{w-2}_{-5/2}$. Indeed, the sequence of smooth functions
 \begin{align*}
f_n := \sum\limits_{l=1}^n \sum\limits_{m=-l}^l \zeta^{(lm)} Y^{(lm)}
\end{align*}
is compactly supported in $\RRRwo$ and converges by \eqref{zetaestfeb25} in $\ol{H}^{w-2}_{-5/2}$ to $\zeta^{[\geq1]}$ as $n\to \infty$. See also the analogous \eqref{conddensity}, \eqref{conddensity2}.\\

The integral identities \eqref{feb25id1} and \eqref{feb25id2} follow from the definition of $\zeta^{[\geq1]}$ in \eqref{def:zet1feb24}-\eqref{def:zet2feb24}. The proof is left to the reader, see the analogous identity \eqref{eq:feb25est45}.\\

{\bf Precise estimate for $\zeta^{[\geq1]}$.} The proof is similar to part (2) of Proposition \ref{prop:boundaryvalues} and therefore only sketched here. \\

Consider first the case $w=2$ of \eqref{inverseestimatelast}. By Proposition \ref{prop:Howtoestimatefunctions} and Lemmas \ref{lem:commutationrelation} and \ref{lem:inversetoDDDD} and given that we already control $\zeta^{[\geq1]}$ above, it suffices to prove in the Hodge-Fourier formalism that for $l\geq1, m \in \{-l, \dots, l\}$,
\begin{align} \label{HFest}
\isinf r^2 \left( \frac{r}{\sqrt{l(l+1)}} \pr_r \zeta^{(lm)} \right)^2 dr \lesssim \isinf r^2 \left( h^{(lm)} \right)^2 dr.
\end{align}
This follows by using the explicit \eqref{def:zet1feb24} which shows that schematically
\begin{align*}
\left\vert \frac{r}{\sqrt{l(l+1)}} \pr_r \zeta^{(lm)} \right\vert \lesssim& \zeta^{(lm)} + c^{(lm)} r^{\sqrt{l(l+1)}/2} \pr_r ((\pr_r \chi)(l(r-1))) \\
&- \tilde{c}^{(lm)} r^{\sqrt{l(l+1)}/2} \pr^2_r ((\pr_r \chi)(l(r-1))) \\
\approx& \zeta^{(lm)}.
\end{align*}
Therefore, by the above control of $\zeta^{[\geq1]}$, \eqref{HFest} follows. This proves \eqref{inverseestimatelast} in the case $w=2$. \\

The case $w>2$ of \eqref{inverseestimatelast} is treated similarly. Indeed, by the explicit expression \eqref{def:zet1feb24}, derivatives can be expressed in the Hodge-Fourier formalism as multiplication by $\frac{\sqrt{l(l+1)}}{r}$. At the same time, the estimates for the constants $c^{(lm)}, \tilde{c}^{(lm)}$ improve, see \eqref{improvedmars41}. This allows to use the estimate above to conclude \eqref{inverseestimatelast} for all $w\geq2$.  \\

It remains to show that $\cDd_1^{-1}(\pr_r \zeta^{[\geq1]},0) \in \ol{\HH}^{w-2}_{-5/2}$. By using \eqref{inverseestimatelast} and the definition of $\zeta^{[\geq1]}$ in \eqref{def:zet1feb24}, it follows that 
\begin{align*}
X_n := \sum\limits_{l=1}^n \sum\limits_{m=-l}^l \left( \cDd_1^{-1}(\pr_r \zeta^{[\geq1]},0) \right)^{(lm)}_E E^{(lm)}
\end{align*}
is a sequence of smooth vectorfields compactly supported in $\RRRwo$ that converges in $\HH^{w-2}_{-5/2}$ to $\cDd_1^{-1}(\pr_r \zeta^{[\geq1]},0)$ as $n \to \infty$. By the definition of $\ol{\HH}^{w-2}_{-5/2}$, this shows that $\cDd_1^{-1}(\pr_r \zeta^{[\geq1]},0) \in \ol{\HH}^{w-2}_{-5/2}$ and hence finishes the precise estimate of $\zeta^{[\geq1]}$.\\

{\bf Regularity and boundary control of $u^{[\geq1]}$.} By Proposition \ref{higherregmars14}, for all $w\geq2$, the scalar function $u^{[\geq1]}$ defined in \eqref{def:Dez20u1} is bounded by
\begin{align}
\Vert u^{[\geq1]} \Vert_{H^{w}_{-1/2}(\RRRwo)} \lesssim \left\Vert  \half h^{[\geq1]}+ \zeta^{[\geq1]} \right\Vert_{\ol{H}^{w-2}_{-5/2}}+C_w \left\Vert  \half h^{[\geq1]}+ \zeta^{[\geq1]} \right\Vert_{\ol{H}^{0}_{-5/2}}.
\label{eq:8j9e}
\end{align}

We show now the improved boundary behaviour
\begin{align*}
u^{[\geq1]} \in \ol{H}^w_{-1/2}.
\end{align*}
By Proposition \ref{higherregmars14444} it suffices to prove the next claim.
\begin{claim} \label{8j2ndclaim}
Let $w\geq2$ be an integer. It holds that 
\begin{align*}
\pr_r u^{[\geq1]} \vert_{r=1} = 0.
\end{align*}
\end{claim}

First, by \eqref{eq:8j9e}, it holds that for $l\geq1$, $m \in \{-l, \dots, l\}$,
\begin{align*}
\isinf \frac{1}{(1+r)^2} \left( u^{(lm)} \right)^2 dr,  \isinf \left( \pr_r u^{(lm)} \right)^2 dr,  \isinf (1+r)^2 \left( \pr^2_r u^{(lm)} \right)^2 dr < \infty.
\end{align*}
By Lemma \ref{sobolev1d}, it follows that
\begin{align}
\sup\limits_{r\in [1, \infty)}  (1+r)^{-1/2} \left\vert u^{(lm)} \right\vert < \infty, \,\, \sup\limits_{r\in [1, \infty)}  (1+r)^{1/2} \left\vert \pr_r u^{(lm)} \right\vert < \infty. \label{eq:8j9f}
\end{align}

We show now that for $w\geq2$ and $l\geq1$, $m \in \{-l, \dots, l\}$,
$$\pr_r u^{(lm)} \vert_{r=1} =0.$$ 
Definition \eqref{def:Dez20u1} is in the Hodge-Fourier formalism equivalent to the following ODEs for $u^{(lm)}$ with $l\geq1, m \in \{ -l, \dots, l\}$, see Lemma \ref{lem:RelationsSpherical},
\begin{align} \begin{cases}
r^{-1+ \sqrt{l(l+1)/2}} \pr_r \left( r^{1-2\sqrt{l(l+1)/2}} \pr_r \left( r^{\sqrt{l(l+1)/2}} u^{(lm)} \right) \right) = \half \left( \half h^{(lm)} + \zeta^{(lm)} \right), \\
u^{(lm)} \vert_{r=1} = 0.
\end{cases} \label{eq:feb25est5}
\end{align}
On the one hand,
\begin{align*}
&\isinf \pr_r \left( r^{1-2\sqrt{l(l+1)/2}} \pr_r \left( r^{\sqrt{l(l+1)/2}} u^{(lm)} \right) \right) dr \\
=&\left[ r^{1-\sqrt{l(l+1)/2}} \pr_r u^{(lm)} + (\sqrt{l(l+1)/2}) r^{-\sqrt{l(l+1)/2}} u^{(lm)} \right]_1^\infty\\
=& -\pr_r u^{(lm)} \vert_{r=1},
\end{align*}
where we used that $l\geq1$, $u^{(lm)} \vert_{r=1}=0$ and \eqref{eq:8j9f}. \\

On the other hand, by \eqref{eq:feb25est5} and integral identity \eqref{feb25id1},
\begin{align*}
&\isinf \pr_r \left( r^{1-2\sqrt{l(l+1)/2}} \pr_r \left( r^{\sqrt{l(l+1)/2}} u^{(lm)} \right) \right) dr \\
=& \half \isinf r^{1-\sqrt{l(l+1)/2}}  \left(\half h^{(lm)} + \zeta^{(lm)} \right) dr
\\=&0.
\end{align*}
This shows that for $l\geq1$, $m \in \{-l, \dots, l\}$, 
$$\pr_r u^{(lm)} \vert_{r=1}= 0.$$ 
This proves Claim \ref{8j2ndclaim} and finishes the control of $u^{[\geq1]}$. This finishes the proof of Proposition \ref{prop:feb24est1}. \end{proof}

The next proposition shows that $u^{[0]} \in \ol{H}^w_{-1/2}$ and $\xi \in \ol{\HH}^w_{-1/2}$ with quantitative estimates.
\begin{proposition} \label{prop:feb242p} Let $w\geq2$ be an integer and $h \in \ol{H}^{w-2}_{-5/2}$. Then, the following holds.
\begin{itemize}
\item {\bf Regularity and boundary behaviour of $u^{[0]}$. } The radial scalar function $u^{[0]}$ defined in \eqref{def:Dez20u0} is bounded by
\begin{align}
\Vert u^{[0]} \Vert_{H^w_{-1/2}(\RRRwo)} \lesssim \Vert h^{[0]} \Vert_{\ol{H}^{w-2}_{-5/2}}+C_w \Vert h^{[0]} \Vert_{\ol{H}^{0}_{-5/2}}. \label{may24eq1}
\end{align}
Furthermore, it holds that $u^{[0]} \in \ol{H}^w_{-1/2}$.
\item {\bf Regularity and boundary behaviour of $\xi$. }The $S_r$-tangent vector field $\xi$ defined in \eqref{eq:defnBelmE} is bounded by
\begin{align}
\Vert \xi \Vert_{\HH^w_{-1/2}(\RRRwo)} \lesssim \Vert h^{[\geq1]} \Vert_{\ol{H}^{w-2}_{-5/2}}+C_w \Vert h^{[\geq1]} \Vert_{\ol{H}^{0}_{-5/2}}. \label{may24eq2} 
\end{align}
Furthermore, it holds that $\xi \in \HHlo$.
\end{itemize}
\end{proposition}

\begin{proof}[Proof of Proposition \ref{prop:feb242p}] We prove each part separately.\\

{ \bf Regularity and boundary behaviour of $u^{[0]}$.} We first show that for $w\geq2$, 
\begin{align}
\Vert u^{[0]} \Vert_{H^w_{-1/2}(\RRRwo)} \lesssim \Vert h^{[0]} \Vert_{\ol{H}^{w-2}_{-5/2}}+C_w \Vert h^{[0]} \Vert_{\ol{H}^{0}_{-5/2}}. \label{necestmars2u0}
\end{align}

Recall that $u^{[0]}$ satisfies on $r>1$ by \eqref{def:Dez20u0}
\begin{align*}\begin{cases}
\frac{1}{r^2} \pr_r \left( r \pr_r \left( r u^{[0]} \right) \right) = \frac{1}{4} h^{[0]},\\
u^{[0]}\vert_{r=1} = \pr_r u^{[0]} \vert_{r=1} =0.
\end{cases}
\end{align*}
By integration, it follows that
\begin{align}
u^{[0]}(r) = \frac{1}{r} \isr \frac{1}{r'} \left( \int\limits_1^{r'} (r''^2) h^{[0]} dr'' \right) dr'. \label{eq:av45}
\end{align}

We now prove the case $w=2$ of \eqref{necestmars2u0} by showing 
\begin{align}
\Vert u^{[0]} \Vert_{H^0_{-1/2}(\RRRwo)} &\lesssim \Vert h^{[0]} \Vert_{\ol{H}^{0}_{-5/2}}, \label{eq:av468j1} \\
\Vert \pr_r u^{[0]} \Vert_{H^0_{-3/2}(\RRRwo)} &\lesssim \Vert h^{[0]} \Vert_{\ol{H}^0_{-5/2}}, \label{eq:may241} \\
\Vert \pr^2_r u^{[0]} \Vert_{H^0_{-5/2}(\RRRwo)} &\lesssim \Vert h^{[0]} \Vert_{\ol{H}^0_{-5/2}}. \label{eq:8j3rd8j}
\end{align}
Indeed, see Lemma \ref{lem:commutationrelation} and note that tangential regularity follows because $u^{[0]}$ is constant on spheres.\\

To prove \eqref{eq:av468j1}, use \eqref{eq:av45} and that $u^{[0]}$ is constant on spheres,
\begin{align*}
\Vert u^{[0]} \Vert^2_{H^0_{-1/2}(\RRRwo)}&= \isinf \frac{1}{r^2} \left( \isr \frac{1}{r'} \left( \int\limits_1^{r'} (r''^2) h^{[0]} dr'' \right) dr' \right)^2 dr \\
&= \left[ -\frac{1}{r} \left( \isr \frac{1}{r'} \left( \int\limits_1^{r'} (r''^2) h^{[0]} dr'' \right) dr' \right)^2 \right]_1^\infty\\
& \qquad + 2 \isinf \frac{1}{r} \left(  \frac{1}{r} \int\limits_1^{r} (r'^2) h^{[0]} dr' \right) \left( \isr \frac{1}{r'} \left( \int\limits_1^{r'} (r''^2) h^{[0]} dr'' \right) dr' \right) dr.
\end{align*} 
The boundary term has negative sign and may thus be discarded. The integral term can be estimated by Cauchy-Schwarz as
\begin{align*}
&\isinf \left(  \frac{1}{r} \int\limits_1^{r} (r'^2) h^{[0]} dr' \right) \left( \frac{1}{r}  \isr \frac{1}{r'} \left( \int\limits_1^{r'} (r''^2) h^{[0]} dr'' \right) dr' \right) dr  \\
\leq& \left( \isinf \frac{1}{r^2} \left( \isr (r')^2 h^{[0]} dr'\right)^2  dr \right)^{1/2}  \Vert u^{[0]} \Vert_{H^0_{-1/2}(\RRRwo)}.
\end{align*}
This proves that 
\begin{align}
\Vert u^{[0]} \Vert^2_{H^0_{-1/2}(\RRRwo)} \lesssim \isinf \frac{1}{r^2} \left( \isr (r')^2 h^{[0]} dr' \right)^2 dr. \label{eq:1mars2}
\end{align}
A similar integration by parts shows that
\begin{align} \begin{aligned}
\isinf \frac{1}{r^2} \left( \isr (r')^2 h^{[0]} dr' \right)^2 dr &\lesssim \isinf  r^4 \left( h^{[0]}\right)^2 dr\\
&= \Vert h^{[0]} \Vert^2_{\ol{H}^{0}_{-5/2}}.
\end{aligned}\label{1marseq1} \end{align}
Together, \eqref{eq:1mars2} and \eqref{1marseq1} prove \eqref{eq:av468j1}.\\

We now prove \eqref{eq:may241}. By differentiating \eqref{eq:av45} in $r$, it follows that on $r>1$
\begin{align*}
 \pr_r u^{[0]} = -\frac{1}{r} u^{[0]} + \frac{1}{r^2} \isr (r')^2 h^{[0]} dr'.
\end{align*}
Therefore, by using \eqref{eq:av468j1} and \eqref{1marseq1},
\begin{align*}
\Vert \pr_r u^{[0]} \Vert_{H^0_{-3/2}(\RRRwo)} &\leq \Vert u^{[0]} \Vert_{H^0_{-1/2}(\RRRwo)} + \left\Vert  \frac{1}{r^2} \isr (r')^2 h^{[0]} dr' \right\Vert_{H^0_{-3/2}(\RRRwo)}\\
&=\Vert u^{[0]} \Vert_{H^0_{-1/2}(\RRRwo)} + \left\Vert  \frac{1}{r} \isr (r')^2 h^{[0]} dr' \right\Vert_{H^0_{-1/2}(\RRRwo)}\\
&\lesssim \Vert h^{[0]} \Vert_{H^0_{-5/2}(\RRRwo)}.
\end{align*} 
This proves \eqref{eq:may241}.\\

By the defining ODE \eqref{def:Dez20u0}, the previous estimates \eqref{eq:av468j1}, \eqref{eq:may241} imply \eqref{eq:8j3rd8j}. This finishes the proof of \eqref{necestmars2u0} in the case $w=2$.\\

We turn now to the case $w>2$ of \eqref{necestmars2u0}. Higher radial derivatives can be estimated by differentiating the defining ODE \eqref{def:Dez20u0} in $r$. Tangential regularity is trivial because $u^{[0]}$ is radial. This proves \eqref{necestmars2u0} for $w\geq2$.\\

It remains to show that $u^{[0]} \in \ol{H}^{w}_{-1/2}(\RRRwo)$. Indeed, this follows by \eqref{def:Dez20u0} and Proposition \ref{prop:TrivialExtensionRegularity}. This finishes the control of $u^{[0]}$. \\


{ \bf Regularity and boundary behaviour of $\xi$.} We prove now that for $w\geq2$
\begin{align}
\Vert \xi \Vert_{\HH^w_{-1/2}(\RRRwo)} \lesssim \Vert h^{[\geq1]} \Vert_{\ol{H}^{w-2}_{-5/2}}+ C_w \Vert h^{[\geq1]} \Vert_{\ol{H}^{0}_{-5/2}}. \label{highermars3}
\end{align}

First, we claim that 
\begin{align}
\Vert \xi \Vert^2_{\HH^0_{-1/2}(\RRRwo)} + \Vert \Nd \xi \Vert^2_{\HH^0_{-3/2}(\RRRwo)} \lesssim \Vert h^{[\geq1]} \Vert_{\ol{H}^{0}_{-5/2}}. \label{mars1est4} 
\end{align}
Indeed, by \eqref{eq:defnBelmE}, $\xi$ solves on each $S_r$, $r\geq1$,
\begin{align*}
\Divd \xi &= \frac{1}{r^3} \isr (r')^3 \zeta^{[\geq1]} dr',\\
\Curld \xi &= 0. 
\end{align*}
Therefore, by Proposition \ref{prop:EllipticityHodgejan}, for all $r\geq1$,
\begin{align} \label{ellestxi13}
\sr \vert \Nd \xi \vert^2 + \frac{1}{r^2} \vert \xi \vert^2 = \sr (\Divd \xi)^2.
\end{align}
We can estimate
\begin{align*}
&\Vert \xi \Vert_{\HH^0_{-1/2}(\RRRwo)}^2 + \Vert \Nd \xi \Vert_{\HH^0_{-3/2}(\RRRwo)}^2 \\
=& \isinf \sr \frac{1}{r^6} \left( \isr (r')^3 \zeta^{[\geq1]} \right)^2dr \\
=& \int\limits_{\mathbb{S}^2} \isinf \frac{1}{r^4} \left( \isr (r')^3 \zeta^{[\geq1]} dr' \right)^2 dr\\
=& \int\limits_{\mathbb{S}^2} \left[ -\frac{1}{3r^3} \left( \isr (r')^3 \zeta^{[\geq1]} dr' \right)^2 \right]_1^\infty\\
& + \frac{2}{3} \int\limits_{\mathbb{S}^2} \isinf \frac{1}{r^3} \left( r^3 \zeta^{[\geq1]} \right) \left( \isr (r')^3 \zeta^{[\geq1]} dr' \right) dr.
\end{align*} 
The first term on the right-hand side is non-positive and discarded. The second term can be estimated by Cauchy-Schwarz as
\begin{align*}
&\int\limits_{\RRRwo} \frac{1}{r^5} \left( r^3 \zeta^{[\geq1]} \right) \left( \isr (r')^3 \zeta^{[\geq1]} dr' \right) \\
\leq& \left( \, \int\limits_{\RRRwo} \frac{1}{r^6} \left( \isr (r')^3 \zeta^{[\geq1]} \right)^2 \right)^{1/2} \left( \, \int\limits_{\RRRwo} r^2 \left( \zeta^{[\geq1]} \right)^2 \right)^{1/2}\\
=& \left( \,\int\limits_{\RRRwo} (\Divd \xi)^2 \right)^{1/2} \left( \, \int\limits_{\RRRwo} r^2 \left( \zeta^{[\geq1]} \right)^2 \right)^{1/2} \\
=& \left( \, \Vert \xi \Vert_{\HH^0_{-1/2}(\RRRwo)}^2 + \Vert \Nd \xi \Vert_{\HH^0_{-3/2}(\RRRwo)}^2 \right)^{1/2} \Vert \zeta^{[\geq1]} \Vert_{\ol{H}^0_{-5/2}},
\end{align*}
where we used \eqref{ellestxi13}. This shows that
\begin{align*}
\Vert \xi \Vert_{\HH^0_{-1/2}(\RRRwo)} &+ \Vert \Nd \xi \Vert_{\HH^0_{-3/2}(\RRRwo)}  \lesssim  \Vert \zeta^{[\geq1]} \Vert_{\ol{H}^0_{-5/2}}.
\end{align*}
By Proposition \ref{prop:feb24est1}, this proves \eqref{mars1est4}.\\

Next, we consider radial regularity of order $1$. We claim that
\begin{align}
\Vert \Ndn \xi \Vert_{\HH^0_{-3/2}(\RRRwo)}^2 + \Vert \Nd \Ndn \xi \Vert_{\HH^0_{-5/2}(\RRRwo)}^2 \lesssim \Vert h^{[\geq1]} \Vert^2_{\ol{H}^0_{-5/2}}. \label{eq:8j2estxi} 
\end{align}
Indeed, by Lemma \ref{lem:commutationrelation}, it follows that on $r>1$
\begin{align*}
r \Divd \Ndn \xi &= \pr_r (r \Divd \xi)\\
&= \pr_r \left( \frac{1}{r^2} r^3 \Divd \xi \right)\\
&= -2 \Divd \xi + r \left( \frac{1}{r^3} \pr_r \left( r^3 \Divd \xi \right) \right),\\
\Curld \Ndn \xi &= 0
\end{align*}
Therefore $\Ndn \xi$ solves on each $S_r$, $r\geq1$, the Hodge system
\begin{align*}
\Divd \Ndn \xi &=-\frac{2}{r} \Divd \xi + \zeta^{[\geq1]},\\
\Curld \Ndn \xi &= 0.
\end{align*}
By Propositions \ref{prop:EllipticityHodgejan} and \ref{prop:feb24est1}, this proves \eqref{eq:8j2estxi}.\\

Similarly, we have the higher radial regularity for each $w\geq2$,
\begin{align} \label{radreg1}
\Vert \Nd_N^{w} \xi \Vert_{\HH^0_{-1/2-w}(\RRRwo)}^2 \lesssim \Vert h^{[\geq1]} \Vert^2_{\ol{H}^{w-2}_{-5/2}}+C_w \Vert h^{[\geq1]} \Vert_{\ol{H}^{0}_{-5/2}}.
\end{align}
This follows by an induction in $w\geq2$, using that by
\begin{align*}
\Divd \left( \Nd_n^w \xi \right) =& \frac{1}{r} \pr_r^w (r \Divd \xi)\\
=& \frac{1}{r} \pr_r^{w-1} \left( -2 \Divd \xi + r \zeta^{[\geq2]} \right)\\
=& -2 \Divd \left( \Nd_N^{w-1} \left( \frac{1}{r} \xi \right) \right) + \frac{1}{r} \pr_r^{w-1} \left( \frac{1}{r} \zeta^{[\geq1]} \right)
\end{align*}
so that we have
\begin{align} \label{normalderxiw}
\Nd_n^w \xi =& -2 \Nd_n^{w-1}\left( \frac{1}{r} \xi \right) + \Nd_N^{w-2} \left( \frac{1}{r} \cDd_1^{-1}(\zeta^{[\geq1]}, 0) \right) \\
&+ \Nd_N^{w-2} \left( \cDd_1^{-1}(\pr_r \zeta^{[\geq1]},0 ) \right),
\end{align}
where the last term is controlled by Proposition \ref{prop:feb24est1}. \\

We turn now to tangential regularity. We claim first that
\begin{align}
\Vert \Nd \Nd \xi \Vert_{\HH^0_{-5/2}(\RRRwo)} \lesssim \Vert h^{[\geq1]} \Vert_{\ol{H}^{0}_{-5/2}}, \label{8jeq1}.
\end{align}
By Proposition \ref{prop:Howtoestimatefunctions}, it suffices to estimate for $l\geq1, m \in \{ -l, \dots l \}$
\begin{align}
\isinf r^2 \left(\frac{l(l+1)}{r^2} \xi_E^{(lm)} \right)^2 dr  &\lesssim \isinf r^2 (h^{(lm)})^2 dr, \label{est:3mars2}. 
\end{align}
Definition \eqref{eq:defnBelmE} is in the Hodge-Fourier formalism equivalent to the following expression for $\xi_E^{(lm)}$, $l\geq1, m \in \{ -l, \dots l \}$,
\begin{align}
\frac{\sqrt{l(l+1)}}{r}\xi_E^{(lm)} = \frac{1}{r^2} \left( \isr (r')^2 \zeta^{(lm)}dr' \right). \label{eq:defininingmars41}
\end{align}
Rewrite first 
\begin{align} \begin{aligned}
\isinf r^2 \left(  \frac{l(l+1)}{r^2} \right)^2 \left( \xi_E^{(lm)} \right)^2 dr &= l(l+1) \isinf \frac{1}{r^4} \left( \isr (r')^2 \zeta^{(lm)} dr' \right)^2dr \\
&= l(l+1) \int\limits_1^{1+\frac{1}{l}} \frac{1}{r^4} \left( \isr (r')^2 \zeta^{(lm)} dr' \right)^2 dr,
\end{aligned} \label{eq:estimateZETALMtang} \end{align}
where in the last integral we bounded the domain of integration by combining the integral identity \eqref{feb25id2} and the fact that $$\mathrm{supp } \zeta^{(lm)}\subset \left[ 1, 1+ \frac{1}{ l} \right].$$
By \eqref{def:zet1feb24} and \eqref{def:zet2feb24}, the right-hand side of \eqref{eq:estimateZETALMtang} can be estimated by
\begin{align} \begin{aligned}
&l(l+1) \int\limits_1^{1+\frac{1}{l}} \frac{1}{r^4} \left( \isr (r')^2 \zeta^{(lm)} dr' \right)^2dr  \\
\lesssim& l(l+1) \int\limits_1^{1+\frac{1}{l}} \frac{1}{r^4} \left( \isr  c^{(lm)} (r')^{ \sqrt{l(l+1)/2}+1} \pr_r \left( \chi (l(r-1) ) \right) dr' \right)^2 dr \\
&+ l(l+1) \int\limits_1^{1+\frac{1}{l}} \frac{1}{r^4} \left(  \isr \tilde{c}^{(lm)} (r')^{ \sqrt{l(l+1)/2} +1 } \pr_r^2 \left( \chi (l(r-1)) \right) dr' \right)^2 dr.
\end{aligned}\label{mars3est6} \end{align}
The first term on the right-hand side is estimated as
\begin{align*}
 &l(l+1) \int\limits_1^{1+\frac{1}{l}} \frac{1}{r^4} \left( \isr  c^{(lm)} (r')^{\sqrt{l(l+1)/2}+1} \pr_r \left( \chi (l(r-1) ) \right) dr' \right)^2 dr \\
  \lesssim& l(l+1) \int\limits_1^{1+\frac{1}{l}} \frac{1}{r^4} \left( c^{(lm)} \right)^2 l ^2 r^{2 \sqrt{l(l+1)/2} +2} \left( \int\limits_1^{1+\frac{1}{l}} dr'  \right)^2 dr \\
 \lesssim& l(l+1) \left( c^{(lm)} \right)^2 \int\limits_1^{1+\frac{1}{l}} r^{2 \sqrt{l(l+1)/2}-2} dr \\
 \lesssim& \isinf \left( r h^{(lm)} \right)^2 dr,
\end{align*}
where we used \eqref{eq:estc1feb25} and the fact that $l\geq1$. The second term on the right-hand side of \eqref{mars3est6} is estimated similarly by using \eqref{eq:estc2feb25}, this is left to the reader. This proves \eqref{est:3mars2} and therefore \eqref{8jeq1}.\\

We also have the higher tangential regularity
\begin{align} \label{highertangreg1}
\Vert \Nd^w \xi \Vert_{\HH^{0}_{-1/2-w}(\RRRwo)} \lesssim \Vert h^{[\geq1]} \Vert_{\ol{H}^{w-2}_{-5/2}}+ C_w \Vert h^{[\geq1]} \Vert_{\ol{H}^{0}_{-5/2}}.
\end{align}
Indeed, in the Hodge-Fourier formalism, see \eqref{eq:defininingmars41},
\begin{align*}
\left\vert \left( \frac{\sqrt{l(l+1)}}{r} \right)^w \xi_E^{(lm)} \right\vert =& \frac{1}{\sqrt{l(l+1)} r} \left\vert \isr (r')^2 \left( \frac{\sqrt{l(l+1)}}{r} \right)^w \zeta^{(lm)}dr' \right\vert\\
\lesssim&  \frac{1}{\sqrt{l(l+1)} r} \left\vert \isr (r')^2 \left( \frac{\sqrt{l(l+1)}}{r'} \right)^w \zeta^{(lm)}dr' \right\vert.
\end{align*}
Together with the higher regularity of $\zeta^{[\geq1]}$ provided by Proposition \ref{prop:feb24est1}, similar estimates as for \eqref{8jeq1} imply \eqref{highertangreg1}.\\

To summarise, \eqref{mars1est4}, \eqref{eq:8j2estxi}, \eqref{radreg1}, \eqref{8jeq1} and \eqref{highertangreg1} imply \eqref{highermars3} for $w\geq2$.\\

It remains to show that $\xi \in \ol{\HH}^{w}_{-1/2}$. This follows by induction from \eqref{normalderxiw} and the fact that $\zeta^{[\geq1]}, \cDd_1^{-1} (\pr_r \zeta^{[\geq1]},0) \in \HHno$ together with Proposition \ref{prop:TrivialExtensionRegularity}. This finishes the proof of Proposition \ref{prop:feb242p}. \end{proof}


\section{Proof of the main Theorem \ref{MainTheorem12} } \label{sec:Iteration}

In this section, we prove the Main Theorem \ref{MainTheorem12}. The idea of the proof is to use Theorems \ref{thm:kextension1} and \ref{thm:MainExtensionScalarTheorem} to set up an iterative scheme. We show that this scheme is well-defined and converges to a fixed point which solves the maximal constraint equations on $\RRR^3$. In Section \ref{SectionHigherreg225342}, we prove higher regularity estimates \eqref{estimateshigherreg24352}.


\subsection{Setup of the iterative scheme} \label{sec:iterationdef}

In this section, we define a sequence of pairs $(g_i, k_i)_{i\geq1}$, where for each $i\geq1$, $g_i$ is an $\HH^2_{-1/2}$-asymptotically flat metric and $k_i \in \HH^{1}_{-3/2}$ a symmetric $2$-tensor on $\RRR^3$.\\

Let $\varep >0$ be a small constant to be determined later. Let $(\bar g, \bar k) \in \HH^w(B_1) \times \HH^{w-1}(B_1)$ be a solution to the maximal constraint equations on $B_1$,
\begin{align} \begin{aligned}
R(\bar g) &= \vert \bar k \vert_{\bar g}^2,\\
\Div_{\bar g} \bar k &= 0, \\
\tr_{\bar g } \bar k&=0
\end{aligned} \label{maximaleqfeb22}
\end{align}
such that 
\begin{align*}
\Vert (\bar g-e, \bar k) \Vert_{\HH^2(B_1) \times \HH^{1}(B_1)} < \varep.
\end{align*}

\begin{itemize}
\item {\bf Definition of $(g_1,k_1)$.} By Proposition \ref{functionextcts}, extend $\bar g$ from $B_1$ to an $\HH^2_{\de}$-asymptotically flat metric $g_1$ on $\RRR^3$ such that
\begin{align*}
\Vert g_1 - e \Vert_{\HH^2_{-1/2}} \lesssim \Vert \bar g - e \Vert_{\HH^2(B_1)}.
\end{align*}
Similarly, extend $\bar k$ from $B_1$ to a symmetric $g_1$-tracefree $2$-tensor $k_1 \in \HH^1_{-3/2}$ such that
\begin{align}
\Vert k_1 \Vert_{\HH^1_{-3/2}} \lesssim \Vert k \Vert_{\HH^{1}(B_1)};\label{eq:estk1feb22}
\end{align}
see Lemma \ref{remark:Bounded}.

\item {\bf Definition of $(g_{i+1}-e,k_{i+1})$ for $i \geq1$.} Given $(g_{i}, k_{i})$, define $(g_{i+1},k_{i+1})$ as follows.\\

\noindent First, let $g_{i+1}$ be the $\HH^2_{-1/2}$-asymptotically flat metric on $\RRR^3$ constructed by Theorem \ref{thm:MainExtensionScalarTheorem} such that 
\begin{align*}
g_{i+1}\vert_{B_1} &= \bar g, \\
R(g_{i+1}) &= \vert k_i \vert_{g_i}^2 \,\,\, \text{on } \RRR^3.
\end{align*}
Here we assumed that $\Vert (g_{i}-e, k_{i} )\Vert_{\HH^2_{-1/2} \times \HH^1_{-3/2}}$ is sufficiently small.\\

\noindent Second, let $k_{i+1} \in \HH^1_{-3/2}$ be the symmetric $2$-tensor on $\RRR^3$ constructed by Theorem \ref{thm:kextension1} such that $$k_{i+1} \vert_{B_1} = \bar k$$ and on $\RRR^3$
\begin{align*} 
\Div_{g_{i+1}} \left( k_{i+1} \right) &= 0, \\
\tr_{g_{i+1}} \left( k_{i+1} \right) &=0.
\end{align*}
Here we assumed further that $\Vert g_{i+1}-e \Vert_{\HH^2_{-1/2}}$ is sufficiently small.
\end{itemize}

In the next section we prove that for $\varep>0$ small enough, the above smallness assumptions are satisfied for all $i\geq1$. In particular, that the sequence is well-defined.

\subsection{Convergence of the iterative scheme}

In this section, we combine the iteration estimates of Theorems \ref{thm:kextension1} and \ref{thm:MainExtensionScalarTheorem} to prove estimates for the iteration scheme defined above in Section \ref{sec:iterationdef}. These estimates then directly imply uniform boundedness and convergence of the sequence $(g_i, k_i)_{i\geq1}$.\\

The next proposition is the main result of this section.
\begin{proposition} \label{seq:estifeb221}
Let $w\geq2$ be an integer. There exist universal constants $\varep>0,c>0$ such that if for an $i\geq2$ it holds that
\begin{align*}
\Vert (g_i-e,k_i) \Vert_{\HH^2_{-1/2} \times \HH^1_{-3/2}}< \varep, \Vert(g_{i-1}-e, k_{i-1}) \Vert_{\HH^2_{-1/2} \times \HH^1_{-3/2}} < \varep,
\end{align*} 
then
\begin{align} \begin{aligned}
&\Vert (g_{i+1}, k_{i+1}) -(g_{i}, k_{i})\Vert_{\HH^2_{-1/2} \times \HH^1_{-3/2}} \\
\leq& c \left( \Vert (g_i-e,k_i) \Vert_{\HH^2_{-1/2} \times \HH^1_{-3/2}} + \Vert (g_{i-1}-e, k_{i-1}) \Vert_{\HH^2_{-1/2} \times \HH^1_{-3/2}}  \right) \\
&  \times\Vert (g_i,k_i)-(g_{i-1}, k_{i-1} ) \Vert_{\HH^2_{-1/2} \times \HH^1_{-3/2}}.
\end{aligned} \label{est:feb221} \end{align}
\end{proposition}

Before proving Proposition \ref{seq:estifeb221}, we state the following technical lemma. Its proof is based on Lemma \ref{ProductEstimates}, see also Lemma \ref{GInverseAnalysis}, and left to the reader.
\begin{lemma} \label{ksquareestifeb22}
Let $g$ and $\tilde g$ be two $\HH^2_{-1/2}$-asymptotically flat metrics on $\RRR^3$. There exists universal constant $\varep>0$ such that if 
\begin{align*}
\Vert g-e \Vert_{\HH^2_{-1/2}} < \varep, \Vert \tilde g-e \Vert_{\HH^2_{-1/2}} < \varep, 
\end{align*}
then for all symmetric $2$-tensors $V \in \HHm$,
\begin{align*}
\Vert \vert V \vert_g^2- \vert V \vert_{\tilde{g}}^2 \Vert_{H^{0}_{-5/2}} \lesssim \Vert g- \tilde{g} \Vert_{\HH^2_{-1/2}} \Vert V \Vert_{\HH^1_{-3/2}}^2.
\end{align*}
\end{lemma}

We turn now to the proof of Proposition \ref{seq:estifeb221}.
\begin{proof}[Proof of Proposition \ref{seq:estifeb221}] On the one hand, for $\varep>0$ sufficiently small, it follows by the iteration estimates of Theorem \ref{thm:MainExtensionScalarTheorem} and Lemmas \ref{ksquareestifeb22} that 
\begin{align}  \begin{aligned}
\Vert g_{i+1} - g_{i} \Vert_{\HH^2_{-1/2}} &\lesssim \Vert R(g_{i+1}) - R(g_i) \Vert_{H^{0}_{-5/2}}\\
&\lesssim \Vert \vert k_i \vert_{g_i}^2 - \vert k_{i-1} \vert_{g_{i-1}}^2 \Vert_{H^{0}_{-5/2}(\RRR^3)} \\
&\lesssim \Vert \vert k_i \vert_{g_i}^2 - \vert k_{i-1} \vert_{g_{i}}^2 \Vert_{H^{0}_{-5/2}(\RRR^3)} + \Vert \vert k_{i-1} \vert_{g_i}^2 - \vert k_{i-1} \vert_{g_{i-1}}^2 \Vert_{H^{0}_{-5/2}(\RRR^3)} \\
&\lesssim \Vert \vert k_i \vert_{g_i}^2 - \vert k_{i-1} \vert_{g_{i}}^2 \Vert_{H^{0}_{-5/2}(\RRR^3)} + \Vert k_{i-1} \Vert_{\HH^1_{-3/2}}^2 \Vert g_i - g_{i-1} \Vert_{\HH^2_{-1/2}}
\end{aligned} \label{star11} \end{align}
Using Lemma \ref{ProductEstimates} and the identity
\begin{align*}
\vert k_i \vert_{g_i}^2 - \vert k_{i-1} \vert_{g_{i}}^2 = (k_i - k_{i-1})^{ab} (k_i + k_{i-1} )_{ab},
\end{align*}
we have, for $\varep>0$ sufficiently small,
\begin{align*}
\Vert \vert k_i \vert_{g_i}^2 - \vert k_{i-1} \vert_{g_{i}}^2 \Vert_{H^{0}_{-5/2}(\RRR^3)} &\lesssim \Vert k_i + k_{i-1} \Vert_{\HH^1_{-3/2}} \Vert k_i - k_{i-1} \Vert_{\HH^1_{-3/2}} \\
& \lesssim \left(\Vert k_i \Vert_{\HH^1_{-3/2}} + \Vert k_{i-1} \Vert_{\HH^1_{-3/2}}\right) \Vert k_i - k_{i-1} \Vert_{\HH^1_{-3/2}}.
\end{align*}
Plugging this into \eqref{star11}, we get
\begin{align*}
&\Vert g_{i+1} - g_{i} \Vert_{\HH^2_{-1/2}} \\
\lesssim& \left( \Vert k_i\Vert_{\HH^1_{-3/2}} +\Vert k_{i-1} \Vert_{\HH^1_{-3/2}} + \Vert k_{i-1} \Vert_{\HH^1_{-3/2}}^2 \right) \Vert (g_i,k_i)- (g_{i-1}, k_{i-1}) \Vert_{\HH^2_{-1/2} \times \HH^1_{-3/2}}.
\end{align*}

On the other hand, for $\varep>0$ sufficiently small, by the iteration estimates of Theorem \ref{thm:kextension1}, 
\begin{align*}
\Vert k_{i+1} - k_i \Vert_{\HH^1_{-3/2}} &\lesssim \Vert \bar k \Vert_{\HH^{1}(B_1)} \Vert g_{i+1} - g_i \Vert_{\HH^2_{-1/2}} \\
&\lesssim \Vert k_i \Vert_{\HH^1_{-3/2}}\Vert g_{i+1} - g_i \Vert_{\HH^2_{-1/2}}.
\end{align*}

Combining the two last estimates, it follows that for $\varep>0$ sufficiently small, there exists a universal constant $c>0$ such that
\begin{align*}
&\Vert (g_{i+1},k_{i+1})- (g_{i}, k_{i}) \Vert_{\HH^2_{-1/2} \times \HH^1_{-3/2}} \\
\leq & c \left( \Vert k_i \Vert_{\HH^1_{-3/2}}+\Vert k_{i-1} \Vert_{\HH^1_{-3/2}} + \Vert k_{i-1} \Vert_{\HH^1_{-3/2}}^2 \right) \Vert (g_i,k_i)- (g_{i-1}, k_{i-1}) \Vert_{\HH^2_{-1/2} \times \HH^1_{-3/2}}.
\end{align*}
This finishes the proof of Proposition \ref{seq:estifeb221}. \end{proof}

Proposition \ref{seq:estifeb221} implies that for sufficiently small data on $B_1$, the iteration is a contraction mapping and hence converges to a fixed point. For completeness, we write out details. \\

First, we to prove that the sequence is well-defined and derive a uniform estimate.
\begin{lemma}\label{cor:feb22}
Let $w\geq2$ be an integer. There is a universal $\varep>0$ small enough such that if 
\begin{align*}
\Vert (\bar g-e, \bar k) \Vert_{\HH^2(B_1) \times \HH^{1}(B_1)} < \varep,
\end{align*}
then the sequence $(g_i, k_i)_{i\geq1}$ is well-defined and for all $i \geq2$,
\begin{align}
\Vert (g_{i+1},k_{i+1})- (g_{i}, k_{i}) \Vert_{\HH^2_{-1/2} \times \HH^1_{-3/2}} \leq \frac{1}{4} \Vert ( g_i,k_i ) -  (g_{i-1}, k_{i-1}) \Vert_{\HH^2_{-1/2} \times \HH^1_{-3/2}}. \label{eqstfebb3}
\end{align}
Furthermore, it is uniformly bounded by
\begin{align*}
\Vert (g_i,k_i) \Vert_{\HH^2_{-1/2} \times \HH^1_{-3/2}} \lesssim \Vert (\bar g-e, \bar k) \Vert_{\HH^2(B_1) \times \HH^{1}(B_1)}.
\end{align*} 
\end{lemma}

\begin{proof} The proof of \eqref{eqstfebb3} goes by induction in $i\geq2$. \\

{\bf The case $i=2$.} By construction, for $\varep>0$ sufficiently small, by Theorems \ref{thm:kextension1} and \ref{thm:MainExtensionScalarTheorem},
\begin{align*} \begin{aligned}
\Vert ( g_1 - e, k_1) \Vert_{\HH^2_{-1/2} \times \HH^1_{-3/2}} &\lesssim \Vert (\bar g - e, \bar k ) \Vert_{\HH^2(B_1) \times \HH^{1}(B_1)}, \\
\Vert (g_2 -e, k_2 )\Vert_{\HH^2_{-1/2} \times \HH^1_{-3/2} } &\lesssim  \Vert  \bar g - e \Vert_{\HH^2(B_1)} + \Vert R(g_2) \Vert_{H^{0}_{-5/2}} + \Vert \bar k \Vert_{\HH^{1}(B_1)} \\
& \lesssim \Vert \bar g - e \Vert_{\HH^2(B_1)} + \Vert \vert k_1 \vert_{g_1}^2 \Vert_{H^{0}_{-5/2}} + \Vert \bar k \Vert_{\HH^{1}(B_1)} \\
&\lesssim \Vert \bar g - e \Vert_{\HH^2(B_1)} + \Vert k_1 \Vert_{\HH^1_{-3/2}}^2 + \Vert \bar k \Vert_{\HH^{1}(B_1)} \\
&\lesssim \Vert (\bar g-e ,\bar k ) \Vert_{\HH^2(B_1) \times \HH^{1}(B_1)},
\end{aligned}
\end{align*}
where we used Lemmas \ref{ProductEstimates}. By Theorems \ref{thm:kextension1} and \ref{thm:MainExtensionScalarTheorem}, for $\varep>0$ sufficiently small, $(g_3,k_3)$ is well-defined. \\

By Proposition \ref{seq:estifeb221}, there exists a universal $c>0$ such that  
\begin{align*}
&\Vert (g_3-g_2, k_3-k_2 )\Vert_{\HH^2_{-1/2} \times \HH^1_{-3/2}} \\
\leq& 3c \Vert (\bar g-e ,\bar k ) \Vert_{\HH^2(B_1) \times \HH^{1}(B_1)} \Vert (g_2-g_{1}, k_2 - k_{1} )\Vert_{\HH^2_{-1/2} \times \HH^1_{-3/2}}
\end{align*}
Let $\varep< \frac{1}{24c}$. This proves the case $i=2$.\\

{\bf The induction step $i\to i+1$.} Using the induction hypothesis, it holds for $j=i-1,i$ that
\begin{align}\begin{aligned}
&\Vert (g_{j+1}-e, k_{j+1}) \Vert_{\HH^2_{-1/2} \times \HH^1_{-3/2}} \\ 
\leq& \Vert (g_{j+1}-g_{j}, k_{j+1} - k_{j} \Vert_{\HH^2_{-1/2} \times \HH^1_{-3/2}} + \dots + \Vert (g_3-g_2,k_3-k_2) \Vert_{\HH^2_{-1/2} \times \HH^1_{-3/2}}\\
&+ \Vert (g_2-e,k_2) \Vert_{\HH^2_{-1/2} \times \HH^1_{-3/2}}\\
\leq &\sum\limits_{m=1}^{j-2}  \frac{1}{4^m}  \Vert (g_2-g_1,k_2-k_1) \Vert_{\HH^2_{-1/2} \times \HH^1_{-3/2}} +  \Vert (g_2-e,k_2) \Vert_{\HH^2_{-1/2} \times \HH^1_{-3/2}}\\
\leq &2  \Vert (g_2-g_1,k_2-k_1) \Vert_{\HH^2_{-1/2} \times \HH^1_{-3/2}}+  \Vert (g_2-e,k_2) \Vert_{\HH^2_{-1/2} \times \HH^1_{-3/2}}\\
\lesssim & \Vert (\bar g-e, \bar k) \Vert_{\HH^2(B_1) \times \HH^1(B_1)}.
\end{aligned}  \label{eq:feb225}\end{align}
This shows that $(g_{i+2},k_{i+2})$ is well-defined for $\varep>0$ sufficiently small. By applying Proposition \ref{seq:estifeb221}, there exists a universal constant $c'>0$ such that
\begin{align*}
\Vert (g_{i+2}-g_{i+1},k_{i+2}-k_{i+1}) \Vert_{\HH^2_{-1/2} \times \HH^1_{-3/2}} \leq c' \varep \Vert (g_{i+1}-g_{i}, k_{i+1}-k_i) \Vert_{\HH^2_{-1/2} \times \HH^1_{-3/2}}.
\end{align*} 
Let $\varep< \frac{1}{24c'}$. This finishes the induction step and hence the proof of \eqref{eqstfebb3}. \\

For $\varep>0$ small, we have, as in \eqref{eq:feb225}, the uniform estimate for all $i\geq1$,
\begin{align*}
\Vert (g_i, k_i) \Vert_{\HH^2_{-1/2} \times \HH^1_{-3/2}} \lesssim \Vert (\bar g-e, \bar k) \Vert_{\HH^2(B_1) \times \HH^{1}(B_1)}.
\end{align*}
This finishes the proof of Lemma \ref{cor:feb22}. \end{proof}

Lemma \ref{cor:feb22} implies convergence of the iterative scheme.
\begin{corollary} \label{lem:ConvergenceOfSeries}
There exists $\varep>0$ small such that if 
\begin{align*}
\Vert (\bar g, \bar k) \Vert_{\HH^2(B_1) \times \HH^{1}(B_1)} <\varep,
\end{align*}
then the sequence $ (g_i-e,k_i)_{i\geq 1}$ converges in $\HH^2_{-1/2} \times \HH^1_{-3/2}$ as $i\to \infty$. Its limit
\begin{align*}
(g, k) := \lim\limits_{i \to \infty} (g_i,k_i) \in \HH^2_{-1/2} \times \HH^1_{-3/2}
\end{align*}
solves the maximal constraint equations on $\RRR^3$ 
\begin{align} \begin{aligned}
R(g) &= \vert k \vert_{g}^2,\\
\Div_{g} k &= 0, \\
\tr_{g} k &=0
\end{aligned}
\label{eq:maximaleqfeb2255}  \end{align}
and satisfies $(g, k) \vert_{B_1} = (\bar g, \bar k)$. Moreover,
\begin{align*}
\Vert (g-e, k) \Vert_{\HH^2_{-1/2} \times \HH^1_{-3/2}} &\lesssim \Vert (\bar g-e, \bar k) \Vert_{H^2(B_1) \times H^{1}(B_1)}.
\end{align*}
\end{corollary}

\begin{proof}[Proof of Corollary \ref{lem:ConvergenceOfSeries}] Lemma \ref{cor:feb22} shows that the iterative scheme $(g_i,k_i)_{i\geq1}$ is a contraction in the Hilbert space $\HH^2_{-1/2} \times \HH^1_{-3/2}$. By the Banach fixed-point theorem, the scheme therefore converges to a fixed point,
\begin{align*}
(g, k) := \lim\limits_{i \to \infty} (g_i,k_i) \in H^2_{-1/2} \times H^{1}_{-3/2}.
\end{align*}
The convergence in $\HH^2_{-1/2} \times \HH^1_{-3/2}$ is strong enough to conclude that $(g,k)$ solves \eqref{eq:maximaleqfeb2255} and moreover, by construction of the sequence, 
$$(g, k) \vert_{B_1} = (\bar g, \bar k).$$
By the uniform estimate in Lemma \ref{cor:feb22},
\begin{align*}
\Vert (g-e, k) \Vert_{\HH^2_{-1/2} \times \HH^1_{-3/2}} 
\lesssim \Vert (\bar g-e, \bar k) \Vert_{\HH^2(B_1) \times \HH^{1}(B_1)}.
\end{align*}
This finishes the proof of Corollary \ref{lem:ConvergenceOfSeries} \end{proof}


\subsection{Higher regularity estimates} \label{SectionHigherreg225342} It remains to prove the higher regularity estimates \eqref{estimateshigherreg24352} of Theorem \ref{MainTheorem12}.\\

Using that the constructed solution $(g,k)$ in the previous section is a fixed point of the iteration scheme, we have by the higher regularity estimates of Theorems \ref{thm:kextension1} and \ref{thm:MainExtensionScalarTheorem} that for integers $w\geq 3$, for $\varep>0$ sufficiently small,
\begin{align*}
&\Vert g-e \Vert_{\HH^w_{-1/2}} + \Vert k \Vert_{\HH^{w-1}_{-3/2}}\\
 \lesssim& \Vert R(g) \Vert_{H^{w-2}_{-5/2}} + \Vert \bar k \Vert_{\HH^1(B_1)} \Vert g-e \Vert_{\HH^w_{-1/2}} + C_w\Big( \Vert \bar g-e \Vert_{\HH^w(B_1)} + \Vert R(g) \Vert_{\HH^0_{-5/2}} + \Vert \bar k \Vert_{\HH^{w-1}(B_1)} \Big) \\
\lesssim& \Vert R(g) \Vert_{H^{w-2}_{-5/2}}+ C_w\Big( \Vert \bar g-e \Vert_{\HH^w(B_1)} + \Vert R(g) \Vert_{\HH^0_{-5/2}} + \Vert \bar k \Vert_{\HH^w(B_1)} \Big)\\
\lesssim& \Vert R(g) \Vert_{H^{w-2}_{-5/2}} + C_w \Big( \Vert \bar g-e \Vert_{\HH^w(B_1)} + \Vert \bar k \Vert_{\HH^w(B_1)} \Big),
\end{align*}
where we used the estimates of Theorems \ref{thm:kextension1} and \ref{thm:MainExtensionScalarTheorem}, and absorbed the second term on the right-hand side after the first estimate.\\

Using that $(g,k)$ satisfies the constraint equations, it holds that $R(g) = \vert k \vert_{g}^2$. Therefore we have by product estimates as in Lemma \ref{ProductEstimates}, for $\varep>0$ sufficiently small,
\begin{align*}
\Vert R(g) \Vert_{\HH^{w-2}_{-5/2}} \lesssim& \Vert g-e \Vert_{\HH^2_{-1/2}} \Vert k \Vert_{\HH^{w-1}_{-3/2}} +\Vert k \Vert_{\HH^1_{-3/2}}  \Vert g-e \Vert_{\HH^w_{-1/2}} \\
&+ C_w \Vert g-e \Vert_{\HH^2_{-1/2}} \Vert k \Vert_{\HH^1_{-3/2}} \\
\lesssim& \Vert \bar g-e \Vert_{\HH^2(B_1)}  \Vert k \Vert_{\HH^{w-1}_{-3/2}} +  \Vert \bar k \Vert_{\HH^1(B_1)}\Vert g-e \Vert_{\HH^w_{-1/2}} \\
&+ C_w \Vert \bar g-e \Vert_{\HH^2_{-1/2}} \Vert \bar k \Vert_{\HH^1_{-3/2}}.
\end{align*}

Therefore, combining the above two estimates and using that $\varep>0$ is sufficiently small to absorb terms on the right-hand side, we have
\begin{align*}
\Vert g-e \Vert_{\HH^w_{-1/2}} + \Vert k \Vert_{\HH^{w-1}_{-3/2}} \leq C_w \Big( \Vert \bar g-e \Vert_{\HH^w(B_1)} + \Vert \bar k \Vert_{\HH^{w-1}(B_1)} \Big),
\end{align*}
where the constant $C_w>0$ depends only on $w$. This finishes the proof of Theorem \ref{MainTheorem12}.

\appendix

\section{The proof of Proposition \ref{prop:Completeness}} \label{sec:ProofCompleteness}

In this section we prove Proposition \ref{prop:Completeness}. First we show that 
\begin{align*}
\Big\{ E^{(lm)}, H^{(lm)} : l \geq 1, m \in \{ -l,\dots,l \} \Big\}
\end{align*}
is a complete orthonormal basis for $L^2$-integrable vectorfields on $(S_r,\gac)$, $r>0$. The orthonormality of $E^{(lm)}, H^{(lm)}$ defined in \eqref{def:Vectorsjan} follows from the orthonormality and completeness of the spherical harmonics $Y^{(lm)}$. Indeed, by \eqref{def:Vectorsjan},
for all index pairs $(lm), (l'm')$,
\begin{align*}
\sr E^{(lm)} \cdot E^{(l'm')} &= \frac{r^2}{l(l+1)} \sr (Y^{(lm)},0) \cdot \left( \cDd_1 \cDd_1^\ast \right)\left(Y^{(l'm')},0\right)\\
&=  \frac{r^2}{l(l+1)} \sr (Y^{(lm)},0) \cdot \left(- \Ld Y^{(l'm')},0\right)\\
&= \sr Y^{(lm)} Y^{(l'm')} \\
&= \de^{l l'} \de^{m m'},
\end{align*}
where we used Lemma \ref{lem:Chrlem1} and denoted the pointwise product of two pairs of functions $$( f_1, f_2) \cdot ( f_3, f_4):= (f_1 f_2, f_3 f_4).$$ The same holds for $H^{(lm)}$. Furthermore, for all index pairs $(lm), (l'm')$,
\begin{align*}
\sr E^{(lm)} \cdot H^{(l'm')} &= \sr \left(Y^{(lm)},0\right) \cdot \left( 0, Y^{(l'm')} \right) \\&= 0.
\end{align*}
This proves the orthonormality of the vectorfields $E^{(lm)}, H^{(lm)}$. \\

We now show that the vectorfields $E^{(lm)}, H^{(lm)}$ form a complete basis of vectorfields in $L^2(S_r)$ for every $r>0$. It suffices to show that for any vector $Z \in L^2(S_r)$,
\begin{align*}
\left( Z^{(lm)}_E=Z^{(lm)}_H = 0 \,\,\, \text{for all } l\geq1, m\in \{-l, \dots,l\}  \right) \Rightarrow Z = 0.
\end{align*}
By the identities of Lemma \ref{lem:RelationsSpherical}, for all $l\geq1, m\in \{-l, \dots,l\}$,
\begin{align*}
0 &= Z^{(lm)}_E = \sr Z \cdot E^{(lm)} = \frac{r}{\sqrt{l(l+1)}} \sr \left( \Divd Z \right)Y^{(lm)},\\
0 &= Z^{(lm)}_H = \sr Z \cdot H^{(lm)} = \frac{r}{\sqrt{l(l+1)}} \sr \left( \Curld Z \right)Y^{(lm)}.
\end{align*}
By the completeness of $Y^{(lm)}$, see Lemma \ref{sphericalharmonics}, this shows that $$\Divd Z = \Curld Z= 0.$$ By the ellipticity of this Hodge system, see Proposition \ref{prop:EllipticityHodgejan}, it follows that $Z=0$. This proves the completeness of the basis $E^{(lm)}, H^{(lm)}$, $l\geq1, m\in \{-l, \dots,l\}$.\\

Second, we show that 
\begin{align*}
\Big\{ \psi^{(lm)}, \phi^{(lm)}: l \geq 2, m \in \{ -l, \dots, l \} \Big\}
\end{align*}
is a complete orthonormal basis of tracefree symmetric $2$-tensors in $L^2(S_r)$, $r>0$. The orthonormality of the $\psi^{(lm)},\phi^{(lm)}$ defined in \eqref{def:2tensorsjan} is proved analogously to the orthonormality of $E^{(lm)}, H^{(lm)}$ and left to the reader.\\

To prove the completeness of the $\psi^{(lm)},\phi^{(lm)}$, we need to prove that for any tracefree symmetric $2$-tensor $V \in L^2(S_r)$, 
\begin{align*}
\left( V^{(lm)}_\psi=V^{(lm)}_\phi = 0 \,\,\, \text{for all } l\geq2, m\in \{-l, \dots,l\}  \right) \Rightarrow V = 0.
\end{align*}
This follows however by the completeness of the $E^{(lm)}, H^{(lm)}$ and Proposition \ref{prop:EllipticityHodgejan}, similar to the above proof for $E^{(lm)}, H^{(lm)}$. This proves the completeness of the basis $\psi^{(lm)}, \phi^{(lm)}$, $l\geq2, m\in \{-l, \dots,l\}$.\\

The equality of norms follows by the orthonormality and completeness properties. This finishes the proof of Proposition \ref{prop:Completeness}.

\section{The proofs of Proposition \ref{prop:Howtoestimatefunctions} and Lemma \ref{lem:commutationrelation} } \label{sec:appidentities}
In this section we prove Proposition \ref{prop:Howtoestimatefunctions} and Lemma \ref{lem:commutationrelation} .
\begin{proof}[Proof of Proposition \ref{prop:Howtoestimatefunctions}] 
Consider the first relation. We show at first that that 
\begin{align}
\Vert \Nd^w u \Vert_{L^2(S_r)}^2 &\lesssim \sumzero \left( \frac{l(l+1)}{r^2} \right)^w  \left( u^{(lm)}\right)^2 + C_w \sumzero \left( u^{(lm)}\right)^2, \label{ineq111}
\end{align}
where $C_w>0$ is a constant that depends on $w$, by induction in $w\geq0$. The cases $w=0,1$ are verified by Lemma \ref{lem:RelationsSpherical} and Propositions \ref{prop:EllipticityHodgejan} and \ref{prop:Completeness}. \\

For the induction step $w \to w+1$, integrate by parts to estimate
\begin{align} \begin{aligned}
&\Vert \Nd^{w+1} u \Vert_{L^2(S_r)}^2 \\
=& - \int\limits_{S_r} \Ld (\Nd^{w} u) \cdot \Nd^{w} u \\
=& -\int\limits_{S_r}  \Nd^{w} \Ld u \cdot \Nd^{w} u  + [\Ld, \Nd^{w}]u \cdot \Nd^{w} u \\
\leq& \int\limits_{S_r} \Nd^{w-1} \Ld u \cdot \Ld \Nd^{n-1} u + \Vert  [\Ld, \Nd^{w}]u \Vert_{L^2(S_r)} \Vert  \Nd^{w} u \Vert_{L^2(S_r)} \\
\lesssim& \Vert \Nd^{w-1} \Ld u \Vert_{L^2(S_r)} \Vert \Nd^{w+1} u \Vert_{L^2(S_r)} + \frac{C_w}{r^2} \sumzero \left( \frac{l(l+1)}{r^2} \right)^{w}  \left( u^{(lm)}\right)^2,
\end{aligned} \label{eq:10j}\end{align}
where we used that 
\begin{align*}
\Vert [\Ld, \Nd^{w}]u \Vert^2_{L^2(S_r)} &\lesssim \frac{1}{r^4} \sumzero \left( \frac{l(l+1)}{r^2} \right)^{w}  \left( u^{(lm)}\right)^2 +\frac{C_w}{r^{2(w+2)}} \sumzero \left( u^{(lm)}\right)^2, \\
\Vert \Nd^{w} u \Vert^2_{L^2(S_r)} &\lesssim \sumzero \left( \frac{l(l+1)}{r^2} \right)^w  \left( u^{(lm)}\right)^2+\frac{C_w}{r^{2w}} \sumzero \left( u^{(lm)}\right)^2, 
\end{align*}
by the fact that we work on the round sphere $(S_r, \gac)$ and the induction hypothesis.\\

It follows from \eqref{eq:10j} that 
\begin{align}\begin{aligned}
\Vert \Nd^{w+1} u \Vert^2_{L^2(S_r)} \lesssim& \Vert \Nd^{w-1}\Ld u \Vert^2_{L^2(S_r)} + \frac{C_w}{r^2} \sumzero \left( \frac{l(l+1)}{r^2} \right)^{w}  \left( u^{(lm)}\right)^2.
\end{aligned} \label{eq:27may}
\end{align}
To estimate the first term on the right-hand side, we use the induction assumption and Lemma \ref{lem:RelationsSpherical},
\begin{align*}
&\Vert \Nd^{w-1} \Ld u \Vert_{L^2(S_r)}^2  \\
\lesssim& \sumzero \left( \frac{l(l+1)}{r^2} \right)^{w-1}  \left( (\Ld u)^{(lm)}\right)^2 + \frac{C_w}{r^{2(w+1)}} \sumzero \left( (\Ld u)^{(lm)}\right)^2\\
\lesssim& \sumzero \left( \frac{l(l+1)}{r^2} \right)^{w-1}  \left( \frac{l(l+1)}{r^2} u^{(lm)}\right)^2 + \frac{C_w}{r^2} \sumzero \left( \frac{l(l+1)}{r^2} \right)^{w}  \left( u^{(lm)}\right)^2 \\
\lesssim& \sumzero \left( \frac{l(l+1)}{r^2} \right)^{w+1}  \left(u^{(lm)}\right)^2 + \frac{C_w}{r^2} \sumzero \left( \frac{l(l+1)}{r^2} \right)^{w}  \left( u^{(lm)}\right)^2.
\end{align*}
Plugging the above into \eqref{eq:27may} yields
\begin{align*}
\Vert \Nd^{w+1} u \Vert_{L^2(S_r)}^2 \lesssim \sumzero \left( \frac{l(l+1)}{r^2} \right)^{w+1}  \left(u^{(lm)}\right)^2 + \frac{C_w}{r^2} \sumzero \left( \frac{l(l+1)}{r^2} \right)^{w}  \left( u^{(lm)}\right)^2.
\end{align*}
Using the interpolation inequality
\begin{align*}
&\frac{C_w}{r^2} \sumzero \left( \frac{l(l+1)}{r^2} \right)^{w}  \left( u^{(lm)}\right)^2 \\
\lesssim&  \sumzero \left( \frac{l(l+1)}{r^2} \right)^{w+1}  \left( u^{(lm)}\right)^2 + \frac{C_w}{r^{2(w+1)}} \sumzero \left(u^{(lm)}\right)^2
\end{align*}
finishes the induction step and therefore the proof of \eqref{ineq111}. The other direction needed for the equivalence relation is proved similarly and left to the reader.\\

It remains to show the second equivalence relation for a vectorfield $X$. This follows as for scalar functions by induction on $w\geq0$, using this time the vectorfields $E^{(lm)}, H^{(lm)}$ with Remark \ref{remarkharm29} and Lemma \ref{lem:RelationsSpherical}. This finishes the proof of Proposition \ref{prop:Howtoestimatefunctions}. \end{proof}

\begin{proof}[Proof of Lemma \ref{lem:commutationrelation}]We prove each part separately. \\

{\bf Part (1).} The estimate \eqref{resscalar1} follows directly by Definition \ref{defnotationsep} and the fact that $Y^{(lm)} \sim r^{-1}$ due to its normalisation, see Section \ref{ssec:Hodgebasis}.\\

{\bf Part (2).}  On the one hand, it generally holds that in standard polar frame components, see \eqref{polarnormalframe}, for $A=1,2$,
\begin{align} \begin{aligned} 
\pr_r \left( X^A \right) &= N \left( X^A \right)\\
&= e(\nab_N X, e_A) - e(X, \nab_N e_A) \\
&= \left( \Nd_N X \right)^A - e(X,\nab_{e_A} N) - e(X, [N,e_A]) \\
&= \left( \Nd_N X \right)^A - X^B \left(-\Theta_{BA} \right) + \frac{1}{r} X^A\\
&= \left( \Nd_N X \right)^A,
\end{aligned}  \label{simid1} \end{align}
where we used that in the Euclidean case, for $A=1,2$,
$$ [N,e_A] = - \frac{1}{r} e_A$$
and, in standard polar frame components,
$$\Theta_{11}=\Theta_{22}= -\frac{1}{r}, \Theta_{12}= \Theta_{21}=0.$$
Consequently,
\begin{align*}
\pr_r \left( X_E^{(lm)} \right) &= \sr \pr_r \left( X^A \right) \left( E^{(lm)} \right)_A + \frac{1}{r} X_E^{(lm)}\\
&= \left( \Nd_N X \right)_E^{(lm)} + \frac{1}{r} X_E^{(lm)},
\end{align*}
where we used in the first equality that on $(S_r,\gac)$, $\sqrt{\mathrm{det}\gac} \sim r^2$ and that in polar frame components, for $A=1,2$, see \eqref{polarnormalframe}, 
\begin{align*}
\left( E^{(lm)} \right)_A = - e_A \left( Y^{(lm)} \right) \frac{r}{\sqrt{l(l+1)}} \sim \frac{1}{r}.
\end{align*}
Repeatedly applying $\Ndn$ then proves the first of \eqref{resscalar2}. The second of \eqref{resscalar2} is proved similarly. \\

On the other hand, it holds generally in standard polar frame components that
\begin{align*}
\Divd X &= e_1\left(X^1\right) + e_2 \left( X^2 \right)\\
&= \frac{1}{r}\pr_{\th^1} \left(X^1\right) + \frac{1}{ r \sin \th^1} \pr_{\th^2} \left( X^2 \right).
\end{align*}
This leads to
\begin{align*}
\pr_r \left( r \Divd X \right) &= \pr_{\th^1} \pr_r \left(X^1\right) + \frac{1}{ \sin \th^1} \pr_{\th^2} \pr_r \left( X^2 \right) \\
&= \pr_{\th^1} \left( \Ndn X \right)^1 +  \frac{1}{ \sin \th^1} \pr_{\th^2} \left( \Ndn X \right)^2 \\
&= r\Divd \Ndn X,
\end{align*}
where we used \eqref{simid1}. This finishes the proof of part (2) of Lemma \ref{lem:commutationrelation}. \\

{\bf Part (3).} The proof of part (3) is similar to part (2) and left to the interested reader. This finishes the proof of Lemma \ref{lem:commutationrelation}. \end{proof}


\section{Elliptic operators on weighted Sobolev spaces} \label{sec:ELLIPTICITYweighted} \label{sec:WEIGHTEDellipticity}

In this section, we first introduce the weak formulation of boundary value problems in weighted spaces. Second, we prove ellipticity and derive higher elliptic regularity estimates in $H^w_\de(\RRRwo) \cap \ol{H}^1_\de$ and $\ol{H}^w_\de$ for PDEs that were used in Sections \ref{sec:EuclideanSurjectivity} and \ref{sec:SurjectivityR}. Here the $w,\de$ depend on the PDE system under consideration. We also derive an elliptic estimate for distributional solutions with $L^2$-regularity to one of the PDEs.

\subsection{Weak formulation of boundary value problems in weighted spaces} \label{ssweakM10}

First, we define corresponding dual spaces.
\begin{definition}[Dual spaces of weighted Sobolev spaces]
Let $ \left( \ol{H}^w_\de \right)^\ast$ denote the space of linear maps $G: \ol{H}^w_\de \to \RRR$ such that there exists a constant $c>0$ so that
\begin{align*}
\vert G(u) \vert \leq c \Vert u \Vert_{\ol{H}^w_\de} \,\,\, \text{for all } u \in \ol{H}^w_\de. 
\end{align*}
Let the norm $\Vert G \Vert_{\left( \ol{H}^w_\de \right)^\ast}$ be defined as the smallest $c>0$ such that the above inequality holds.\end{definition}

The next lemma shows how weights behave with respect to the dual spaces.
\begin{lemma}\label{lem:simple1}
Let $w\geq0$, $v \in \ol{H}^{0}_{\de}$ and $\a \in \mathbb{N}^3$ a multi-index such that $\vert \a \vert=w$. Denote by $\pr^\a v$ the $\a$-th weak derivative of $v$. Then 
\begin{align*}
\pr^\a v \in \left( \ol{H}^w_{-\de+w-3} \right)^\ast.
\end{align*}
\end{lemma}

\begin{proof}
For $u \in \ol{H}^w_{-\de+w-3}$, it holds that
\begin{align*}
\int\limits_{\RRRwo} u \pr^\a v &= (-1)^{\vert \a \vert} \int\limits_{\RRRwo} \pr^\a u \, v\\
&\leq \Vert u \Vert_{\ol{H}^w_{-\de+w-3}} \Vert v \Vert_{\ol{H}^0_{\de}}.
\end{align*}
This concludes the proof of Lemma \ref{lem:simple1}.
\end{proof}

In Sections \ref{sec:EuclideanSurjectivity} and \ref{sec:SurjectivityR}, we consider PDEs of the form 
\begin{align}
\begin{cases}
\triangle u + \frac{a}{r} \pr_r u + \frac{b}{r^2} u = f \,\,\, \text{on } \RRRwo,\\
u \vert_{r=1} =0, 
\end{cases} \label{eq:mars8}
\end{align}
where $a,b \in \RRR$ are constants and $u \in \ol{H}^1_\de, f \in \left( \ol{H}^1_{-\de-1} \right)^\ast$.\\

Note that if $v\in \ol{H}^1_{\de}$, then $r^{-1-2\de} v \in \ol{H}^1_{-\de-1}$. Therefore, to apply the standard theory of generalised solutions, see for example \cite{GilbargTrudinger}, we consider \emph{weighted} weak formulations after a formal integration by parts of \begin{align*}
\int\limits_{\RRRwo} r^{-2\de-1} \left(-\triangle u - \frac{a}{r} \pr_r u - \frac{b}{r^2} u \right) v
\end{align*}
where $u, v \in \ol{H}^1_\de$. This leads to the following definition.

\begin{definition}[Weak solutions] 
Let $ f \in \left( \ol{H}^1_{-\de-1}\right)^\ast, a,b \in \RRR$ be given. A function $u \in \ol{H}^1_\de$ is called \emph{weak solution} to
\begin{align}
\begin{cases}
\triangle u + \frac{a}{r} \pr_r u + \frac{b}{r^2} u = f \,\,\, \text{on } \RRRwo,\\
u \vert_{r=1} =0, 
\end{cases}
\end{align}
if for all $v\in \ol{H}^1_{\de}$ it holds that
\begin{align*}
\BB_{\de,a,b}(u,v) = \int\limits_{\RRRwo} r^{-1-2\de } f v,
\end{align*}
where $\BB_{\de,a,b}(u,v): \ol{H}^1_\de \times \ol{H}^1_\de \to \RRR$ is the symmetric bilinear form defined by
\begin{align}\begin{aligned}
\BB_{\de,a,b}(u,v) := &\int\limits_{\RRRwo} r^{-2\de-1} \nab u \cdot \nab v - \frac{a+2\de+1}{2}  r^{-2\de-2} \left(v \pr_r u - u \pr_r v \right)\\
& -\int\limits_{\RRRwo} (\de(a+2\de+1)+b) r^{-2\de-3} uv. \end{aligned}\label{mars111}
\end{align}
\end{definition}
It is left to the reader to verify that $\BB_{\de,a,b}(u,v): \ol{H}^1_\de \times \ol{H}^1_\de \to \RRR$ is bounded for all $\de,a,b\in \RRR$, that is,
\begin{align*}
\vert \BB_{\de,a,b}(u,v) \vert \lesssim \Vert u \Vert_{\ol{H}^1_\de} \Vert v \Vert_{\ol{H}^1_\de} \,\,\, \text{for all } u,v \in \ol{H}^1_\de.
\end{align*}


Let $w\geq2$ be an integer. Introduce three PDEs on $\RRRwo$.
\begin{itemize}
\item Consider a scalar function $u^{[\geq2]}$ on $\RRRwo$ that verifies 
\begin{align} \begin{cases}
\triangle u^{[\geq2]} +\frac{4}{r} \pr_r u^{[\geq2]} +\frac{6}{r^2} u^{[\geq2]} = f^{[\geq2]} \,\,\, \text{on } \RRRwo, \\
u^{[\geq2]}\vert_{r=1} = 0,
\end{cases} \label{eq:weightedProblem} \tag{\textbf {E1}} \end{align}
where $f^{[\geq2]}$ is a given scalar function on $\RRRwo$.
\item Consider a scalar function $u^{[\geq2]}$ on $\RRRwo$ that verifies
\begin{align} \begin{cases}
\triangle u^{[\geq2]} + \frac{1}{r} \pr_r u^{[\geq2]} - \frac{3}{r^2} u^{[\geq2]}= f^{[\geq2]} \,\,\, \text{on } \RRRwo, \\
u^{[\geq2]} \vert_{r=1} = 0,
\end{cases} \label{eq:weightedProblem1} \tag{\textbf {E2}} \end{align}
where $f^{[\geq2]}$ is a given scalar function on $\RRRwo$.
\item Consider a scalar function $u^{[\geq1]}$ on $\RRRwo$ that verifies
\begin{align} \begin{cases}
\triangle u^{[\geq1]} -\half \Ld u^{[\geq1]} + \frac{1}{r} \pr_r u^{[\geq1]} + \frac{1}{r^2} u^{[\geq1]}= f^{[\geq1]} \,\,\, \text{on } \RRRwo, \\
u^{[\geq1]} \vert_{r=1} = 0,
\end{cases} \label{eq:weightedProblem2} \tag{\textbf {E3}} \end{align}
where $f^{[\geq1]}$ is a given scalar function on $\RRRwo$.
\end{itemize}
Notice that ({\bf E1}) corresponds to \eqref{eq:delta4}, ({\bf E2}) to \eqref{eq:sigma2} and ({\bf E3}) to \eqref{def:Dez20u1}.


\subsection{Elliptic estimates in $\ol{H}^1_\de$} \label{ellmay1}

The next proposition shows existence and first elliptic estimates for the above PDEs in weighted Sobolev spaces.
\begin{proposition}\label{prop:ellmars10} The following holds.
\begin{itemize}
\item Let $ f^{[\geq2]} \in  \left( \ol{H}^1_{1/2} \right)^\ast$. There exists a unique weak solution $u^{[\geq2]} \subset \ol{H}^1_{-3/2}$ to \eqref{eq:weightedProblem} bounded by
\begin{align}
\Vert u^{[\geq2]} \Vert_{\ol{H}^1_{-3/2}} \lesssim \Vert f^{[\geq2]} \Vert_{ \left( \ol{H}^1_{1/2} \right)^\ast} \label{est:mars101}
\end{align}
\item Let $ f^{[\geq2]} \in  \left( \ol{H}^1_{3/2} \right)^\ast$. There exists a unique weak solution $u^{[\geq2]} \subset \ol{H}^1_{-5/2}$ to \eqref{eq:weightedProblem1} bounded by
\begin{align}
\Vert u^{[\geq2]} \Vert_{\ol{H}^1_{-5/2}} \lesssim \Vert f^{[\geq2]} \Vert_{ \left( \ol{H}^1_{3/2} \right)^\ast} \label{est:mars1011}
\end{align}
\item Let $f^{[\geq1]} \in  \left( \ol{H}^{1}_{-1/2} \right)^\ast $. There exists a unique weak solution $u^{[\geq1]} \in \ol{H}^1_{-1/2}$ to \eqref{eq:weightedProblem2} bounded by
\begin{align}
\Vert u^{[\geq1]} \Vert_{\ol{H}^1_{-1/2}} \lesssim \Vert f^{[\geq1]} \Vert_{ \left( \ol{H}^{1}_{-1/2} \right)^\ast } \label{est:mars102}
\end{align}
\end{itemize}
\end{proposition}
To prove the above proposition, we use the following Poincar\'e inequality.
\begin{lemma} \label{lem:hardy1} Let $n\geq1$ be an integer. Let the scalar function $u^{[\geq n]} \in C^\infty(\RRR^3)$. For any $r>0$ it holds that 
\begin{align}
\int\limits_{S_r} \frac{\left(u^{[\geq n]} \right)^2}{r^2} \leq \frac{1}{n(n+1)} \int\limits_{S_r} \vert \Nd u^{[\geq n]} \vert^2. \label{eq:hardyestimate}
\end{align}
\end{lemma}
\begin{proof}
Indeed, write
\begin{align*}
\int\limits_{S_r} \frac{\left( u^{[\geq n]} \right)^2}{r^2} &= \sum\limits_{l\geq n}  \sum\limits_{m=-l}^l \frac{ \left(u^{(lm)} \right)^2}{r^2}, \\
&= \sum\limits_{l\geq n}  \sum\limits_{m=-l}^l \frac{1}{l(l+1)}  \frac{l(l+1)}{r^2} \left( u^{(lm)} \right)^2, \\
&\leq\sum\limits_{l\geq n}  \sum\limits_{m=-l}^l \frac{1}{n(n+1)}  \frac{l(l+1)}{r^2} \left( u^{(lm)} \right)^2, \\
&= \frac{1}{n(n+1)} \isphere \vert \Nd u^{[\geq n]} \vert^2,
\end{align*}
where we used Lemma \ref{lem:RelationsSpherical}. This proves Lemma \ref{lem:hardy1}. \end{proof}
\begin{proof}[Proof of Proposition \ref{prop:ellmars10}] We show for each PDE ({\bf E1})-({\bf E3}) that the corresponding bilinear form $\BB_{\de,a,b}$ defined in \eqref{mars111} is coercive on the respective weighted space. The $a,b$ corresponding to the PDEs are specified by comparing to \eqref{eq:mars8}. By the Lax-Milgram Theorem, see for example \cite{GilbargTrudinger}, existence, uniqueness and the claimed estimates follow.\\

{\bf Estimate \eqref{est:mars101}.} For ({\bf E1}), $a=4, b=6$. We derive the coercivity of $\BB_{-3/2,4,6}$ by Lemma \ref{lem:hardy1} with $n=2$ as follows,
\begin{align} \begin{aligned}
\BB_{-3/2,4,6}(u^{[\geq2]},u^{[\geq2]})&= \int\limits_{\RRRwo} r^2 \vert \nab u^{[\geq2]}\vert^2 -3\left(u^{[\geq2]} \right)^2 \\
&\geq \int\limits_{\RRRwo} r^2 \vert \nab u^{[\geq2]}\vert^2 - \frac{1}{2} r^2 \vert \Nd u^{[\geq2]} \vert^2, \\
& \geq \half  \int\limits_{\RRRwo} r^2 \vert \nab u^{[\geq2]}\vert^2\\
& \gtrsim \Vert u^{[\geq2]} \Vert_{\overline{H}^1_{-3/2}}^2, \end{aligned} \label{mars161}
\end{align}
where we used Lemma \ref{lem:hardy1} in the second and the last line. This proves \eqref{est:mars101}. \\


{\bf Estimate \eqref{est:mars1011}.}  For ({\bf E2}), $a=1,b=-3$. We estimate from below with Lemma \ref{lem:hardy1}
\begin{align} \begin{aligned}
\mathcal{B}_{-5/2,1,-3}(u^{[\geq2]},u^{[\geq2]})&= \int\limits_{\RRRwo} r^4 \vert \nab u^{[\geq2]}\vert^2 -\frac{9}{2} r^2 \left( u^{[\geq2]} \right)^2\\
&\geq \int\limits_{\RRRwo} r^4 \vert \nab u^{[\geq2]}\vert^2 - \frac{9}{12} r^4 \vert \Nd u^{[\geq2]}\vert^2\\
&\gtrsim \Vert u^{[\geq2]} \Vert_{\overline{H}^1_{-5/2}}^2.
\end{aligned} \label{mars162} \end{align}
This proves \eqref{est:mars1011}.\\

{\bf Estimate \eqref{est:mars102}.} The symmetric bilinear form $\tilde{\BB}$ associated to the weighted weak formulation of ({\bf E3}) in $\ol{H}^1_{-1/2}$ is in fact given by
\begin{align*}
\tilde{\BB}(u,v):= \int\limits_{\RRRwo} \nab u \nab v - \half \Nd u \Nd v - \frac{1}{2r} \left( v \pr_r u -u \pr_r v\right) - \frac{1}{2r^2} uv.
\end{align*}
Estimate this from below by Lemma \ref{lem:hardy1}
\begin{align} \begin{aligned}
\tilde{\BB}(u^{[\geq1]},u^{[\geq1]}) &=  \int\limits_{\RRRwo} \vert \nab u^{[\geq1]}  \vert^2 - \half \vert \Nd u^{[\geq1]}  \vert^2 - \frac{1}{2r^2} \left( u^{[\geq1]} \right)^2 \\
&\geq \int\limits_{\RRRwo} \vert \nab u^{[\geq1]}  \vert^2  - \half \vert \Nd u^{[\geq1]}  \vert^2  - \frac{1}{4} \vert \Nd u^{[\geq1]}  \vert^2 \\
&\gtrsim \Vert u^{[\geq1]} \Vert_{\overline{H}^1_{-1/2}}^2.
\end{aligned} \end{align}
This proves \eqref{est:mars102} and hence finishes the proof of Proposition \ref{prop:ellmars10}. \end{proof}

\subsection{Higher elliptic regularity in $H^w_{\de}(\RRRwo) \cap \ol{H}^1_{\de}$ and $\ol{H}^w_\de$} \label{highermay1}

In this section, we derive higher elliptic regularity estimates in $H^w_\de(\RRRwo) \cap \ol{H}^1_{\de}$, $w\geq2$ and $\ol{H}^w_\de$, for the boundary value problems \eqref{eq:weightedProblem}-\eqref{eq:weightedProblem2}, on the domain $\RRRwo$. \\

The next proposition is a generalisation of standard higher elliptic regularity estimates to weighted Sobolev spaces.
\begin{proposition}[Higher regularity in $H^w_\de(\RRRwo) \cap \ol{H}^1_{\de}$] \label{higherregmars14}
Let $w\geq2$ be an integer. The following holds.
\begin{itemize}
\item Let $ f^{[\geq2]} \in H^{w-2}_{-7/2}(\RRRwo)$. Then the solution $u^{[\geq2]}$ to \eqref{eq:weightedProblem} satisfies $$u^{[\geq2]} \in H^{w}_{-3/2}(\RRRwo) \cap \ol{H}^1_{-3/2}$$ and
\begin{align}
\Vert u^{[\geq2]} \Vert_{H^{w}_{-3/2}(\RRRwo)} \lesssim \Vert f^{[\geq2]} \Vert_{H^{w-2}_{-7/2}(\RRRwo)} + C_w \Vert f^{[\geq2]} \Vert_{H^{0}_{-7/2}(\RRRwo)}. \label{est:mars141}
\end{align}
\item Let $ f^{[\geq2]} \in H^{w-2}_{-9/2}(\RRRwo)$. Then the solution $u^{[\geq2]}$ to \eqref{eq:weightedProblem1} satisfies $$u^{[\geq2]} \in H^{w}_{-5/2}(\RRRwo)\cap \ol{H}^1_{-5/2}$$ and
\begin{align}
\Vert u^{[\geq2]} \Vert_{H^{w}_{-5/2}(\RRRwo)} \lesssim \Vert f^{[\geq2]} \Vert_{H^{w-2}_{-9/2}(\RRRwo)} + C_w \Vert f^{[\geq2]} \Vert_{H^{0}_{-9/2}(\RRRwo)}. \label{est:mars1411}
\end{align}
\item Let $f^{[\geq1]} \in H^{w-2}_{-5/2}(\RRRwo)$. Then the solution $u^{[\geq1]}$ to \eqref{eq:weightedProblem2} satisfies $$u^{[\geq1]} \in H^{w}_{-1/2}(\RRRwo) \cap \ol{H}^1_{-1/2}$$ and
\begin{align}
\Vert u^{[\geq1]} \Vert_{H^{w}_{-1/2}(\RRRwo)} \lesssim \Vert f^{[\geq1]} \Vert_{ H^{w-2}_{-5/2}(\RRRwo) }+ C_w\Vert f^{[\geq1]} \Vert_{ H^{0}_{-5/2}(\RRRwo) }. \label{est:mars142}
\end{align}
\end{itemize}
\end{proposition}

\begin{proof}[Proof of Proposition \ref{higherregmars14}] The standard interior and global elliptic regularity estimates (see for example Theorems 8.8 and 8.13 in \cite{GilbargTrudinger}) can be generalised from bounded smooth domains to weighted Sobolev spaces, see for example \cite{Bartnik} \cite{BruhatChr} \cite{Maxwell} and \cite{MaxwellDissertation}. For example, in case of \eqref{est:mars141}, we have for integers $w\geq2$,
\begin{align*}
\Vert u^{[\geq2]} \Vert_{H^{w}_{-3/2}(\RRRwo)} \lesssim& \left\Vert \triangle u^{[\geq2]} +\frac{4}{r} \pr_r u^{[\geq2]} +\frac{6}{r^2} u^{[\geq2]} \right\Vert_{H^{w-2}_{-7/2}(\RRRwo)} + C_w \Vert u^{[\geq2]} \Vert_{H^{1}_{-3/2}(\RRRwo)} \\
\lesssim& \left\Vert f^{[\geq2]} \right\Vert_{H^{w-2}_{-7/2}(\RRRwo)} + C_w \Vert u^{[\geq2]} \Vert_{H^{1}_{-3/2}(\RRRwo)}.
\end{align*}
From this, the estimate \eqref{est:mars141} follows by using the ellipticity of $({\bf E1})$ proved in Proposition \ref{prop:ellmars10}. The estimates \eqref{est:mars1411} and \eqref{est:mars142} are derived similarly. \end{proof}


Furthermore, we have the following result in $\ol{H}^w_\de$.
\begin{proposition}[Higher regularity for ({\bf E1}), ({\bf E2}),({\bf E3})] \label{higherregmars14444} \label{corollary:realsystems}
Let $w\geq2$ be an integer. The following holds.
\begin{itemize}
\item Let $ f^{[\geq2]} \in \ol{H}^{w-2}_{-7/2}$. If the solution $u^{[\geq2]} \in \ol{H}^{1}_{-3/2} \cap H^2_{-3/2}(\RRRwo)$ to \eqref{eq:weightedProblem} satisfies
\begin{align*}
\pr_r u^{[\geq2]} \vert_{r=1} =0,
\end{align*}
then it holds that $u^{[\geq2]} \in \ol{H}^{w}_{-3/2}$.
\item Let $ f^{[\geq2]} \in \ol{H}^{w-2}_{-9/2}$. If the solution $u^{[\geq2]} \in \ol{H}^{1}_{-5/2}\cap H^2_{-5/2}(\RRRwo)$ to \eqref{eq:weightedProblem1} satisfies
\begin{align*}
\pr_r u^{[\geq2]} \vert_{r=1} =0,
\end{align*}
then it holds that $u^{[\geq2]} \in \ol{H}^{w}_{-5/2}$.
\item Let $w\geq2$ be an integer. Let $f^{[\geq1]} \in \ol{H}^{w-2}_{-5/2}$. If the solution $u^{[\geq1]} \in \ol{H}^{1}_{-1/2} \cap H^2_{-1/2}(\RRRwo)$ to \eqref{eq:weightedProblem2} satisfies
\begin{align*}
\pr_r u^{[\geq1]} \vert_{r=1} =0,
\end{align*}
then it holds that $u^{[\geq1]} \in \ol{H}^{w}_{-1/2}$.
\end{itemize}
\end{proposition}

\begin{proof} Proposition \ref{higherregmars14444} follows by Propositions \ref{higherregmars14} and \ref{prop:TrivialExtensionRegularity}. Indeed, all necessary normal derivatives on $r=1$ can be expressed via the equations and shown to vanish.
\end{proof}

\subsection{An elliptic estimate in $L^2$}

In Section \ref{sssec:EuclideanSurjectivity}, we considered the following Dirichlet problem on $\RRRwo$ for a scalar function $u^{[\geq2]} \in H^{w-2}_{-5/2}(\RRRwo)$,
\begin{align}\begin{cases}
\triangle u^{[\geq2]} + \frac{1}{r} \pr_r u^{[\geq2]} - \frac{3}{r^2}u^{[\geq2]} = \pr_r \left(  \Curld \left( f^{[\geq2]}_H \right) \right), \\
u^{[\geq2]} \vert_{r=1} = 0,
\end{cases} \label{eq:sigmaROUGH}
\end{align}
where $f^{[\geq2]}_H \in \HH^{w-2}_{-5/2}(\RRRwo)$, $w\geq2$ was a given vectorfield. In the previous sections \ref{ellmay1} and \ref{highermay1} where this PDE was denoted ({\bf E2}), we derived elliptic estimates in case $w\geq3$. In this section we derive estimates for the case $w=2$. \\ 

First, we derive a distributional formulation of \eqref{eq:sigmaROUGH}. Let 
\begin{itemize}
\item $f^{[\geq2]}_H \in C^\infty(\RRR^3)$, 
\item $u \in C^\infty(\RRR^3)$ be a solution to \eqref{eq:sigmaROUGH},
\item $\phi \in C_c^\infty(\RRR^3)$ such that $\phi \vert_{r=1}=0$.
\end{itemize}
Then, by integrating by parts twice, 
\begin{align*}
0=&\int\limits_{\RRRwo} r^4 \Big( \triangle u^{[\geq2]} + \frac{1}{r} \pr_r u^{[\geq2]} - \frac{3}{r^2}u^{[\geq2]} -  \pr_r \left(  \Curld \left( f^{[\geq2]}_H \right) \right) \Big) \phi \\
=& \int\limits_{\RRRwo} r^4 u^{[\geq2]} \left( \triangle \phi^{[\geq2]} + \frac{7}{r} \pr_r \phi^{[\geq2]} + \frac{12}{r^2} \phi^{[\geq2]} \right) \\
&- \int\limits_{\RRRwo} r^4 f_H^{[\geq2]} \cdot \left( \frac{6}{r}  {}^\ast \mkern-5mu \left(\Nd \phi^{[\geq2]}\right) + {}^\ast \mkern-5mu \left(\Nd (\pr_r \phi^{[\geq2]}) \right)\right),
\end{align*}
where here ${}^\ast \mkern-5mu \left(\Nd \phi \right)_A := \iin_{AB} \left(\Nd \phi \right)^B$ denotes the Hodge dual of $\Nd \phi$. Here the boundary terms 
\begin{align*}
\int\limits_{S_1} \pr_r u \, \phi, \, \int\limits_{S_1} u \pr_r \phi, \, \int\limits_{S_1} u\phi, \, \int\limits_{S_1} \Curld f_H^{[\geq2]} \phi
\end{align*}
vanished by the assumptions. The right-hand side still makes sense for 
\begin{align*}
f_H^{[\geq2]} &\in \HH^0_{-5/2}(\RRRwo),\\
u^{[\geq2]} &\in H^0_{-5/2}(\RRRwo), \\
\phi &\in H^2_{-5/2}(\RRRwo) \cap \ol{H}^1_{-5/2}. 
\end{align*}
Note the dense inclusion 
\begin{align*}
\{ \phi \in C_c^{\infty}(\RRR^3) : \phi \vert_{r=1} =0 \} \,\, \subset \,\, H^2_{-5/2}(\RRRwo) \cap \ol{H}^1_{-5/2} \cap C^\infty(\RRRwo),
\end{align*}
which is proved by using cut-off functions and left to the reader. This leads to the following definition.
\begin{definition} \label{def:distributionalsol}
Let $f^{[\geq2]}_H \in \HH^{w-2}_{-5/2}(\RRRwo)$ be a vectorfield. A function $u^{[\geq2]} \in H^0_{-5/2}(\RRRwo)$ is called a \emph{distributional solution} to 
\begin{align*}\begin{cases}
\triangle u^{[\geq2]} + \frac{1}{r} \pr_r u^{[\geq2]} - \frac{3}{r^2}u^{[\geq2]} = \pr_r \left(  \Curld \left( f^{[\geq2]}_H \right) \right), \\
u^{[\geq2]} \vert_{r=1} = 0,
\end{cases} 
\end{align*}
if for all $\phi^{[\geq2]} \in H^{2}_{-5/2}(\RRRwo) \cap \ol{H}^1_{-5/2} \cap C^\infty(\RRRwo)$,
\begin{align} \begin{aligned}
&\int\limits_{\RRRwo} r^4 u^{[\geq2]} \left( -\triangle \phi^{[\geq2]} - \frac{7}{r} \pr_r \phi^{[\geq2]} - \frac{12}{r^2} \phi^{[\geq2]} \right) \\
=& - \int\limits_{\RRRwo} r^4 f_H^{[\geq2]} \cdot \left( \frac{6}{r} {}^\ast \mkern-5mu \left(\Nd \phi^{[\geq2]}\right) + {}^\ast \mkern-5mu \left(\Nd (\pr_r \phi^{[\geq2]}) \right)\right).
\end{aligned}
\label{eq:def13}
\end{align}
\end{definition}
Note that this distributional solution is unique in $H^0_{-5/2}(\RRRwo)$ in view of Lemma \ref{lowreg16mars} below.\\

The next lemma is the main result of this section. 
\begin{lemma} \label{lowreg16mars}
Let $f^{[\geq2]}_H \in \HH^0_{-5/2}(\RRRwo)$. Let $u^{[\geq2]} \in H^0_{-5/2}(\RRRwo)$ be a distributional solution to \eqref{eq:sigmaROUGH}. Then it holds that 
\begin{align*}
\Vert u^{[\geq2]} \Vert_{H^0_{-5/2}} \lesssim \Vert f_H^{[\geq2]} \Vert_{\HH^0_{-5/2}}.
\end{align*}
\end{lemma}

\begin{proof}[Proof of Lemma \ref{lowreg16mars}] To prove that $u^{[\geq2]} \in H^0_{-5/2} = \ol{H}^0_{-5/2} = \left( \ol{H}^0_{-1/2} \right)^\ast$ with
$$\Vert u^{[\geq2]} \Vert_{H^0_{-5/2}} \lesssim \Vert f_H^{[\geq2]}\Vert_{\HH^0_{-5/2} },$$
it suffices to show that for all $\varphi^{[\geq2]} \in C^\infty_c(\RRRwo)$,
\begin{align*}
\int\limits_{\RRRwo} u^{[\geq2]} \varphi^{[\geq2]} \lesssim \Vert f_H^{[\geq2]} \Vert_{\HH^0_{-5/2}} \Vert \varphi^{[\geq2]} \Vert_{H^0_{-1/2}}.
\end{align*}
In the following, we will prove that for all $\varphi^{[\geq2]} \in C^\infty_c(\RRRwo)$,
\begin{align}
\int\limits_{\RRRwo} r^4 u^{[\geq2]} \varphi^{[\geq2]} \lesssim \Vert f^{[\geq2]}_H \Vert_{\HH^0_{-5/2}} \Vert \varphi^{[\geq2]} \Vert_{H^0_{-9/2}}, \label{eq:roughestimate}
\end{align}
which implies the above estimate by the fact that  
\begin{align*}
\Vert r^{-4} f^{[\geq2]}_H \Vert_{\HH^0_{-9/2}(\RRRwo)} \simeq \Vert f^{[\geq2]}_H \Vert_{\HH^0_{-1/2}(\RRRwo)}
\end{align*}

It remains to prove \eqref{eq:roughestimate}. For given $\varphi^{[\geq2]} \in C^\infty_c(\RRRwo)$, let $\Psi^{[\geq2]}$ be defined as solution to 
\begin{align} \begin{cases}
-\triangle \Psi^{[\geq2]} - \frac{7}{r} \pr_r \Psi^{[\geq2]} - \frac{12}{r} \Psi^{[\geq2]} = \varphi^{[\geq2]} \,\,\, \text{on } \RRRwo, \\
\Psi^{[\geq2]} \vert_{r=1}=0.
\end{cases}  \label{eq:label111} \end{align}
The operator in \eqref{eq:label111} is the adjoint to ({\bf E2}) with respect to weighted scalar product $$(u,v) \mapsto \int\limits_{\RRRwo} r^4 uv.$$ Therefore \eqref{eq:label111} has the same weak formulation, ellipticity and higher regularity estimates for its generalised solutions as ({\bf E2}). By Proposition \ref{higherregmars14} and standard local interior and boundary elliptic regularity,
$$\Psi \in H^2_{-5/2}(\RRRwo) \cap \ol{H}^1_{-5/2} \cap C^\infty(\RRRwo)$$
with 
\begin{align}
\Vert \Psi^{[\geq2]} \Vert_{H^2_{-5/2}(\RRRwo)} \lesssim \Vert \varphi^{[\geq2]} \Vert_{H^0_{-9/2}}. \label{est:7jjj}
\end{align}
Plugging now $\Psi$ into \eqref{eq:def13}, using \eqref{est:7jjj} and \eqref{eq:label111}, we get
\begin{align*}
\int\limits_{\RRRwo} r^4 u^{[\geq2]} \varphi^{[\geq2]} &= \int\limits_{\RRRwo} r^4 u^{[\geq2]} \left(-\triangle \Psi^{[\geq2]} - \frac{7}{r} \pr_r \Psi^{[\geq2]} - \frac{12}{r} \Psi^{[\geq2]} \right) \\
&= - \int\limits_{\RRRwo} r^4 f_H^{[\geq2]} \cdot \left( \frac{6}{r} {}^\ast \mkern-5mu \left(\Nd \Psi^{[\geq2]}\right) + {}^\ast \mkern-5mu \left(\Nd (\pr_r \Psi^{[\geq2]}) \right) \right) \\
&\lesssim  \Vert f_H^{[\geq2]} \Vert_{\HH^0_{-5/2}(\RRRwo)} \Vert \Psi^{[\geq2]} \Vert_{H^2_{-5/2}(\RRRwo)} \\
&\lesssim \Vert f_H^{[\geq2]} \Vert_{\HH^0_{-5/2}(\RRRwo)} \Vert \varphi^{[\geq2]} \Vert_{H^0_{-9/2}(\RRRwo)}.
\end{align*}
This proves \eqref{eq:roughestimate} and finishes the proof of Lemma \ref{lowreg16mars}. \end{proof}

\begin{remark} The existence of a solution $u^{[\geq2]} \in H^0_{-5/2}$ to $({\bf E2})$ can be deduced from Lemma \ref{lowreg16mars} by a limit argument. Indeed, it suffices to take $\left(f^{[\geq2]}_H\right)_n \in C^\infty_c(\RRRwo),  n\in \mathbb{N}$ a sequence such that $\left(f^{[\geq2]}_H\right)_n \to f^{[\geq2]}_H$ in $\HH^0_{-5/2}$ as $n\to \infty$. The corresponding solutions $(u^{[\geq2]} )_n \in \ol{H}^1_{-5/2}$ whose existence is assured by Proposition \ref{prop:ellmars10} will by Lemma \ref{lowreg16mars} converge to the distributional solution $u^{[\geq2]}$ in $H^0_{-5/2}$.
\end{remark}

\subsection{Estimates to apply Lemma \ref{lem:simple1}}

To apply the above elliptic theory in Section \ref{sec:EuclideanSurjectivity}, 
the following corollary is used. It follows by the operator analysis in Section \ref{sec:DivergenceEquation} and its proof is left to the reader.
\begin{corollary} \label{corprac8j}
The right-hand sides of the PDEs \eqref{eq:delta4} and \eqref{eq:sigma2} can be estimated as follows. We have
\begin{align*}
\left\Vert \frac{1}{r^3} \pr_r \left( r^3 \left( \rh_N \right)^{[\geq2]} \right) - \Divd \left( \rhod_E^{[\geq2]} + \zeta_E \right) 
\right\Vert_{\left( \ol{H}^1_{1/2} \right)^\ast} &\lesssim \Vert \rh \Vert_{\ol{\HH}^0_{-5/2}} + \Vert \zeta_E\Vert_{\ol{\HH}^0_{-5/2}}, \\
\left\Vert \rhod_H^{[\geq2]} + \zeta_H \right\Vert_{ \HH^0_{-5/2}(\RRRwo)} &\lesssim \Vert \rh \Vert_{\HH^{0}_{-5/2}(\RRRwo)} +\Vert \zeta_H \Vert_{\HH^{0}_{-5/2}(\RRRwo)} \\
\left\Vert \pr_r \Curld \left( \rhod_H^{[\geq2]} + \zeta_H  \right) \right\Vert_{\left( \ol{H}^1_{3/2}(\RRRwo) \right)^\ast} &\lesssim \Vert \rh \Vert_{\ol{\HH}^{1}_{-5/2}(\RRRwo)} +\Vert \zeta_H \Vert_{\ol{\HH}^{1}_{-5/2}(\RRRwo)}.
\end{align*}
Also for $w\geq3$, 
\begin{align*}
\left\Vert \frac{1}{r^3} \pr_r  \left( r^3 \left( \rh_N \right)^{[\geq2]} \right) - \Divd \left( \rhod_E^{[\geq2]} + \zeta_E \right) \right\Vert_{\ol{H}^{w-3}_{-7/2} } \lesssim& \Vert \rh \Vert_{\ol{\HH}^{w-2}_{-5/2}} + \Vert \zeta_E \Vert_{\ol{\HH}^{w-2}_{-5/2}} \\
&+ C_w\Big( \Vert \rh \Vert_{\ol{\HH}^{0}_{-5/2}} + \Vert \zeta_E \Vert_{\ol{\HH}^{0}_{-5/2}}\Big),
\end{align*}
and for $w\geq4$,
\begin{align*}
\left\Vert \pr_r \left( \Curld \left( \rhod_H^{[\geq2]} + \zeta_H \right) \right) \right\Vert_{\ol{H}^{w-4}_{-9/2}} \lesssim& \Vert \rh \Vert_{\ol{\HH}^{w-2}_{-5/2}}+\Vert \zeta_H\Vert_{\ol{\HH}^{w-2}_{-5/2}} \\
&+ C_w \Big( \Vert \rh \Vert_{\ol{\HH}^{0}_{-5/2}}+\Vert \zeta_H\Vert_{\ol{\HH}^{0}_{-5/2}} \Big).
\end{align*}
\end{corollary}

\end{document}